\documentclass[11pt]{article}

\usepackage{amsmath}
\usepackage{amsfonts}
\usepackage{amsthm}
\usepackage{ifthen}
\usepackage{mathrsfs}
\usepackage{mathtools}
\usepackage[shortlabels]{enumitem}
\usepackage{tikz}
\usepackage{tikz-3dplot}
\usetikzlibrary{calc}
\usepackage{MnSymbol} 
\usepackage{verbatim}
\usepackage{hyperref}
\usepackage{scrextend}
\usepackage{caption}

\newtheorem{theorem}{Theorem}[section]
\newtheorem{lemma}[theorem]{Lemma}
\newtheorem{proposition}[theorem]{Proposition}
\newtheorem{corollary}[theorem]{Corollary}

\newtheorem*{conjecture*}{Conjecture}

\theoremstyle{definition}
\newtheorem{remark}[theorem]{Remark}
\newtheorem{definition}[theorem]{Definition}

\newcommand{\diffII}[3]{\ifthenelse{\equal{#2}{#3}}
{\frac{d^2 #1}{d #2^2}}
{\frac{d^2 #1}{d #2 d #3}}
}
\newcommand{\diffIIat}[4]{\left.\diffII{#1}{#2}{#3}\right|_{#4}}

\newcommand{\diffat}[3]{\left.\frac{d #1}{d #2}\right|_{#3}}
\newcommand{\pdiff}[2]{\frac{\partial #1}{\partial #2}}
\newcommand{\pdiffat}[3]{\left.\frac{\partial #1}{\partial #2}\right|_{#3}}

\newcommand{\rank}{\mathop{\mathrm{rank}}}
\newcommand{\vol}{\mathrm{vol}}
\newcommand{\inr}[2]{\langle #1, #2 \rangle}
\newcommand{\norm}[1]{\left\Vert#1\right\Vert}
\newcommand{\snorm}[1]{\Vert#1\Vert}
\newcommand{\abs}[1]{\left\vert#1\right\vert}
\newcommand{\tabs}[1]{\widetilde{\abs{#1}}}
\newcommand{\tbra}[1]{\widetilde{[#1]}}
\newcommand{\set}[1]{\left\{#1\right\}}
\newcommand{\brac}[1]{\left(#1\right)}
\newcommand{\rbrac}[1]{\left[#1\right]}
\newcommand{\scalar}[1]{\left \langle #1 \right \rangle}
\newcommand{\eps}{\epsilon}
\newcommand{\R}{\mathbb{R}}
\newcommand{\B}{\mathcal{B}}
\newcommand{\E}{\mathbb{E}}
\newcommand{\Real}{\mathbb{R}}

\newcommand{\n}{\mathbf{n}}
\newcommand{\tang}{\mathbf{t}}
\newcommand{\calH}{\mathcal{H}}

\newcommand{\cyclic}{\mathcal{C}}

\newcommand{\per}{P}
\newcommand{\Id}{\mathrm{Id}}
\newcommand{\II}{\mathrm{I\!I}}
\newcommand{\simplex}{\Delta}
\renewcommand{\div}{\mathrm{div}}

\newcommand{\navg}{\overline{\mathbf{n}}}
\newcommand{\pot}{W}

\newcommand{\Y}{\mathbf{Y}}
\newcommand{\T}{\mathbf{T}}
\newcommand{\Tr}{\mathcal{T}}
\newcommand{\M}{\mathbf{M}}
\renewcommand{\H}{\mathcal{H}}
\renewcommand{\S}{\mathcal{S}}
\newcommand{\C}{\mathcal{C}}
\newcommand{\Eps}{\mathcal{E}}
\newcommand{\F}{\mathcal{F}}
\newcommand{\modelV}{\Psi}
\newcommand{\Sphere}{\mathbb{S}}

\newcommand{\G}{\mathbb{G}}

\DeclareMathOperator{\interior}{int}

\DeclareMathOperator{\tr}{tr}
\DeclareMathOperator{\supp}{supp}
\DeclareMathOperator{\relint}{relint}
\DeclareMathOperator{\spn}{span}

\DeclareMathOperator{\Ker}{Ker}
\DeclareMathOperator{\Img}{Im}

\numberwithin{equation}{section}

\begin{document}

\title{The Gaussian Double-Bubble and Multi-Bubble Conjectures}

\author{Emanuel Milman\textsuperscript{1} and Joe Neeman\textsuperscript{2}}

\date{}

\maketitle

\abstract{We establish the Gaussian Multi-Bubble Conjecture: 
the least Gaussian-weighted perimeter way to decompose $\R^n$ into $q$ cells of prescribed (positive) Gaussian measure when $2 \leq q \leq n+1$, is to use a ``simplicial cluster", obtained from the Voronoi cells of $q$ equidistant points. Moreover, we prove that simplicial clusters are the unique isoperimetric minimizers (up to null-sets).  
In particular, the case $q=3$ confirms the Gaussian Double-Bubble Conjecture: the unique least Gaussian-weighted perimeter way to decompose $\R^n$ ($n \geq 2$) into three cells of prescribed (positive) Gaussian measure is to use a tripod-cluster, whose interfaces consist of three half-hyperplanes meeting along an $(n-2)$-dimensional plane at $120^{\circ}$ angles (forming a tripod or ``Y" shape in the plane). The case $q=2$ recovers the classical Gaussian isoperimetric inequality. 

To establish the Multi-Bubble conjecture, we show that in the above range of $q$, stable regular clusters must have flat interfaces, therefore consisting of convex polyhedral cells (with at most $q-1$ facets). In the Double-Bubble case $q=3$, it is possible to avoid establishing flatness of the interfaces by invoking a certain dichotomy on the structure of stable clusters, yielding a simplified argument. 
}

\footnotetext[1]{Department of Mathematics, Technion - Israel
Institute of Technology, Haifa 32000, Israel. Email: emilman@tx.technion.ac.il.}

\footnotetext[2]{Department of Mathematics, University of Texas at Austin. Email: joeneeman@gmail.com. \\
The research leading to these results is part of a project that has received funding from the European Research Council (ERC) under the European Union's Horizon 2020 research and innovation programme (grant agreement No 637851). This material is also based upon work supported by the National Science Foundation under Grant No. 1440140, while the authors were in residence at the Mathematical Sciences Research Institute in Berkeley, California, during the Fall semester of 2017.
}

\begingroup    \renewcommand{\thefootnote}{}    
    \endgroup

\section{Introduction}

Let $\gamma = \gamma^n$ denote the standard Gaussian probability measure on Euclidean space $(\R^n,\abs{\cdot})$: \[
 \gamma^n := \frac{1}{(2\pi)^{\frac{n}{2}}} e^{-\frac{|x|^2}{2}} \, dx =: e^{-\pot(x)}\, dx.
\]
More generally, if $\H^{k}$ denotes the $k$-dimensional Hausdorff measure, let $\gamma^k$ denote its Gaussian-weighted counterpart: 
\[
\gamma^k := e^{-\pot(x)} \H^{k} .
\]
The Gaussian-weighted (Euclidean) perimeter of a Borel set $U \subset \R^n$ is defined as:
\[ \per_\gamma(U) := \sup \left\{\int_U (\div X - \inr{\nabla \pot}{X}) \, d\gamma: X \in C_c^\infty(\R^n; T \R^n), |X| \le 1 \right\}.
\] For nice sets (e.g. open sets with piecewise smooth boundary), $\per_\gamma(U)$ is known to agree with $\gamma^{n-1}(\partial U)$ (see e.g.~\cite{MaggiBook}). The weighted perimeter $\per_\gamma(U)$ has the advantage of being lower semi-continuous with respect to $L^1(\gamma)$ convergence, and thus fits well with the direct method of calculus-of-variations.

The classical Gaussian isoperimetric inequality, established independently by Sudakov--Tsirelson \cite{SudakovTsirelson} and Borell \cite{Borell-GaussianIsoperimetry} in 1975, asserts that among all Borel sets $U$ in $\R^n$ having prescribed Gaussian measure $\gamma(U) = v \in [0,1]$, halfspaces minimize Gaussian-weighted perimeter $\per_\gamma(U)$ (see also \cite{EhrhardPhiConcavity, Latala-EhrhardForOneConvexSet,BakryLedoux, BobkovGaussianIsopInqViaCube, BartheMaureyIsoperimetricInqs,BobkovLocalizedProofOfGaussianIso,Borell-EhrhardForBorelSets,MorganManifoldsWithDensity}).
Later on, it was shown by Carlen--Kerce \cite{CarlenKerceEqualityInGaussianIsop} (see also \cite{EhrhardGaussianIsopEqualityCases,MorganManifoldsWithDensity, McGonagleRoss:15}), that up to null-sets, halfspaces are in fact the \emph{unique} minimizers for the Gaussian isoperimetric inequality. 
These results are the Gaussian analogue of the classical \emph{unweighted} isoperimetric inequality in Euclidean space $(\R^n,\abs{\cdot})$, which states that Euclidean balls uniquely minimize (up to null-sets) unweighted perimeter among all sets of prescribed Lebesgue measure \cite{BuragoZalgallerBook,MaggiBook}.

\medskip

In this work, we extend these classical results for the Gaussian measure to the case of clusters. A \emph{$q$-cluster} $\Omega = (\Omega_1, \ldots, \Omega_q)$ is a $q$-tuple of Borel subsets $\Omega_i \subset \R^n$ called cells, such that $\set{\Omega_i}$ are pairwise disjoint, $\per_\gamma(\Omega_i) <
\infty$ for each $i$, and  $\gamma(\R^n \setminus \bigcup_{i=1}^q \Omega_i) = 0$. Note that the cells are not required to be connected. The total Gaussian
perimeter of a cluster $\Omega$ is defined as:
\[
  \per_\gamma(\Omega) := \frac 12 \sum_{i=1}^q \per_\gamma(\Omega_i) .
\]
The Gaussian measure of a cluster is defined as:
\[
  \gamma(\Omega) := (\gamma(\Omega_1), \ldots, \gamma(\Omega_q)) \in \simplex^{(q-1)} ,
\]
where $\simplex^{(q-1)} := \{v \in \R^q: v_i \ge 0 ~,~ \sum_{i=1}^q v_i = 1\}$ denotes the $(q-1)$-dimensional probability simplex. The isoperimetric problem for $q$-clusters consists of identifying those clusters $\Omega$ of prescribed Gaussian measure $\gamma(\Omega) = v \in \simplex^{(q-1)}$ which minimize the total Gaussian perimeter $\per_\gamma(\Omega)$. 

\smallskip
Note that easy properties of the perimeter ensure that for a $2$-cluster $\Omega$, $\per_\gamma(\Omega) = \per_\gamma(\Omega_1) = \per_\gamma(\Omega_2)$, and so the case $q=2$ corresponds to the classical isoperimetric setup when testing the perimeter of a single set of prescribed measure. In analogy to the classical 
unweighted isoperimetric inequality in $\R^n$,
we will refer to the case $q=2$ as the ``single-bubble" case (with the bubble being $\Omega_1$ and having complement $\Omega_2$). Accordingly, the case $q=3$ is called the ``double-bubble" problem, and the case of general $q$ is referred to as the ``multi-bubble" problem.  See below for further motivation behind this terminology and related results.

\smallskip
A natural conjecture is then the following:
\begin{conjecture*}[Gaussian Multi-Bubble Conjecture]
For all $2 \leq q \leq n+1$, the least Gaussian-weighted perimeter way to decompose $\R^n$ into $q$ cells of prescribed Gaussian measure $v \in \interior \simplex^{(q-1)}$ is by using ``simplicial clusters", obtained as the Voronoi cells of $q$ equidistant points in $\R^n$. \end{conjecture*}
\noindent 
Recall that the Voronoi cells of $\set{x_1,\ldots,x_q} \subset \R^n$ are defined as:
\[
\Omega_i = \interior \set{ x \in \R^n : \min_{j=1,\ldots,q} \abs{x-x_j} = \abs{x-x_i} } \;\;\; i=1,\ldots, q ~ ,
\]
where $\interior$ denotes the interior operation (to obtain pairwise disjoint cells).

\begin{figure}
\begin{center}
\begin{tikzpicture}
    \coordinate(O) at (0, 0);   
    \coordinate(A) at ({-1.5 * sqrt(3)}, 1.5);
    \coordinate(B) at ({1.5 * sqrt(3)}, 1.5);
    \coordinate(C) at (0, -2.6);   
    
    \draw[thick] (O) -- (A); 
    \draw[thick] (O) -- (B); 
    \draw[thick] (O) -- (C); 
\end{tikzpicture}
\tdplotsetmaincoords{10}{20} 
\begin{tikzpicture}[tdplot_main_coords]
    \coordinate (O) at (0, 0, 0);  
    \coordinate (X1) at (3, 0, 0);    
    \coordinate (X2) at (-1, {2*sqrt(2)}, 0);
    \coordinate (X3) at (-1, {-sqrt(2)}, {sqrt(6)});
    \coordinate (X4) at (-1, {-sqrt(2)}, {-sqrt(6)});

        \draw[thick] (O) -- (X1);    
    \draw[thick] (O) -- (X2);    
    \draw[thin,dashed] (O) -- (X3);    
    \draw[thick] (O) -- (X4);    

        \foreach \n/\m in {1/2, 1/4, 2/3, 2/4, 3/4} {
            \foreach \t in {0,5,...,85} {
                \draw[thick] (${sin(\t)}*(X\n)+{cos(\t)}*(X\m)$) -- (${sin(\t+5)}*(X\n)+{cos(\t+5)}*(X\m)$);
            }
    }
    \foreach \t in {0,5,...,85} {  
        \draw[thin,dashed] (${sin(\t)}*(X1)+{cos(\t)}*(X3)$) -- (${sin(\t+5)}*(X1)+{cos(\t+5)}*(X3)$);
    }

\end{tikzpicture}
\captionsetup{width=0.7\linewidth} 
\caption{Left, a simplicial $3$-cluster in $\R^2$; right, a simplicial $4$-cluster in $\R^3$ (intersected with a Euclidean ball).}              
\end{center}
\end{figure}

\smallskip

Indeed, when $q=2$, the cells of a simplicial cluster are precisely halfspaces, and the single-bubble conjecture holds by the classical Gaussian isoperimetric inequality.  
The simplest case beyond $q=2$ is the double-bubble case ($q=3$ and $v \in \interior \simplex^{(2)}$), when the interfaces of a simplicial cluster's cells are three half-hyperplanes meeting along an $(n-2)$-dimensional plane at $120^\circ$ angles (forming a tripod or ``Y" shape in the plane).
Simplicial clusters are the naturally conjectured minimizers in the range $3 \leq q \leq n+1$ as their interfaces have constant Gaussian mean-curvature (being flat), and meet at $120^{\circ}$ angles in threes, both of which are necessary conditions for any extremizer of Gaussian perimeter under a Gaussian measure constraint -- see Section \ref{sec:minimizers}.
Equal measure simplicial clusters also naturally appear in other related problems (see e.g. \cite{HeilmanThesis,HeilmanJagannathNaor-PropellerInR3,IsakssonMossel-GaussianNoiseStability} and the references therein).

\smallskip
Our main results in this work are the following:

\begin{theorem}[Gaussian Multi-Bubble Theorem] \label{thm:main1} \label{thm:main-I-I_m}
The Gaussian Multi-Bubble Conjecture holds true for all $2 \leq q \leq n+1$. 
\end{theorem}

\begin{theorem}[Uniqueness of Minimizing Clusters] \label{thm:main-uniqueness}
Up to null-sets, simplicial $q$-clusters are the \emph{unique} minimizers of Gaussian perimeter among all $q$-clusters of prescribed Gaussian measure $v \in \interior \Delta^{(q-1)}$,
for all $2 \leq q \leq n+1$.
\end{theorem}

\subsection{Previously Known and Related Results}

The Gaussian Multi-Bubble Conjecture is known to experts. Presumably, its origins may be traced to an analogous problem of J.~Sullivan from 1995 in the \emph{unweighted} Euclidean setting \cite[Problem 2]{OpenProblemsInSoapBubbles96}, where the conjectured uniquely minimizing
$q$-cluster (up to null-sets) for $q \leq n+2$ is a standard $(q-1)$-bubble -- spherical caps bounding connected cells $\set{\Omega_i}_{i=1}^{q}$ which are obtained by taking the Voronoi cells of $q$ equidistant points in $\mathbb{S}^{n} \subset \R^{n+1}$, and applying all stereographic projections to $\R^n$. That the standard $(q-1)$-bubble in unweighted Euclidean space exists and is unique (up to isometries) for all volumes was proved by Montesinos-Amilibia \cite{MontesinosStandardBubbleE!}. 
 Similarly, on the (unweighted) canonical $n$-sphere $\mathbb{S}^n$, where spherical caps are known to uniquely minimize perimeter (up to null-sets) \cite{BuragoZalgallerBook}, the Spherical Multi-Bubble Conjecture asserts that the uniquely perimeter minimizing $q$-cluster (up to null-sets) for $q \leq n+2$ is a standard spherical $(q-1)$-bubble -- a stereographic projection of a standard $(q-1)$-bubble in $\R^n$. See also Schechtman \cite{Schechtman-ApproxGaussianMultiBubble} for a formulation of the equal volumes case of the Gaussian Multi-Bubble Conjecture.

\smallskip

In the double-bubble case $q=3$, various prior results have been established in a variety of settings. To put our results into appropriate context, let us first go over some related double-bubble results in the three settings mentioned above: the unweighted Euclidean setting $\R^n$, the canonical sphere $\Sphere^n$ (normalized for convenience to have total mass one), 
and our Gaussian-weighted Euclidean setting $\G^n$. Further results may be found in F.~Morgan's excellent book \cite[Chapters 13,14,18,19]{MorganBook5Ed}.

\begin{itemize}
\item 
Long believed to be true, but appearing explicitly as a conjecture in an undergraduate thesis by J.~Foisy in 1991 \cite{Foisy-UGThesis}, 
the Euclidean double-bubble problem in $\R^n$ was considered in the 1990's by various authors \cite{SMALL93, HHS95}, culminating in the work of Hutchings--Morgan--Ritor\'e--Ros \cite{DoubleBubbleInR3}, who proved that up to null-sets, the standard double-bubble is uniquely perimeter minimizing in $\R^3$; this was later extended to $\R^n$ in \cite{SMALL03,Reichardt-DoubleBubbleInRn}. \item  
The double-bubble conjecture was resolved on $\Sphere^2$ by Masters \cite{Masters-DoubleBubbleInS2}, but on $\Sphere^n$ for $n\geq 3$ only partial results are known \cite{CottonFreeman-DoubleBubbleInSandH, CorneliHoffmanEtAl-DoubleBubbleIn3D,CorneliCorwinEtAl-DoubleBubbleInSandG}.  
In particular, we mention a result by Corneli-et-al \cite{CorneliCorwinEtAl-DoubleBubbleInSandG}, which confirms the double-bubble conjecture on $\Sphere^n$ for all $n \geq 3$ when the prescribed measure $v \in \simplex^{(2)}$ satisfies $\max_{i} \abs{v_i - 1/3} \leq 0.04$. Their proof employs a result of Cotton--Freeman \cite{CottonFreeman-DoubleBubbleInSandH} stating that if the minimizing cluster's cells are known to be connected, then it must be the standard double-bubble. 
\item In the Gaussian setting $\G^n$ which we consider in this work, the original proofs of the single-bubble case made use of 
the classical fact that the projection onto a fixed $n$-dimensional subspace of the uniform measure on a rescaled sphere $\sqrt{N} \Sphere^N$, converges to the Gaussian measure $\gamma^n$ as $N \rightarrow \infty$. Building upon this idea, it was shown by Corneli-et-al \cite{CorneliCorwinEtAl-DoubleBubbleInSandG} that verification of the double-bubble conjecture on $\Sphere^N$ for a sequence of $N$'s tending to $\infty$ and a fixed $v \in \interior \simplex^{(2)}$ will verify the double-bubble conjecture on $\G^n$ for the same $v$ and for all $n \geq 2$. As a consequence, they confirmed the double-bubble conjecture on $\G^n$ for all $n \geq 2$ when the prescribed measure $v \in \simplex^{(2)}$ satisfies $\max_{i} \abs{v_i - 1/3} \leq 0.04$. Note that this approximation argument precludes any attempts to establish uniqueness in the double-bubble conjecture, and to the best of our knowledge, no uniqueness in the double-bubble conjecture on $\G^n$ was known for any $v \in \interior \simplex^{(2)}$ prior to our Theorem \ref{thm:main-uniqueness}.
\end{itemize} 

\smallskip

In the triple-bubble case $q=4$, Wichiramala proved in \cite{Wichiramala-TripleBubbleInR2} that the standard triple-bubble is uniquely perimeter minimizing in the unweighted Euclidean plane $\R^2$. 
\emph{To the best of our knowledge, no other (provably) multi-bubble isoperimetric minimizers have been determined when $q > 3$ prior to this work}, not in any of the above settings, nor in any other. 

\medskip

It is therefore not surprising that we do not employ in this work the traditional ingredients used in the above mentioned results, such a B.~White's symmetrization argument (see \cite{Foisy-UGThesis,Hutchings-StructureOfDoubleBubbles}), or Hutchings' theory \cite{Hutchings-StructureOfDoubleBubbles} of bounds on the number of connected components comprising each cell of a minimizing cluster. Contrary to previous approaches, we do not identify the minimizing clusters by systematically ruling out competitors, and in fact we do not characterize them at all in our proof of Theorem \ref{thm:main1} (the characterization is obtained only later in Theorem \ref{thm:main-uniqueness}). We do not know how to use other tools from the Gaussian single-bubble setting such as tensorization \cite{BobkovGaussianIsopInqViaCube}, semi-group methods \cite{BakryLedoux}, martingale methods \cite{BartheMaureyIsoperimetricInqs} or parabolic techniques \cite{Borell-EhrhardForBorelSets}, even after a possible reduction to the case $n=q-1$ via methods such as localization (cf. \cite{BobkovLocalizedProofOfGaussianIso}). We do not even know how to \emph{directly} obtain a sharp lower bound on the perimeter of a cluster having prescribed measure $v \in \simplex^{(q-1)}$. Our approach, which appears new in this context, is based on obtaining a matrix-valued partial-differential inequality (MPDI) for the associated isoperimetric profile, concluding the desired lower bound in one fell swoop simultaneously for \emph{all} $v \in \simplex^{(q-1)}$ by an application of the maximum-principle.

\subsection{Matrix-valued Partial Differential Inequality}

Let $I^{(q-1)} : \simplex^{(q-1)} \to \R_+$ denote the Gaussian isoperimetric profile for $q$-clusters, defined as:
\[
  I^{(q-1)}(v) := \inf\{\per_\gamma(\Omega): \text{$\Omega$ is a $q$-cluster with $\gamma(\Omega) = v$}\}.
\]
Our goal will be to show that $I^{(q-1)} = I^{(q-1)}_m$ on $\interior \simplex^{(q-1)}$, where $I^{(q-1)}_m : \interior \simplex^{(q-1)} \to \R_+$ denote the Gaussian multi-bubble \emph{model} profile:
\[
  I^{(q-1)}_m(v) := \per_\gamma(\Omega^m) \text{ where $\Omega^m$ is a simplicial $q$-cluster with $\gamma(\Omega^m) = v$} 
\]
(see Section \ref{sec:model} for why this is well-defined). 

To establish that $I^{(q-1)} = I^{(q-1)}_m$, we draw motivation from the single-bubble case. By identifying $\simplex^{(1)}$ with $[0,1]$ using the map $\simplex^{(1)} \ni (v_1,v_2) \mapsto v_1 \in [0,1]$, we will think of the single-bubble profile $I^{(1)}$ as defined on $[0,1]$. The single-bubble Gaussian isoperimetric inequality asserts that $I^{(1)} = I^{(1)}_m$, where $I^{(1)}_m : [0,1] \rightarrow \Real_+$ denotes the single-bubble \emph{model} profile: 
\[
I^{(1)}_m(v) := P_\gamma(U) \text{ where $U$ is a halfplane with $\gamma(U) = v$} . 
\]
The product structure of the Gaussian measure implies that $I^{(1)}_m$ may be calculated in dimension one, readily yielding $I^{(1)}_m(v) = \varphi \circ \Phi^{-1}(v)$,
where $\varphi(x) = (2\pi)^{-1/2} e^{-x^2/2}$ is the one-dimensional Gaussian density and $\Phi(x) = \int_{-\infty}^x \varphi(y)\, dy$. It is well-known and immediate to check that $I^{(1)}_m$ satisfies the following ordinary differential equation:
\begin{equation} \label{eq:intro-ODE}
(I^{(1)}_m)^{\prime\prime} = -\frac{1}{I^{(1)}_m} \text{ on $[0,1]$.} 
\end{equation}

Our starting observation is that $I^{(q-1)}_m$ satisfies a similar \emph{matrix-valued partial} differential equation on $\simplex^{(q-1)}$. Let $E = E^{(q-1)}$ denote the tangent space to $\simplex^{(q-1)}$, which we identify with $\{x \in \R^q: \sum_{i=1}^q x_i = 0\}$. Given $A = \{A_{ij}\}_{1 \leq i < j \leq q}$ where $A_{ij} = \gamma^{n-1}(\partial^* \Omega_i \cap \partial^* \Omega_j) \geq 0$ denote the weighted areas of the interfaces of a general $q$-cluster $\Omega$ (and $\partial^* U$ denotes the reduced boundary of a Borel set $U$ having finite perimeter -- see Section \ref{sec:prelim}), we consider the following $q \times q$ positive semi-definite matrix:
\begin{equation} \label{eq:intro-LA}
L_{A} := \sum_{1 \leq i < j \leq q} A_{ij} (e_i - e_j) (e_i-e_j)^T .
\end{equation}
In fact, as a quadratic form on $E$, it is easy to show that $L_{A}$ is strictly positive-definite whenever $\gamma(\Omega) \in \interior \simplex^{(q-1)}$ (see Lemma \ref{lem:L-nondegenerate}). 
Given $v \in \interior \simplex^{(q-1)}$, let $A^m_{ij}(v) = \gamma^{n-1}(\partial^* \Omega^m_i \cap \partial^* \Omega^m_j)   > 0$ denote the weighted areas of the interfaces of a  model simplicial $q$-cluster $\Omega^m$ satisfying $\gamma(\Omega^m) = v$. A calculation then verifies that:
\begin{equation} \label{eq:intro-MDE}
\nabla^2 I^{(q-1)}_m(v) = -L_{A^m(v)}^{-1} \text{ on $\interior \simplex^{(q-1)}$,}
\end{equation}
where differentiation and inversion are both carried out on $E$. This constitutes the right matrix-valued extension of (\ref{eq:intro-ODE}) to the multi-bubble setting. Taking inverses and contracting, we obtain the PDE analogue of (\ref{eq:intro-ODE}): 
\begin{equation} \label{eq:intro-PDE-m}
- \tr[ (\nabla^2 I^{(q-1)}_m(v))^{-1} ] = \tr(L_{A^m(v)}) = 2 \sum_{i<j} A^m_{ij}(v) = 2 I^{(q-1)}_m(v)  \text{ on $\interior \simplex^{(q-1)}$.}
\end{equation}
Note that the factor of $2$ which is present in (\ref{eq:intro-PDE-m}) but not in (\ref{eq:intro-ODE}) is explained by the fact that the map identifying $\simplex^{(1)}$ with $[0, 1]$ contracts distances by a factor of $\sqrt 2$.

To establish that $I^{(q-1)} = I^{(q-1)}_m$ on $\simplex^{(q-1)}$, our idea is as follows. First, it is not too hard to show that $I^{(q-1)}_m$ may be extended continuously to the entire $\simplex^{(q-1)}$ by setting $I^{(q-1)}_m(v) := I^{(k)}_m(v_J)$ for $v \in J$, where $J$ is a $k$-dimensional face of $\partial \simplex^{(q-1)}$, and $v_J$ is the natural restriction of $v$ to the coordinates defined by $J$. Our goal is to show that $I^{(q-1)} = I^{(q-1)}_m$ on $\interior \simplex^{(q-1)}$; we clearly have $I^{(q-1)} \leq I^{(q-1)}_m$ on $\simplex^{(q-1)}$, and we may assume that equality occurs on the boundary by induction on $q$.

Assume for the sake of this sketch that $I^{(q-1)}$ is twice continuously differentiable on $\interior \simplex^{(q-1)}$. Given an isoperimetric minimizing cluster $\Omega$ with $\gamma(\Omega) = v \in \interior \simplex^{(q-1)}$, let $A_{ij}(v) := \gamma^{n-1}(\partial^*\Omega_i \cap \partial^* \Omega_j)$ denote the weighted areas of the cluster's interfaces. Assume again for simplicity that $A_{ij}(v)$ are well-defined, that is, depend only on $v$. We will then show that the following matrix-valued partial differential inequality (MPDI) holds:
\begin{equation} \label{eq:intro-MDI}
\nabla^2 I^{(q-1)}(v) \leq -L_{A(v)}^{-1} \text{ on $\interior \simplex^{(q-1)}$,}
\end{equation}
as quadratic forms on $E^{(q-1)}$. Consequently:
\begin{equation} \label{eq:intro-PDE}
- \tr[ (\nabla^2 I^{(q-1)}(v))^{-1} ] \leq \tr(L_{A(v)}) = 2 \sum_{i<j} A_{ij}(v) = 2 I^{(q-1)}(v)  \text{ on $\interior \simplex^{(q-1)}$.}
\end{equation}
On the other hand, by (\ref{eq:intro-PDE-m}), we have equality above when $I^{(q-1)}$ and $A$ are replaced by $I^{(q-1)}_m$ and $A^m$, respectively. Since $I^{(q-1)} \leq I^{(q-1)}_m$ on $\simplex^{(q-1)}$ with equality on the boundary, an application of the maximum principle for the (fully non-linear) second-order elliptic PDE (\ref{eq:intro-PDE}) will yield the desired $I^{(q-1)} = I^{(q-1)}_m$.

\subsection{Variations along Vector-Fields -- challenges which arise}

The bulk of this work is thus aimed at establishing a rigorous (approximate) version of (\ref{eq:intro-MDI}). To this end, we consider an isoperimetric minimizing cluster $\Omega$, and perturb it using a flow $F_t$ along an (admissible) vector-field $X$. Since:
\begin{equation} \label{eq:intro-variation}
 I^{(q-1)}(\gamma(F_t(\Omega))) \le \per_\gamma(F_t(\Omega)) 
\end{equation}
with equality at $t=0$, we deduce (at least, conceptually) that the first variations must coincide and that the second variations must obey the inequality; this idea is well-known in the single-bubble setting, see e.g. \cite{BavardPansu,SternbergZumbrun,MorganJohnson,Kuwert, RosIsoperimetryInCrystals, BayleIsoperimetricODE, BayleRosales, KleinerProofOfCartanHadamardIn3D, BayleThesis}. 

\smallskip
Unfortunately, there is no straightforward adaptation of this approach to the double-bubble, and more generally, multi-bubble settings, for several reasons: 

\begin{itemize}
\item The desired differential inequality (\ref{eq:intro-MDI}) is a MPDI, not an ODI, and so necessarily a $(q-1)$-dimensional \emph{family} of vector-fields $X$ must be constructed and used, not just a single one as in the single-bubble setting. Moreover, in order to get the sharp inequality (\ref{eq:intro-MDI}), we need to use vector-fields which preserve the family of model simplicial clusters to second order (as the inequality becomes an equality in this case). The most natural family of vector-fields with this property is the family of constant fields (generating translations), but unfortunately this is not sufficient. To explain why, let us say that a cluster is \emph{effectively $k$-dimensional} if it is of the form $\tilde \Omega \times F^{\perp}$, where $\tilde \Omega$ is a $q$-cluster in a $k$-dimensional subspace $F$. It is clear that translations of such a cluster will only contribute $k$ degrees of information, which is not enough for deriving (\ref{eq:intro-MDI}) if the effective-dimension $k$ is strictly smaller than $q-1$ -- a possibility we cannot rule out beforehand. Consequently, we are required to use more complicated vector-fields $X$ and derive formulas for the second variation along them.  

\item Contrary to the single-bubble setting, a cluster's cells will meet not only in pairs along an $(n-1)$-dimensional boundary $\Sigma^1$ (consisting of the interfaces $\Sigma_{ij} = \partial^*\Omega_i \cap \partial^* \Omega_j$),
but also in threes along an $(n-2)$-dimensional boundary $\Sigma^2$ in the double (and higher) bubble setting, and in fours along an $(n-3)$-dimensional boundary $\Sigma^3$ in the triple (and higher) bubble setting -- let us informally call these meeting points the triple and quadruple points, respectively. Thanks to the isoperimetric stationarity of our cluster, the contributions of triple-points to the first variation of perimeter cancel out, and it is easy to check that the latter will only depend on $\Sigma^1$. However, for the second variation, the triple-points $\Sigma^2$ will in general have a contribution, involving moreover the three interface curvatures at the triple-point; we will return to this point below. 

\item Unfortunately, it turns out that the \emph{regularity} of the boundary around the \emph{quadruple-points} $\Sigma^3$ also plays an indirect yet crucial role in the \emph{justification of the formulas} for the second variation one is able to \emph{rigorously} derive. To illustrate this point, consider the following heuristic derivation of the formula for the second variation of perimeter of magnitude $\varphi : \Sigma^1 \rightarrow \R$ in the normal direction (for simplicity, in the unweighted setting and ignoring the contribution of $\Sigma^2$): it is well-known that the first normal variation of perimeter is $\int_{\Sigma^1} H \varphi \; d\H^{n-1}$, where $H$ denotes the interface's mean-curvature; taking another normal variation, and using that the normal variation of $H$ is the Jacobi operator $-\Delta_{\Sigma^1} - \norm{\II}^2$ (where $\Delta_{\Sigma^1}$ and $\II$ are the interface Laplacian and second fundamental form, respectively), one might be tempted to immediately conclude that the second normal variation of perimeter  is $\int_{\Sigma^1} (- \varphi \Delta_{\Sigma^1} \varphi - \norm{\II}^2 \varphi^2) \; d\H^{n-1}$. However, to rigorously justify differentiating inside the integral, at the very least one needs to \emph{a-priori} know that the latter expression, which involves the curvature, is integrable. 
This subtle difficulty is not apparent in the single nor double-bubble settings, where there are no quadruple-points; Geometric Measure Theory (GMT) and elliptic regularity theory guarantee that $\Sigma^1$ is smooth all the way up to and including $\Sigma^2$, and so it follows that the curvature is guaranteed to be locally bounded, and one can easily justify differentiating inside the integral. However, the present (just recently established) state-of-the-art GMT results \cite{CES-RegularityOfMinimalSurfacesNearCones} only guarantee that the boundary near quadruple-points is $C^{1,\alpha}$ diffeomorphic  to a neighborhood of a quadruple-point on a model simplicial cluster, and so the curvature may in principle be blowing up near these points. Note that it is not possible to circumvent around this issue by using cutoff functions to truncate a neighborhood of $\Sigma^3$, since in general $\H^{n-3}(\Sigma^3) > 0$, and thus a second variation of $\Sigma^1$ will feel this truncation and its contribution cannot be made arbitrary small. Consequently, we are forced to \emph{a-priori} justify the $L^2$ integrability of the curvature on $\Sigma^1$ (and in fact also its $L^1$ integrability on $\Sigma^2$), as well as the various integration by parts we perform on $\Sigma^1$.

\item In the single-bubble setting, the vector-field $X$ used  in (\ref{eq:intro-variation}) is typically the unit-normal field to the interface $\Sigma_{12}$ (perhaps after attenuation by some cutoff functions to keep away from the singular parts of the topological interface, which are known to have sufficiently small Hausdorff dimension to be negligible). However, in the multi-bubble setting, the interfaces will meet in threes in $120^\circ$ angles at the triple-points, and so it is not possible to construct a vector-field $X$ which is normal to all interfaces and in addition continuous, to obtain a well-defined flow $F_t$ 
(using cutoff functions to keep away from the triple-points fails in this instance as they are $(n-2)$-dimensional and thus contribute to the second variation). 
Hence, contrary to the single-bubble setting, the non-trivial tangential component of $X$ and the three interface curvatures at the triple-points necessarily come into play and must be somehow accounted for. As we a-priori have no control over the curvatures (and their sign can be arbitrary), in order to obtain a useful differential inequality, we are forced to \emph{first} establish that the curvatures vanish, i.e. that \emph{the interfaces are in fact flat}. 

\item As a testament of these complications, we remark that our approach \emph{cannot} directly handle a more general probability measure $\mu = e^{-W} dx$ with $\nabla^2 W \geq \Id$ as in the single-bubble setting (see \cite{BakryLedoux,CaffarelliContraction,BobkovLocalizedProofOfGaussianIso}). However, after having established Theorem \ref{thm:main1} for the Gaussian measure, we can \emph{a-posteriori} treat such more general measures using Caffarelli's Contraction theorem \cite{CaffarelliContraction}, see Remark \ref{rem:Q-sign} and Theorem \ref{thm:CaffarelliCor}.
\end{itemize}

\subsection{Ingredients of Proof}

We now detail some of the ingredients which we employ to establish (\ref{eq:intro-MDI}):

\begin{itemize}
\item Classical results from Geometric Measure Theory due to Almgren \cite{AlmgrenMemoirs} (see also Maggi's excellent book \cite{MaggiBook}) ensure the existence of isoperimetric minimizing clusters and $C^\infty$ regularity of their interfaces $\Sigma^1 = \bigcup_{i<j} \Sigma_{ij}$. We will in addition require regularity and structural information on
the codimension-$1$ and $2$ boundary (as a manifold) of $\Sigma^1$, namely the triple-point set $\Sigma^2$ and quadruple-point set $\Sigma^3$, respectively, as well as information on the negligibility of the residual topological boundary. Such results have been obtained by various authors:  J.~Taylor when $n=2,3$ \cite{Taylor-SoapBubbleRegularityInR3} (see also F.~Morgan when $n=2$ \cite{MorganSoapBubblesInR2}), and B.~White \cite{White-SoapBubbleRegularityInRn} and very recently M.~Colombo, N.~Edelen and L.~Spolaor \cite{CES-RegularityOfMinimalSurfacesNearCones} when $n \geq 4$. Together with elliptic regularity results of Kinderlehrer, Nirenberg and Spruck \cite{KNS}, these results imply that $\Sigma^2$ consists of  $(n-2)$-dimensional $C^\infty$ manifolds where three of the cells locally meet (like the cells of a simplicial $3$-cluster), $\Sigma^3$ consists of $(n-3)$-dimensional $C^{1,\alpha}$ manifolds where four of the cells locally meet (like the cells of a simplicial $4$-cluster), and the residual topological boundary $\Sigma^4$ satisfies $\H^{n-3}(\Sigma^4)=0$. Being slightly inaccurate in this introductory section, we call such clusters ``regular". 
\item For any stationary regular cluster with respect to a measure with smooth positive density on $\R^n$, we establish new integrability properties of the curvature of $\Sigma^1$, which may a-priori be blowing up near $\Sigma^3$. We show in Appendix \ref{sec:Schauder} that the curvature is in $L^2(\Sigma^1 \cap K) \cap L^1(\Sigma^2 \cap K)$ for any compact $K$ disjoint from $\Sigma^4$. This is achieved using Schauder estimates for elliptic systems of PDEs, following the reflection technique of Kinderlehrer--Nirenberg--Spruck \cite{KNS}. As explained in the previous subsection, this information is completely crucial for us, both for rigorously establishing the formula for second variation of perimeter in the multi-bubble setting, as well as for rigorously constructing our approximate inward fields, described below. 
\item To justify the various integration by parts we need to employ on each interface $\Sigma_{ij}$, we establish a version of Stokes' theorem on the manifold-with-boundary $M_{ij} := \Sigma_{ij} \cup \partial \Sigma_{ij}$ for vector-fields which are non-compactly supported and which may be blowing up near $\overline{M_{ij}} \setminus M_{ij}$ (here and below, $\partial \Sigma_{ij}$ denotes the codimension-$1$ manifold boundary of $\Sigma_{ij}$). 
The fact that the latter set has locally-finite $\H^{n-3}$-measure enables us to handle vector-fields in  $L^2(\Sigma_{ij},\gamma^{n-1}) \cap L^1(\partial \Sigma_{ij},\gamma^{n-2})$  whose divergence is in $L^1(\Sigma_{ij},\gamma^{n-1})$, and the above-mentioned integrability properties of the curvature ensure that the relevant vector-fields indeed belong to this class. We thus finally obtain the following formula for the second variation of weighted area under a volume constraint:
\begin{equation} \label{eq:intro-formula}
        \sum_{i < j} \Big[\int_{\Sigma_{ij}} \brac{ |\nabla^\tang X^{\n_{ij}}|^2 - (X^{\n_{ij}})^2 \|\II^{ij} \|_2^2 - (X^{\n_{ij}})^2  } d\gamma^{n-1} - \int_{\partial \Sigma_{ij}} X^{\n_{ij}} X^{\n_{\partial ij}} \II^{ij}_{\partial,\partial} \, d\gamma^{n-2}\Big]
 \end{equation}
  (see Theorem \ref{thm:formula} for the precise assumptions on the cluster $\Omega$, the vector-field $X$, and the notation used above). As already explained in the previous subsection, while this formula may be considered as standard in the single and double bubble settings, this is not the case in the multi-bubble case when $q > 3$, and we require the entire contents of Section \ref{sec:second-var} and Appendices \ref{sec:calculation} and \ref{sec:Schauder} to rigorously justify it here for the first time. 
\item Our $(q-1)$-parameter family of vector-fields is constructed as follows. For each of the $q$ cells $\Omega_i$, we would like to construct an ``inward field": a vector-field $X_i$ on $\R^n$ so that $X_i$'s normal component $X_i^{\n}$ is constant $1$ on $\Sigma^1 \cap \partial \Omega_i$  with respect to the inward normal to $\Omega_i$, and constant $0$ on $\Sigma^1 \setminus \partial \Omega_i$. The family is obtained by taking linear combinations of the $X_i$'s, and one degree of freedom is obviously lost because $\sum_{i=1}^q X^{\n}_i \equiv 0$ on $\Sigma^1$. The idea behind this particular choice of vector-fields is that if $X = \sum_{i=1}^q a_i X_i$, $F_t$ is the flow along $X$, and $\n_{ij}$ denotes the unit-normal to $\Sigma_{ij}$ pointing from $\Omega_i$ to $\Omega_j$, then: 
\[
\delta_X V(\Omega) := \left .  \frac{d}{dt} \right |_{t=0} \gamma(F_t(\Omega)) = \brac{\sum_{j \neq i} \int_{\Sigma_{ij}} X^{\n_{ij}} d\gamma^{n-1}}_i = (\sum_{j \neq i} A_{ij} (a_j - a_i))_i = - L_A a ,
\]
which already shows that the family is always $(q-1)$-dimensional (as $L_A$ is positive definite and hence full-rank on $E^{(q-1)}$). In addition, when the \emph{interfaces have vanishing curvature} (such as for a model simplicial cluster), the expression in (\ref{eq:intro-formula}) for the second variation of weighted area under a volume constraint reduces to:
\begin{equation} \label{eq:intro-formula-simplified}
\sum_{i < j} \int_{\Sigma_{ij}} \brac{|\nabla^\tang X^{\n_{ij}}|^2 - (X^{\n_{ij}})^2} d\gamma^{n-1} = - \sum_{i < j} A_{ij} (a_j - a_i)^2 = - a^T L_A a = - (\delta_X V)^T L_A^{-1} \delta_X V ,
\end{equation}
revealing the relation to our desired MPDI (\ref{eq:intro-MDI}).

In practice, we can only \emph{approximately} ensure that our inward fields have constant normal components, due to the fact that the normal to $\Sigma_{ij}$ is only guaranteed to be $C^{0,\alpha}$ smooth near the quadruple-set $\Sigma^3$, whereas to obtain a uniquely defined flow $F_t$, we need to construct, at the very least, a \emph{Lipschitz} vector-field. To ensure that the $|\nabla^\tang X^{\n_{ij}}|^2$ term in (\ref{eq:intro-formula-simplified}) can be made arbitrarily small for our approximate inward fields, our curvature-integrability is crucially used once again -- see Proposition \ref{prop:inward-fields}. 
\item All of the above work culminates in the proof of Theorem \ref{thm:flat}, which is the key ingredient in the proof of Theorem \ref{thm:main-I-I_m} and of independent interest.
In particular, it verifies:
\begin{theorem}[Stable Regular Clusters] \label{thm:intro-stable}
A Gaussian-stable regular $q$-cluster in $\R^n$ ($2 \leq q \leq n+1$) has flat interfaces. Moreover, its cells are (up to null-sets) convex polyhedra with at most $q-1$ facets, and the effective dimension is at most $q-1$. 
\end{theorem}
Note that in the single-bubble case $q=2$, this implies that the cells must be complementary halfspaces. If instead of regularity one makes the stronger assumption that the entire topological boundary of the single-bubble is smooth, the case $q=2$ of Theorem \ref{thm:intro-stable} was proved by McGonagle and Ross in \cite{McGonagleRoss:15} (see also \cite{Italians-SharpGaussianIsopStability,Rosales-StableSetsForGaussianMeasures}). For $q \geq 3$ (including the double-bubble case $q=3$), Theorem \ref{thm:intro-stable} is new. 

\smallskip

That stability forces the effective dimension to be at most $q-1$ is elementary (see Lemma \ref{lem:MKernelAndDef}). 
The main challenge in establishing Theorem \ref{thm:intro-stable} is to show that the interfaces must be flat. To this end, we assume in the contrapositive that the curvature is non-zero at some point, and produce a vector-field $X$ which violates stability, i.e. decreases area in second-order while maintaining the volume constraint. As a first step, we take a linear combination $X = \sum_{i=1}^q a_i X_i$ of our inward fields, and show that there is some linear combination so that the contribution of the curvature terms in formula (\ref{eq:intro-formula}) for the second variation is strictly negative; the other terms work in our favor. The problem is to ensure that the volume constraint is preserved to first order: $\delta_X V = 0$. When the cluster is full-dimensional (or just of maximal effective dimension $q-1$), then the map $\R^n \ni w \mapsto \delta_w V \in E^{(q-1)}$ is surjective (where we think of $w$ as a constant vector-field on $\R^n$), and we can make sure that $\delta_X V = 0$ by modifying $X$ by an appropriate $w \in \R^n$. However, as already explained, we do not a-priori know that a minimizing cluster will be effectively $(q-1)$-dimensional,
 and moreover, a general stable cluster may very well be effectively $k$-dimensional with $k < q-1$.  To handle the case that the cluster $\Omega$ is dimension-deficient, we write $\Omega$ as $\tilde \Omega \times \R$, and perform the above construction of $\tilde X$ for $\tilde \Omega$ in $\R^{n-1}$; we then define $X$ to be the product of $\tilde X$ and a linear function on the one-dimensional fiber, yielding a flow $F_t$ which \emph{skews} the cluster out of its product state.
By oddness of the linear function, we are ensured that $\delta_X V = 0$, and it turns out that, by the Poincar\'e inequality for the one-dimensional Gaussian measure (which is saturated by linear functions), we have just enough room to still get a negative overall second variation. This yields a contradiction to stability, and shows that a stable regular cluster must be flat. 

Stability similarly implies connectedness of the cells; convexity follows since the flat interfaces meet on their codimension-1 boundary at $120^{\circ}$ angles while their lower-dimensional boundary satisfies $\H^{n-2}(\Sigma^3 \cup \Sigma^4) = 0$. 

See also Theorem \ref{thm:pull-back} for a refinement of Theorem \ref{thm:intro-stable}, where a further description of the convex polyhedral cells is obtained, yielding an interesting relation between a general stable regular cluster and the canonical model simplicial cluster.
\item 
Once we know that a stable regular cluster (and in particular, a minimizing cluster) is flat, the formula (\ref{eq:intro-formula}) simplifies to (\ref{eq:intro-formula-simplified}), and we can use the entire $(q-1)$-dimensional family of inward fields to deduce a rigorous (approximate) version of the $(q-1)$-dimensional MPDI (\ref{eq:intro-MDI}) -- see Theorem \ref{thm:hessian-bound-for-I}.
\end{itemize}

\medskip

To establish the uniqueness of minimizers, we observe that all of our inequalities in the derivation of (\ref{eq:intro-MDI}) must have been equalities, and this already provides enough information for characterizing simplicial-clusters. An alternative approach for establishing uniqueness is presented in Section \ref{sec:conclude}: with any stable regular $q$-cluster on $\R^n$ ($2 \leq q \leq n+1$) we associate a two-dimensional simplicial complex and show that its first (simplicial) homology must be trivial; from this we deduce that such a cluster must be the pull-back of a shifted canonical model cluster on $E^{(q-1)}$ by a linear-map having certain properties (see Theorem \ref{thm:pull-back}); for a minimizing cluster, we show that this map must be adjoint to an isometry, and hence the cluster must itself be a simplicial cluster. 

\subsection{Simplified argument in the Double-Bubble case}

In the double-bubble case $q=3$, there is a drastic simplification of the argument outlined above, which is based on the following dichotomy: either there is  enough information from testing (\ref{eq:intro-variation}) on constant (translation) fields to deduce our desired MPDI (\ref{eq:intro-MDI}), or else the cluster is necessarily effectively one-dimensional, in which case the interfaces are automatically flat, and we can deduce (\ref{eq:intro-MDI}) by constructing the inward fields by hand in dimension one; the simple argument, based on a certain Cauchy--Schwarz inequality for quadratic forms, is described in Section \ref{sec:double-bubble}. In particular, there is no need to establish the key Theorem \ref{thm:intro-stable} on the flatness of a minimizing cluster's interfaces, nor to use any of its ingredients, such as the higher-codimensional regularity theory for triple and quadruple points, curvature-integrability, the general formula (\ref{eq:intro-formula}), or the construction of our approximate inward fields. Consequently, the reader only interested in the resolution of the Gaussian Double-Bubble Conjecture may safely skip over Sections \ref{sec:higher} through \ref{sec:stable} as well as Appendices \ref{sec:calculation} and \ref{sec:Schauder}. 
Unfortunately, this simplified argument is particular to the case $q=3$, since when $q \geq 4$, the alternative in the above dichotomy is that the cluster's effective dimension is between $1$ and $q-2 \geq 2$.

\medskip

The contents of this work were posted on the arXiv in two installments: the above simplified argument for resolving the Gaussian Double-Bubble Conjecture (case $q=3$ of Theorems \ref{thm:main-I-I_m} and \ref{thm:main-uniqueness}) was posted in January 2018 \cite{EMilmanNeeman-GaussianDoubleBubbleConj}, and the full resolution of the Gaussian Multi-Bubble Conjecture was announced in \cite{EMilmanNeeman-GaussianDoubleBubbleConj} and posted in May 2018 \cite{EMilmanNeeman-GaussianMultiBubbleConj}. 

\medskip

The rest of this work is organized as follows. In Section \ref{sec:model}, we construct the model simplicial-clusters and associated model isoperimetric profile, and establish their properties.  In Section \ref{sec:prelim} we recall relevant definitions and provide some preliminaries for the ensuing calculations. In Sections \ref{sec:minimizers} and \ref{sec:higher} we collect and prove the results regarding isoperimetric minimizing clusters we require: Section \ref{sec:minimizers} pertains to properties one can deduce just from the regularity theory of the interfaces $\Sigma^1$, while Section \ref{sec:higher} is dedicated to consequences of the higher-codimensional regularity theory of $\Sigma^2$ and $\Sigma^3$; to streamline the presentation, proofs of various statements are deferred to the appendices, and in particular, the proof of the curvature integrability is deferred to Appendix \ref{sec:Schauder}. In Section \ref{sec:second-var} we derive a generalized version of Stokes' theorem; we then apply it to the general formula for the second variation of weighted area under a volume constraint, whose calculation is deferred to Appendix \ref{sec:calculation}, and obtain the simplified formula (\ref{eq:intro-formula}) for regular stationary clusters. In Section \ref{sec:inward-fields} we construct our inward fields and record some of their useful properties. In Section \ref{sec:stable} we prove that stable regular $q$-clusters in $\R^n$ have convex polyhedral cells when $2 \leq q \leq n+1$. In Section \ref{sec:multi-bubble} we establish an (approximate) rigorous version of (\ref{eq:intro-MDI}) and prove Theorem \ref{thm:main1}. In Section \ref{sec:double-bubble} we present the simplified argument for the double-bubble case. In Section \ref{sec:uniqueness} we prove Theorem \ref{thm:main-uniqueness}. In Section \ref{sec:conclude} we provide some concluding remarks and highlight some questions which remain open.

\medskip

\noindent 
\textbf{Acknowledgments.}
We thank Ramon van Handel for his comments on a preliminary version of this work, Francesco Maggi for his continued interest and support and for pointing us towards a simpler version of the maximum principle than our original one, Frank Morgan for many helpful references,  and Brian White for informing us of the reference to the recent \cite{CES-RegularityOfMinimalSurfacesNearCones}. We also acknowledge the hospitality of MSRI where part of this work was conducted.

\section{Model Simplicial Clusters} \label{sec:model}

In this section, we construct the model $q$-clusters which are conjectured to be optimal on $\R^n$, for all $q \geq 2$ and $n \geq q-1$. It will be enough to construct them on $\R^{q-1}$, since by taking Cartesian product with $\Real^{n-q+1}$ and employing the product structure of the Gaussian measure, these clusters extend to $\Real^n$ for all $n \geq q-1$. Actually, it will be convenient to construct them on $E^{(q-1)}$, rather than on $\R^{q-1}$, 
where recall that $E = E^{(q-1)}$ denotes the tangent space to $\simplex = \simplex^{(q-1)}$, which we identify with $\{x \in \R^q: \sum_{i=1}^q x_i = 0\}$.
Strictly speaking, we will construct them on the dual space $E^*$, but we will freely identify between $E$ and $E^*$ via the standard Euclidean structure. 
Consequently, in this section, let $\gamma = \gamma_{q-1}$ denote the standard $(q-1)$-dimensional Gaussian measure on $E$, and if $\varphi = e^{-W}$ denotes its (smooth, Gaussian) density on $E$,  set $\gamma^{q-2} = \gamma_{q-1}^{q-2} := \varphi\H^{q-2}$. 
\medskip

We will frequently use the fact that if $Z \in E^{(q-1)}$ is distributed according to $\gamma_{q-1}$, then:
\begin{equation} \label{eq:GaussianProj} 
\text{$\sqrt{\frac{q}{q-1}} Z_i$ and $\frac{Z_i - Z_j}{\sqrt{2}}$ ($i \neq j$) are distributed according to $\gamma_1$.} 
\end{equation}
(for the former statement, note that the orthogonal projection of $e_i$ onto $E^{(q-1)}$ is $e_i - \frac{1}{q} \sum_{j=1}^q e_j$ which has norm $\sqrt{\frac{q-1}{q}}$). We use $\Phi$ to denote the one-dimensional Gaussian cumulative distribution function $\Phi(t) = \gamma_1(-\infty,t]$.  

\medskip

Define $\Omega^m_i = \interior \{x \in E: \max_j x_j = x_i \}$ (``$m$'' stands for ``model'').
For any $x \in E$, $x + \Omega^m = (x + \Omega^m_1, \ldots, x + \Omega^m_q)$ is a $q$-cluster, which we call a canonical model \emph{simplicial} cluster, as it corresponds to 
the Voronoi cells of the $q$-equidistant points $\{x + e_i -  \frac{1}{q}\sum_{j} e_j\}_{i=1,\ldots,q} \subset E$. 
Let $\Sigma^m_{ij} := \partial \Omega^m_i \cap \partial \Omega^m_j$ denote the codimension-$1$ interfaces of $\Omega^m$. Observe that $x + \Sigma^m_{ij}$, the interfaces of $x + \Omega^m$, are flat, and meet at $120^\circ$ angles along codimension-$2$ flat surfaces. We denote:  \[
A^m_{ij}(x) := \gamma^{q-2}(x + \Sigma^m_{ij}) .
\]

\begin{lemma} \label{lem:diffeo}
 The map:
 \[
 E \ni x \mapsto \modelV(x) := \gamma(x + \Omega^m) \in \interior \simplex 
 \]
 is a diffeomorphism between $E$ and $\interior \simplex$. Its differential is given by:
 \begin{equation}\label{eq:DPsi}
      D \modelV(x) = - \frac{1}{\sqrt 2} L_{A^m(x)} ,
  \end{equation}
where, recall, $L_A$ was defined in (\ref{eq:intro-LA}). 
\end{lemma}

\begin{proof}

    Clearly $\modelV(x)$ is $C^\infty$,
    since the Gaussian density is $C^\infty$ and all of its derivatives vanish rapidly
    at infinity.
    
    To see that $\modelV$ is injective, simply note that if $y \neq x$, then there exists $i \in \{1,\ldots,q\}$ such that $y \in x + \overline{\Omega^m_i}$, and as $\overline{\Omega^m_i}$ is a convex cone, it follows that $y + \Omega^m_i \subsetneq x + \Omega^m_i$ and hence $\gamma(y + \Omega^m_i) < \gamma(x + \Omega^m_i)$. 
  
                    We now compute:
    \[
      \nabla_v \Psi_i(x) = \int_{x + \Omega^m_i} \nabla_v e^{-\pot(y)} \, dy
      = \int_{x + \partial \Omega^m_i} \inr{v}{\n} e^{-\pot(y)}\, d\calH^{q-2},
    \]
    where $\n$ denotes the outward unit-normal. Note that
    $\partial \Omega^m_i = \bigcup_{j \ne i} \Sigma^m_{ij}$ (the union is disjoint up to $\calH^{q-2}$-null sets),
    and that the outward
    unit-normal on $\Sigma^m_{ij}$ is the constant vector-field $(e_j - e_i)/\sqrt{2}$.
    Therefore,
    \[
      \nabla_v \Psi_i(x) = \frac{1}{\sqrt 2} \sum_{j \ne i} \inr{v}{e_j - e_i} A^m_{ij}(x) = -\frac{1}{\sqrt 2} v^T L_{A^m}(x) e_i ,
    \]
    thereby confirming (\ref{eq:DPsi}). Since each of the $A^m_{ij}(x)$ is strictly positive, $L_{A^m(x)}$ is non-singular (in fact, strictly positive-definite) as a quadratic form on $E$, and hence $D\modelV(x)$ is non-singular as well. It follows that $\modelV$ is a diffeomorphism onto its image. 
            
    Finally, to show that $\Psi$ is surjective, fix $R>0$ and consider the open simplex $K_R =
    \{x \in E: \max_i x_i < R\}$. For any $x \in \partial K_R$, there is some $i$ with $x_i = R$, and hence
                    \begin{align*}
      x + \overline{\Omega^m_i}
      &= \{z \in E : z_i - x_i \geq \max_{j} (z_j - x_j)\} \\
      &= \{z \in E : z_i - R \geq \max_{j} (z_j - x_j)\} \\
      &\subset \{z \in E : z_i - R \ge 0\},
    \end{align*}
    since $z - x \in E$ implies that $\max_j (z_j - x_j) \ge 0$. It follows by (\ref{eq:GaussianProj}) that for all $x \in \partial K_R$, 
     $\gamma(x + \Omega^m_i) \le 1 - \Phi(\sqrt{q/(q-1)} R) < 1 - \Phi(R)$. In other words, if we consider the shrunken
    simplex $\simplex_R = \{v \in \simplex: \min_i v_i \ge 1 -\Phi(R)\}$, then $\Psi(\partial K_R) \cap \simplex_R = \emptyset$. 
        On the other hand, $\Psi : K_R \rightarrow \Psi(K_R)$ is continuous and injective, and is therefore a homeomorphism by the Invariance of Domain theorem \cite{Munkres-TopologyBook2ndEd}. 
            Since $\partial \Psi(K_R) = \Psi(\partial K_R)$ is disjoint from $\simplex_R$, it follows that $\simplex_R \cap \Psi(K_R)$ and $\simplex_R \setminus \Psi(K_R)$ are two complementing relatively-open subsets of the connected set $\simplex_R$. 
    Since $\Psi(0) = (\frac 1 q,\ldots, \frac 1 q) \in \simplex_R$ for $R > 0$ large enough, it follows that necessarily $\simplex_R  \subset \Psi(K_R)$. Taking $R \to \infty$, we see that $\Psi: E \to \interior \simplex$ is surjective. 
    
\end{proof}

We can now give the following definition, which is well-posed according to the bijection established in Lemma \ref{lem:diffeo}. Adopting the jargon from the Euclidean setting, note that a cluster with $q$ cells corresponds to $q-1$ bubbles (with the last cell corresponding to the ``exterior" bubble, which is not counted). 

\begin{definition}
The model isoperimetric $(q-1)$-bubble profile $I_m : \interior \simplex \to \R$ is defined as:
\[
I_m(v) := \per_\gamma(x + \Omega^m) = \sum_{i < j} A^m_{ij}(x)  \;\;\; \text{for $x$ such that} \;\;\; \Psi(x) = \gamma(x + \Omega^m) = v . 
\]
When we need to explicitly refer to the number of bubbles involved, we will write $I^{(q-1)}_m$ and $\simplex^{(q-1)}$ instead of $I_m$ and $\simplex$, respectively. For completeness, we define $I_m^{(0)}(1) = 0$. \end{definition}
 
Note that thanks to the product structure of the Gaussian measure and its invariance under orthogonal transformations, the above definition coincides with the one given in the Introduction involving arbitrary simplicial clusters in $\R^n$ (instead of canonical ones on $E$).

Thanks to the next lemma, we may (and do) extend $I_m$ by continuity to the entire $\simplex$. 
Given $J \subset \{ 1,\ldots,q \}$ with $0 < |J| < q$, let $\simplex_J$ denote the face of $\simplex$ consisting of all $v \in \simplex$ such that $\{i: v_i > 0\} = J$. Given $v \in \simplex_J$, let $v_J$ denote its natural projection to $\interior \simplex^{(|J|-1)}$ obtained by only keeping the coordinates in $J$. 

\begin{lemma}\label{lem:tripod-profile-continuous}
  $I^{(q-1)}_m$ is $C^\infty$ on $\interior \simplex^{(q-1)}$, and continuous up to $\partial
  \simplex^{(q-1)}$. Moreover, the continuous extension of $I^{(q-1)}_m$ satisfies 
  \[
    I^{(q-1)}_m(v) = I^{(|J|-1)}_m(v_J)
  \]
  whenever $v \in \simplex^{(q-1)}_J$, for all $J \subset \set{1,\ldots,q}$ with $0 < |J| < q$.
\end{lemma}

The main idea in the proof of Lemma~\ref{lem:tripod-profile-continuous} is to
take a non-degenerate $q$-simplicial cluster with a few small cells, and compare
it to the lower-dimensional simplicial cluster obtained by deleting the small cells
and ``extending'' the large cells in the natural way (see Figure~\ref{fig:comparison} for an illustration).
By comparing the measure
and perimeter of these two clusters, we can see how $I_m$ behaves
near $\partial \simplex$.

\begin{figure}
    \begin{center}
    \begin{tikzpicture}
        \draw (0, 0) -- (0, 2);
        \draw (0, 0) -- (0.866, -0.5);
        \draw (0, 0) -- (-0.866, -0.5);
        \node at (-0.5, 1.0) {$\Omega_1$};
        \node at (0.5, 1.0) {$\Omega_2$};
        \node at (0.0, -0.3) {$\Omega_3$};
    \end{tikzpicture}
    \hspace{5em}
    \begin{tikzpicture}
        \draw (0, -0.5) -- (0, 2);
        \draw[dotted] (0, 0) -- (0.866, -0.5);
        \draw[dotted] (0, 0) -- (-0.866, -0.5);
        \node at (-0.5, 1.0) {$\Omega_1$};
        \node at (0.5, 1.0) {$\Omega_2$};
    \end{tikzpicture}
    \end{center}

    \captionsetup{width=0.7\linewidth}
    \caption{
        Given the configuration on the left, we delete the small chamber and extend the large chambers,
        as seen on the right.
    }
    \label{fig:comparison}
\end{figure}

\begin{proof}[Proof of Lemma~\ref{lem:tripod-profile-continuous}]
    Since $I_m(v) = \per_\gamma(\modelV^{-1}(v) + \Omega^m)$ and both $\modelV^{-1}$ and the map $x \mapsto \per_\gamma(x +
  \Omega^m)$ are $C^\infty$ on their respective domains, it follows that $I_m$ is $C^\infty$ on $\interior \simplex$.

  For the rest of this proof, we write $\Omega = \Omega^{q}$ for the model $q$-cluster $\Omega^m$ on $E = E^{(q-1)}$, and recall the definition of $\Psi = \Psi^{(q)}$ from Lemma \ref{lem:diffeo}. 
  Given $0 < k < q$, consider the embedding $E^{(k-1)} \subset E^{(q-1)}$ by padding with zeros the last $q-k$ coordinates. Let $\Pi_{k}: E^{(q-1)} \to E^{(k-1)}$ denote the orthogonal projection onto $E^{(k-1)}$, defined by $(\Pi_{k} x)_i = x_i - \frac 1k \sum_{j=1}^k x_j$ for $i=1,\ldots,k$. 
    
  We now define the $q$-cluster $\Omega^{q \leftarrow k}$ on $E^{(q-1)}$ by
  taking $\Omega^{q \leftarrow k}_i$ empty for $i > k$, and
  \[
      \Omega^{q \leftarrow k}_i = \interior \{z \in E^{(q-1)} :  \max_{1 \le j \le k} z_j  = z_i \}
  \]
  otherwise. In other words, $\Omega^{q \leftarrow k} = \Pi_{k}^{-1}(\Omega^{k})$. 
  Observe that $\Omega_i \subset \Omega^{q \leftarrow k}_i$ if $i \le k$, and so by the product structure of the Gaussian measure,
         \begin{equation}\label{eq:volume-ineq}
      \gamma(x + \Omega_i) \le \Psi^{(k)}_i(\Pi_{k} x)  \;\;\; \forall i \leq k . 
  \end{equation}
  
  Since the functions $I^{(q-1)}$ are symmetric in their arguments for every $q$, it suffices to prove that
  if $\{ v_m \}_m \subset \interior \simplex^{(q-1)}$ is a convergent sequence with
  $(v_{m,1}, \dots, v_{m,k}) \to v \in \interior \simplex^{(k-1)}$ then $I^{(q-1)}(v_m) \to I^{(k-1)}(v)$.
  So, fix such a sequence, and define:
  \[
  x_m := \Psi^{-1}(v_m) \in E^{(q-1)} ~,~ y := (\Psi^{(k)})^{-1}(v) \in E^{(k-1)} ~,~  y_m := \Pi_{k} x_m \in E^{(k-1)} .
  \]
  First, we observe that $y_m \to y$. Indeed,~\eqref{eq:volume-ineq} implies that for $i \le k$,
  \[
    v_i = \liminf_{m \to \infty} \Psi_i(x_m) \le \liminf_{m \to \infty} \Psi^{(k)}_i(y_m).
  \]
  But the inequality cannot be strict, since $\sum_{i=1}^k v_i = 1$ and
  \[
  \sum_{i=1}^k \liminf_{m \to \infty} \Psi^{(k)}_i(y_m) \leq \liminf_{m \to \infty} \sum_{i=1}^k \Psi^{(k)}_i(y_m) = 1 .
  \]
     It follows that
  \begin{equation} \label{eq:Psiyv}
   \Psi^{(k)}(y_m) \to v ,
   \end{equation} 
   and since $v = \Psi^{(k)}(y)$ and $\Psi^{(k)}: E^{(k-1)} \to \simplex^{(k-1)}$
  is a diffeomorphism, we must have $y_m \to y$.

  Next, we claim that for every $i \le k$ and $j > k$, $x_{m,j} - x_{m,i} \to \infty$.
  Assume in the contrapositive this is not the case. 
  Then we can choose $C > 0$, $i \le k$, and $j > k$ such that (after passing to a subsequence)
  $x_{m,j} = \min_{\ell > k} x_{m,\ell}$
  and $x_{m,j} - x_{m,i} \le C$ for all $m$. Since $y_m = \Pi_k x_m$ converges, it follows that
  the first $k$ coordinates of $x_m$ differ by at most a constant independent of $m$.
  Hence, $x_{m,j} \le \min_{\ell \le k} x_{m,\ell} + C'$ (for a different constant $C'$).
  Since $x_{m,j} \le x_{m,\ell}$
  for $\ell > k$, we have $x_{m,j} \le \min_{\ell=1, \dots, q} x_{m,\ell} + C'$. But then
  \begin{align*}
    x_m + \Omega_j
    &= \{z \in E : z_j - x_{m,j} > \min_{\ell \ne j} (z_\ell - x_{m,\ell})\} \\
    &\supset \{z \in E : z_j - x_{m,j} > \min_{\ell \ne j} (z_\ell + C' - x_{m,j})\} \\
    &= \{z \in E : z_j > \min_{\ell \ne j} z_\ell - C'\}.
  \end{align*}
  The Gaussian measure of the latter set is strictly positive,  in contradiction to the assumption 
  that $v_m = \Psi(x_m) = \gamma(x_m + \Omega)$ satisfies $(v_{m,1}, \dots, v_{m,k}) \to v \in \interior \simplex^{(k-1)}$, which in particular implies that $v_{m,j} \to 0$
  for $j > k$.

  Having understood something about $x_m = \Psi^{-1}(v_m)$, we turn to the interfaces
  $x_m + \Sigma_{ij}$.   Recall that the interfaces $\Sigma_{ij}$ of
  $\Omega$ are given by $\Sigma_{ij} = \{z \in E^{(q-1)} : z_i = z_j = \max_\ell z_\ell\}$,
  while the interfaces $\Sigma_{ij}^{q \leftarrow k}$ of $\Omega^{q \leftarrow k}$ are given by
  \[
    \Sigma_{ij}^{q \leftarrow k} = 
    \begin{cases}
      \{z \in E^{(q-1)} : z_i = z_j = \max_{\ell \le k} z_\ell\} & \text{if $i, j \le k$} \\
      \emptyset & \text{otherwise}.
    \end{cases}
  \]
  In particular, $\Sigma_{ij}^{q \leftarrow k} \supset \Sigma_{ij}$ as long as $i, j \le k$.
  On the other hand,
  \[
    \Sigma_{ij}^{q \leftarrow k}
    \subset \Sigma_{ij} \cup \bigcup_{\ell > k} \{z \in E^{(q-1)} : z_i = z_j \text{ and } z_\ell > z_i\},
  \]
  meaning that
  \[
    x + \Sigma_{ij}^{q \leftarrow k}
    \subset (x + \Sigma_{ij}) \cup \bigcup_{\ell > k} \{z \in E^{(q-1)} : z_i - z_j = x_i - x_j \text{ and } z_\ell - z_i > x_\ell - x_i\}.
  \]
  Recalling that $i \le k$ and $\ell > k$ implies that $x_{m,\ell} - x_{m,i} \to \infty$, we have
  \[
      \gamma^{q-2}(\{z \in E^{(q-1)} : z_i - z_j = x_{m,i} - x_{m,j} \text{ and } z_\ell - z_i > x_{m,\ell} - x_{m,i}\}) \to 0
  \]
  for every $\ell > k$, and so it follows that
  \begin{equation} \label{eq:arrow-same}
    \left|\gamma^{q-2}(x_m + \Sigma_{ij}) - \gamma^{q-2}(x_m + \Sigma_{ij}^{q \leftarrow k})\right| \to 0
  \end{equation}
  as $m \to \infty$ for every $i, j \le k$. By the product structure of the Gaussian measure, (\ref{eq:Psiyv}), and continuity of $I^{(k-1)}$ on $\interior \simplex^{(k-1)}$, we know that:
  \[
    \sum_{1 \le i < j \le k} \gamma^{q-2}(x_m + \Sigma_{ij}^{q \leftarrow k}) = \sum_{1 \le i < j \le k} \gamma^{k-2}_{k-1}(y_m + \Sigma_{ij}^{k})
    = I^{(k-1)}(\Psi^{(k)}(y_m))
    \to I^{(k-1)}(v),
  \]
  and hence we deduce by (\ref{eq:arrow-same}) that:
  \[
      \sum_{1 \leq i < j \le k} \gamma^{q-2}(x_m + \Sigma_{ij}) \to I^{(k-1)}(v).
  \]
  
  To conclude that $I^{(q-1)}(v_m) \to I^{(k-1)}(v)$, we will show that the remaining terms of $I^{(q-1)}(v_m) = \per_\gamma(x_m + \Omega) = \sum_{1 \leq i < j \leq q} \gamma^{q-2}(x_m + \Sigma_{ij})$ are negligible: 
  \[
    \gamma^{q-2}(x_m + \Sigma_{ij}) \to 0
  \]
  whenever $i > k$ (and by symmetry, we will reach the same conclusion whenever either $i > k$ or $j > k$).
  For $j \le k$, this follows from the inclusion
  \begin{align*}
    x_m + \Sigma_{ij}
    &= \{z \in E : z_i - x_{m,i} = z_j - x_{m,j} = \max_\ell (z_\ell - x_{m,\ell})\} \\
    &\subset \{z \in E : z_i - z_j = x_{m,i} - x_{m,j}\},
  \end{align*}
  and the fact that $x_{m,i} - x_{m,j} \to \infty$, which implies by (\ref{eq:GaussianProj}) that:
  \[
  \gamma^{q-2}( x_m + \Sigma_{ij} ) \leq \gamma_1^0 \set{ \frac{x_{m,i} - x_{m,j}}{\sqrt{2}}  } \rightarrow 0 . 
  \]
  For $j > k$, it follows similarly   by fixing a single $\ell \le k$ and using the inclusion
  \begin{align*}
    x_m + \Sigma_{ij}
    &\subset \{z \in E : z_i - x_{m,i} = z_j - x_{m,j} \ge z_\ell - x_{m,\ell} \} \\
    & = \{z \in E : z_i - z_j = x_{m,i} - x_{m,j} \text{ and } z_i - z_\ell \ge x_{m,i} - x_{m,\ell} \},
  \end{align*}
  noting again that $x_{m,i} - x_{m,\ell} \to \infty$.
\end{proof}

We have constructed the canonical model simplicial $(k-1)$-bubble clusters $\Omega$ in $E^{(k-1)}$ with $\gamma_{k-1}(\Omega) = v$ for all $v$ in the \emph{interior} $\interior \simplex^{(k-1)}$ and for all $k \geq 2$. For $k=1$, the trivial $0$-bubble ($1$-cell) cluster on $E^{(0)} = \{ 0 \}$ is $\Omega_1 = \{ 0 \}$. Clearly, by taking Cartesian products with $E^{(q-k+1)}$ and employing the product structure of the Gaussian measure, we obtain the canonical model $(q-1)$-bubble clusters $\Omega$ in $E^{(q-1)}$ with $\gamma_{q-1}(\Omega) = v$ for all $v \in \simplex^{(q-1)}$. Similarly, these clusters extend to $\R^n$ for all $n \geq q-1$. Consequently, $I(v) \le I^{(q-1)}_m(v)$ for all $v \in \simplex^{(q-1)}$ and $n \geq q-1$, and our main goal in this work is to establish the converse inequality. 

\medskip
To this end, we observe that the model profile $I_m = I_m^{(q-1)}$ satisfies a remarkable
differential equation. 
\begin{proposition}\label{prop:I_m-equation}
    At any point in $\interior \Delta$,
    \[
        \nabla I_m = \frac{1}{\sqrt{2}} \Psi^{-1} \qquad \text{and} \qquad \nabla^2 I_m = - (L_{A^m \circ \Psi^{-1}})^{-1}
    \]
    as tensors on $E$, where, recall, $L_A$ was defined in (\ref{eq:intro-LA}). 
\end{proposition}

\begin{proof}
    Recalling that $A^m_{ij}(x) = \int_{x + \Sigma^m_{ij}} e^{-\pot}\, d\calH^{q-2}$,
    we have
    \[
        \nabla_v A^m_{ij}(x) = \int_{x + \Sigma^m_{ij}} \nabla_v e^{-\pot} \, d\calH^{q-2}.
    \]
    Decomposing $v$ into its normal and tangential components and applying
    Stokes' theorem to the tangential part (which is justified since the Gaussian density decays faster than any polynomial)
    \[
        \nabla_v A^m_{ij}(x) = -\int_{x + \Sigma^m_{ij}} \inr{v}{\n_{ij}} \inr{\n_{ij}}{y} e^{-\pot(y)} \, d\calH^{q-2}
        + \int_{x + \partial \Sigma^m_{ij}} \inr{v}{\n_{\partial ij}} e^{-\pot(y)} \, d\calH^{q-3},
    \]
    where $\n_{ij} = \frac{1}{\sqrt{2}} (e_j - e_i)$ 
        is the unit-normal to $\Sigma^m_{ij}$, $\partial \Sigma^m_{ij}$ denotes the $(q-3)$-dimensional boundary of $\Sigma^m_{ij}$ in the manifold sense,
    and $\n_{\partial ij}$ is the boundary outer unit-normal. Now, $\inr{\n_{ij}}{y} =
    \inr{\n_{ij}}{x}$ for all $y \in x + \Sigma^m_{ij}$. Note also that the $(q-2)$-dimensional interfaces meet at their $(q-3)$-dimensional boundaries
    in threes, and that the boundary outer normals at these points sum to zero. Hence, summing over all $i < j$, the total contribution of the last integral above vanishes, and we are left with:
    \begin{align*}
        \nabla_v \sum_{i < j} A^m_{ij}(x)
        & = - \sum_{i < j} \inr{v}{\n_{ij}} \inr{\n_{ij}}{x} A^m_{ij}(x) \\
        & = - \frac{1}{2} \sum_{i < j} \inr{v}{e_j - e_i} \inr{e_j - e_i}{x} A^m_{ij}(x) =  - \frac{1}{2}  v^T L_{A^m} x.
    \end{align*}
    Finally, recall that $I_m = \sum_{i < j} A^m_{ij} \circ \Psi^{-1}$.
    By~\eqref{eq:DPsi}, $D (\Psi^{-1}) = (D \Psi)^{-1} \circ \Psi^{-1} = - \sqrt 2 (L_{A^m \circ \Psi^{-1}})^{-1}$, and so the chain rule implies that
    \[
        \nabla I_m = \frac{1}{\sqrt 2} \Psi^{-1},
    \]
    thus establishing the first claim. The second claim follows from differentiating the first claim. \end{proof}

\section{Definitions and Notation}\label{sec:prelim}

We will be working in Euclidean space $(\R^n,\abs{\cdot})$ endowed with a
measure $\mu = \mu^n$ having $C^\infty$-smooth and strictly-positive density
$e^{-\pot}$ with respect to Lebesgue measure; with the exception of Theorem \ref{thm:trans-formula} and Lemma \ref{lem:MKernelAndDef},
we will not restrict ourselves to the Gaussian measure $\gamma$ until reaching Section \ref{sec:stable}. 
Recall that the cells $\set{\Omega_i}_{i=1,\ldots,q}$
of a cluster $\Omega$ are assumed to be pairwise disjoint Borel subsets of $\R^n$, and
satisfy $\mu(\R^n \setminus \cup_{i=1}^q \Omega_i) = 0$. In addition, they are
assumed to have finite $\mu$-weighted perimeter $P_\mu(\Omega_i) < \infty$, to
be defined below. Given distinct $i,j,k \in \{1,\ldots,q\}$, we define the set of cyclically ordered pairs in $\{i,j,k\}$:
\[
\cyclic(i,j,k) := \{ (i,j) , (j,k) , (k,i) \}  .
\]

\subsection{Weighted divergence and Mean-curvature}

We write $\div X$ to denote divergence of a smooth vector-field $X$, and $\div_\mu X$ to denote its weighted divergence:
\begin{equation} \label{eq:weighted-div}
\div_{\mu} X := \div (X e^{-W}) e^{+W} = \div X - \nabla_X \pot . 
\end{equation}
For a smooth hypersurface $\Sigma \subset \Real^n$ co-oriented by a unit-normal field
$\n$, let $H_\Sigma: \Sigma \to \R$ denote its mean-curvature, defined as the trace of its second fundamental form $\II_{\Sigma}$. 
 The weighted mean-curvature $H_{\Sigma,\mu}$ is defined as:
\[
H_{\Sigma,\mu} := H_{\Sigma} - \nabla_\n \pot .
\]
We write $\div_\Sigma X$ for the surface divergence of a vector-field $X$ defined on $\Sigma$, i.e. $\sum_{i=1}^{n-1} \scalar{\tang_i,\nabla_{\tang_i} X}$ where $\{\tang_i\}$ is a local orthonormal frame on $\Sigma$; this coincides with $\div X - \inr{\n}{\nabla_\n X}$ for any smooth extension of $X$ to a neighborhood of $\Sigma$. 
The weighted surface divergence $\div_{\Sigma,\mu}$ is defined as:
\[
\div_{\Sigma,\mu} X = \div_{\Sigma} X - \nabla_X \pot,
\]
so that $\div_{\Sigma,\mu} X = \div_{\Sigma} (X e^{-\pot}) e^{+\pot}$ if $X$ is tangential to $\Sigma$.  
Note that $\div_{\Sigma} \n = H_{\Sigma}$ and $\div_{\Sigma,\mu} \n = H_{\Sigma,\mu}$.  We will also abbreviate $\inr{X}{\n}$ by $X^\n$, and we will write $X^\tang$ for the tangential part of $X$, i.e. $X - X^{\n} \n$.  

We will frequently use that:
\[
\div_{\Sigma,\mu} X = \div_{\Sigma,\mu} (X^\n \n) + \div_{\Sigma,\mu} X^{\tang} = H_{\Sigma,\mu} X^{\n} + \div_{\Sigma,\mu} X^{\tang} . 
\]
Note that the above definitions ensure the following weighted version of Stokes' theorem (which we interchangeably refer to in this work as the Gauss--Green or divergence theorem): if $\Sigma$ is a smooth $(n-1)$-dimensional manifold with $C^1$ boundary, denoted $\partial \Sigma$, (completeness of $\Sigma \cup \partial \Sigma$ is not required), and $X$ is a smooth vector-field on $\Sigma$, continuous up to $\partial \Sigma$, with compact support in $\Sigma \cup \partial \Sigma$, then:
\begin{equation} \label{eq:Stokes}
\int_\Sigma \div_{\Sigma,\mu} X d\mu^{n-1} = \int_{\Sigma} H_{\Sigma,\mu} X^{\n} d\mu^{n-1} + \int_{\partial \Sigma} X^{\n_{\partial}} d\mu^{n-2} ,
\end{equation}
where $\n_{\partial}$ denotes the exterior unit-normal to $\partial \Sigma$, and:
\[
\mu^{k} := e^{-W} \H^{k} . 
\]

\subsection{Reduced boundary and cluster interfaces}

Given a Borel set $U \subset \Real^n$ with locally-finite perimeter, its reduced boundary $\partial^* U$ is defined (see e.g. \cite[Chapter 15]{MaggiBook}) as the subset of $\partial U$ for which there is a uniquely defined outer unit normal vector to $U$ in a measure theoretic sense. While the precise definition will not play a crucial role in this work, we provide it for completeness. The set $U$ is said to have locally-finite (unweighted) perimeter, if for any compact subset $K \subset \Real^n$ we have:
\[
\sup \left\{\int_{U} \div X \; dx : X \in C_c^\infty(\R^n; T \R^n) \; ,\; \text{supp}(X) \subset K \; , \; |X| \le 1 \right\} < \infty . 
\]
With any Borel set with locally-finite perimeter one may associate a vector-valued Radon measure $\mu_U$ on $\Real^n$, called the Gauss--Green measure, so that:
\[
\int_U \div X \; dx = \int_{\Real^n} \inr{X}{d\mu_U} \;\;\; \forall X \in C_c^\infty(\R^n; T \R^n) .
\]
The reduced boundary $\partial^* U$ of a set $U$ with locally-finite perimeter is defined as the collection of $x \in \text{supp } \mu_U$ so that the vector limit:
\[
\n_U := \lim_{\eps \to 0+} \frac{\mu_U(B(x,\eps))}{\abs{\mu_U}(B(x,\eps))}  
\]
exists and has length $1$ (here $\abs{\mu_U}$ denotes the total-variation of $\mu_U$ and $B(x,\eps)$ is the open Euclidean ball of radius $\eps$ centered at $x$). When the context is clear, we will abbreviate $\n_U$ by $\n$. Note that modifying $U$ on a null-set does not alter $\mu_U$ nor $\partial^* U$. It is known that $\partial^* U$ is a Borel subset of $\partial U$ and that $\abs{\mu_U}(\R^n \setminus \partial^* U) = 0$. 
If $U$ is an open set with $C^1$ smooth boundary, it is known (e.g. \cite[Remark 15.1]{MaggiBook}) that $\partial^* U = \partial U$.

Recall that the $\mu$-weighted perimeter of $U$ was defined in the Introduction as:
\[
\per_\mu(U) := \sup \left\{\int_U \div_\mu X \, d\mu: X \in C_c^\infty(\R^n; T \R^n), |X| \le 1 \right\}.
\]
Clearly, if $U$ has finite weighted-perimeter $\per_\mu(U) < \infty$, it has locally-finite (unweighted) perimeter. It is known \cite[Theorem 15.9]{MaggiBook} that in that case:
\[
\per_\mu(U) = \mu^{n-1}(\partial^* U)  .
\]
In addition, by the Gauss--Green--De Giorgi theorem, the following integration by parts formula holds for any $C_c^1$ vector-field $X$ on $\R^n$, and Borel subset $U \subset \Real^n$ with locally finite perimeter (see \cite[Theorem 15.9]{MaggiBook} and recall (\ref{eq:weighted-div})):
\begin{equation}\label{eq:integration-by-parts}
  \int_U \div_\mu X \, d\mu^n = \int_{\partial^* U} X^\n \, d\mu^{n-1} .
\end{equation}

Given a cluster $\Omega = (\Omega_1, \dots, \Omega_q)$, we define the interface
between cells $i$ and $j$ (for $i \ne j$) as:
\[
    \Sigma_{ij} = \Sigma_{ij}(\Omega) := \partial^* \Omega_i \cap \partial^* \Omega_j,
\]
and we define:
\[
    A_{ij} = A_{ij}(\Omega) := \mu^{n-1}(\Sigma_{ij}).
\]
It is standard to show (see \cite[Exercise 29.7, (29.8)]{MaggiBook}) that for any $S \subset \set{1,\ldots,q}$:
\begin{equation} \label{eq:nothing-lost-many}
\H^{n-1}\brac{\partial^*(\cup_{i \in S} \Omega_i) \setminus \cup_{i \in S , j \notin S} \Sigma_{ij}} = 0 .
\end{equation}
In particular:
\begin{equation} \label{eq:nothing-lost}
\H^{n-1} \brac{ \partial^* \Omega_i  \setminus \cup_{j \neq i} \Sigma_{ij} } = 0 \;\;\; \forall i=1,\ldots,q ,
\end{equation}
and hence:
\[
\per_\mu(\Omega_i) = \sum_{j \neq i} A_{ij}(\Omega) , 
\]
and:
\[
\per_\mu(\Omega) = \frac{1}{2} \sum_{i=1}^q \per_\mu(\Omega_i) = \sum_{i < j} A_{ij}(\Omega)  .
\]
In addition, if follows that:
\begin{equation} \label{eq:top-nothing-lost}
\forall i \;\;\; \overline{\partial^* \Omega_i} = \overline{\cup_{j \neq i} \Sigma_{ij}} .
\end{equation}
Indeed, since $\H^{n-1}|_{\partial^* \Omega_i} = \H^{n-1}|_{\cup_{j \neq i} \Sigma_{ij}}$ they must have the same support. But the support of $\H^{n-1}|_{\cup_{j \neq i} \Sigma_{ij}}$ is clearly contained in the right-hand-side of (\ref{eq:top-nothing-lost}), whereas 
$\H^{n-1}|_{\partial^* \Omega_i} = \abs{\mu_{\Omega_i}}$
by \cite[Theorem 15.9]{MaggiBook}, and hence its support is (e.g. \cite[Remark 15.3]{MaggiBook}) the left-hand-side of (\ref{eq:top-nothing-lost}); the converse inclusion is trivial. 

\subsection{Volume and perimeter regular sets} \label{subsec:variations}

\begin{definition}[Admissible Vector-Fields] \label{def:admissible}
    A vector-field $X$ on $\R^n$ is called \emph{admissible} if it is
    $C^\infty$-smooth and satisfies
    \begin{equation} \label{eq:field-bdd}
        \forall i \ge 0 \quad \max_{x \in \R^n}\|\nabla^i X(x)\| \le C_i < \infty.
    \end{equation}
\end{definition}
Any smooth, compactly supported vector-field (denoted $C_c^\infty$) is clearly admissible, and of course so is a constant vector-field; 
these two types and their linear combinations will be the main kinds of vector-fields which we will use in this work.

Let $F_t$ denote the associated flow along an admissible vector-field $X$,
defined as the family of maps $\set{F_t : \R^n \to \R^n}$ solving
the following ODE:
\begin{equation} \label{eq:flow-gen-by-X}
\frac{d}{dt} F_t(x) = X \circ F_t(x) ~,~ F_0(x) = x .
\end{equation}
It is well-known that a unique smooth solution in $t \in \Real$ exists for all
$x \in \R^n$, and that the resulting maps $F_t : \R^n \rightarrow \R^n$ are
$C^\infty$ diffeomorphisms, so that the partial derivatives in $t$ and $x$ of
any fixed order are uniformly bounded in $(x,t) \in \Real^n \times [-T,T]$, for
any fixed $T > 0$. 

Note that if $\Omega=(\Omega_1,\ldots,\Omega_q)$ is a cluster then so is its image
$F_t(\Omega) = (F_t(\Omega_1),\ldots,F_t(\Omega_q))$: obviously, its cells remain
Borel and pairwise disjoint; the fact that they have finite $\mu$-weighted perimeter
and satisfy $\mu(\Real^n \setminus \cup_i F_t(\Omega_i)) = \mu(F_t(\Real^n
\setminus \cup_i \Omega_i)) = 0$ follows from the fact that $F_t$ is a Lipschitz map.

We define the $r$-th variations of weighted volume and perimeter of $\Omega$ as:
\begin{align*}
  \delta_X^r V(\Omega) &:= \left. \frac{d^r}{(dt)^r}\right|_{t=0} \mu(F_t(\Omega)) ,\\
  \delta_X^r A(\Omega) &:= \left. \frac{d^r}{(dt)^r}\right|_{t=0} P_\mu(F_t(\Omega)) ,
\end{align*}
whenever the right-hand sides exist. When $\Omega$ is clear from the context, we will simply write $\delta_X^r V$ and $\delta_X^r A$; when $r = 1$, we will write $\delta_X V$ and $\delta_X A$. 

\smallskip

It will be of crucial importance for us in this work to calculate the first and especially
second variations of weighted volume and perimeter for \textbf{non-}compactly
supported vector-fields, for which even the existence of
$\delta_X^r V(\Omega)$ and especially $\delta_X^r A(\Omega)$ is not immediately
clear.  Indeed, even for the case of the standard Gaussian measure, the
derivatives of its density are asymptotically larger at infinity than the
Gaussian density itself.  We consequently introduce the following:

\begin{definition}[Volume / Perimeter Regular Set]
A Borel set $U$ is said to be \emph{volume regular} with respect to the measure $\mu = e^{-W(x)} dx$ if:
\[
\forall i,j \geq 0 \;\;\; \exists \delta > 0 \;\; \int_{U} \sup_{z \in B(x,\delta)} \norm{\nabla^i \pot(z)}^j e^{-\pot(z)} dx < \infty . 
\]
It is said to be \emph{perimeter regular} with respect to the measure $\mu$ if:
\[
\forall i,j \geq 0 \;\;\; \exists \delta > 0 \;\; \int_{\partial^* U} \sup_{z \in B(x,\delta)} \norm{\nabla^i \pot(z)}^j e^{-\pot(z)} d\H^{n-1}(x) < \infty .
\]
If $\delta > 0$ above may be chosen uniformly for all $i,j \geq 0$, $U$ is called \emph{uniformly} volume / perimeter regular. 
\end{definition}
\noindent
Here and throughout this work $\norm{\cdot} = \norm{\cdot}_2$ denotes the Hilbert-Schmidt norm of a tensor, defined as the square-root of the sum of squares of its coordinates in any local orthonormal frame. 
Note that volume (perimeter) regular sets clearly have finite weighted volume (perimeter). 

\medskip
Write $JF_t = \text{det}(dF_t)$ for the Jacobian of $F_t$, and observe that by the change-of-variables formula for smooth injective functions:
\begin{equation} \label{eq:Jac-vol}
\mu(F_t(U)) = \int_{U} J F_t  e^{-\pot \circ F_t} \, dx,
\end{equation}
for any Borel set $U$. Similarly, if $U$ is in addition of locally finite-perimeter, let $\Phi_t = F_t|_{\partial^* U}$ and write $J \Phi_t = \text{det}((d_{\n_U^{\perp}} F_t)^T d_{\n_U^{\perp}} F_t)^{1/2}$ for the Jacobian of $\Phi_t$ on $\partial^* U$. Since $\partial^* U$ is locally $\H^{n-1}$-rectifiable, \cite[Proposition 17.1 and Theorem 11.6]{MaggiBook} imply:
\begin{equation} \label{eq:Jac-area}
\mu^{n-1}(\partial^* F_t(U)) = \mu^{n-1}(F_t(\partial^* U)) = \int_{\partial^* U} J \Phi_t e^{-\pot \circ F_t} \, d\H^{n-1} . 
\end{equation}

\begin{lemma} \label{lem:regular}
\hfill
\begin{enumerate}
\item
If $U$ is volume regular with respect to $\mu$ and $X$ is admissible, then for any $r \geq 1$, $t \mapsto \mu(F_t(U))$ is $C^r$ in an open neighborhood of $t=0$, and:
\[
\delta_X^r V(U) = \int_{U} \left . \frac{d^r}{(dt)^r} \right |_{t=0}  (J F_t e^{-\pot \circ F_t}) dx .
\]
Furthermore, if $U$ is uniformly volume regular, or alternatively, if $X$ is $C_c^\infty$, then there exists an open neighborhood of $t=0$ where $t \mapsto \mu(F_t(U))$ is $C^\infty$ and the above formula holds. 
\item
If $U$ is perimeter regular with respect to $\mu$ and $X$ is admissible, then for any $r \geq 1$, $t \mapsto P_\mu(F_t(U))$ is $C^r$ in an open neighborhood of $t=0$, and:
\[
\delta_X^r A(U) = \int_{\partial^* U} \left . \frac{d^r}{(dt)^r} \right |_{t=0} (J \Phi_t e^{-\pot \circ F_t}) d\H^{n-1} .  \]
Furthermore, if $U$ is uniformly perimeter regular, or alternatively, if $X$ is $C_c^\infty$, then there exists an open neighborhood of $t=0$ where $t \mapsto P_\mu(F_t(U))$ is $C^\infty$ and the above formula holds. 
\end{enumerate}
\end{lemma}

\begin{proof}[Proof of Lemma \ref{lem:regular}]
Our task is to justify taking derivative inside the integral representations (\ref{eq:Jac-vol}) and (\ref{eq:Jac-area}). Taking difference quotients, applying Taylor's theorem with Lagrange remainder and induction on $r$, it is enough to establish by Dominant Convergence Theorem that for some $\eps > 0$:
\begin{align} 
\label{eq:dominant1} & \int_U \sup_{t \in [-\eps,\eps]} \frac{d^r}{(dt)^r} (J F_t(x) e^{-\pot(F_t(x))}) dx < \infty ~,\\
\label{eq:dominant2} & \int_{\partial^* U} \sup_{t \in [-\eps,\eps]} \frac{d^r}{(dt)^r} (J \Phi_t(x) e^{-\pot(F_t(x))}) d\H^{n-1}(x) < \infty .
\end{align}
By the Leibniz product rule, for $\F = F,\Phi$:
\[
\frac{d^r}{(dt)^r} (J \F_t(x) e^{-\pot(F_t(x))}) = \sum_{p+q=r} {r \choose p} \frac{d^p}{(dt)^p} J \F_t(x) \frac{d^{q}}{(dt)^{q}}  e^{-\pot(F_t(x))} .
\]
For each $x$, $t \mapsto J \F_t(x)$ is a smooth function of the differential $d F_t(x)$ (note that for $\F = \Phi$, the normal $n_{U}(x)$ remains fixed for all $t$).
This differential satisfies:
\[
\frac{d}{dt} dF_t(x) = \nabla X(F_t(x))  dF_t(x) , 
\]
and since $X$ satisfies (\ref{eq:field-bdd}), it follows that $\sup_{t \in [-\eps,\eps]} \frac{d^p}{(dt)^p} J \F_t(x)$ is uniformly bounded in $x \in \Real^n$ for all fixed $\eps > 0$ and $p$. Moreover, the latter expression has compact support $K$ if $X$ is assumed $C_c^\infty$. 

It remains to handle the $\frac{d^{q}}{(dt)^{q}}  e^{-\pot(F_t(x))}$ term. Repeated differentiation and application of the chain rule results in a polynomial expression in $\nabla^a \pot$ and $\nabla^b X$ times $e^{-\pot}$ evaluated at $F_t(x)$, where the degree of the polynomial, $a$ and $b$ are bounded by a function of $q$. Since $X$ satisfies (\ref{eq:field-bdd}), we may bound the magnitudes of $\nabla^b X$ by a constant depending on $q$. Now let $\delta > 0$ be arbitrary if $X$ is $C_c^\infty$, or be the corresponding parameter if $U$ is assumed volume or perimeter regular with respect to $\mu$ for appropriately bounded above $i,j$. 
Since $\abs{F_t(x) - x} \leq \abs{t} \max_{x \in \R^n} \abs{X(x)}$, we may find $\eps > 0$ so that $\abs{F_t(x) - x} \leq \delta$ uniformly in $x \in \R^n$. It follows that for an appropriate constant $D_r >0$:
\[
\sup_{t \in [-\eps,\eps]} \frac{d^r}{(dt)^r} (J \F_t(x) e^{-\pot(F_t(x))}) \leq D_r \sup_{z \in B(x,\delta)} \norm{\nabla^i \pot(z)}^j e^{-\pot(z)} ,
\]
and so the required (\ref{eq:dominant1}) and (\ref{eq:dominant2}) are established by definition of a regular set if $U$ is regular, or simply by integration inside the compact set $K$ if $X$ is $C_c^\infty$. Note that when the set is assumed to be uniformly regular or when $X \in C^\infty_c$, $\delta > 0$ and hence $\eps  > 0$ above may be chosen uniformly in $r \geq 1$, and $C^\infty$ smoothness is established for $t \in (-\eps,\eps)$. 
\end{proof}

\subsection{The quadratic form $L_A$ of interface measures}

Recall the definition of the $q \times q$ symmetric matrix $L_A$, associated to $A = \{ A_{ij} \}_{1 \leq i < j \leq q}$:
\begin{equation} \label{eq:L_A}
        L_A := \sum_{1 \leq i<j \leq q} A_{ij} (e_i - e_j) (e_i - e_j)^T .
                                  \end{equation}
We will mostly consider $L_A$ as a quadratic form on $E^{(q-1)}$. 

\begin{lemma}\label{lem:L-nondegenerate}
    Let $\Omega$ be a $q$-cluster with $\mu(\Omega) \in \interior \Delta^{(q-1)}$
    and let $A_{ij} = \mu^{n-1}(\Sigma_{ij})$. 
    Consider the undirected graph $G$ with vertices $\{1, \dots, q\}$ and an edge
    between $i$ and $j$ if $A_{ij} > 0$.
    Then $G$ is connected and $L_A$ is positive-definite as an operator on $E^{(q-1)}$ (in particular, it has full-rank $q-1$). 
    \end{lemma}
\begin{proof}
        The graph $G$ is connected because if
    $S \subset \{1, \dots, q\}$ were a non-trivial connected-component then
    $U = \bigcup_{i \in S} \Omega_i$ would satisfy $\mu(U) \in (0, 1)$, and hence $\gamma(U) \in (0,1)$. 
    On the other hand, we would have $\per_\mu(U) = \sum_{i \in S , j \notin S} A_{ij} = 0$ by (\ref{eq:nothing-lost-many}), and hence $\per_\gamma(U) = 0$, contradicting the (single-bubble) Gaussian isoperimetric inequality.

    Clearly $L_A$ is positive semi-definite as $v^T L_A v = \sum_{i \ne j} A_{ij} (v_i - v_j)^2$ with $A_{ij} \geq 0$. 
    Moreover, the fact that $G$ is connected implies that $v^T L_A v = 0$ for $v \in \R^q$ if and only
    if the coordinates of $v$ are all identical, which means that $v=0$ if $v \in E^{(q-1)}$. It follows that $L_A$
    is strictly positive-definite on $E^{(q-1)}$.
\end{proof}

The fact that $L_A$ is invertible on $E^{(q-1)}$ will play a crucial role in this work.

\section{Isoperimetric Minimizing Clusters -- Interface Regularity} \label{sec:minimizers}

A cluster $\Omega$ is called an \emph{isoperimetric minimizer with respect to
$\mu$} (or simply $\mu$-\emph{minimizing}) if $\per_\mu(\Omega') \ge \per_\mu(\Omega)$
for every other cluster $\Omega'$ satisfying $\mu(\Omega') =
\mu(\Omega)$.  We will at times invoke the following additional assumption: \begin{equation} \label{eq:mu-regular}
\begin{array}{l}
\text{$\mu$ is a probability measure for which all cells of any isoperimetric}\\
\text{minimizing cluster are volume and perimeter regular.}
\end{array}
\end{equation}
It is shown in Corollary \ref{cor:Gaussian-regular} in the Appendix that the Gaussian
measure $\gamma$ satisfies~\eqref{eq:mu-regular}. More precisely, it is shown that any Borel set is uniformly volume regular, and that the cells of any minimizing cluster are uniformly perimeter regular (all with respect to $\gamma$).

\subsection{Existence and interface-regularity}

The following theorem is due to Almgren \cite{AlmgrenMemoirs} (see also~\cite[Chapter 13]{MorganBook5Ed} and~\cite[Chapters 29-30]{MaggiBook}). 

\begin{theorem}[Almgren] \label{thm:Almgren}
    Let $\mu = e^{-W} dx$ with $W \in C^\infty(\R^n)$.  
    \begin{enumerate}[(i)]
\item \label{it:Almgren-i}
   If $\mu$ is a probability measure, then for any prescribed $v \in \simplex^{(q-1)}$, an isoperimetric $\mu$-minimizing $q$-cluster $\Omega$ satisfying $\mu(\Omega) = v$ exists. 
   \end{enumerate}
    For every isoperimetric $\mu$-minimizing cluster $\Omega$: 
    \begin{enumerate}[(i)] \setcounter{enumi}{1}
\item \label{it:Almgren-ii}
  $\Omega$ may and will be modified on a $\mu$-null set (thereby not altering $\set{\partial^* \Omega_i}$) so that all of its cells are open, and so that for every $i$, 
$\overline{\partial^* \Omega_i} = \partial \Omega_i$  and $\mu^{n-1}(\partial \Omega_i \setminus \partial^* \Omega_i) = 0$. 
\item \label{it:Almgren-iii} 
    For all $i \neq j$ the interfaces $\Sigma_{ij} = \Sigma_{ij}(\Omega)$ are
    $C^\infty$-smooth $(n-1)$-dimensional manifolds, relatively open in
    $\Sigma := \bigcup_k \partial \Omega_k$, and for every $x \in
    \Sigma_{ij}$ there exists $\epsilon > 0$ such that $B(x,\epsilon) \cap
    \Omega_k = \emptyset$ for all $k \neq i,j$. 
\item \label{it:Almgren-density}
    For any compact set $K$ in $\R^n$, there exist constants $\Lambda_K,r_K > 0$ so that:
    \begin{equation} \label{eq:density}
    \mu^{n-1}(\Sigma \cap B(x,r)) \leq \Lambda_K r^{n-1} \;\;\; \forall x \in \Sigma \cap K \;\;\; \forall r \in (0,r_K) . 
    \end{equation}
\item \label{it:Almgren-M} For any open bounded set $U$, $\Sigma \cap U$ is an $(\M,\eps(r) = \Lambda_U r, \delta_U)$-minimizing set in $U$. 
\end{enumerate}
\end{theorem}
We refer to \cite{AlmgrenMemoirs,MorganBook5Ed,CES-RegularityOfMinimalSurfacesNearCones} for Almgren's definition of an $(\M,\eps,\delta)$-minimizing set, which we will not directly require for the purposes of this work. Whenever referring to the cells of a minimizing cluster or their topological boundary in this work, we will always choose a representative such as in Theorem \ref{thm:Almgren} \ref{it:Almgren-ii}. 
\begin{proof}[Proof of Theorem \ref{thm:Almgren}]
\hfill
\begin{enumerate}[(i)]
\item
  It is well-known (e.g.~\cite[Proposition~12.15]{MaggiBook}) that the (weighted)
  perimeter is lower semi-continuous with respect to (weighted) $L^1$
  convergence: if $U^r \to U$ in $L^1(\mu)$ then $\liminf_r \per_\mu(U^r) \geq
  P_\mu(U)$ for all Borel sets $U^r$ of finite (weighted) perimeter. Clearly, the
  same applies to clusters, where $L^1(\mu)$ convergence is understood for each of
  the individual cells.
  
   It is also well-known that, since $\mu$ has finite mass, the set $\set{ U \in \B(\R^n) : P_\mu(U) \leq C}$ (where
  $\B(\R^n)$ denotes the collection of Borel subsets of $\R^n$) is compact in $L^1(\mu)$ -- for bounded sets, this follows from ~\cite[Theorem~12.26]{MaggiBook}, and the general case follows by truncation with a large ball and a standard diagonalization argument (see e.g. \cite[Theorem 2.1]{RitoreRosalesMinimizersInEulideanCones}).
  
  Set $I(v) := \inf \set{ P_\mu(\Omega) : \text{$\Omega$ is a $q$-cluster with }\mu(\Omega) = v}$. As the latter set is clearly non-empty, obviously $I(v) < \infty$. Given $v \in \Delta^{(q-1)}$, let $\Omega^r$ be a sequence of $q$-clusters with $\mu(\Omega^r) = v$ and $P_\mu(\Omega^r) \rightarrow I(v)$. As $P_\mu(\Omega^r_i) \leq P_\mu(\Omega^r) \leq I(v)+1$ for large enough $r$, by passing to a subsequence, it follows that each of the cells $\Omega^r_i$ converges in $L^1(\mu)$ to $\Omega_i$. By Dominated Convergence (as the total mass is finite), we must have $\mu(\Omega_i) = v_i$, and the limiting $\Omega$ is easily seen to be a cluster (possibly after a measure-zero modification to ensure disjointness of the cells). It follows by lower semi-continuity that:
  \[
    I(v) \le P_\mu(\Omega) \le \liminf_{r \to \infty} P_\mu(\Omega^r) = I(v),
  \]
  and consequently $P_\mu(\Omega) = I(v)$. Hence $\Omega$ is a minimizing cluster with $\mu(\Omega) = v$. 
  
  Note that the proof is much simpler than the one in the unweighted setting, where the total mass is infinite, but on the other hand perimeter and volume are translation-invariant. 
  
  \item That $\Omega$ may be chosen so that 
$\overline{\partial^* \Omega_i} = \partial \Omega_i$ and
  $\mu^{n-1}(\partial \Omega_i \setminus \partial^* \Omega_i) = 0$ for all $i$ follows from \cite[Theorems 12.19 and 30.1]{MaggiBook}; the proof of the latter carries over to the weighted setting (see below). In particular, the topological boundary of each cell has zero $\mu^n$-measure, and so by replacing each cell by its interior, we do not change its measure nor its reduced boundary (and hence its $\mu$-weighted perimeter), so $\Omega$ remains an isoperimetric minimizer. 
  In addition, $\partial \interior \Omega_i \supseteq  \overline{\partial^* \interior \Omega_i} =  \overline{\partial^* \Omega_i} = \partial \Omega_i  \supseteq \partial \interior \Omega_i$
 and so the topological boundary remains unaltered and the above still holds.
        \item
 The assertions follow from \cite[Theorem 30.1, Lemma 30.2 and Corollary 3.3]{MaggiBook}, whose proofs carry over to the weighted setting. 
  Indeed, all of the arguments used in those proofs are local in nature, and so as long as our density $e^{-W}$ is $C^{\infty}$-smooth and positive, and hence locally bounded above and below (away from zero), the proofs of  the relative openness of $\Sigma_{ij}$ in $\partial \Omega_i \cap \partial \Omega_j$, and of the disjointness of $B(x,\eps)$ from $\Omega_k$ carry over by adjusting constants (which are local). To justify the relative openness of $\Sigma_{ij}$ in $\cup_k \partial \Omega_k$, note that $\partial \Omega_i \subset \bigcup_{k \neq i} \partial \Omega_k$, 
and hence $\partial \Omega_i \setminus \partial \Omega_j \subset \cup_{k \neq i,j} \partial \Omega_k$. It follows that given $x \in
    \Sigma_{ij}$ and $\eps > 0$ such that $B(x,\eps) \cap \Omega_k = \emptyset$ for all $k \neq i,j$, we have:
    \[
    B(x,\eps/2) \cap \cup_k \partial \Omega_k = B(x,\eps/2) \cap  (\partial \Omega_i \cup \partial \Omega_j) = B(x,\eps/2) \cap  \partial \Omega_i \cap \partial \Omega_j , 
    \]
    and so the relative openness in $\cup_k \partial \Omega_k$ is equivalent to that in $\partial \Omega_i \cap \partial \Omega_j$, already established above. 
 
 By \cite[Corollary 3.3 and Remark 3.4]{MaggiBook}, we also know that for any $x \in \Sigma_{ij}$ there exists $r_x > 0$ so that $\Omega_i$ and $\Omega_j$ are measure-constrained weighted-perimeter minimizers in $B(x,r_x)$. Consequently, by regularity theory for volume-constrained perimeter minimizers (see Morgan \cite[Section 3.10]{MorganRegularityOfMinimizers} for an adaptation to the weighted Riemannian setting), since our density is $C^{\infty}$-smooth, it follows that $\Sigma_{ij} \cap B(x,r_x) = \partial \Omega_i \cap \partial \Omega_j \cap B(x,r_x)$ is a $C^{\infty}$-smooth $(n-1)$-dimensional submanifold.

\item For a proof in the unweighted setting see \cite[Lemma 30.5]{MaggiBook} or \cite[Example 21.3 and Theorem 21.11]{MaggiBook}; the proof easily transfers to the weighted setting on any compact set $K$, where the density is bounded between two positive constants (depending on $K$), only resulting in modified constants. 

\item It is well known in the unweighted setting that $\Sigma$ is $(\M, \epsilon(r) = \Lambda r, \delta)$-minimizing for any minimizing cluster $\Omega$ (see \cite[Theorem 3.8]{CES-RegularityOfMinimalSurfacesNearCones} and the references therein), and by inserting the effect of the smooth positive density into the excess function, it follows that the same holds in the weighted setting inside any bounded open set $U$.
   
\end{enumerate}
\end{proof}

\begin{definition}[interface--regular cluster] \label{def:interface-regular}
A cluster $\Omega$ satisfying parts \ref{it:Almgren-ii} and \ref{it:Almgren-iii} of Theorem \ref{thm:Almgren}
is called interface--regular. \end{definition}

The definition of interface-regular cluster should not be confused with the stronger definition of regular cluster, introduced in the next section (cf. Definition \ref{def:regular}).  
Given an interface--regular cluster, let $\n_{ij}$ be the (smooth) unit normal field along $\Sigma_{ij}$ that
points from $\Omega_i$ to $\Omega_j$. We use $\n_{ij}$ to co-orient $\Sigma_{ij}$, and since $\n_{ij} = -\n_{ji}$, note that $\Sigma_{ij}$ and $\Sigma_{ji}$ have opposite orientations.
When $i$ and $j$ are clear from the context, we will simply write $\n$. 
We will typically abbreviate $H_{\Sigma_{ij}}$ and $H_{\Sigma_{ij},\mu}$ by $H_{ij}$ and $H_{ij,\mu}$, respectively.

\subsection{Stationarity and Stability}

Using Theorem \ref{thm:Almgren} and Lemma \ref{lem:regular}, a standard first variation argument (see Appendix \ref{app:stst}) gives
necessary first-order conditions for the minimality of a cluster.

\begin{lemma}[First-order conditions] \label{lem:first-order-conditions}
  For any $\mu$-minimizing cluster $\Omega$:
  \begin{enumerate}[(i)]
    \item On each $\Sigma_{ij}$, $H_{ij,\mu}$ is constant. \label{it:first-order-constant}
    \item There exists $\lambda \in E^*$ such that $H_{ij,\mu} = \lambda_i - \lambda_j$ for all $i \neq j$; moreover, $\lambda \in E^*$ is unique whenever $\mu(\Omega) \in \interior \simplex$. \label{it:first-order-cyclic}
                                \item \label{it:weak-angles}
     For every $C_c^\infty$ vector-field $X$:
      \[
        \sum_{i<j} \int_{\Sigma_{ij}} \div_{\Sigma,\mu} X^\tang\, d\mu^{n-1} = 0.
      \]
      Moreover, assuming the cells of $\Omega$ are volume and perimeter regular with respect to $\mu$, the above holds for any admissible vector-field $X$.
          \end{enumerate}
\end{lemma}
 
\begin{definition}[Stationary Cluster]
An interface--regular cluster $\Omega$ satisfying the three conclusions of Lemma \ref{lem:first-order-conditions} is called stationary (with respect to $\mu$ and with Lagrange multiplier $\lambda \in E^*$), or $\mu$-stationary. 
\end{definition}

Indeed, the following lemma provides an insightful interpretation of $\lambda \in E^*$ as a Lagrange multiplier for the isoperimetric constrained minimization problem (see Appendix \ref{app:stst}):
\begin{lemma}[Lagrange Multiplier] \label{lem:Lagrange}      Let $\Omega$ be stationary cluster with respect to $\mu$ and with Lagrange multiplier $\lambda \in E^*$.
    Then for every $C_c^\infty$ vector-field $X$: 
    \begin{align} 
    \label{eq:first-variation-volume}
     \delta_X V(\Omega)_i &= \sum_{j \neq i} \int_{\Sigma_{ij}} X^{\n_{ij}}\, d\mu^{n-1} \;\;\; \forall i ,\\
     \nonumber
      \delta_X A(\Omega) &= \sum_{i<j} H_{ij,\mu} \int_{\Sigma_{ij}} X^{\n_{ij}}\, d\mu^{n-1} .
    \end{align} 
    In particular, 
    \[
      \delta_X A = \inr{\lambda}{\delta_X V} .
    \]
  Moreover, assuming the cells of $\Omega$ are volume and perimeter regular with respect to $\mu$, the above holds for any admissible vector-field $X$. 
\end{lemma}

\medskip
Similarly, a standard second variation argument (see Appendix \ref{app:stst}) gives
necessary second-order conditions for the local minimality of a cluster.

\begin{definition}[Index Form $Q$] The Index Form $Q$ associated to a stationary cluster with respect to $\mu$ and with Lagrange multiplier $\lambda \in E^*$ is defined as the following quadratic form:
\[
Q(X) := \delta^2_X A - \scalar{\lambda,\delta^2_X V} ,
\]
defined on vector-fields $X$ for which the right-hand-side is well defined. 
\end{definition}

\begin{lemma}[Stability]\label{lem:unstable}
For any $\mu$-minimizing cluster $\Omega$ and $C_c^\infty$ vector-field $X$:
    \begin{equation}  \label{eq:Q}
      \delta_X V = 0 \;\;\; \Rightarrow \;\;\; Q(X) \ge 0 .
     \end{equation}
 Moreover, assuming the cells of $\Omega$ are volume and perimeter regular with respect to $\mu$, the above holds for any admissible vector-field $X$.
\end{lemma}

\begin{definition}[Stable Cluster]
A $\mu$-stationary cluster satisfying the conclusion of Lemma \ref{lem:unstable} 
is called stable (with respect to $\mu$), or $\mu$-stable.
\end{definition}

\subsection{Second Variation under Translations} 

Unlike the formulas for first variation in Lemma \ref{lem:Lagrange} above, the following identity for the second
variation under translations appears to be new, and moreover, reveals a crucial feature of the Gaussian measure. 

Let $\delta_w$ denotes the variation under the translation of the cluster generated by the constant vector-field $X \equiv w \in \R^n$ (note that it is not compactly supported). The following formula follows from the calculations in Appendix \ref{app:translations}:

\begin{lemma} \label{lem:cor-Q} 
Let $\Omega$ be a stationary cluster with respect to $\mu = e^{-\pot} dx$ and with Lagrange multiplier $\lambda \in E^*$, and assume its cells are volume and perimeter regular. Then for any $w \in \R^n$:
\[
Q(w) = \delta_w^2 A - \inr{\lambda}{\delta_w^2 V} = -  \sum_{i < j} \int_{\Sigma_{ij}} \scalar{w, \n_{ij}} \nabla^2_{\n_{ij},w} \pot  d\gamma^{n-1} . 
\]
\end{lemma}

\noindent
Specializing this to the Gaussian measure $\gamma$, and recalling that any Borel set is volume regular with respect to $\gamma$, we immediately deduce (as $\nabla^2 W = \text{Id}$):

\begin{theorem}\label{thm:trans-formula}
  If $\Omega$ is a Gaussian-stationary cluster whose cells are perimeter regular, then for any $w \in \R^n$:
    \begin{equation}
        Q(w) = \delta_w^2 A - \inr{\lambda}{\delta_w^2 V} = - \sum_{i < j} \int_{\Sigma_{ij}} \inr{w}{\n_{ij}}^2\, d\gamma^{n-1} \leq 0 . 
            \label{eq:trans-formula}
    \end{equation}
\end{theorem}

\begin{remark} \label{rem:Q-sign} 
We see that a crucial feature of the Gaussian measure, which we will constantly use in this work, is that $Q$ is negative semi-definite under translation-fields. 
Note that for a measure $\gamma = e^{-\pot} dx$  with $\nabla^2 W \geq \Id$, it is in general \emph{false} that:
\[
Q(w) \leq  -  \sum_{i < j} \int_{\Sigma_{ij}} \inr{w}{\n_{ij}}^2 d\gamma^{n-1} ,
\]
or even that $Q$ is negative semi-definite. Consequently, our approach in this work cannot directly handle such measures, in sharp contrast to a variety of methods which are able to  handle such measures in the single-bubble setting (cf. \cite{BakryLedoux,BobkovLocalizedProofOfGaussianIso}). 
\end{remark}

\subsection{Effective Dimension of Gaussian-Stable Clusters} \label{subsec:M}

\begin{definition}[Effective Dimension and Dimension of Deficiency]  \label{def:effective} Given an interface--regular cluster $\Omega$, consider the linear span $N$ of $\set{\n_{ij}(x)}$ for all $x \in \Sigma_{ij}$ and all $i \neq j$. 
The dimension $e$ of $N$ is called the effective dimension of $\Omega$ and $\Omega$ is said to be effectively $e$-dimensional. The dimension of $N^{\perp}$ is called the dimension of deficiency of $\Omega$, and $\Omega$ is said to be dimension-deficient in the directions of $N^{\perp}$. If $N= \R^n$, $\Omega$ is called full-dimensional, otherwise it is called dimension-deficient. 
\end{definition}

\begin{remark}\label{rem:empty-interfaces}
Clearly the effective dimension is always positive, unless all the $\Sigma_{ij}$'s are empty, which by (\ref{eq:top-nothing-lost}) would imply that $\partial^* \Omega_i = \emptyset$ and hence $\gamma^{n-1}(\partial^* \Omega_i) = 0$ for all $i=1,\ldots,q$, and so by the single-bubble Gaussian isoperimetric inequality the cluster must (up to null-sets and rearranging indices) be the trivial cluster $\Omega = (\R^n,\emptyset,\ldots,\emptyset)$. For consistency in the statements below, we say that $\Omega$ is a $q$-cluster on the 0-dimensional subspace $\{0\}$ if (up to rearranging indices) $\Omega = ( \{0\} , \emptyset, \ldots, \emptyset)$. 
\end{remark}

\begin{lemma} \label{lem:dim-deficient}
An interface--regular $q$-cluster $\Omega$ is dimension-deficient in the direction $\theta \in \mathbb{S}^{n-1}$ if and only if there exists an interface--regular $q$-cluster $\tilde \Omega$ in $\theta^{\perp}$, so that $\Omega = \tilde \Omega \times \R$ up to null-sets. 
\end{lemma}
\begin{proof}
By Remark \ref{rem:empty-interfaces} there is nothing to prove if $n=1$, so assume $n \geq 2$. 
The ``if" direction is obvious as $\partial^* (\tilde \Omega_i \times \R) = (\partial^* \Omega_i) \times \R$ and modification of a cluster by null-sets does not change its interfaces. 
For the ``only if" direction, assume without loss of generality that the cluster $\Omega$ is dimension-deficient in the direction $e_n \in \mathbb{S}^{n-1}$. Observe that by the Gauss--Green--De Giorgi theorem (\ref{eq:integration-by-parts}) and (\ref{eq:nothing-lost}):
\[
\int_{\Omega_i} \partial_{x_n}\varphi \, d\H^{n} = \int_{\partial^* \Omega_i} \varphi \scalar{e_n,\n_{\Omega_i}} d\H^{n-1} = \sum_{j \neq i} \int_{\Sigma_{ij}} \varphi \scalar{e_n, \n_{ij}} d\H^{n-1} =  0 ,
\]
for all $\varphi \in C^1_c(\R^n)$. Applying this to $-\varphi_{\eps}(x-\cdot)$ where $\varphi_\eps$ is a compactly-supported approximation of identity, and denoting $u_\eps := 1_{\Omega_i} \ast \varphi_\eps$, it follows that:
\[
\partial_{x_n} u_\eps (x) = \int_{\Omega_i} \partial_{x_n} \varphi_{\eps}(x-y) \, d\H^{n}(y) = 0 \;\;\; \forall x \in \R^n . 
\]
Since $u_\eps \in C^\infty(\R^n)$, it follows that $u_{\eps}(x) = \tilde u_\eps(\tilde x)$ for some $\tilde u_\eps \in C^\infty(\R^{n-1})$, where $\tilde x = (x_1,\ldots,x_{n-1})$ and $x = (\tilde x , x_n)$.  But since $u_\eps \rightarrow 1_{\Omega_i}$ in $L^1_{loc}(\R^n)$ as $\eps \rightarrow 0$, it follows by Fubini's theorem that 
for a.e. $\tilde x \in \R^{n-1}$, for a.e. $x_n \in \R$, $1_{\Omega_i}(\tilde x,x_n) = 1_{\tilde \Omega_i}(\tilde{x})$ for some Borel set $\tilde \Omega_i$ in $\R^{n-1}$. In other words,  $\Omega_i$ coincides with $\tilde \Omega_i \times \R$ up to a null-set. Since $\{\Omega_i\}$ are disjoint, by modifying $\{ \tilde \Omega_i \}$ on null-sets, we can ensure that they are also disjoint. Since $\partial^* (\tilde \Omega_i \times \R) = (\partial^* \Omega_i) \times \R$, it follows that $P_{\tilde \mu}(\tilde \Omega_i) = P_{\mu}(\Omega_i) < \infty$ and $\tilde \mu(\R^{n-1} \setminus \cup_i \tilde \Omega_i) = \mu(\R^{n} \setminus \cup_i \Omega_i) = 0$, where $\tilde \mu$ denotes the marginal of $\mu$ on $\theta^{\perp}$, verifying that $\tilde \Omega$ is indeed a legal cluster (with respect to $\tilde \mu$). Clearly $\tilde \Omega$ is interface--regular if and only if $\Omega$ is. 
\end{proof}

Applying Lemma \ref{lem:dim-deficient} recursively, we obtain:
\begin{corollary}
The effective dimension of an interface--regular $q$-cluster $\Omega$ is the minimal $e$ so that there exists an $e$-dimensional subspace $F$ and a $q$-cluster $\tilde \Omega$ in $F$ with the property that $\Omega = \tilde \Omega \times F^{\perp}$ up to null-sets. 
\end{corollary}

We will make use of the following linear operator associated to a $q$-cluster $\Omega$ (whose cells are volume regular):
\[
 M : \R^n \to E^{(q-1)} ~,~ M w := \delta_w V(\Omega) ,
\]
which accounts for the first variation in (weighted) volume under translation of the cluster in the direction $w \in \R^n$. Recall that by Corollary \ref{cor:Gaussian-regular}, any Borel set is volume regular with respect to the Gaussian measure $\gamma$. 

\medskip
The fact that for the Gaussian measure, $Q$ is negative semi-definite under translation-fields, allows us to completely characterize a Gaussian-stable cluster's effective dimension in terms of the rank of $M$.

\begin{lemma} \label{lem:MKernelAndDef}
For any Gaussian-stable $q$-cluster $\Omega$ whose cells are perimeter regular, $\Ker M$ coincides with $\Omega$'s directions of dimension-deficiency. 
In particular, if $d$ denotes $\Omega$'s dimension of deficiency and $e$ denotes its effective dimension, then:
\begin{enumerate}[(i)]
\item $d = \dim \Ker M$ and $e = \rank M$. 
\item $\Omega$ is full-dimensional iff  $M$ is injective.
\item $\Omega$ is effectively $(q-1)$-dimensional iff $M$ is surjective. 
\item The effective dimension of $\Omega$ is at most $q-1$. 
\item If $q < n+1$ then $e < n$ and $d > 0$, so $\Omega$ must be dimension-deficient. 
\end{enumerate}
\end{lemma}
\begin{proof}
Clearly, if $w$ is a direction of dimension-deficiency then $\delta_w V = 0$ by (\ref{eq:first-variation-volume}). To see the other implication, recall that by Theorem \ref{thm:trans-formula}, the Gaussian measure satisfies:     \[
        Q(w) = -\sum_{i < j} \int_{\Sigma_{ij}} \inr{w}{\n_{ij}}^2 \, d\gamma^{n-1} \leq 0 ,
    \]
for any $w \in \R^n$ and Gaussian-stationary cluster (whose cells are perimeter regular). On the other hand, stability implies that if $\delta_w V = 0$ then $Q(w) \geq 0$, which means that $w$ must be perpendicular to all normals $\n_{ij}$, and is thus a direction of dimension-deficiency. Since $M : \R^n \to E^{(q-1)}$, we have $\dim \Img M = n - \dim \Ker M = n - d = e$, and $e \leq \dim E^{(q-1)} = q-1$.  
\end{proof}

Thanks to the product structure of the Gaussian measure, Lemma \ref{lem:MKernelAndDef} already reduces the proof of Gaussian Multi-Bubble Conjecture from $\R^n$ to $\R^{k}$ where $k$ is a minimizing cluster's maximal effective dimension, namely to $k = q-1 \leq n$. However, to handle the case $q > 3$, we will see in Section \ref{sec:stable} that there is actually some advantage in working with dimension-deficient clusters.

\section{Isoperimetric Minimizing Clusters - Higher Codimension Regularity} \label{sec:higher}

To handle the case $3 < q \leq n+1$ of the Gaussian Multi-Bubble Conjecture, we will need more information on the structure and regularity of the higher codimensional boundary of an isoperimetric minimizing cluster. \textbf{The reader only interested in the simplified proof of the Double-Bubble case $q=3$ may safely skip over the present and subsequent sections, and jump straight to Section \ref{sec:double-bubble}.} 

\medskip

Given a minimizing cluster $\Omega$, denote: 
\[
    \Sigma := \cup_{i} \partial \Omega_i ~,~  \Sigma^1 := \cup_{i < j} \Sigma_{ij} ,
\]
and observe that $\Sigma = \overline{\Sigma^1}$ by (\ref{eq:top-nothing-lost}) and our convention from Theorem \ref{thm:Almgren} \ref{it:Almgren-ii}.
We will also require additional information on the lower-dimensional structure of $\Sigma$. To this end, define two special cones:
\begin{align*}
    \Y &= \{x \in E^{(2)}: \text{ there exist $i \ne j \in \{1,2,3\}$ with $x_i = x_j = \max_{k \in \{1,2,3\}} x_k$}\} , \\
    \T &= \{x \in E^{(3)}: \text{ there exist $i \ne j \in \{1,2,3,4\}$ with $x_i = x_j = \max_{k \in \{1,2,3,4\}} x_k$}\}.
\end{align*}
In other words, $\Y$ is the boundary of a model $3$-cluster in $E^{(2)}$ and $\T$ is the boundary of a model $4$-cluster in $E^{(3)}$. Note that $\Y$ consists of 
$3$ half-lines meeting at the origin in $120^\circ$ angles, and that $\T$ consists of $6$ two-dimensional sectors meeting in threes at $120^{\circ}$ angles along $4$ half-lines, which in turn all meet at the origin in $\cos^{-1}(-1/3) \simeq 109^{\circ}$ angles. 
It turns out that on the codimension-$2$ and codimension-$3$ parts of a minimizing cluster, $\Sigma$ locally looks
like $\Y \times \R^{n-2}$ and $\T \times \R^{n-3}$, respectively.

\begin{theorem}[Taylor, White, Colombo--Edelen--Spolaor] \label{thm:regularity}
    Let $\Omega$ be a minimizing cluster for the measure $\mu = \exp(-W) dx$ in $\R^n$ with $W \in C^\infty(\R^n)$.         Then there exist $\alpha > 0$ and sets $\Sigma^2, \Sigma^3, \Sigma^4 \subset \Sigma$ such that:
    \begin{enumerate}[(i)]
        \item $\Sigma$ is the disjoint union of $\Sigma^1,\Sigma^2,\Sigma^3,\Sigma^4$;             \label{it:regularity-union}
        \item \label{it:regularity-Sigma2}         $\Sigma^2$ is a locally-finite union of embedded $(n-2)$-dimensional $C^{1,\alpha}$ manifolds, and for every $p
            \in \Sigma^2$ there is a $C^{1,\alpha}$ diffeomorphism mapping a neighborhood of
            $p$ in $\R^n$ to a neighborhood of the origin in $E^{(2)} \times \R^{n-2}$, so that $p$ is mapped to the origin and $\Sigma$ is locally mapped to $\Y \times \R^{n-2}$ ; 
        \item \label{it:regularity-Sigma3}
         $\Sigma^3$ is a locally-finite union of embedded $(n-3)$-dimensional $C^{1,\alpha}$ manifolds, and for every $p
            \in \Sigma^3$ there is a $C^{1,\alpha}$ diffeomorphism mapping a neighborhood of
            $p$ in $\R^n$ to a neighborhood of the origin in $E^{(3)} \times \R^{n-3}$, so that $p$ is mapped to the origin and $\Sigma$ is locally mapped to $\T \times \R^{n-3}$ ;
                \item $\Sigma^4$ is closed and $\dim_H(\Sigma^4) \leq n-4$.             \label{it:regularity-dimension}
    \end{enumerate}
\end{theorem}
\begin{remark}
Clearly, when $n=2$, necessarily $\Sigma^2$ is discrete and $\Sigma^3 = \Sigma^4 = \emptyset$, and when $n=3$, necessarily $\Sigma^3$ is discrete and $\Sigma^4 = \emptyset$. 
The $C^{1,\alpha}$ regularity in \ref{it:regularity-Sigma2} will be improved to $C^\infty$ regularity in Corollary \ref{cor:smooth-Sigma2} using an argument of Kinderlehrer--Nirenberg--Spruck \cite{KNS}. The work of Naber and Valtorta \cite{NaberValtorta-MinimizingHarmonicMaps} implies that $\Sigma^4$ is actually $\H^{n-4}$-rectifiable and has locally-finite $\H^{n-4}$ measure, but we will not require this here. 
 \end{remark}
\begin{proof}[Proof of Theorem \ref{thm:regularity}]
Let us first give references in the classical unweighted setting (when $\Omega$ is a minimizing cluster with respect to the Lebesgue measure in $\R^n$). 
The case $n=2$ was shown by F.~Morgan in \cite{MorganSoapBubblesInR2} building upon the work of Almgren \cite{AlmgrenMemoirs}, and also follows from the results of  J.~Taylor~\cite{Taylor-SoapBubbleRegularityInR3}. 
  The case $n=3$  was established by Taylor \cite{Taylor-SoapBubbleRegularityInR3} for general $(\M,\eps,\delta)$ sets in the sense of Almgren. 
 When $n \geq 4$, Theorem \ref{thm:regularity} was announced by B.~White \cite{White-AusyAnnouncementOfClusterRegularity,White-SoapBubbleRegularityInRn} for general $(\M,\eps,\delta)$ sets. 
Theorem \ref{thm:regularity} with part \ref{it:regularity-Sigma3} replaced by $\dim_H(\Sigma^3) \leq n-3$ follows from the work of L.~Simon \cite{Simon-Codimension2Regularity}.
A version of Theorem \ref{thm:regularity} for multiplicity-one integral varifolds in an open set $U \subset \R^n$ having associated cycle structure, no boundary in $U$, bounded mean-curvature and whose support is $(\M,\eps,\delta)$ minimizing, was very recently established by M.~Colombo, N.~Edelen and L.~Spolaor  \cite[Theorem~1.3, Remark~1.4, Theorem~3.10]{CES-RegularityOfMinimalSurfacesNearCones}; in particular, this applies to isoperimetric minimizing clusters in $\R^n$ \cite[Theorem~3.8]{CES-RegularityOfMinimalSurfacesNearCones}, yielding Theorem \ref{thm:regularity} when $\mu$ is the Lebesgue measure.

All the the above regularity results equally apply in the weighted setting. 
Let us explain how to deduce Theorem \ref{thm:regularity} from the results of \cite{CES-RegularityOfMinimalSurfacesNearCones}.
Given an open set $U \subset \R^n$ and a set $\Sigma_U \subset U$ satisfying a certain requirement, to be described next, \cite[Theorem~3.10]{CES-RegularityOfMinimalSurfacesNearCones} asserts a stratification of $\Sigma_U$ into a disjoint union of $\Sigma^1_U,\Sigma^2_U,\Sigma^3_U,\Sigma^4_U$, where $\Sigma^i_U$ for $i=1,2,3$ is a locally-finite union of $C^{1,\alpha}$ $(n-i)$-dimensional manifolds for some $\alpha > 0$ depending only $n$, with $\Sigma_U$ diffeomorphic to $\R^{n-1}$, $Y \times \R^{n-2}$ and $\T \times \R^{n-3}$ near $\Sigma^1_U$, $\Sigma^2_U$ and $\Sigma^3_U$, respectively, and with $\Sigma^4_U$ satisfying \ref{it:regularity-dimension} in $U$. The stratification of $\Sigma_U$ into $\Sigma^i_U$ may be described explicitly using symmetry properties of the possible tangent cones of $\Sigma_U$ at a given $x \in \Sigma_U$, but we will not require this here; it suffices to note that the stratification is determined by the \emph{local} structure of $\Sigma_U$, so that if $\Sigma_U \subset \Sigma_V$ with $U \subset V$ open, then $\Sigma_U^i \subset \Sigma_V^i$ for all $i=1,2,3,4$. 

It will thus be enough to establish the above decomposition for $\Sigma_U := \Sigma \cap U$, where $U = B_R$ is an open ball of arbitrary radius $R > 0$. Indeed, setting $\tilde{\Sigma}^i := \cup_{R > 0} \Sigma^i_{B_R}$, it would follow that $\{\tilde{\Sigma}^i\}_{i=1,2,3,4}$ satisfy all the asserted properties \ref{it:regularity-union} through \ref{it:regularity-dimension} of Theorem \ref{thm:regularity}. Moreover, recalling that $\Sigma^1 = \cup_{ij} \Sigma_{ij}$, clearly $\Sigma^1$ is disjoint from $\tilde{\Sigma}^2$ and $\tilde{\Sigma}^3$ since $\Sigma$ is diffeomorphic to $\R^{n-1}$ near $\Sigma^1$, and therefore $\Sigma^1 \subset \tilde{\Sigma}^1 \cup \tilde{\Sigma}^4$. On the other hand, we claim that $\tilde{\Sigma}^1 \subset \Sigma^1$. To see this, let $p \in \tilde{\Sigma}^1$, and let $N_p$ be a neighborhood of $p$ where $\Sigma$ is locally diffeomorphic to $\{x_n = 0\} \subset \R^n$. 
 Applying our local diffeomorphism, we conclude that there exist $j \neq k$ so that $p \in \partial \Omega_j \cap \partial \Omega_k$, and on $N_p$, $\Omega_j$ is mapped to $\{x_n < 0\}$ and $\Omega_k$ is mapped to $\{ x_n > 0 \}$. Pulling back via our diffeomorphism (and since for open sets with $C^1$ smooth boundary, the reduced boundary coincides with the topological one), it follows that $p \in \partial^* \Omega_j \cap \partial^* \Omega_k = \Sigma_{jk} \subset \Sigma^1$, as asserted. Hence, defining $\Sigma^2 := \tilde{\Sigma}^2$, $\Sigma^3 := \tilde{\Sigma}^3$ and $\Sigma^4 := \Sigma \setminus (\Sigma^1 \cup \Sigma^2 \cup \Sigma^3) \subset \tilde{\Sigma}^4$, all the asserted properties \ref{it:regularity-union} through \ref{it:regularity-dimension} would hold for $\{\Sigma^i\}_{i=1,2,3,4}$, thereby concluding the proof. 

As for the aforementioned requirement on $\Sigma_U$ in \cite[Theorem~3.10]{CES-RegularityOfMinimalSurfacesNearCones} --  it holds, in particular, if $\Sigma_U$ is $(\M, \epsilon, \delta)$-minimizing in $U$ in the sense of Almgren, and in addition it is the support of a multiplicity-one integral $(n-1)$-current with bounded mean curvature and no boundary in $U$. Let us verify this requirement for our $\Sigma_U = \Sigma \cap U$ for an arbitrary open bounded set $U$. 
Theorem \ref{thm:Almgren} \ref{it:Almgren-M} already asserts that $\Sigma_U$ is $(\M, \epsilon, \delta)$-minimizing, so
it remains to check that $\Sigma_U$ is the support of an integral $(n-1)$-current with bounded mean curvature and no boundary in $U$. 

Indeed, consider $\Sigma^1$ as the  $(n-1)$-dimensional smooth manifold $\sqcup_{i < j} \Sigma_{ij}$, which is in addition oriented (recall that the co-orientation is given by $\n_{ij}$), and define the associated rectifiable $(n-1)$-current $T$ on $U$ acting by integration:
\[
T(\alpha) = \int_{\Sigma^1 \cap U} \alpha ,
\]
for any smooth differential $(n-1)$-form $\alpha$ compactly supported in $U$. Since $\Sigma^1$ is a smooth manifold, which is relatively open in its closure $\Sigma$, the support of $T$ in $U$ is easily seen to be $\overline{\Sigma^1} \cap U = \Sigma \cap U = \Sigma_U$.

Next, note that $\Sigma^1 \cap U$ and hence $T$ has bounded mean-curvature, because the weighted mean-curvature $H_{\Sigma_{ij},\mu}$ is constant by Lemma~\ref{lem:first-order-conditions}, and it differs from $H_{\Sigma_{ij}}$ by $\nabla_{\n_{ij}} W$, which is a bounded quantity on any bounded set $U$. In addition, we claim that $T$ has no boundary as a current, i.e. $T(d\beta) = 0$ for any smooth differential $(n-2)$-form $\beta$ compactly supported in $U$. Indeed, on $\Sigma^1$, 
write $\beta = \beta_0 + \eta$, where $\eta$ is identically zero when acting on $\Lambda^{n-2} T \Sigma^1$, the space of $(n-2)$-vectorfields in the tangent bundle $T \Sigma^1$, and $\beta_0$ is an $(n-2)$-form on $\Lambda^{n-2} T \Sigma^1$. 
 It is then easy to check that $d \eta$ remains identically zero when acting on $\Lambda^{n-1} T \Sigma^1$, and that $\beta_0 = i_Y \vol_{\Sigma^1}$ for some smooth vector-field $Y$ tangential to $\Sigma^1$ and compactly supported in $U$ (where $i$ denotes interior product). It follows that as forms on $\Lambda^{n-1} T\Sigma^1$:
\[
d \beta = d (i_Y \vol_{\Sigma^1}) = \div_{\Sigma^1} Y \; \vol_{\Sigma^1} . \]
Hence:
\[
T(d\beta) = \int_{\Sigma^1} \div_{\Sigma^1} Y \; \vol_{\Sigma^1} = \int_{\Sigma^1} \div_{\Sigma^1,\mu}(Y e^{+W}) e^{-W} \vol_{\Sigma^1} = \sum_{i < j} \int_{\Sigma_{ij}} \div_{\Sigma_{ij},\mu}(Y e^{+W}) d\mu^{n-1} .  
\]
But since $Y e^{+W}$ is tangential to $\Sigma^1$ and compactly supported, Lemma \ref{lem:first-order-conditions} \ref{it:weak-angles} asserts that the latter integral is zero, confirming that $T$ has no boundary, and concluding the proof.

\end{proof}

\begin{definition}[Regular Cluster] \label{def:regular} A cluster satisfying parts \ref{it:Almgren-ii}, \ref{it:Almgren-iii} and \ref{it:Almgren-density} of Theorem \ref{thm:Almgren} and the conclusion of Theorem \ref{thm:regularity}, and whose cells are volume and perimeter regular with respect to $\mu$, is called regular (with respect to $\mu$). 
\end{definition}

\begin{remark} \label{rem:relaxed-regularity}
Note that a regular cluster is in particular interface--regular (cf. Definition \ref{def:interface-regular}). 
For the results of this work, we can slightly relax the above definition of regularity: we do not need the local-finiteness in parts \ref{it:regularity-Sigma2} and \ref{it:regularity-Sigma3} of Theorem \ref{thm:regularity}, and instead of $\dim_H(\Sigma^4) \leq n-4$ in part \ref{it:regularity-dimension} we will only require $\H^{n-3}(\Sigma^4) = 0$. 
In addition, it is worthwhile noting that parts \ref{it:Almgren-ii}, \ref{it:Almgren-iii} and \ref{it:Almgren-density} of Theorem \ref{thm:Almgren} are a consequence of $\Sigma$ being a locally $(\M,\eps,\delta)$-minimizing set (part \ref{it:Almgren-M} of Theorem \ref{thm:Almgren}) and stationarity, and that the latter two properties are the only ones which were used in the proof of Theorem \ref{thm:regularity}; as we will only consider regularity in conjunction with stationarity in this work, one could therefore replace ``satisfying parts \ref{it:Almgren-ii}, \ref{it:Almgren-iii} and \ref{it:Almgren-density} of Theorem \ref{thm:Almgren} and the conclusion of Theorem \ref{thm:regularity}" in the above definition by (the a-priori stronger, but simpler to state) ``satisfying part \ref{it:Almgren-M} of Theorem \ref{thm:Almgren}". 
\end{remark}

In summary, Theorem \ref{thm:Almgren}, Lemmas \ref{lem:first-order-conditions} and \ref{lem:unstable} and Theorem \ref{thm:regularity} assert that, assuming that $\mu$ satisfies (\ref{eq:mu-regular}), an isoperimetric $\mu$-minimizing cluster $\Omega$ is necessarily stationary, stable and regular (with respect to $\mu$). We proceed to describe additional properties of $\Sigma$, which hold for any stationary and regular cluster $\Omega$.

\subsection{Angles and $C^\infty$ regularity}

Theorem~\ref{thm:regularity} implies that every point in $\Sigma^2$, which we will call the \emph{triple-point set}, belongs to the closure of exactly three cells, as well as to the closure of exactly three interfaces. Given distinct $i,j,k$, we will write $\Sigma_{ijk}$ for the subset of $\Sigma^2$ which belongs to the closure of $\Omega_i$, $\Omega_j$ and $\Omega_k$, or equivalently, to the closure of $\Sigma_{ij}$, $\Sigma_{jk}$ and $\Sigma_{ki}$. Similarly, we will call $\Sigma^3$ the \emph{quadruple-point set}, and given distinct $i,j,k,l$, denote by $\Sigma_{ijkl}$ the subset of $\Sigma^3$ which belongs to the closure of $\Omega_i$, $\Omega_j$, $\Omega_k$ and $\Omega_l$, or equivalently, to the closure of all six $\Sigma_{ab}$ for distinct $a,b \in \{ i , j , k, l \}$. 

We will extend the normal fields $\n_{ij}$ to $\Sigma_{ijk}$ and $\Sigma_{ijkl}$ by continuity (thanks to $C^1$ regularity). We will also define $\n_{\partial ij}$ on
$\Sigma_{ijk}$ to be the outward-pointing unit boundary-normal to $\Sigma_{ij}$.
When $i$ and $j$ are clear from the context, we will write $\n_\partial$ for $\n_{\partial ij}$.

Let us also denote:
\[
\partial \Sigma_{ij} := \bigcup_{k \neq i,j} \Sigma_{ijk}.
\]
Note that $\Sigma_{ij} \cup \partial \Sigma_{ij}$ is a (possibly incomplete) $C^\infty$ manifold with $C^{1,\alpha}$ boundary (the boundary $\partial \Sigma_{ij}$ will be shown to actually be $C^\infty$ in Corollary \ref{cor:smooth-Sigma2}).

\begin{corollary}\label{cor:boundary-normal-sum}
    For any stationary regular cluster, the following holds at every point in $\Sigma_{ijk}$:
    \[
     \sum_{(\ell,m) \in \cyclic(i, j, k)} \n_{\partial \ell m} = 0 \; \text{ and } \sum_{(\ell,m) \in \cyclic(i, j, k)} \n_{\ell m} = 0  .
     \]
      In other words, $\Sigma_{ij}$, $\Sigma_{jk}$ and $\Sigma_{ki}$ meet at $\Sigma_{ijk}$ in $120^{\circ}$ angles. 
\end{corollary}

\begin{proof}
    Let $X$ be any $C_c^\infty$ vector field whose support intersects $\Sigma$ only in $\Sigma_{ijk} \cup \Sigma_{ij} \cup\Sigma_{jk} \cup \Sigma_{ki}$.
    According to Lemma~\ref{lem:first-order-conditions} \ref{it:weak-angles}, we have:
    \[
        \sum_{(\ell, m) \in \cyclic(i, j, k)} \int_{\Sigma_{\ell m}} \div_{\Sigma,\mu} X^\tang \, d\mu^{n-1} = 0 . 
    \]
    On the other hand, since $\Sigma$ is closed, note that $X$ has compact support in $\Sigma_{\ell m} \cup \partial \Sigma_{\ell m}$ for all $(\ell, m) \in \cyclic(i, j, k)$, and so applying Stokes' theorem (\ref{eq:Stokes}), we see that the above quantity is equal to 
    \[
        \sum_{(\ell, m) \in \cyclic(i, j, k)} \int_{\Sigma_{ijk}} \inr{X}{\n_{\partial \ell m}}\, d\mu^{n-2}.
    \]
    It follows that for any such vector field $X$,
    \[
    \int_{\Sigma_{ijk}} \left\langle X , \sum_{(\ell,m) \in \cyclic(i, j, k)} \n_{\partial \ell m}\right\rangle \, d\mu^{n-2} = 0
    \]
    Since $\n_{\partial \ell m}$ is a continuous vector-field on $\Sigma_{ijk}$ and $\Sigma_{ijk} \cup \Sigma_{ij} \cup \Sigma_{jk} \cup \Sigma_{ki}$ is relatively open in $\Sigma$ (by Theorems \ref{thm:Almgren} and ~\ref{thm:regularity}),
    the first assertion follows. For the second assertion, let $W$ denote the span of $\{\n_{\partial ij}, \n_{\partial jk}, \n_{\partial ki}\}$, which we already know is two-dimensional. We will orient $W$ so that $\n_{\partial ij}$, $\n_{\partial jk}$, $\n_{\partial ki}$ are in counterclockwise order. Then $\n_{\ell m} = R \, \n_{\partial \ell m}$ for all $(\ell,m) \in \cyclic(i,j,k)$, where $R$ is a $90^\circ$ clockwise rotation in $W$, and the second assertion follows from the first.
\end{proof}

\begin{corollary}[Kinderlehrer--Nirenberg--Spruck] \label{cor:smooth-Sigma2}
For any stationary regular cluster, Theorem \ref{thm:regularity} \ref{it:regularity-Sigma2} holds with $C^\infty$ regularity (instead of just $C^{1,\alpha}$) for $\Sigma^2$ and its associated local diffeomorphisms. 
\end{corollary}
\begin{proof}
It was shown by Nitsche \cite{Nitsche-ThreeMinimalSurfacesMeetOnSmoothCurve} that three minimal surfaces in $\R^3$ meeting at equal $120^{\circ}$ angles must do so on a smooth curve. Kinderlehrer, Nirenberg and Spruck \cite[Theorem 5.2]{KNS} extended this to the case that three hypersurfaces with analytic mean-curvature in $\R^n$ meet at arbitrary \emph{constant} transversal angles on a $C^{1,\alpha}$ $(n-2)$-dimensional manifold $\Sigma^2$, and showed that $\Sigma^2$ must in fact be analytic. As we shall essentially reproduce in Appendix \ref{sec:Schauder} for a different purpose, the key step in their argument is to use elliptic methods to show that the $C^{k,\alpha}$ regularity of $\Sigma^2$ 
(as well as of the local diffeomorphism as in Theorem \ref{thm:regularity} \ref{it:regularity-Sigma2}) may be improved to $C^{k+1,\alpha}$ regularity if the mean-curvature of the three hypersurfaces is a smooth-enough function of the first and zeroth order derivatives of the surfaces' parametrization. 
By Corollary \ref{cor:boundary-normal-sum}, $\Sigma_{ij}$, $\Sigma_{jk}$ and $\Sigma_{ki}$ meet at equal $120^{\circ}$ angles on the $C^{1,\alpha}$ manifold $\Sigma_{ijk}$. Since the weighted mean-curvature of $\Sigma_{\ell m}$ is constant $H_{\ell m,\mu}$, its unweighted mean-curvature is of the form $H_{\ell m,\mu} + \scalar{\nabla W, \n_{\ell m}}$, which is a $C^\infty$ function of the zeroth and first order terms since $W \in C^\infty(\R^n)$. The assertion therefore follows by applying the Kinderlehrer--Nirenberg--Spruck argument. 
\end{proof}

\subsection{Local integrability of curvature}

We will also crucially require the following local integrability properties of our various curvatures. Let $\II^{ij}$ denotes the second fundamental form on $\Sigma_{ij}$, which may be
extended by continuity to $\partial \Sigma_{ij}$. When $i$ and $j$ are clear from the context, we will write $\II$ for $\II^{ij}$. 
Recall that $\norm{\II^{ij}}$ denotes (say) the Hilbert-Schmidt norm of $\II^{ij}$.

\begin{proposition} \label{prop:Schauder} Let $\Omega$ be a stationary regular cluster. 
For any compact set $K \subset \R^n$ which is disjoint from $\Sigma^4$:
\begin{enumerate}[(1)]
\item \label{it:Schauder-Sigma1} $\displaystyle \sum_{i,j} \int_{\Sigma_{ij} \cap K} \norm{\II^{ij}}^2 d\mu^{n-1} < \infty$. 
\item \label{it:Schauder-Sigma2} $\displaystyle \sum_{i,j,k} \int_{\Sigma_{ijk} \cap K} \norm{\II^{ij}} d\mu^{n-2} < \infty$. 
\end{enumerate}
\end{proposition}

By compactness and the fact that $\Sigma^4$ is closed, it is enough to verify the above integrability locally, in arbitrarily-small relatively open subsets of $\overline{\Sigma^1}$ and $\overline{\Sigma^2}$, respectively, which are disjoint from $\Sigma^4$. Since $\Sigma_{ij}$ is $C^\infty$ smooth all the way up to $\partial \Sigma_{ij} \subset \Sigma^2$, it remains to verify the integrability around small neighborhoods of quadruple points in $\Sigma^3$. Proposition \ref{prop:Schauder} follows from combining the locally uniform $C^{1,\alpha}$ regularity of $\Sigma$ around quadruple points, together with the elliptic regularity method used by Kinderlehrer--Nirenberg--Spruck \cite{KNS} to self-improve the $C^{1,\alpha}$ regularity around triple points $\Sigma^2$ to $C^\infty$ smoothness. The idea is to use Schauder estimates to show that the curvature (or $C^2$ semi-norm of the graph of $\Sigma_{ij}$) at a distance of $r$ from $\Sigma^3$ blows up at a rate of at most $1/r^{1-\alpha}$ as $r \rightarrow 0$, and hence is locally in $L^1$ on $\Sigma^2$ and locally in $L^2$ on $\Sigma^1$. As the tools we require for establishing Proposition \ref{prop:Schauder} will not be used in any other part of this work, we defer its proof to Appendix \ref{sec:Schauder}.

\medskip
Note that the same compactness argument as above, together with the $C^{1,\alpha}$ regularity of $\Sigma$ around quadruple-points in $\Sigma^3$,  immediately yields the following simple corollary of Theorem \ref{thm:regularity}:
\begin{corollary} \label{cor:locally-finite-Sigma2}
For any compact set $K \subset \R^n$ which is disjoint from $\Sigma^4$, $\mu^{n-2}(\Sigma^2 \cap K) < \infty$. 
\end{corollary}
In fact, by the local finiteness statement in Theorem \ref{thm:regularity} \ref{it:regularity-Sigma2}, the same holds without assuming that $K$ is disjoint from $\Sigma^4$, but we will not require this here.

\section{Second Variation for Tame Vector-Fields} \label{sec:second-var}

Throughout this section, unless otherwise stated, $\Omega$ is assumed to be a $\mu$-stationary regular cluster. 
Recall that $\partial \Sigma_{ij}$ denotes $\bigcup_{k \neq i,j} \Sigma_{ijk}$ and that
$\Sigma_{ij} \cup \partial \Sigma_{ij}$ is a $C^\infty$ manifold with boundary.
Recall also that $\II^{ij}$ denotes the second fundamental form of $\Sigma_{ij}$, which we extend by continuity to $\partial \Sigma_{ij}$.
We will abbreviate $\II^{ij}_{\partial,\partial}$ for $\II^{ij}(\n_{\partial ij}, \n_{\partial ij})$ on $\partial \Sigma_{ij}$. When $i$ and $j$ are clear from the context, we use $\II$ for $\II^{ij}$. 
Finally, recall the definition of the index-form $Q(X) = \delta_X^2 A - \inr{\lambda}{\delta_X^2 V}$.

\begin{definition}[Tame Field]
    The vector-field $X$ is called \emph{tame} for a regular cluster $\Omega$ if it is $C_c^\infty$ and supported in $\R^n \setminus \Sigma^4$.
\end{definition}

In this section, we establish the following key formula for $Q(X)$. We denote by $\nabla^\tang$ the tangential component of the derivative.

\begin{theorem}[Index-form formula for tame fields]\label{thm:formula}
    If $\Omega$ is a stationary regular cluster and $X$ is a tame vector-field, then $Q(X)$ is given by
    \begin{equation}         \sum_{i < j} \Big[\int_{\Sigma_{ij}} \brac{ |\nabla^\tang X^\n|^2 - (X^\n)^2 \|\II\|_2^2 - (X^\n)^2 \nabla^2_{\n,\n} W } d\mu^{n-1}     - \int_{\partial \Sigma_{ij}} X^\n X^{\n_\partial} \II_{\partial,\partial} \, d\mu^{n-2}\Big].
        \label{eq:formula}
    \end{equation} \end{theorem}

We will also require the following variant of Theorem \ref{thm:formula}, which will require additional work: 
\begin{theorem}[Index-form polarization formula with respect to constant fields]
\label{thm:combined-vector-fields}     With the same assumptions as in Theorem \ref{thm:formula}, if $Y = X + w$ where $w$ is a constant
    vector-field, then $Q(Y)$ is given by
    \begin{multline*}
        \sum_{i < j} \Bigg[
            \int_{\Sigma_{ij}} \left ( |\nabla^\tang X^\n|^2 - (X^\n)^2 \|\II\|_2^2 - (Y^\n)^2 \nabla^2_{\n,\n} W                         - (X+Y)^{\n} \nabla^2_{w^{\tang} ,\n} W \right ) d\mu^{n-1} \\ - \int_{\partial \Sigma_{ij}} X^\n X^{\n_\partial} \II_{\partial,\partial}\, d\mu^{n-2} 
        \Bigg] .
    \end{multline*}
\end{theorem} 
\noindent Note that when $X \equiv 0$, the above expression for $Q(w)$ indeed coincides with the one derived in (the much simpler) Lemma \ref{lem:cor-Q}. 

\medskip

Our starting point for establishing Theorem \ref{thm:formula} will be the following formula for $Q(X)$, which may be derived starting from Lemma~\ref{lem:regular}; its proof consists of a long computation, which is deferred to Appendix \ref{sec:calculation}.
\begin{lemma}\label{lem:formula}
    For any stationary cluster $\Omega$ and vector-field $X$, so that either $X$ is $C_c^\infty$, or alternatively, $X$ is admissible and the cells of $\Omega$ are volume and perimeter regular, $Q(X)$ is given by     \begin{multline*}
        \sum_{i < j}
        \int_{\Sigma_{ij}}   \Bigg[ |\nabla^\tang X^\n|^2 - (X^\n)^2 \|\II\|_2^2 - (X^\n)^2 \nabla_{\n,\n}^2 W - \div_{\Sigma,\mu} (X^\n \nabla_{X^\tang} \n)  \\
         + \div_{\Sigma,\mu} (X^\n \nabla_\n X) + \div_{\Sigma,\mu} (X \div_{\Sigma,\mu} X) - H_{ij,\mu} X^\n \div_\mu X \Bigg] \, d\mu^{n-1}.
    \end{multline*}
\end{lemma}

In order to simplify the expression in Lemma \ref{lem:formula}, we will want to apply Stokes' theorem to the
various divergence terms. Even if we assume that $X$ has compact support in $\R^n$, this requires some justification; each
$M_{ij} = \Sigma_{ij} \cup \partial \Sigma_{ij}$ is a smooth manifold with boundary, but
the vector-fields $Y_{ij}$ whose divergence is integrated may not have compact support inside $M_{ij}$, as their support may intersect $\Sigma^3 \cup \Sigma^4$. 
Furthemore, $Y_{ij}$ may be blowing up near $\Sigma^3 \cup \Sigma^4$, precluding us from invoking a version of Stokes' theorem for vector-fields with bounded magnitude and divergence  
which only requires that $\H^{n-2}(\overline{M_{ij}} \setminus M_{ij}) = 0$, see e.g.~\cite[XXIII, Section 6]{LangRealAndFunctionalAnalysis}.
However, in our setting, we actually know that $\overline{M_{ij}} \setminus M_{ij}$ has locally-finite $\H^{n-3}$-measure, which enables us to handle vector-fields in  $L^2(\Sigma_{ij},\mu^{n-1}) \cap L^1(\partial \Sigma_{ij},\mu^{n-2})$  whose divergence is in $L^1(\Sigma_{ij},\mu^{n-1})$; it turns out that whenever $X$ is tame, our $Y_{ij}$'s will indeed belong to this class thanks to the integrability properties of our curvatures established in Proposition \ref{prop:Schauder}.
The idea is to use a well-known truncation argument to cut out the low-dimensional sets (e.g. as in \cite{LangRealAndFunctionalAnalysis,SternbergZumbrun}). 

\subsection{Truncation via cutoff functions} \label{subsec:cutoff}

A function on $\R^n$ which is $C^\infty$ smooth and takes values in $[0, 1]$ will be called a cutoff function. 

\begin{lemma}\label{lem:sigma4-cutoff}
    For every $\epsilon > 0$, there is a compactly supported cutoff function $\eta$
    such that $\eta \equiv 0$ on a neighborhood of $\Sigma^{4}$,
    \[
        \int_{\Sigma^1} |\nabla \eta|^2 \, d\mu^{n-1} \le \epsilon,
    \]
    and
    \[
        \mu^{n-1} \{p \in \Sigma^1: \eta(p) < 1\} \le \epsilon.
    \]
    \end{lemma}

\begin{proof}
    Fix $\epsilon > 0$, and choose $R \geq 4$ sufficiently large so that     $\mu^{n-1}(\Sigma^1 \setminus B_R) \le \epsilon$. Let $\zeta$
    be a $C_c^\infty$ function that is identically one on $B_{R+1}$,
    identically zero outside $B_{2R-1}$, and
    satisfies $|\nabla \zeta| \le 1$.     By the upper density estimate (\ref{eq:density}), there exist $\Lambda > 0$ and $r_0 \in (0,1)$ so that:
    \begin{equation} \label{eq:cutoff-density}
    \mu^{n-1}(\Sigma^1 \cap B(x,4 r)) \leq \Lambda r^{n-1} \;\;\; \forall x \in \Sigma \cap \overline{B_{2R}} \;\; \forall r \in (0,r_0) . 
    \end{equation}

    Since $\H^{n-3}(\Sigma^{4})=0$ and $\Sigma^{4} \cap \overline{B_{2R}}$
    is compact,     we may find finite sequences $x_i \in \Sigma^{4} \cap \overline{B_{2R}}$ and     $\delta_i \in (0, r_0)$, $i=1,\ldots,N$, such that $\{B(x_i,\delta_i)\}_{i =1,\ldots,N}$ cover
    $\Sigma^{4} \cap \overline{B_{2R}}$ and $\sum_i \delta_i^{n-3} \le \epsilon / \Lambda$. 
        For each $i$, choose a cutoff function $\eta_i$ such that $\eta_i \equiv 1$ outside
    of $B(x_i,3 \delta_i)$, $\eta_i \equiv 0$ inside $B(x_i,2 \delta_i)$, and
    $|\nabla \eta_i| \le 2/\delta_i$. 
                        Now define $\tilde \eta(x) = \min\{\zeta, \min_{i=1,\ldots,N} \eta_i(x)\}$;
    $\tilde \eta$ is compactly supported, vanishes on $(\R^n \setminus \overline{B_{2R-1}}) \bigcup \cup_i  B(x_i,2 \delta_i)$, and is identically $1$ on $B_{R+1} \setminus \cup_i B(x_i,3 \delta_i)$. Since it is only piecewise $C^\infty$, let us mollify $\tilde \eta$ with a smooth mollifier supported in $B(0,\delta)$, for $\delta = \min_{i=1,\ldots,N} \delta_i \in (0,1)$ - this will be our desired function $\eta$. The cutoff function $\eta$ is compactly supported and  vanishes on $(\R^n \setminus \overline{B_{2R}}) \bigcup \cup_i  B(x_i, \delta_i)$, an open neighborhood of $\Sigma^4$ and $\infty$. 
        Since $\delta_i < r_0 < 1$, it follows by (\ref{eq:cutoff-density}) that:
    \begin{align*}
        \mu^{n-1}\{p \in \Sigma^1: \eta(p) < 1\}
        &\le \mu^{n-1}(\Sigma^1 \setminus B_R) + \sum_i \mu^{n-1}(\Sigma^1 \cap B(x_i,4 \delta_i)) \\
        &\le \epsilon + \Lambda \sum_i \delta_i^{n-1} \le \epsilon + \Lambda \sum_i \delta_i^{n-3}  \leq 2\epsilon.
    \end{align*}
    Finally, we have for a.e. $x \in \R^n$:
    \[
        |\nabla \tilde \eta(x)|^2 \le \max \{|\nabla \zeta(x)|^2, \max_i |\nabla \eta_i(x)|^2\}
        \le |\nabla \zeta(x)|^2 + \sum_i |\nabla \eta_i(x)|^2 ,
    \]
    and hence after mollification it follows that for all $x \in \R^n$:
    \[
    | \nabla \eta(x) |^2 \leq 1_{B_{2R} \setminus B_{R}}(x) + \sum_{i} (2 / \delta_i)^2 1_{B(x_i,4 \delta_i) \setminus B(x_i,\delta_i)}(x) .
    \]
    Consequently:
        \begin{align*}
    \int_{\Sigma^1} |\nabla \eta|^2\, d\mu^{n-1} & \leq \mu^{n-1}(\Sigma^1 \setminus B_R) + \sum_i (2 / \delta_i)^2 \mu^{n-1}(\Sigma^1 \cap B(x_i,4 \delta_i) ) \\
    & \leq \epsilon + \sum_{i} 4 \Lambda \delta_i^{n-3} \leq 5 \epsilon . 
    \end{align*}
                        Appropriately modifying the value of $\eps$, the assertion follows. 
    \end{proof}

A slight variation on the proof of Lemma~\ref{lem:sigma4-cutoff} implies
that we can cut off $\Sigma^3$ also, while only paying a bounded amount
in the $W^{1,2}$ norm. 
\begin{lemma}\label{lem:sigma3-cutoff}
    Let $U$ be an open neighborhood of $\Sigma^4$ and infinity.     There is a constant $C_U$ depending on the cluster $\Omega$ and $U$,
    such that for every $\delta > 0$, 
    there is a cutoff function $\xi$ such that
    $\xi \equiv 0$ on an open neighborhood of $\Sigma^{3} \setminus U$,     \[
        \int_{\Sigma^1} |\nabla \xi|^2 \, d\mu^{n-1} \le C_U,
    \]
    and
    \[
        \mu^{n-1} \{p \in \Sigma^1: \xi(p) < 1\} \le \delta.
    \]
    \end{lemma}

\begin{proof}
    First, note that $\overline{\Sigma^3} \subset \Sigma^3 \cup \Sigma^4$ by the relative openness of $\Sigma^1\cup \Sigma^2$ in the closed $\Sigma$. It follows that $\Sigma^3 \setminus U$ is compact, since it is clearly bounded, but also closed, as it coincides with $\overline{\Sigma^3} \setminus U$.     
    Since $\Sigma^3$ has finite $\calH^{n-3}$-measure in a neighborhood of every point in $\Sigma^3$ (by Theorem \ref{thm:regularity}), compactness implies that
    $\calH^{n-3}(\Sigma^3 \setminus U)$ is finite. In addition, by the upper density estimate (\ref{eq:density}), there exist $\Lambda_U,r_U > 0$ so that:
    \begin{equation} \label{eq:cutoff-density2}
    \mu^{n-1}(\Sigma^1 \cap B(x,4 r)) \leq \Lambda_U r^{n-1} \;\;\; \forall x \in \Sigma^3 \setminus U \;\; \forall r \in (0,r_U) . 
    \end{equation}
    
             For any $\delta_0 \in (0,r_U)$, the definition of Hausdorff measure and compactness imply the existence of finite sequences $x_i \in \Sigma^3 \setminus U$ and $\delta_i \in (0, \delta_0)$, $i=1,\ldots,N$, such that the sets $\{B(x_i, \delta_i)\}$ cover $\Sigma^3 \setminus U$ and $\sum_i \delta_i^{n-3} \le C_n \calH^{n-3} (\Sigma^3 \setminus U)$, where $C_n >0$ is a numeric constant depending only on $n$. 
         The rest of the proof is identical to that of Lemma~\ref{lem:sigma4-cutoff}.
    For each $i$, we select a cutoff function $\xi_i$ which is identically zero on $B(x_i, 2 \delta_i)$,
    identically one outside of $B(x_i, 3 \delta_i)$, and satisfies $|\nabla \xi_i| \leq 2/\delta_i$.
    We then set $\tilde \xi(x) = \min_{i=1,\ldots,N} \xi_i(x)$ and define $\xi$ to be the mollification of $\tilde \xi$ using a smooth mollifier compactly supported in $B_{\delta}$, where $\delta = \min_{i=1,\ldots,N} \delta_i$. The cutoff function $\xi$ vanishes on $\cup_i  B(x_i, \delta_i)$, an open neighborhood of $\Sigma^3 \setminus U$. 
        It follows that:
    \[
        \int_{\Sigma^1} |\nabla \xi|^2 \, d\mu^{n-1} \le \sum_i (2/\delta_i)^2 \mu^{n-1}(B(x_i,4 \delta_i)) \le 4 \Lambda_U \sum_{i} \delta_i^{n-3} \leq 4 \Lambda_U C_n \calH^{n-3}(\Sigma^3 \setminus U),
    \]
    which proves the claimed gradient bound. On the other hand,
    \begin{align*}
        \mu^{n-1} \{p \in \Sigma^1: \xi(p) < 1\}
        &\le \sum_i \mu^{n-1}(\Sigma^1 \cap B(x_i, 4 \delta_i)) \\
        &\le \Lambda_U \sum_i \delta_i^{n-1} \le \Lambda_U C_n \delta^2_0 \calH^{n-3}(\Sigma^3 \setminus U),
    \end{align*}
    which can be made smaller than $\delta$ by an appropriate choice of $\delta_0$.
\end{proof}

\subsection{A version of Stokes' theorem}

\begin{lemma}[Stokes' theorem on $\Sigma_{ij} \cup \partial \Sigma_{ij}$] \label{lem:stokes}
    Let $\Omega$ be a stationary regular cluster.     Suppose that $Y_{ij}$ is a vector-field which is $C^\infty$ on $\Sigma_{ij}$ and continuous
    up to $\partial \Sigma_{ij}$.
    Suppose, moreover, that
    \begin{equation} \label{eq:stokes-assumptions}
        \int_{\Sigma_{ij}} |Y_{ij}|^2 \, d\mu^{n-1}, \quad
        \int_{\Sigma_{ij}} |\div_{\Sigma,\mu} Y_{ij}| \, d\mu^{n-1}, \quad
        \int_{\partial \Sigma_{ij}} |Y_{ij}^{\n_\partial}| \, d\mu^{n-2} < \infty . 
    \end{equation}
        Then:
    \[
        \int_{\Sigma_{ij}} \div_{\Sigma,\mu} Y_{ij} \, d\mu^{n-1}
        = \int_{\Sigma_{ij}} H_{ij,\mu} Y_{ij}^\n \, d\mu^{n-1} + \int_{\partial \Sigma_{ij}} Y_{ij}^{\n_{\partial}} \, d\mu^{n-2}.
    \]
\end{lemma}
\begin{remark} \label{rem:stokes} 
Observe that whenever $Y_{ij}$ is the restriction to $\Sigma_{ij} \cup \partial \Sigma_{ij}$ of a tame vector-field on $\R^n$, all of the integrability requirements in (\ref{eq:stokes-assumptions}) are automatically satisfied, since $\mu^{n-1}(\Sigma_{ij}) < \infty$ and $\mu^{n-2}(\partial \Sigma_{ij} \cap K) < \infty$ for any compact $K$ disjoint from $\Sigma^4$ by Corollary \ref{cor:locally-finite-Sigma2}.  
\end{remark}
While we will only apply Lemma \ref{lem:stokes} to vector-fields which are supported outside a neighborhood of $\Sigma^4$ and infinity, we provide a proof for general vector-fields as above, as this comes at almost no extra cost.
\begin{proof}[Proof of Lemma \ref{lem:stokes}]
    Fix $\epsilon > 0$ and a cutoff function $\eta = \eta_\eps$ as in Lemma~\ref{lem:sigma4-cutoff}.
    Let $U = U_\eta$ be a neighborhood of $\Sigma^4$ and infinity on which $\eta$ vanishes,
    and let $C_\eta$ denote the constant $C_U$ from Lemma~\ref{lem:sigma3-cutoff}.
    Fix $\delta > 0$ and a function $\xi = \xi_{\eps,\delta}$ as in Lemma~\ref{lem:sigma3-cutoff}.
    Note that $\eta \xi$ vanishes on an open neighborhood of $\Sigma^3 \cup \Sigma^4$ and infinity, and therefore 
    $\eta \xi Y_{ij}$  has compact support in $\Sigma_{ij} \cup \partial \Sigma_{ij}$. Applying Stokes' theorem (\ref{eq:Stokes}) to $\eta \xi Y_{ij}$, we obtain:
    \begin{equation}
        \int_{\Sigma_{ij}} \div_{\Sigma,\mu} (\eta \xi Y_{ij}) \, d\mu^{n-1}
        = \int_{\Sigma_{ij}} H_{ij,\mu} \eta \xi Y_{ij}^\n \, d\mu^{n-1}
        + \int_{\partial \Sigma_{ij}} \eta \xi Y_{ij}^{\n_{\partial}} \, d\mu^{n-2}.
            \label{eq:stokes-1}
    \end{equation}
    
    Next, we check what happens when we send $\delta \to 0$ and then $\epsilon \to 0$.
    By the Dominated Convergence theorem, the right-hand-side of~\eqref{eq:stokes-1} converges to
    \[
        \int_{\Sigma_{ij}} H_{ij,\mu} Y_{ij}^\n \, d\mu^{n-1}
            + \int_{\partial \Sigma_{ij}} Y_{ij}^{\n_{\partial}} \, d\mu^{n-2} .
    \]
   Indeed, the second term is absolutely integrable directly by assumption, and the first one is too
   since  $H_{ij,\mu}$ is constant and:    \[
    \brac{\int_{\Sigma_{ij}} \abs{Y_{ij}^\n} d\mu^{n-1}}^2 \leq \mu^{n-1}(\Sigma_{ij}) \int_{\Sigma_{ij}} \abs{Y_{ij}}^2 d\mu^{n-1} < \infty .
    \]
    
    As for the left-hand-side of~\eqref{eq:stokes-1}, we split it as
    \[
        \int_{\Sigma_{ij}} \div_{\Sigma,\mu} (\eta \xi Y_{ij}) \, d\mu^{n-1}
        = 
        \int_{\Sigma_{ij}} \eta \xi \div_{\Sigma,\mu} Y_{ij} d\mu^{n-1} + \int_{\Sigma_{ij}} \nabla_{Y^{\tang}_{ij}} (\eta \xi) \, d\mu^{n-1}.
    \]
    Taking $\delta \to 0$ and then $\epsilon \to 0$,
    $\int_{\Sigma_{ij}} \eta \xi \div_{\Sigma,\mu} Y_{ij} \, d\mu^{n-1}
    \to \int_{\Sigma_{ij}} \div_{\Sigma,\mu} Y_{ij} \, d\mu^{n-1}$
    because we assumed $\div_{\Sigma,\mu} Y_{ij}$ to be absolutely integrable.
    On the other hand,
    \[
        \left|\int_{\Sigma_{ij}} \nabla_{Y^{\tang}_{ij}} (\eta \xi)\, d\mu^{n-1}\right|
        \le \int_{\Sigma_{ij}} (|\nabla \eta| + |\nabla \xi|) |Y_{ij}| \, d\mu^{n-1}.
    \]
    Applying the Cauchy-Schwarz inequality,
    \[
        \int_{\Sigma_{ij}} |\nabla \eta| |Y_{ij}| \, d\mu^{n-1} \le \left(\epsilon \int_{\Sigma_{ij}} |Y_{ij}|^2\, d\mu^{n-1}\right)^{1/2},
    \]
    which converges to zero as $\epsilon \to 0$, and
    \[
        \int_{\Sigma_{ij}} |\nabla \xi| |Y_{ij}| \, d\mu^{n-1} \le \left(C_\eta \int_{\Sigma_{ij}} 1_{\{\xi < 1\}} |Y_{ij}|^2\, d\mu^{n-1}\right)^{1/2}
    \]
    because $\nabla \xi = 0$ on the set $\{\xi = 1\}$. For any fixed $\epsilon$ and $\eta$, this
    last integral converges to zero as $\delta \to 0$ by the Dominated Convergence theorem.
    Taking $\delta \to 0$ and then $\epsilon \to 0$, it follows that
    $\int_{\Sigma_{ij}} \nabla_{Y^{\tang}_{ij}} (\eta \xi) \, d\mu^{n-1} \to 0$ and so
    \[
        \int_{\Sigma_{ij}} \div_{\Sigma,\mu} (\eta \xi Y_{ij}) \, d\mu^{n-1}
        \to \int_{\Sigma_{ij}} \div_{\Sigma,\mu} Y_{ij} \, d\mu^{n-1}.
    \]
    Plugging this back into~\eqref{eq:stokes-1} proves the claim.
\end{proof}

\subsection{Cancellation identities}

 To see that the boundary integrals simplify or even vanish after an application of Stokes' theorem on the divergence terms appearing in Lemma \ref{lem:formula}, we will require a couple of useful identities. Recall that by Corollary~\ref{cor:boundary-normal-sum}, for any distinct $i,j,k$, the three interfaces $\Sigma_{ij}$, $\Sigma_{jk}$ and $\Sigma_{ki}$ meet on $\Sigma_{ijk}$ at $120^\circ$ angles:
 \begin{equation} \label{eq:120}
 \sum_{(\ell, m) \in \cyclic(i, j, k)} \n_{\ell m} = 0 . 
 \end{equation}

\begin{lemma}\label{lem:boundary-curvature-cancellation0}
    At every point of $\Sigma_{ijk}$ and for every $x \in \R^n$,
    \[
        \sum_{(\ell, m) \in \cyclic(i, j, k)} \inr{x}{\n_{\ell m}} \inr{\nabla_{x^{\tang_{\ell m}}} \n_{\ell m}}{\n_{\partial \ell m}}
        = \sum_{(\ell, m) \in \cyclic(i,j,k)} \inr{x}{\n_{\ell m}} \inr{x}{\n_{\partial \ell m}} \II^{\ell m}_{\partial,\partial}.
    \]
 \end{lemma}
 \begin{proof}
 To simplify notation, we will assume
    that $\{i, j, k\} = \{1, 2, 3\}$ and fix $p \in \Sigma_{123}$. 
    Let $y$ be the component of $x$ which is tangent to $\Sigma_{123}$ at $p$,
    so that for any distinct $i, j,k \in \{1, 2, 3\}$,
    $x^{\tang_{ij}} = y + \n_{\partial ij} \inr{x}{\n_{\partial ij}}$. Hence,
    \begin{align}
        \inr{\nabla_{x^{\tang_{ij}}} \n_{ij}}{\n_{\partial ij}}
        &= 
        \inr{x}{\n_{\partial ij}} \inr{\nabla_{\n_{\partial ij}} \n_{ij}}{\n_{\partial ij}}
        + \inr{\nabla_{y} \n_{ij}}{\n_{\partial ij}} \notag \\
        &= 
        \inr{x}{\n_{\partial ij}} \II^{ij}_{\partial,\partial}
        + \inr{\nabla_{y} \n_{ij}}{\n_{\partial ij}} .
        \label{eq:boundary-curvature}
    \end{align}
    Next, we observe that the second term is independent of
    $(i, j) \in \cyclic(1, 2, 3)$. Indeed, by (\ref{eq:120}), we know that
        $\n_{jk} = \frac{\sqrt 3}{2} \n_{\partial ij} - \frac 12 \n_{ij}$
    and $\n_{\partial jk} = - \frac{\sqrt 3}{2} \n_{ij} - \frac 12 \n_{\partial ij}$.
    Clearly, $\inr{\nabla_y \n_{ij}}{\n_{ij}} = \frac{1}{2} \nabla_y \inr{\n_{ij}}{\n_{ij}} = 0$ and similarly $\inr{\nabla_y \n_{\partial ij}}{\n_{\partial ij}} = 0$ and $\inr{\nabla_y \n_{\partial ij}}{\n_{ij}} = - \inr{\n_{\partial ij}}{\nabla_y \n_{ij}}$.
    Hence,
    \[
        \inr{\nabla_y \n_{jk}}{\n_{\partial jk}} = -\frac 34 \inr{\nabla_y \n_{\partial ij}}{\n_{ij}}
        + \frac 14 \inr{\nabla_y \n_{ij}}{\n_{\partial ij}}
        = \inr{\nabla_y \n_{ij}}{\n_{\partial ij}}.
    \]
    It follows by (\ref{eq:120}) again that:
        \[
        \sum_{(i,j) \in \cyclic(1,2,3)} \inr{x}{\n_{ij}} \inr{\nabla_y \n_{ij}}{\n_{\partial ij}} = \scalar{x,\sum_{(i,j) \in \cyclic(1,2,3)}  \n_{ij}} \inr{\nabla_y \n_{12}}{\n_{\partial 12}}  = 0.
    \]
    Multiplying~\eqref{eq:boundary-curvature} by $\inr{x}{\n_{ij}}$ and
    summing over $(i,j) \in \cyclic(1, 2, 3)$ completes the proof.
\end{proof}

 \begin{lemma}\label{lem:cancellation}
    At every point of $\Sigma_{ijk}$, the following $3$-tensor is identically zero:
    \[
    T^{\alpha \beta \gamma} = \sum_{(\ell,m) \in \cyclic(i,j,k)} 
    \brac{\n_{\ell m}^{\alpha} \n_{\ell m}^\beta \n_{\partial \ell m}^{\gamma} - \n_{\partial \ell m}^{\alpha} \n_{\ell m}^\beta \n_{\ell m}^{\gamma}} .
    \]
 \end{lemma}
\begin{proof}

    Consider a two-dimensional Euclidean space $W$ and let $R : W \to W$ denote a $90^\circ$ clockwise rotation. 
    Note that for any unit vector $u \in W$, $u (R u)^T - (R u) u^T$ is independent of $u$ and is equal to $-R$, as immediately seen by expressing this operator in the orthogonal basis $\{u, Ru\}$. 
   
    Now assume, without loss of generality as before, that $\{i, j, k\} = \{1, 2, 3\}$, and fix $p \in \Sigma_{123}$.
    Let $W$ be the span of $\{\n_{12}, \n_{23}, \n_{31}\}$, which by (\ref{eq:120}) is two-dimensional. We will orient $W$ so that $\n_{12}$, $\n_{23}$, $\n_{31}$ are in clockwise order.
    It follows that for every $(\ell, m) \in \cyclic(1,2,3)$, $\n_{\partial \ell m}$ is a $90^\circ$
    clockwise rotation of $\n_{\ell m}$.
    By the previous remarks, $\n^\alpha_{\ell m} \n_{\partial \ell m}^\gamma - \n_{\partial \ell m}^\alpha \n_{\ell m}^\gamma = -R^{\alpha \gamma}$
    is independent of $(\ell, m) \in \cyclic(1, 2, 3)$. Hence,
    \[
    T^{\alpha \beta \gamma} = -R^{\alpha \gamma} \sum_{(\ell, m) \in \cyclic(1, 2, 3)} n^\beta_{\ell m}  ,
    \]
    which vanished identically by (\ref{eq:120}), as asserted. 
\end{proof}

\subsection{Proof of Theorem~\ref{thm:formula}}

    The first step is to apply Stokes' theorem to show
    that the second line in the formula of Lemma~\ref{lem:formula} vanishes. 
    Observe that on $\Sigma_{ij}$:
    \begin{equation} \label{eq:div-Y}
    \div_{\Sigma,\mu} (X^{\n} \nabla_{\n} X) + \div_{\Sigma,\mu} (X \div_{\Sigma,\mu} X) - H_{\Sigma,\mu} X^{\n} \div_\mu X 
        = \div_{\Sigma,\mu} Y_{ij},
    \end{equation}
     where $Y_{ij}$ is the following vector-field on $\Sigma_{ij} \cup \partial \Sigma_{ij}$:
   \begin{equation} \label{eq:Y-def}
        Y_{ij} := X^{\n_{ij}} \nabla_{\n_{ij}} X - X \inr{\nabla_{\n_{ij}} X}{\n_{ij}} + X^{\tang_{ij}} \div_\mu X . 
    \end{equation}
    Note that $Y_{ij}$ is $C^\infty$ on $\Sigma_{ij} \cup \partial \Sigma_{ij}$ thanks to Corollary \ref{cor:smooth-Sigma2},     and satisfies $|Y_{ij}| \le (2+n)|X| \|\nabla X\| + |X|^2 |\nabla \pot|$.  Since $X$ is tame, $|Y_{ij}|$ is
    uniformly bounded and has bounded support disjoint from $\Sigma^4$, and hence satisfies the first integrability condition
    of Lemma~\ref{lem:stokes}, as well as the third one by Corollary \ref{cor:locally-finite-Sigma2}. 
    In addition, one sees that $\div_{\Sigma,\mu} Y_{ij}$ is absolutely integrable on $\Sigma_{ij}$ by inspecting (\ref{eq:div-Y});
  each of the terms on the left-hand-side is bounded by $\|\II\|$ times a polynomial in $|X|$, $\|\nabla X\|$, and $\|\nabla^2 X\|$, and is therefore
    in $L^2(\Sigma_{ij},\mu^{n-1})$ thanks to tameness of $X$ and Proposition \ref{prop:Schauder} \ref{it:Schauder-Sigma1}, and hence in $L^1(\Sigma_{ij},\mu^{n-1})$ since $\mu^{n-1}(\Sigma_{ij}) < \infty$.
    It follows that we may apply Lemma~\ref{lem:stokes} (Stokes' theorem) to each $Y_{ij}$ on $\Sigma_{ij} \cup \partial \Sigma_{ij}$. Note that $Y_{ij}$ is tangential to $\Sigma_{ij}$ since $\inr{Y_{ij}}{\n_{ij}} = 0$, 
    and so
    \begin{align}
       \nonumber
        \sum_{i < j} \int_{\Sigma_{ij}} \div_{\Sigma,\mu} Y_{ij} \, d\mu^{n-1}
        &= \sum_{i < j} \int_{\partial \Sigma_{ij}} \inr{Y_{ij}}{\n_{\partial ij}}\, d\mu^{n-2} \\
        \label{eq:cyclic-boundary}
        &= \sum_{i < j < k} \int_{\Sigma_{ijk}} \sum_{(\ell,m) \in \cyclic(i,j,k)} \inr{Y_{\ell m}}{\n_{\partial \ell m}}\, d\mu^{n-2}.
    \end{align}
    In order to see that this vanishes, we will show that for every distinct $i,j,k$,
    \[
        \sum_{(\ell, m) \in \cyclic(i,j,k)} \inr{Y_{\ell m}}{\n_{\partial \ell m}} = 0 \text{ on } \Sigma_{ijk} . 
    \]
    Indeed, for the last term in (\ref{eq:Y-def}), note that 
    \[
        \inr{X^{\tang_{\ell m}}}{\n_{\partial \ell m}} \div_\mu X = \inr{X}{\n_{\partial \ell m}} \div_\mu X,
    \]
    which vanishes when summed over $(\ell, m) \in \cyclic(i,j,k)$ by (\ref{eq:120}). 
        The other two terms in (\ref{eq:Y-def}) vanish after taking inner product with $\n_{\partial \ell m}$ and summing, because by Lemma~\ref{lem:cancellation}:
    \[
    \sum_{(\ell, m) \in \cyclic(i,j,k)} \inr{X}{\n_{\ell m}} \inr{ \nabla_{\n_{\ell m}} X}{\n_{\partial \ell m}} - \inr{X}{\n_{\partial \ell m}} \inr{\nabla_{\n_{\ell m}} X}{\n_{\ell m}} = T^{\alpha \beta \gamma} X_\alpha \nabla_\beta X_\gamma = 0 .
    \]

    Summarizing, we see that the entire second line of the formula in Lemma~\ref{lem:formula}
    vanishes, and we are left with the formula
    \begin{equation}\label{eq:formula-1}
        \sum_{i < j}
        \int_{\Sigma_{ij}} \Bigg[
         |\nabla^\tang X^\n|^2 - (X^\n)^2 \|\II\|_2^2 - (X^\n)^2 \nabla_{\n,\n}^2 W - \div_{\Sigma,\mu} (X^\n \nabla_{X^\tang} \n)
        \Bigg] d\mu^{n-1}
    \end{equation}
    for $Q(X)$. Note that the integrand in Lemma \ref{lem:formula} as a whole is integrable by Lemma~\ref{lem:regular}, and hence so is the integrand in~\eqref{eq:formula-1} 
   (since the terms we have already removed were integrable and integrated to zero). 
    On the other hand, all of the terms in~\eqref{eq:formula-1} besides the last
    are individually integrable: the third term is integrable because $\norm{\nabla^2 W}$ and $\abs{X}$ are bounded on $X$'s compact support, 
    the second term is integrable by Proposition \ref{prop:Schauder} \ref{it:Schauder-Sigma1},
    and the first term is integrable because $|\nabla^\tang X^\n| \le \|\nabla X\| + |X| \|\II\|$,
    which is square integrable again by Proposition \ref{prop:Schauder} \ref{it:Schauder-Sigma1}.
    It follows that the remaining term, $\div_{\Sigma,\mu} (X^\n \nabla_{X^\tang} \n)$,
    is  integrable on $\Sigma_{ij}$. 
    
    It remains to establish that:
    \begin{equation} \label{eq:formula-final}
        \sum_{i < j} \int_{\Sigma_{ij}} \div_{\Sigma,\mu} (X^\n \nabla_{X^\tang} \n) \, d\mu^{n-1}
        = \sum_{i < j} \int_{\partial \Sigma_{ij}} X^\n X^{n_\partial} \II_{\partial,\partial} \, d\mu^{n-2} .
    \end{equation}
    This will follow by applying Lemma~\ref{lem:stokes} (Stokes' theorem) to the $C^\infty$ vector-field
    \[
       Z_{ij} := X^{\n_{ij}} \nabla_{X^{\tang_{ij}}} \n_{ij} ,
    \]
    defined on $\Sigma_{ij} \cup \partial \Sigma_{ij}$ for each $i < j$. Let us first verify the integrability conditions of Lemma~\ref{lem:stokes}: $\div_{\Sigma,\mu}(Z_{ij})$ is     integrable by the preceding paragraph, and as $|Z_{ij}| \le |X|^2 \|\II\|$ and $X$ is tame, we see that $Z_{ij}$ is square-integrable on $\Sigma_{ij}$
    and integrable on $\Sigma_{ijk}$ by Proposition~\ref{prop:Schauder}. Since $Z_{ij}$ is clearly tangential to $\Sigma_{ij}$, Lemma~\ref{lem:stokes} yields:
    \[
    \sum_{i < j} \int_{\Sigma_{ij}} \div_{\Sigma,\mu}(Z_{ij}) \, d\mu^{n-1} = \sum_{i < j} \int_{\partial \Sigma_{ij}} X^{\n_{ij}} \inr{\nabla_{X^{\tang_{ij}}} \n_{ij}}{\n_{\partial ij}}
     \, d\mu^{n-2} .
    \]
    An application of Lemma~\ref{lem:boundary-curvature-cancellation0} then establishes (\ref{eq:formula-final}), and completes the proof of Theorem \ref{thm:formula}.

\subsection{Proof of Theorem \ref{thm:combined-vector-fields}}

To establish Theorem \ref{thm:combined-vector-fields}, we will need the following polarization formula with respect to constant vector fields, whose proof is deferred to Appendix \ref{sec:calculation}.  
\begin{lemma} \label{lem:polarization}
For any stationary cluster $\Omega$ whose cells are volume and perimeter regular, and for any admissible vector-field $X$:
\begin{align*}
& Q(X + w) = Q(X) + Q(w) \\
& + \sum_{i < j} \int_{\Sigma_{ij}} \Big[  \div_{\Sigma,\mu} (\nabla_w X) - 
    2\div_{\Sigma,\mu} (\nabla_w W \cdot X) - 2 X^\n \nabla^2_{w,\n} W    \\
& \hspace{46pt} - H_{ij,\mu} \brac{w^\n \div_\mu X - X^\n \nabla_w \pot} \Big ] d\mu^{n-1} .
\end{align*}
\end{lemma}

Recall that $\Omega$ is assumed to be stationary and regular, and $X$ is assumed to be tame. 
To simplify the integral in Lemma \ref{lem:polarization}, we start by applying Stokes'  theorem  (Lemma~\ref{lem:stokes}) to $Z_{ij}$, the restriction of the tame vector-field $Z = \nabla_w X - 2 X \nabla_w W$ to $\Sigma_{ij} \cup \partial \Sigma_{ij}$; by Remark \ref{rem:stokes}, the integrability assumptions of Lemma~\ref{lem:stokes} are in force. 
    It follows that:
   \begin{align*}
   & \sum_{i<j} \int_{\Sigma_{ij}} \Big [ \div_{\Sigma,\mu} (\nabla_w X) - 2\div_{\Sigma,\mu} (\nabla_w W \cdot X) \Big ] d\mu^{n-1} = \\
   & \sum_{i<j} \int_{\Sigma_{ij}} H_{ij,\mu} \brac{ \inr{\nabla_w X}{\n} - 2X^\n \nabla_w W} d\mu^{n-1} + \sum_{i<j} \int_{\partial \Sigma_{ij}} \inr{Z}{\n_{\partial ij}} d\mu^{n-2} . 
   \end{align*}
   But since $\sum_{(\ell,m) \in \cyclic(i,j,k)} n_{\partial \ell m} = 0$, the total contribution of the boundary integral is zero. 

   Plugging this into the formula of Lemma \ref{lem:polarization}, we deduce that:
   \begin{align}
   \label{eq:Q-almost-polarized}
   & Q(X + w) =  Q(X) + Q(w) \\
   \nonumber
   & +  \sum_{i < j} \int_{\Sigma_{ij}}  \Big [ H_{ij,\mu} \brac{\inr{\nabla_w X}{\n} - X^\n \nabla_w W - w^\n \div_\mu X  } - 2 X^\n \nabla^2_{w,\n} W \Big ] d\mu^{n-1} . 
   \end{align}
   We now claim that the first three terms of the above integrand are equal to $H_{ij,\mu} \div_{\Sigma,\mu} (X^\n w - w^\n X)$. Indeed, using the Leibniz rule $\div_{\Sigma,\mu} (f X) = f \div_{\Sigma,\mu} X + \nabla_{X^\tang} f$, the facts that $\div_{\Sigma,\mu} w = - \nabla_w W$ and that $\nabla w = 0$, and the symmetry of the second fundamental-form:
   \begin{align*}
   \div_{\Sigma,\mu} (X^\n w - w^\n X)
   &= X^\n \div_{\Sigma,\mu} w - w^\n \div_{\Sigma,\mu} X + \nabla_{w^{\tang}} (X^\n) - \nabla_{X^\tang} (w^\n) \\
   &= - X^\n \nabla_w W - w^\n \div_{\Sigma,\mu} X + \inr{\nabla_{w^\tang} X}{\n} \\
   &= - X^\n \nabla_w W - w^\n \div_{\Sigma,\mu} X + \inr{\nabla_{w} X}{\n} - w^\n \inr{\nabla_\n X}{\n} \\
   &= - X^\n \nabla_w W - w^\n \div_{\mu} X + \inr{\nabla_{w} X}{\n}.
                      \end{align*}
   Setting $Y_{ij} = X^{\n_{ij}} w - w^{\n_{ij}} X$ on $\Sigma_{ij} \cup \partial \Sigma_{ij}$, we would like to apply Stokes' theorem (Lemma~\ref{lem:stokes}) again. 
   The first and third integrability assumptions from (\ref{eq:stokes-assumptions}) are satisfied because $X$ is tame (as in Remark \ref{rem:stokes}). The second assumption is also satisfied because, as usual, $\div_{\Sigma,\mu} Y_{ij}$ is bounded by $\norm{\II^{ij}}$ times a polynomial in $\abs{w}$, $\abs{X}$ and $\norm{\nabla X}$, and hence is integrable on $\Sigma_{ij}$ by tameness, Proposition \ref{prop:Schauder} \ref{it:Schauder-Sigma1} and $\mu(\Sigma_{ij}) < \infty$; alternatively, $\div_{\Sigma,\mu} Y_{ij}$ must be integrable by Lemma \ref{lem:regular}, as all of the other terms in (\ref{eq:Q-almost-polarized}) are integrable. 
 
 Applying Lemma~\ref{lem:stokes}, since $Y_{ij}$ is clearly tangential, we obtain as in (\ref{eq:cyclic-boundary}):
 \[
 \sum_{i<j} \int_{\Sigma_{ij}} H_{ij, \mu} \div_{\Sigma,\mu} Y_{ij} d\mu^{n-1} = \sum_{i<j<k} \int_{\Sigma_{ijk}} \sum_{(\ell,m) \in \cyclic(i,j,k)} H_{\ell m ,\mu} \scalar{Y_{\ell m} , \n_{\partial \ell m}} d\mu^{n-2} . 
\]
We now claim that the above integrand is pointwise zero. Indeed, introducing the following vector-field on $\Sigma_{ijk}$:
\[
\Lambda  := - \frac{2}{\sqrt{3}} (\lambda_i \n_{\partial jk} + \lambda_j \n_{\partial ki} + \lambda_k \n_{\partial ij}) ,
\]
it is immediate to check that $\scalar{\Lambda , n_{\ell m}} = \lambda_\ell - \lambda_m = H_{\ell m , \mu}$. Hence, \[
 H_{\ell m ,\mu} \scalar{Y_{\ell m} , \n_{\partial \ell m}}  =X^{\n_{\ell m}} \Lambda^{\n_{\ell m}} w^{\n_{\partial \ell m}} - w^{\n_{\ell m}} \Lambda^{\n_{\ell m}} X^{\n_{\partial \ell m}} ,
\]
and so summing over all $(\ell, m) \in \cyclic(i,j,k)$ we obtain $T^{\alpha \beta \gamma} X_\alpha \Lambda_\beta w_\gamma = 0$, where $T^{\alpha \beta \gamma}$ is the $3$-tensor from Lemma~\ref{lem:cancellation} which vanishes identically. 

We conclude from the above discussion that:
\begin{equation} \label{eq:QX+w}
Q(X + w) =  Q(X) + Q(w) - 2 \sum_{i<j} \int_{\Sigma_{ij}} X^\n \nabla^2_{w,\n} W  d\mu^{n-1} .
\end{equation}
In particular, we formally deduce by setting $X = w$ and using $Q(2 w) = 4 Q(w)$, that:
\begin{equation} \label{eq:Qw}
Q(w) = - \sum_{i<j} \int_{\Sigma_{ij}} w^\n \nabla^2_{w,\n} W  d\mu^{n-1} .
\end{equation}
This is only formal, since $w$ is not a tame vector-field, but this actually holds rigorously by Lemma \ref{lem:cor-Q}. 
 Plugging (\ref{eq:Qw}) and the expression for $Q(X)$ from Theorem~\ref{thm:formula} into (\ref{eq:QX+w}), the assertion of Theorem \ref{thm:combined-vector-fields} readily follows.

\section{Approximate Inward Fields} \label{sec:inward-fields}

An additional crucial consequence of Theorem~\ref{thm:regularity} is that we can define a
consistent family of (approximate) ``inward'' vector-fields: for each cell $\Omega_i$, we
will try to define a smooth vector-field $X_i$ such that
$\inr{X_i}{\n_{\Omega_i}} \equiv -1$ (recall that $\n_{\Omega_i}$ is the
outward unit normal to $\Omega_i$, defined on $\partial^* \Omega_i$).
This family of approximate inward fields will be crucial for our variational arguments.

Theorem~\ref{thm:Almgren} allows us to construct the inward fields on $\Sigma^1
= \bigcup_{ij} \Sigma_{ij}$, since $X_i := -\n_{ij}$ is smooth on $\Sigma_{ij}$.
The issue is to extend these vector fields smoothly to all of $\R^n$,
and this is where Theorem~\ref{thm:regularity} comes in: it will allow
us to extend our inward fields smoothly to $\Sigma^{2}$ (the triple points), and to construct
approximate inward fields near $\Sigma^3$ (the quadruple points).

As usual, let $\delta_{ij}$ denote the Kronecker delta (i.e., $1$ if $i = j$ and $0$
otherwise), and let $\nabla^\tang$ denote the tangential component of the derivative.

\begin{proposition}[Existence of Approximate Inward Fields] \label{prop:inward-fields}
    Let $\Omega$ be a stationary regular cluster with respect to $\mu$. 
    For every $\epsilon_1 > 0$, there is a subset $K \subset \R^n$,
        such that for every $\epsilon_2 > 0$, there is a family of vector-fields
    $X_1, \dots, X_q$ with the following properties:
    \begin{enumerate}[(1)]
        \item $K$ is compact, disjoint from $\Sigma^4$, satisfies $\mu^{n-1}(\Sigma \setminus K) \le \epsilon_1$, and each $X_k$ is $C_c^\infty$ and supported inside $K$;
            \label{it:inward-fields-smooth}
        \item for every $k$,             $\displaystyle
                \int_{\Sigma^1} |\nabla^\tang X_k^\n|^2 \, d\mu^{n-1} \le \epsilon_1;
            $
            \label{it:inward-fields-gradient}
        \item             $\sum_{i < j} \mu^{n-1}\{p \in \Sigma_{ij} : \exists k \in \{ 1,\ldots,q \} \;\;\; X_k^{\n_{ij}}(p) \ne \delta_{kj} - \delta_{ki}\} \le \epsilon_1 + \epsilon_2;$ 
            \label{it:inward-fields-inward}
        \item For every $i \neq j$, for every $p \in \Sigma_{ij}$, there is
            some $\alpha \in [0, 1]$ such that for every $k$, 
            $|X_k^{\n_{ij}}(p) - \alpha (\delta_{kj} - \delta_{ki})| \le \alpha \epsilon_2;$
            \label{it:inward-fields-pointwise}
        \item for every $k$, and at every point in $\R^n$, $|X_k| \le \sqrt {3/2}$.             \label{it:inward-fields-bounded}
    \end{enumerate}
\end{proposition}

\begin{definition}[Approximate Inward Fields]
A family $X_1,\ldots,X_q$ of vector-fields satisfying properties (1)-(5) above is called a family of $(\eps_1,\eps_2)$-approximate inward fields. 
\end{definition}

The distinction between $\eps_1$ and $\eps_2$ and their order of quantification will be important in only a single instance in this work (Lemma \ref{lem:R}), but in all other applications of Proposition \ref{prop:inward-fields} we will simply use $\eps_1=\eps_2 = \eps$.

\subsection{Conformal properties}

To construct our inward fields we will require a bit more information on the conformal properties of the diffeomorphisms appearing in Theorem~\ref{thm:regularity}. Below, $\Omega$ is assumed to be a stationary regular cluster. 

\begin{lemma}\label{lem:angle-preservation}
    Suppose that $\phi$ is a local $C^1$ diffeomorphism defined on a neighborhood $N_p$ of $p \in
    \Sigma_{ijk}$ that maps $p$ to the origin and $\Sigma$ to $\Y \times \R^{n-2}$.
    Let $E = E^{(2)}$, and write $P_E$ for the orthogonal projection onto $E$.
    For $r \in N_p \cap \Sigma_{ijk}$, let
    $W_{ijk}(r) = \spn\{\n_{\partial ij}(r), \n_{\partial jk}(r), \n_{\partial ki}(r)\}$.

    Then for any $r \in N_p \cap \Sigma_{ijk}$, $d_r \phi$ maps $W_{ijk}(r)^\perp$ into
    $\{0\} \times \R^{n-2}$, and
    $P_E \circ d_r \phi$ is a conformal (i.e. angle-preserving) transformation from $W_{ijk}(r)$ to $E$.
\end{lemma}

\begin{proof}
    First, note that if $A: E \to E$ is any linear transformation that
    preserves $\Y$ then $A$ is conformal.
    Indeed, up to relabelling coordinates we may assume that $A$
    preserves each ``arm'' of the $\Y$. For $\{i, j, k\} = \{1, 2, 3\}$,
    let $w_i = e_j + e_k - 2e_i$, so that each arm of the $\Y$ has the
    form $\{\lambda w_i : \lambda \ge 0\}$. Since $A$ preserves the arms,
    $A w_i$ is a positive multiple of $w_i$; say, $A w_i = \lambda_i w_i$.
    On the other hand, $\sum_i \lambda_i w_i = A \sum_i w_i = 0$, and since the $w_i$'s are affinely independent, it follows that the
    $\lambda_i$'s are all equal, and so $A$ is a multiple of the identity.

    Now consider $r \in N_p \cap \Sigma_{123}$ (without loss of generality, $\{i,j,k\} = \{1,2,3\}$).
    Up to relabeling the coordinates,
    we may assume that for distinct $i,j,k \in \{1, 2, 3\}$, $\phi$ maps $\Sigma_{ij}$ into $\Sigma_{ij}^m := \Y_{ij} \times \R^{n-2}$, where $\Y_{ij} := \{x \in E : x_i = x_j > x_k\}$.
    To prove the first claim, note that if $v \in W_{123}(r)^\perp$ then $v$ is
    tangent to each of the surfaces $\Sigma_{ij}$ at $r$. It follows that
    $(d_r \phi) v$ is tangent to each of the surfaces $\Sigma_{ij}^m$,
    and so $(d_r \phi) v \in \{0\} \times \R^{n-2}$.

    To prove the second claim, note that
    by Corollary~\ref{cor:boundary-normal-sum}, the boundary
    normals $\n_{\partial ij}$ meet at $120^\circ$ angles, and so the cone
    $\tilde \Y = \bigcup_{i < j} \{\lambda \n_{\partial ij}(r): \lambda \le 0\}$
    is isometric to $\Y$. Since $\n_{\partial ij}$ is tangent to
    $\Sigma_{ij}$, $(d_r \phi) \n_{\partial ij}$ is tangent to $\Sigma_{ij}^m$,
    and so $P_E (d_r \phi) \n_{\partial ij}$ is tangent to $\Y_{ij}$.
    It follows that (when restricted to $W_{123}(r)$) $P_E \circ d_r \phi$ maps $\tilde \Y$
    to $\Y$, and so by the first paragraph, it is conformal on $W_{123}(r)$. 
    \end{proof}

A convenient consequence of Lemma \ref{lem:angle-preservation} is that, after composing with a linear function if necessary, we may always assume that each local diffeomorphism as in Theorem~\ref{thm:regularity} \ref{it:regularity-Sigma2}
is an isometry at $p$: 
\begin{corollary}\label{cor:isometry-Sigma2}
        For every $p \in \Sigma^2$ there is a $C^{\infty}$ diffeomorphism $\phi$ defined on a neighborhood $N_p \ni p$
    such that $\phi(p) = 0$, $\phi(N_p \cap \Sigma) \subset \Y \times \R^{n-2}$, and $d_p \phi$ is an isometry.
\end{corollary}

\begin{proof}
    Suppose that $p \in \Sigma_{ijk}$ and let $W_{ijk} = \spn\{\n_{\partial ij}(p), \n_{\partial jk}(p), \n_{\partial ki}(p)\}$.
    By Corollary~\ref{cor:smooth-Sigma2}, there exists a $C^\infty$ diffeomorphism $\phi$ satisfying $\phi(p) = 0$
    and $\phi(N_p \cap \Sigma) \subset \Y \times \R^{n-2}$. By Lemma~\ref{lem:angle-preservation}, 
    $P_E \circ d_p \phi$ is conformal on $W_{ijk}$. Let $\{v_1, \dots, v_n\}$ be an orthonormal basis
    for $\R^n$ such that $\{v_1, v_2\}$ is an orthonormal basis for $W_{ijk}$.
    Since $P_E \circ d_p \phi$
    is conformal, there exists an orthonormal basis $\{w_1,w_2\}$ for $E$ and $\alpha > 0$ so that 
    $P_E (d_p \phi) v_i = \alpha w_i$ for $i=1,2$; let us complete $\{w_1,w_2\}$  to
    an orthonormal basis $\{w_1, \dots, w_n\}$ for $E \times \R^{n-2}$.
        
    Now consider the linear operator $d_p \phi$ expressed as an invertible matrix $M$ in the bases $\{v_1, \dots, v_n\}$
    and $\{w_1, \dots, w_n\}$; that is, $M_{\ell m} = \inr{w_\ell}{(d_p \phi) v_m}$.
    We will write $M$ in blocks as $(\begin{smallmatrix} A & B \\ C &D\end{smallmatrix})$,
    where $A$ is $2 \times 2$ and $D$ is $(n-2) \times (n-2)$. Our choice of $w_1$ and $w_2$ ensures
    that $A = \alpha \Id$. By Lemma~\ref{lem:angle-preservation},
    $(d_p \phi) v_m \in \{0\} \times \R^{n-2}$ for every $m \ge 3$, and it follows that $B = 0$.
    Then
    \[
        M^{-1} = \begin{pmatrix} A^{-1} & 0 \\ -D^{-1} C A^{-1} & D^{-1} \end{pmatrix}
        = \begin{pmatrix} \frac{1}{\alpha} \Id & 0 \\ - \frac{1}{\alpha} D^{-1} C & D^{-1} \end{pmatrix}.
    \]
    Let $f: E \times \R^{n-2} \to E \times \R^{n-2}$ be the linear function that sends $w_m$ to $\sum_\ell M^{-1}_{\ell m} w_\ell$; 
    that is, $f$ is the linear function that is represented in the basis $\{w_1, \dots, w_n\}$
    by the matrix $M^{-1}$. The form of $M^{-1}$ above implies that $f(\Y \times \R^{n-2}) \subset \Y \times \R^{n-2}$,
    because if $(x, y) \in \Y \times \R^{n-2}$ then the $E$-component of $f(x,y)$ is $\alpha^{-1} x \in \Y$.
    Hence, $\tilde \phi = f \circ \phi$ is a local $C^\infty$ diffeomorphism defined on $N_p$ satisfying
    $\tilde \phi(p) = 0$ and $\tilde \phi(N_p \cap \Sigma) \subset \Y \times \R^{n-2}$. Moreover, $d_p \tilde \phi$
    is an isometry because $(d_p \tilde \phi) v_m = (d_0 f) (d_p \phi) v_m = w_m$ by the definition of $f$.
\end{proof}

By an analogous argument,
we may also choose the local diffeomorphisms of Theorem~\ref{thm:regularity} \ref{it:regularity-Sigma3}
to be isometries at $p$:

\begin{corollary}\label{cor:isometry-Sigma3}
        For every $p \in \Sigma^3$ there is a $C^{1,\alpha}$ diffeomorphism $\phi$ defined on a neighborhood $N_p \ni p$
    such that $\phi(p) = 0$, $\phi(N_p \cap \Sigma) \subset \T \times \R^{n-3}$, and $d_p \phi$ is an isometry.
\end{corollary}

\begin{proof} The proof is essentially identical to the proof of
    Corollary~\ref{cor:isometry-Sigma2}. The first main point is that (like
    $\Y$) any linear transformation that preserves $\T$ must be conformal. The second is that, by Lemma~\ref{lem:angle-preservation} and continuity up to $\Sigma^3$, 
    at every $p \in \Sigma_{1234}$ and for every distinct $i,j,k \in \{1, 2, 3, 4\}$, $\n_{ij}$, $\n_{jk}$
    and $\n_{ki}$ are co-planar and meet at $120^\circ$ angles, and it follows that at any $r \in \Sigma^3$,
    the cone
    \[
        \tilde \T = \bigcup_{\{i,j,k,\ell\} = \{1,2,3,4\}} \{x \in \R^n: \inr{x}{\n_{ij}} = 0, \inr{x}{\n_{ik}} \le 0, \inr{x}{\n_{jk}} \le 0\}
    \]
    is isometric to $\T \times \R^{n-3}$. The rest of the proof remains unchanged. 
\end{proof}

\subsection{Construction}
 
Denote for brevity $\Sigma^{\ge 3} = \Sigma^3 \cup \Sigma^4$. 
We begin with the exact construction of the inward fields
on $\R^n \setminus \Sigma^{\ge 3}$.

\begin{lemma}\label{lem:inward-fields}
    There is a family of $C^\infty$ vector-fields $Z_1, \dots, Z_q$ defined 
    on $\R^n \setminus \Sigma^{\ge 3}$ such that for every $k$, for every $i \ne j$,
    and for every $p \in \Sigma_{ij}$,
    \[
        |Z_k(p)| \le \sqrt {3/2} \text{ and }
        \inr{Z_k(p)}{\n_{ij}(p)} = \delta_{kj} - \delta_{ki}.
    \]
\end{lemma}

\begin{proof}
    By applying a partition of unity, it suffices to prove the claim locally:
    we will show that for every $p \in \R^n \setminus \Sigma^{\ge 3}$, we may
    define vector fields $Z_1, \dots, Z_q$ on a neighborhood of $p$ to satisfy the
    required properties on that neighborhood.

    For $p$ outside of (the closed) $\Sigma$, this is easy:
    we may choose a neighborhood of $p$ which is disjoint from $\Sigma$,
    and define $Z_i \equiv \cdots \equiv Z_q \equiv 0$ on that neighborhood.
    
    For $p \in \Sigma_{ij}$, we may (by
    Theorem~\ref{thm:Almgren}) choose a neighborhood of $p$ that does not
    intersect any other $\Sigma_{k\ell}$; on that neighborhood, define $Z_i$ to
    be a smooth extension of $\n_{ji}$, define $Z_j$ to be a smooth extension of $\n_{ij}$,
    and define all other $Z_k$'s to be identically zero.

    Finally, we will describe the construction on the triple-point set $\Sigma^2$: suppose (in
    order to simplify notation) that $p \in \Sigma_{123}$ and set $E = E^{(2)}$. 
    We will choose a neighborhood $N_p$ and a local $C^\infty$ diffeomorphism $\phi$
    according to Corollary~\ref{cor:isometry-Sigma2}.
    By relabeling the coordinates,
    we may assume that $\phi$ sends $N_p \cap \Omega_i$ into $\{x \in E :  x_i > \max\{x_j , x_k\} \} \times \R^{n-2}$,
    and it follows that $\phi$ sends
    $N_p \cap \Sigma_{ij}$ into $\Sigma_{ij}^m := Y_{ij} \times \R^{n-2}$, for all distinct $i,j,k \in \set{1,2,3}$ (recall that $Y_{ij} := \{x \in E : x_i = x_j > x_k\}$).
    Let $w_i := \frac{1}{\sqrt{6}} (e_j + e_k - 2 e_i) \in E \times \R^{n-2}$,
            and note that the vector-fields $\{ -\frac{2}{\sqrt{3}} w_i \}$ satisfy the assertion on the model cone $\Sigma^m := \Y \times \R^{n-2}$: $w_i$ is tangent to $\Sigma^m_{jk}$, and it has a constant normal component on $\Sigma^m_{ij}$ and $\Sigma^m_{ik}$. 
        The basic idea is to define $Z_i$ as the pull-back of $-w_i$ via $\phi$, renormalized so as to satisfy the condition on
    $\inr{Z_i}{\n_{\ell m}}$. Let us describe the construction a little more carefully, in
    order to verify that the renormalization preserves smoothness.
                            
    Consider the metric $\inr{\cdot}{\cdot}_\phi$ on $\phi(N_p)$ obtained by pushing
    forward the Euclidean metric under $\phi$.
    According to
    Lemma~\ref{lem:angle-preservation}, at every point $r \in N_p \cap \Sigma_{123}$,
    there is some $c_r > 0$ such that for every $u, v \in W_{123}(r)$, $\inr{P_E (d_r \phi) u}{P_E (d_r\phi) v} = c_r \inr{u}{v}$. The same formula holds even if only one of $u$ and $v$ belongs to $W_{123}(r)$, since Lemma~\ref{lem:angle-preservation} implies that $P_E (d_r \phi) (u - P_{W_{123}(r)} u) = 0$ for every $u \in \R^n$.
    If, in addition, either $(d_r \phi) u$ or $(d_r \phi) v$ belongs to $E \times \{0\}$, then
    $\inr{(d_r \phi) u}{(d_r \phi) v} =  \inr{P_E (d_r \phi) u}{P_E (d_r\phi) v} = c_r \inr{u}{v}$.
    After changing variables, it follows that whenever at least one of $u$ and $v$ belongs to $E \times \{0\}$
    and at least one of $(d_r \phi)^{-1} u$ and $(d_r \phi)^{-1} v$ belongs to $W_{123}(r)$
    then $\inr{u}{v} = c_r \inr{(d_r \phi)^{-1} u}{(d_r \phi)^{-1} v} = c_r \inr{u}{v}_\phi$.

            Let $\tilde \n_{ij}$ be a vector-field on $\spn \Sigma^m_{ij} = \{x \in E^{(2)} : x_i = x_j\} \times \R^{n-2}$
    obtained by smoothly extending $(d\phi) \n_{ij}$ (for example, such a vector-field can be constructed by starting with $e_j - e_i$ and applying the
    Gram-Schmidt process with respect to the inner product $\inr{\cdot}{\cdot}_\phi$).
    Since, for any $r \in \Sigma_{123} \cap N_p$, 
    $w_k \in E \times \{0\}$ and $(d_r \phi)^{-1} \tilde \n_{ij} \in W_{123}(r)$,
    the previous paragraph implies that
    $\inr{w_k}{\tilde \n_{ij}(\phi(r))} = c_r \inr{(d_r \phi)^{-1} (w_k)}{\n_{ij}(r)} = 0$,
    because $w_k$ is tangent to $\Sigma^m_{ij}$ and so $(d_r \phi)^{-1} (w_k)$ is tangent to $\Sigma_{ij}$ at $r$.
    On the other hand,
    $\inr{P_E \tilde \n_{ij}}{P_E \tilde \n_{ik}} = c_r \inr{\n_{ij}}{\n_{ik}} = c_r/2$ and $|P_E \tilde\n_{ij}|^2 = c_r$ at the point $\phi(r)$. Combining the previous observations, it follows that
    $P_E \tilde \n_{ij}(\phi(r)) = \pm \sqrt{c_r/2} (e_j - e_i)$ for every $i \ne j \in \{1, 2, 3\}$.
    In fact, the assumption that $\phi$ maps $N_p \cap \Omega_i$ into $\{x \in E : x_i > \max \{x_j,x_k\} \} \times \R^{n-2}$ implies
    that $P_E \tilde \n_{ij}(\phi(r)) = \sqrt{c_r/2} (e_j - e_i)$.

    Now define the function $f_i$ on $\spn \Sigma^m_{ij} \cup \spn \Sigma^m_{ik}$
    by
    \[
        f_i(x) = \begin{cases}
            \inr{w_i}{\tilde \n_{ij}(x)}_\phi & \text{for $x \in \spn \Sigma^m_{ij}$} \\
            \inr{w_i}{\tilde \n_{ik}(x)}_\phi & \text{for $x \in \spn \Sigma^m_{ik}$}.
        \end{cases}
    \]
    This definition is consistent for $x \in \spn \Sigma^m_{ij} \cap \spn \Sigma^m_{ik} =
    \{0\} \times \R^{n-2}$: if $r = \phi^{-1}(x)$ then $w_i \in E \times \{0\}$ and $(d_r \phi)^{-1} \tilde \n_{ij} = \n_{ij} \in W_{123}(r)$, and so
    $c_r \inr{w_i}{\tilde \n_{ij}}_\phi = \inr{w_i}{\tilde \n_{ij}} = \sqrt{c_r/2} \inr{w_i}{e_j - e_i} = \sqrt{c_r/3}$
    (and similarly with $j$ replaced by $k$).
    Hence, $f_i$ may be extended to a $C^\infty$ function on all of $\R^n$, for example
    by setting $f_i(\frac{x + y}{2}) = \frac{f_i(x) + f_i(y)}{2}$ whenever $x \in \spn \Sigma^m_{ij}$, $y \in \spn \Sigma^m_{ik}$
    and $x-y \in \spn \Sigma^m_{ij} \cap \spn \Sigma^m_{ik}$, i.e. $x$ and $y$ agree in the last $n-2$ coordinates.
        Clearly $f_i$ does not vanish at the origin, and so by continuity, after possibly shrinking $N_p$ if necessary, $f_i$ does not vanish on the entire $\phi(N_p)$ for all $i=1,2,3$.

    Finally, for $i = 1, 2, 3$, define $Z_i = -(d \phi)^{-1} (w_i / f_i)$.
    By the definition of $f_i$, we have
    \[
        \inr{Z_i}{\n_{ij}} = -\inr{w_i/f_i}{\tilde \n_{ij}}_\phi = -1 \text{ on $\Sigma_{ij}$,}
    \]
    and similarly on $\Sigma_{ik}$. On the other hand, $w_i$
    is tangential (with respect to $\inr{\cdot}{\cdot}$ and therefore also with respect to $\inr{\cdot}{\cdot}_\phi$) to $\Sigma^m_{jk}$, and so $\inr{Z_i}{\n_{jk}} \equiv 0$ on
    $\Sigma_{jk}$.

    To prove the claim on the boundedness of $Z_i$, recall that $\inr{\cdot}{\cdot}_\phi$
    agrees with the Euclidean metric at the origin (as $d_p \phi$ is an isometry). It follows that
    $|Z_i(p)| = \sqrt{4/3}$. By continuity, we may shrink the neighborhood $N_p$
    in order to ensure that $|Z_i| \le \sqrt{3/2}$ on the whole neighborhood.
\end{proof}

Next, we will construct approximate inward fields on $\R^{n}\setminus
\Sigma^4$. Note that
we cannot simply imitate our construction from Lemma~\ref{lem:inward-fields},
because in a neighborhood of $p \in \Sigma^3$, $\Sigma$ is only $C^{1,\alpha}$-diffeomorphic
to the model cone $\T \times \R^{n-3}$. In particular, if we were to imitate that construction
then we would end up with merely $C^{0,\alpha}$ vector-fields. This would
not be sufficient for our purposes, for example because such non-Lipschitz vector-fields
do not necessarily admit flows.
To avoid this issue, we will drop the requirement that the normal components
be exactly constant, and settle for a pointwise approximation. This in fact makes the construction a fair amount simpler. 

\begin{lemma}\label{lem:approx-inward-fields}
    For every $\epsilon > 0$,
    there is a family of $C^\infty$ vector fields $Y_1, \dots, Y_q$ defined
    on $\R^n \setminus \Sigma^4$ such that
    for every $k$, for every $i \ne j$ and every $p \in \Sigma_{ij}$, 
    \[
        |Y_k(p)| \le \sqrt{3/2} \text{ and } |\inr{\n_{ij}(p)}{Y_k} - (\delta_{kj} - \delta_{ki})| \le \epsilon.
    \]
\end{lemma}

\begin{proof}
    By applying a partition of unity, it suffices to prove the claim locally:
    for every $p \in \R^n \setminus \Sigma^4$, we will locally define $Y_1, \dots, Y_{q}$
    with the required properties. As long as $p \in \R^n \setminus \Sigma^{\ge 3}$,
    we may set $Y_i$ to be identical to $Z_i$ from Lemma~\ref{lem:inward-fields}
    in a neighborhood of $p$; it remains to describe the construction in a neighborhood of 
    $p \in \Sigma^3$.

    Suppose (to simplify notation) that $p \in \Sigma_{1234}$. 
    Choose a neighborhood $N_p \ni p$ and a local $C^{1,\alpha}$ diffeomorphism $\phi$
    according to Corollary~\ref{cor:isometry-Sigma3}.
    By relabeling the coordinates, we may assume
    that $\phi$ sends $N_p \cap \Omega_i$ into $\{x \in E^{(3)} : x_i > \max_{j \neq i} x_j\} \times \R^{n-3}$.
    Since $d_p \phi$ is an isometry, it follows that $(d_p \phi) \n_{ij} = \frac{1}{\sqrt 2}(e_j - e_i)$
    for every distinct $i,j \in \{1,2,3,4\}$.

    For $\{i, j, k, \ell\} = \{1, 2, 3, 4\}$, let
    $w_i = \frac{1}{\sqrt{12}} (e_j + e_k + e_\ell - 3 e_i) \in E^{(3)} \times \R^{n-3}$.
    Then each $w_i$ satisfies
    $\inr{\n_{ij}(p)}{(d_p \phi)^{-1} w_i} = \sqrt{2/3}$,
    and $\inr{\n_{jk}(p)}{(d_p \phi)^{-1}w_i} = 0$ if $i \not \in \{j,k\}$. For each $i=1,\ldots,4$, define
    $Y_i$ in a neighborhood of $p$ to be the constant vector-field
    $-\sqrt{3/2} (d_p \phi)^{-1} w_i$. By continuity, there is a neighborhood
    of $p$ on which $\inr{Y_i}{\n_{ij}} \in [-1-\epsilon, -1+\epsilon]$
    and $\inr{Y_i}{\n_{jk}} \in [-\epsilon, \epsilon]$ whenever
    $i \not \in \{j, k\}$. Moreover, the fact $d_p \phi$ is an isometry implies
    that $|Y_i| = \sqrt{3/2}$. Finally, we define $Y_5, \dots, Y_q$ to be zero on
    this neighborhood.
\end{proof}

We can now combine the exact inward fields of Lemma~\ref{lem:inward-fields}
with the approximate inward fields of Lemma~\ref{lem:approx-inward-fields},
using the cutoff functions of Subsection \ref{subsec:cutoff}. 

\begin{proof}[Proof of Proposition~\ref{prop:inward-fields}]
    Fix $\epsilon_1 > 0$, and
    choose $\eta$ as in Lemma~\ref{lem:sigma4-cutoff}, with $\epsilon=\epsilon_1$. 
    Let $K$ be the support of $\eta$, and let $K_1 := \{ \eta = 1 \} \subset K$; by Lemma~\ref{lem:sigma4-cutoff}, $K$ and $K_1$ 
    are compact, disjoint from $\Sigma^4$, and satisfy $\mu^{n-1}(\Sigma \setminus K) \le \mu^{n-1}(\Sigma \setminus K_1) \le \epsilon_1$.
        Let $Z_1, \dots, Z_q$ be vector fields as in Lemma~\ref{lem:inward-fields}.

    Let $C_\eta$ be the constant $C_U$ of Lemma~\ref{lem:sigma3-cutoff} applied to $U = \R^n \setminus K$.
    Then let $Y_1, \dots, Y_q$ be vector-fields as in Lemma~\ref{lem:approx-inward-fields} with parameter $\eps = \sqrt{\epsilon_1/C_\eta}$.
    For some $\delta_0 > 0$ to be determined (depending on $K$, $\{Y_k\}$ and $\Sigma$),
    choose $\xi$ as in Lemma~\ref{lem:sigma3-cutoff}, with parameter $\delta=\min(\epsilon_2,\delta_0)$. 
    Define
    \[
        X_k = \eta \cdot (\xi Z_k + (1-\xi) Y_k).
    \]
    Note that this defines a $C^\infty$ vector-field on $\R^n$, because
    $\supp (\eta \xi)$ is contained in the domain of $Z_k$, namely $\R^n \setminus \Sigma^{\geq 3}$, 
    and $\supp \eta$ is contained in the domain of $Y_k$, namely $\R^n \setminus \Sigma^4$; as $\supp X_k \subset \supp \eta = K$, this establishes requirement~\ref{it:inward-fields-smooth}.

    Next, since $\nabla^{\tang} Z_k^{\n} = 0$ on $\Sigma^1$, observe that
    \[
        \nabla^\tang X_k^\n = (\xi Z_k + (1-\xi) Y_k)^\n \nabla^\tang \eta
        + \eta \cdot (Z_k - Y_k)^\n \nabla^\tang \xi
        + \eta (1-\xi) \nabla^\tang Y_k^\n.
    \]
    To control the first term above, note that $|(\xi Z_k + (1-\xi) Y_k)^\n| \le \sqrt {3/2}$.
    For the second term, note that $|\eta \cdot (Z_k - Y_k)^\n| \le \sqrt{\epsilon_1/C_\eta}$.
    For the third term (which vanishes outside the set $\{\xi < 1\}$), we split $\nabla^\tang Y^\n_k$
    into its normal and tangential components as $(\nabla^\tang Y_k)^\n + \nabla_{Y^\tang_k} \n$, and observe that
    $|\nabla_{Y^\tang_k} \n| \le |Y_k| \|\II\|_2 \le \sqrt{3/2} \|\II\|_2$.
    Hence,
    \[
        |\nabla^\tang X_k^\n|^2 \le 3 |\nabla \eta|^2 + \frac{2\epsilon_1}{C_\eta} |\nabla \xi|^2
        + 1_{\{\xi < 1\}} \eta^2 \brac{2 \|\nabla^{\tang} Y_k \|_2^2 + 3 \|\II\|_2^2}.
    \]
    Integrating, we use Lemma~\ref{lem:sigma4-cutoff} to control $|\nabla \eta|^2$
    and Lemma~\ref{lem:sigma3-cutoff} to control $|\nabla \xi|^2$:
    \[
        \int_{\Sigma^1} |\nabla^\tang X_k^\n|^2 \, d\mu^{n-1}
        \le 5 \epsilon_1 + 3 \int_{\Sigma^1} 1_{\{\xi < 1\}} \eta^2 \brac{ \|\nabla^\tang Y_k \|_2^2 + \|\II\|_2^2} d\mu^{n-1}.
    \]
    Now, $\nabla Y_k$ is uniformly bounded on the compact support of $\eta$, and hence $\eta^2 \|\nabla^\tang Y_k\|^2$
    is uniformly bounded and therefore integrable, as $\mu^{n-1}(\Sigma^1) < \infty$. In addition,
    $\eta^2 \|\II\|_2^2$ is integrable by Proposition~\ref{prop:Schauder} \ref{it:Schauder-Sigma1}, as the cutoff function $\eta$ is compactly supported away from $\Sigma^4$.
    Hence, we may choose $\delta_0 > 0$ to be sufficiently small so that the integral
    on the right hand side above is at most $\epsilon_1$.
    By appropriately modifying $\epsilon_1$ by a constant factor, this proves requirement~\ref{it:inward-fields-gradient}.

    For requirement~\ref{it:inward-fields-inward}, note that
    $X_k^{\n_{ij}} = Z_k^{\n_{ij}} = \delta_{kj} - \delta_{ki}$ whenever $\eta = \xi = 1$.
    Since $K_1 = \set{\eta = 1}$, we deduce that
    \[
    \bigcup_{i<j} \{p \in \Sigma_{ij} \cap K_1: \exists k \;\; X_k^{\n_{ij}} \ne \delta_{kj} - \delta_{ki}\} \subset \{p \in \Sigma^1 \cap K_1: \xi < 1\}.
    \]
     Hence, requirement~\ref{it:inward-fields-inward} follows by the union-bound from Lemma~\ref{lem:sigma3-cutoff} and the fact that $\mu^{n-1}(\Sigma^1 \setminus K_1) \leq \eps_1$.
    Requirements~\ref{it:inward-fields-pointwise} and~\ref{it:inward-fields-bounded}
    follow from Lemmas~\ref{lem:inward-fields}
    and~\ref{lem:approx-inward-fields}, and the fact that both of these properties
    are preserved by pointwise convex combinations (the $\alpha$ in requirement~\ref{it:inward-fields-pointwise}
    will be $\eta(p)$).
\end{proof}

\subsection{Some useful estimates}

In this subsection we record some useful estimates which follow from the properties of approximate inward fields. 
Recall the definition of the quadratic form $L_A$ from (\ref{eq:L_A}). 

\begin{lemma} \label{lem:inward-combination} 
Let $\Omega$ be a stationary regular cluster, let $A_{ij} = \mu^{n-1}(\Sigma_{ij})$, and let $(X_1,\ldots,X_q)$ be a collection of $(\eps_1,\eps_2)$-approximate inward fields. Given $a \in \R^q$, set $X := \sum_{k=1}^q a_k X_k$. Then:
\begin{enumerate}[(i)]
\item \label{it:inward-nabla} $\int_{\Sigma^1}  |\nabla^{\tang} X^{\n}|^2 d\mu^{n-1} \leq q |a|^2 \eps_1$ . 
\item \label{it:inward-L2} $\int_{\Sigma^1} (X^{\n})^2 d\mu^{n-1} \geq a^T L_A a -  2 |a|^2  (\eps_1 + \eps_2)$. 
\item \label{it:inward-delta-V} For all $i$, $\abs{(\delta_X V(\Omega) + L_A a)_i} \leq \max(1,\eps_2)(\eps_1+\eps_2) \sqrt{q} |a|$. 
\end{enumerate}
\end{lemma}
\begin{proof}
The first assertion follows since by property \ref{it:inward-fields-gradient} of approximate inward fields and Cauchy--Schwarz:
\[
\int_{\Sigma^1}  |\nabla^{\tang} X^{\n}|^2 d\mu^{n-1} \leq \int_{\Sigma^1} |a|^2 \sum_{k=1}^q |\nabla^{\tang} X^{\n}_k|^2 d\mu^{n-1} .
\]
The second assertion follows since, denoting:
\[
q_{ij} := \mu^{n-1}(p \in \Sigma_{ij} \; ; \;  \exists k=1,\ldots,q \;\; X^{\n}_k \neq \delta_{kj} - \delta_{ki} ),
\]
we have $\sum_{i<j} q_{ij} \leq \eps_1 + \eps_2$ by property \ref{it:inward-fields-inward}.  Therefore:
\[
\int_{\Sigma_{ij}} (X^{\n})^2 d\mu^{n-1} \geq (a_j - a_i)^2 \brac{\mu^{n-1}(\Sigma_{ij}) - q_{ij} } ,
\]
and after summation over $i<j$, we obtain:
\begin{align*}
\sum_{i < j} \int_{\Sigma_{ij}} (X^{\n})^2 d\mu^{n-1} & \geq \sum_{i<j} A_{ij} (a_j - a_i)^2 - (\eps_1 + \eps_2) \max_{i < j}  (a_j - a_i)^2 \\
& \geq a^T L_A a -   (\eps_1 + \eps_2) 2 |a|^2  .
\end{align*}

To see the third assertion, first note that:
\[
\int_{\Sigma_{ij}} |X_k^\n  - (\delta_{kj} - \delta_{ki})| d\mu^{n-1} \leq q_{ij} \max(1,\eps_2) ,
\]
 since by property \ref{it:inward-fields-pointwise}:
\[
\sup_{p \in \Sigma_{ij}} |X_k^\n(p) - (\delta_{jk} - \delta_{ik})| \le \max(1,\eps_2) . 
\]
It follows by (\ref{eq:first-variation-volume}) that:
\begin{align*}
     |(\delta_X V(\Omega) + L_A a)_i| & \leq \abs{\sum_{j \neq i} \int_{\Sigma_{ij}} X^{\n} \, d\gamma^{n-1} - \sum_{j \neq i} (a_j - a_i) A_{ij}} \\
    & \leq \sum_{j \neq i} \int_{\Sigma_{ij}} \abs{ \sum_{k} a_k X^\n_k  - (a_j - a_i) } \, d\gamma^{n-1} \\
    & \le \sum_{j \neq i} \int_{\Sigma_{ij}}  \sum_k  \abs{a_k} \abs{X^\n_k -  (\delta_{jk} - \delta_{ik})} \,  d\gamma^{n-1} \\
        &\le \max(1,\eps_2) \sum_{j \neq i}  q_{ij} \sum_k |a_k| \le  \max(1,\eps_2)(\eps_1+\eps_2) \sqrt{q} |a|. 
\end{align*}        
\end{proof}

\begin{lemma} \label{lem:trace} 
Let $(X_1,\ldots,X_q)$ be a collection of $(\eps_1,\eps_2)$-approximate inward vector-fields with $\eps_2 < 1$. Then for all distinct $i,j,k$, and at every point of $\Sigma_{ijk}$, 
\begin{enumerate}
\item There exists $\alpha \in [0,1]$ so that:
\[
\norm{P_{W_{ijk}} \brac{\sum_{\ell=1}^q X_\ell X_\ell^T -  2 \alpha \text{Id}} P_{W_{ijk}}} \leq C_q \alpha \eps_2 ,
\]
where $P_{W_{ijk}}$ denote the orthogonal projection onto the two-dimensional subspace $W_{ijk} = \spn(\n_{ij},\n_{jk},\n_{ki})$, $\norm{\cdot}$ denotes the Hilbert-Schmidt norm, and $C_q$ is a constant depending solely on $q$. 
\item For any linear operator $Z : \R^n \rightarrow \R^n$ so that $P_{W_{ijk}} Z  = Z P_{W_{ijk}} = Z$ and $\tr(Z) =0$:
\[
\abs{\tr\brac{ (\sum_{\ell=1}^q X_\ell X_\ell^T) Z }} \leq C_q \norm{Z} \eps_2  .
\]
\end{enumerate}
\end{lemma}
\begin{proof}
By continuity, property \ref{it:inward-fields-pointwise} extends from $\Sigma_{ij}$ to $\Sigma_{ijk}$. Given $p \in \Sigma_{ijk}$, the constant $\alpha = \alpha_p \in [0,1]$ from property \ref{it:inward-fields-pointwise} is a common scaling factor to all $\{X_\ell\}_{\ell=1}^q$, and so it is enough to prove the claim for $\alpha=1$. Since $|\scalar{X_i,\n_{ji}} - 1|\leq \eps_2$, $|\scalar{X_i,\n_{ki}} - 1|\leq \eps_2$ and $|\scalar{X_i,\n_{jk}}| \leq \eps_2$, an easy calculation verifies that:
\[
|P_{W_{ijk}} X_i - \frac{2}{\sqrt{3}} \n_{\partial jk} | \leq \frac{2}{\sqrt{3}} \eps_2 .  \]
Denoting $V_i = \frac{2}{\sqrt{3}} \n_{\partial jk}$ and $E_i = P_{W_{ijk}} X_i - V_i$, note that:
\[
\norm{P_{W_{ijk}} X_i X_i^T P_{W_{ijk}} - V_i V_i^T} \leq 2 \norm{E_i V_i^T} + \norm{E_i E_i^T} = 2 \abs{E_i} \abs{V_i} + \abs{E_i}^2 \leq \frac{8}{3} \eps_2 + \frac{4}{3} \eps^2_2 . 
\]
On the other hand, if $\ell \notin \set{i,j,k}$ then similarly: 
\[
\norm{P_{W_{ijk}} X_\ell X_\ell^T P_{W_{ijk}}} = \abs{P_{W_{ijk}} X_\ell}^2 \leq \frac{4}{3} \eps_2^2 . 
\]
Since:
\[
V_i V_i^T + V_j V_j^T + V_k V_k^T = 2 P_{W_{ijk}}  ,
\]
it follows that:
\[
\norm{P_{W_{ijk}} \brac{\sum_{\ell=1}^q X_\ell X_\ell^T - 2 \text{Id}} P_{W_{ijk}} } \leq 8 \eps_2 + \frac{4}{3} q \eps_2^2 ,
\]
establishing the first assertion. 

For the second assertion, denote $T = \sum_{\ell=1}^q X_\ell X_\ell^T$ and $W = W_{ijk}$, and write:
\[
\tr(T Z) = \tr(P_W T P_W Z) + \tr((T - P_W T P_W) Z) = \tr(P_W (T - 2\alpha\text{Id}) P_W Z) ,
\]
using both properties of $Z$. It remains to apply the Cauchy--Schwarz inequality and the estimate from the first assertion. 
\end{proof}

\section{Stable Regular Clusters for the Gaussian Measure} \label{sec:stable}

From this point on, unless otherwise stated, we will switch from working with the general measure $\mu$
to the standard Gaussian measure $\gamma$. In this section, we identify several crucial properties 
of Gaussian-stable regular clusters when the number
of cells is at most $n+1$. In fact, we can say a bit more, as described below.

\begin{definition}[Flat Cluster]      A regular cluster $\Omega$ is called \emph{flat} and is said to have \emph{flat cells} if for every $i \ne j$, $\II^{ij} = 0$ on $\Sigma_{ij}$. 
    \end{definition}

\begin{theorem}[Stable Regular Clusters] \label{thm:flat}
   Let $\Omega$ denote a Gaussian-stable regular $q$-cluster, and assume that $q \leq n+1$. Then:
   \begin{enumerate}[(i)]
   \item $\Omega$ has flat, connected and convex cells. 
   \item    $\n_{ij}$ is constant on every non-empty $\Sigma_{ij}$, and $\Sigma_{ij}$ is contained in a single hyperplane perpendicular to $\n_{ij}$. In addition $\dim \spn \{ \n_{ij} \}_{i < j} \leq q-1$. 
   \item  All non-empty cells are convex polyhedra with at most $q-1$ facets, given by:
    \begin{equation} \label{eq:polyhedron-formula}
    \Omega_i = \bigcap_{j \neq i : A_{ij} > 0 } \{ x \in \R^n : \scalar{\n_{ij},x} < \lambda_j - \lambda_i \} ,
    \end{equation}
    where, recall, $\lambda \in E^*$ is given by Lemma \ref{lem:first-order-conditions} \ref{it:first-order-cyclic}. 
        \end{enumerate}
\end{theorem}

A further refinement of this theorem will be given in Theorem \ref{thm:pull-back}. 
The main task in the proof of Theorem~\ref{thm:flat} is to prove that the
cells are flat. This task will be broken up into two cases, depending on
whether the cluster is full-dimensional or dimension-deficient (recall Definition \ref{def:effective}). 

In both cases, we will assume that the cluster is not flat, and produce a vector-field to demonstrate that it is not stable. 
When the cluster is dimension-deficient (and without any restrictions on $q$!), 
we will build such a vector field by taking a carefully chosen linear combination of inward fields, multiplied by a ``height''-dependent linear function in a direction of dimension-deficiency.
When the cluster is full-dimensional and $q \leq n+1$, we will take again a carefully chosen linear combination of inward fields together with a constant field. Vector-fields of the latter form were previously considered in~\cite{McGonagleRoss:15} but in the context of a single set with smooth boundary instead of a cluster (of course, in a single set setting, having moreover a smooth boundary, there is no need to construct the inward field as it is simply given by the interior unit-normal, and its variations are well-known and classical).

\subsection{Negative contribution of curvature terms}

In this subsection we temporarily revert to treating a general measure $\mu$, as the Gaussian measure's properties will not be used.

The key lemma we will require pertains to the following quantity associated to a regular cluster $\Omega$ (with respect to $\mu$) and a tame vector-field $X$:
\[ R(X) :=
        \sum_{i < j} \Big[\int_{\Sigma_{ij}} \brac{ |\nabla^\tang X^\n|^2 - (X^\n)^2 \|\II\|_2^2 } d\mu^{n-1}
    - \int_{\partial \Sigma_{ij}} X^\n X^{\n_\partial} \II_{\partial,\partial} \, d\mu^{n-2}\Big].
\] 
\begin{lemma} \label{lem:R}
Let $\Omega$ be a stationary regular cluster (with respect to $\mu$) which is non-flat. Then there exist $\eps_1,\eps_2 > 0$, a collection $\set{X_1,\ldots,X_q}$ of $(\eps_1,\eps_2)$-approximate inward fields, and a linear combination $X = \sum_{i=1}^q a_i X_i$, so that $R(X) < 0$. 
\end{lemma}
\begin{proof}
   By assumption there is some non-flat $\Sigma_{ij}$, and we may assume without loss of generality that it is $\Sigma_{12}$. By continuity and relative openness of $\Sigma_{12}$ in $\Sigma$, there exists a bounded neighborhood $N_p$ of $p \in\Sigma_{12}$ so that $\II^{12} \neq 0$ on $\Sigma_{12} \cap N_p =  \Sigma \cap  N_p$, 
         and we denote
    \[
    \delta :=  \int_{\Sigma_{12} \cap N_p} \|\II\|_2^2 \, d\mu^{n-1}  > 0 . 
    \]

    Fix $\eps_1 := \delta/ (4 q^2)$, and let $K$ be the compact set from Proposition~\ref{prop:inward-fields}, which is disjoint from $\Sigma^4$ and satisfies $\mu^{n-1}(\Sigma^1 \setminus K) \leq \eps_1$. We can always assume that $K$ contains $N_p$. Then:
                        \begin{equation} \label{eq:non-flat-start}
        \infty > \int_{\Sigma_{12} \cap  K} \| \II\|_2^2 \, d\mu^{n-1} \ge \delta > 0 ,
    \end{equation}
    where the finiteness is ensured by Proposition \ref{prop:Schauder} \ref{it:Schauder-Sigma1}. 
                    By Proposition~\ref{prop:inward-fields}, for every $\eps_2 \in (0,1)$, there exists a collection $\{ X_1, \dots,  X_q\}$ of $(\eps_1,\eps_2)$-approximate inward vector-fields for $\Omega$ which are supported inside $K$. Thanks to the finiteness in (\ref{eq:non-flat-start}), by taking $\eps_2>0$ small-enough, we may always ensure by part (3) of Proposition~\ref{prop:inward-fields} that 
    \begin{equation}\label{eq:large-curvature}
        \int_{\Sigma_{12} \cap K} (( X_1^{\n})^2 + (X_2^{\n})^2) \|\II\|_2^2 \, d\mu^{n-1} \ge \delta .
    \end{equation}
    As $\{ X_i \}$ are supported in $K$, we will henceforth drop the intersection with $K$ in all subsequent integrals. 

    For $a \in \R^{q}$, define $X_a := \sum_{\ell=1}^q a_\ell X_\ell$, and introduce the following quadratic forms:
    \begin{align*}         R_1(a)  & := \sum_{i < j} \int_{\Sigma_{ij}} (X_a^{\n})^2 \| \II \|_2^2\, d\mu^{n-1} , \\
        R_2(a)  & := \sum_{i < j} \int_{\partial \Sigma_{ij}} X_a^{\n} X_a^{\n_{\partial}} \II_{\partial,\partial} \, d\mu^{n-2} .
    \end{align*}     Note that $R_1$ and $R_2$ are well-defined, as $\II^{ij}$ is square-integrable on $\Sigma_{ij} \cap K$ and $\II^{ij}_{\partial,\partial}$ is integrable on $\partial\Sigma_{ij} \cap K$ by Proposition \ref{prop:Schauder}. Note that:
    \[
    \tr(R_1) \geq \int_{\Sigma_{12}}  ((X_1^{\n})^2 + (X_2^{\n})^2) \|\II\|_2^2 \, d\mu^{n-1} \ge \delta 
    \]
    by (\ref{eq:large-curvature}). In addition, observe that given $p \in \partial \Sigma_{ij}$ and setting $Z = \n_{ij} \n_{\partial ij}^T$, Lemma \ref{lem:trace} implies:
    \[
    \abs{\sum_{\ell=1}^q X_\ell^{\n_{ij}} X_\ell^{\n_{\partial ij}}} = \abs{\tr\brac{ (\sum_{\ell=1}^q  X_\ell  X_\ell^{T} ) Z }} \leq C_{q} \eps_2 .
    \]
     Consequently, the integrability of $\II_{\partial,\partial}$ on $\Sigma^2 \cap K$ implies that $\tr(R_2) \geq - C_{K} C_{q} \eps_2$. Hence, by choosing $\eps_2 >0$ small-enough, we can ensure that:
     \[
     \tr(R_1 + R_2) \geq \delta - C_{K} C_{q} \eps_2 \geq \frac{\delta}{2} .
     \]
     It follows that we may find a unit-vector $a \in \R^q$ so that 
     \begin{equation} \label{eq:trace-R1-R2}
     R_1(a) + R_2(a) \geq \frac{\delta}{2q}.
     \end{equation}
     Let us fix this $a = a_{\eps_1,\eps_2}$ and set $X = X_{\eps_1,\eps_2} := X_{a}$.
     
     It remains to apply Lemma \ref{lem:inward-combination} \ref{it:inward-nabla}, which together with $|a|=1$ yields:
                            \[
        R(X_{\eps_1,\eps_2}) = \sum_{i<j} \int_{\Sigma_{ij}} |\nabla^\tang X^{\n}|^2 \, d\mu^{n-1} - R_1(a) - R_2(a) 
        \le q \eps_1  - \frac{\delta}{2q} .
     \]
     Recalling that $\eps_1$ was chosen to be $\frac{\delta}{4 q^2}$, we see that $R(X_{\eps_1,\eps_2}) < 0$ for all $\eps_2 > 0$ sufficiently small. 
\end{proof}

We now return to treat the Gaussian measure $\gamma$ solely. 

\subsection{Dimension-deficient stable regular clusters are flat}

\begin{proposition}\label{prop:dimension-deficient-implies-flat}
    If a Gaussian-stable regular cluster is dimension-deficient then it is flat.
\end{proposition}

\begin{proof}
    Suppose that $\Omega$ is stationary, regular, dimension-deficient, and non-flat; we will prove
    that it is unstable. Without loss of generality, we will assume that $\Omega$ is
    dimension-deficient in the last coordinate. By Lemma \ref{lem:dim-deficient}, there exists a cluster  
    $\tilde \Omega$ in $\R^{n-1}$ such that $\Omega = \tilde \Omega \times \R$ up to null-sets.
    Since $\partial (\tilde \Omega \times \R) = \partial \tilde \Omega \times \R$ and since $\Omega$ is stationary and regular with respect to $\gamma^n$,  it is immediate to verify that $\tilde{\Omega}$ is equally stationary and regular with respect to $\gamma^{n-1}$. Since $\Omega$ is assumed non-flat, the same clearly holds for $\tilde \Omega$. 
    We will use a tilde to denote everything related to $\tilde \Omega$: $\tilde \Sigma_{ij}$
    are the interfaces of $\tilde \Omega$, $\tilde \n_{ij}$ are the normals to those
    interfaces,     and so on. 
    
    By Lemma \ref{lem:R} applied to $\tilde \Omega$, there exist $\eps_1,\eps_2 > 0$, a collection $\{\tilde X_1,\ldots, \tilde X_q \}$ of $(\eps_1,\eps_2)$-approximate inward fields for $\tilde \Omega$, and a linear combination $\tilde X = \tilde X_{\eps_1,\eps_2} := \sum_{i=1}^q a_i \tilde X_i$, so that $\tilde R(\tilde X) < 0$, where:
    \[ 
    \tilde R(\tilde X) :=
    \sum_{i < j} \Big[\int_{\tilde \Sigma_{ij}} \brac{ |\nabla^\tang \tilde X^{\tilde \n}|^2 - (\tilde X^\n)^2 \|\tilde \II\|_2^2 } d\gamma^{n-2}
    - \int_{\partial \tilde \Sigma_{ij}} \tilde X^{\tilde \n} \tilde X^{\tilde \n_\partial} \tilde \II_{\partial,\partial} \, d\gamma^{n-3}\Big].
    \]

        Now take $f = f_{\eps} \in C_c^\infty(\R)$ so that $\int f \, d\gamma^1 = 0$, $\int f^2 \, d\gamma^1 = 1$
    and $\int (f'(x))^2 \, d\gamma^1 \le 1+\eps$; note that the Gaussian Poincar\'e inequality (e.g. \cite{BGL-Book}) ensures that $\int (f'(x))^2 \, d\gamma^1 \geq 1$ with equality for the (non-compactly supported) linear function $f(x) = x$, so an appropriate smooth cutoff of this function will do the job. 
    We define a tame vector-field $X = X_{\eps}$ on $\R^n$ by $X(x) = f(x_n) \tilde X(\tilde x)$, where $x = (\tilde x, x_n)$.
    By~\eqref{eq:formula}, as $\nabla^2_{\n,\n} W = 1$ for the Gaussian measure, we have:
    \begin{equation} \label{eq:dim-def-Q1}
    Q(X) = R(X) - \sum_{i<j} \int_{\Sigma_{ij}} (X^{\n})^2 \, d\gamma^{n-1} ,
    \end{equation}
    where:
    \[
    R(X) :=
        \sum_{i < j} \Big[\int_{\Sigma_{ij}} \brac{ |\nabla^\tang X^\n|^2 - (X^\n)^2 \|\II\|_2^2 } d\gamma^{n-1}
    - \int_{\partial \Sigma_{ij}} X^\n X^{\n_\partial} \II_{\partial,\partial} \, d\gamma^{n-2}\Big].
    \]
        Note that for $x = (\tilde x, x_n) \in \Sigma_{ij}$, we have $\tilde x \in \tilde \Sigma_{ij}$,  $X^\n(x) = f(x_n) \tilde X^{\tilde \n} (\tilde x)$, $X^{\n_{\partial}}(x) = f(x_n) \tilde X^{\tilde \n_{\partial}} (\tilde x)$, 
    $\|\II(x)\|_2^2 = \|\tilde \II(\tilde x)\|_2^2$, and for $x \in \partial \Sigma_{ij}$, we have $\tilde x \in \partial \tilde \Sigma_{ij}$ and $\II_{\partial,\partial}(x) = \tilde \II_{\partial,\partial}(\tilde x)$. Consequently, thanks to the product structure of $\Omega$ and $\gamma^{n-1} = \gamma^{n-2} \otimes \gamma^1$, Fubini's theorem and the fact that $\int_\R f^2 \, d\gamma^1 = 1$ imply:
    \begin{align*}
        \sum_{i < j} \int_{\Sigma_{ij}} (X^\n)^2 \|\II\|_2^2  \, d\gamma^{n-1} & =
        \sum_{i < j} \int_{\tilde \Sigma_{ij}} (\tilde X^{\tilde \n})^2 \|\tilde \II\|_2^2  \, d\gamma^{n-2} , \\
                \sum_{i < j} \int_{\partial \Sigma_{ij}} X^\n X^{\n_\partial} \II_{\partial,\partial} \, d\gamma^{n-2} & =
        \sum_{i < j} \int_{\partial \tilde \Sigma_{ij}} \tilde X^{\tilde \n} \tilde X^{\tilde \n_\partial} \tilde \II_{\partial,\partial} \, d\gamma^{n-3} ,\\
                 \sum_{i < j} \int_{\Sigma_{ij}} (X^\n)^2 \, d\gamma^{n-1} & = \sum_{i < j} \int_{\tilde \Sigma_{ij}} (\tilde X^\n)^2 \, d\gamma^{n-2}  .
             \end{align*}
            As for the integral involving $|\nabla^\tang X^\n|^2$, observe that for $x \in \Sigma_{ij}$:
    \[
     \nabla^\tang X^\n (x) = e_n f'(x_n) \tilde X^{\tilde \n} (\tilde x) + f(x_n) (\tilde \nabla^\tang \tilde X^{\tilde \n}(\tilde x), 0) ,
    \]
    and since the two terms on the right are orthogonal,
    \[
        |\nabla^\tang X^\n(x)|^2 = (f'(x_n))^2 (\tilde X^{\tilde \n}(\tilde x))^2 + f^2(x_n) |\tilde \nabla^\tang \tilde X^{\tilde \n}(\tilde x)|^2.
    \]
    By Fubini's theorem and the fact that $\int_\R f^2 \, d\gamma^1 = 1$ and $\int_\R (f')^2 \, d\gamma^1 \le 1 + \eps$,
    \[
        \int_{\Sigma_{ij}} |\nabla^\tang X^\n|^2 \, d\gamma^{n-1}
        \le \int_{\tilde \Sigma_{ij}} \brac{(1 + \eps) (\tilde X^{\tilde \n})^2 + |\tilde \nabla^\tang \tilde X^{\tilde \n}|^2 } d\gamma^{n-2} .
    \]
    Plugging all of these relations into (\ref{eq:dim-def-Q1}), we conclude that:
    \[
    Q(X_{\eps}) \leq \sum_{i < j} \int_{\tilde \Sigma_{ij}} \eps (\tilde X^{\tilde \n})^2 d\gamma^{n-2} + \tilde R(\tilde X) .
    \]
    But by Proposition~\ref{prop:inward-fields}, $\tilde X^\n = \sum_{i=1}^q a_i \tilde X^\n_i$ is uniformly bounded on $\tilde \Sigma^1$, and since $\gamma^{n-2}(\tilde \Sigma^1) = \gamma^{n-1}(\Sigma^1) < \infty$, we see that the contribution of the first term on the right can be made arbitrarily small. 
                            Recalling that $\tilde R(\tilde X) < 0$, it follows that $Q(X_{\eps}) < 0$ for $\eps > 0$ sufficiently small.
    Finally, recalling the simple formula (\ref{eq:first-variation-volume}) for the first variation of volume, it immediately follows that $\delta_{X_{\eps}} V = 0$ by Fubini's theorem and the fact that $\int f\, d\gamma^1 = 0$. 
    By definition, this verifies that the cluster $\Omega$ is unstable. 
\end{proof}

\subsection{Full-dimensional stable regular clusters are flat}

\begin{proposition}\label{prop:full-dimensional-implies-flat}
    If $q \le n + 1$ then every full-dimensional, Gaussian-stable regular $q$-cluster is flat and necessarily $q = n+1$. 
\end{proposition}

Given a $q$-cluster $\Omega$, recall that $M : \R^n \rightarrow E^{(q-1)}$ (defined in Subsection \ref{subsec:M}) denotes the first variation of volume operator under translation fields:
\[
M w := \delta_w V . 
\]
By Lemma \ref{lem:MKernelAndDef}, if $\Omega$ is full-dimensional, regular and stable then $M$ has trivial kernel and $q \geq n+1$, so if we are given that $q \leq n+1$ then necessarily $q=n+1$ and $M$ must be surjective. This reduces the proof of Proposition \ref{prop:full-dimensional-implies-flat} to the following slightly more general claim:
  
\begin{proposition}\label{prop:surjective-implies-flat}
    A Gaussian-stable regular cluster for which $M$ is surjective is flat. 
\end{proposition}
\begin{proof}   
    Let $\Omega$ be a stationary, regular and non-flat cluster,  for which $M$ is surjective; we will show that it is unstable.
    
   By Lemma \ref{lem:R}, there exist $\eps_1,\eps_2 > 0$ and a linear combination $X = \sum_{i=1}^q a_i X_i$ of $(\eps_1,\eps_2)$-approximate inward fields $\{ X_1,\ldots, X_q \}$, so that $R(X) < 0$, where recall:
   \[
    R(X) :=
        \sum_{i < j} \Big[\int_{\Sigma_{ij}} \brac{ |\nabla^\tang X^\n|^2 - (X^\n)^2 \|\II\|_2^2 } d\gamma^{n-1}
    - \int_{\partial \Sigma_{ij}} X^\n X^{\n_\partial} \II_{\partial,\partial} \, d\gamma^{n-2}\Big].
    \]
    
    Now choose $w \in \R^n$ such that $\delta_w V(\Omega) = M w = - \delta_{X} V(\Omega)$, which is always possible because $M$ is surjective, and denote $Y = X + w$. By
    Theorem~\ref{thm:combined-vector-fields} and the fact that $\nabla^2_{\n,\n} W = 1$ and $\nabla^2_{w^{\tang},\n} W = 0$ for the Gaussian measure, we have:
    \[
    Q(Y) = R(X) - \sum_{i < j} \int_{\Sigma_{ij}} (Y^{\n})^2 d\gamma^{n-1} \leq R(X) < 0 . 
    \]
    But since $\delta_Y V(\Omega) = \delta_X V(\Omega) +\delta_w V(\Omega) = 0$, it follows by definition that $\Omega$ is unstable.

\end{proof}

By combining Propositions~\ref{prop:dimension-deficient-implies-flat} and~\ref{prop:full-dimensional-implies-flat}, we know that all Gaussian-stable regular
$q$-clusters are flat whenever $2 \leq q\leq n+1$. Consequently, 
the reader only interested in the proof of the Gaussian Multi-Bubble Conjecture and not in the structure of Gaussian-stable regular clusters may safely skip to the next section; the remaining parts of the present section pertain to establishing further properties of Gaussian-stable regular clusters, which are not required for the proof of Theorem \ref{thm:main-I-I_m}. However, the resulting formula (\ref{eq:polyhedron-formula}) for the cluster's convex polyhedral cells which we will obtain at the end of this section, will be useful in the proof of the uniqueness Theorem \ref{thm:main-uniqueness} in Section \ref{sec:uniqueness}.

\subsection{Stable regular clusters have connected cells}

To show that the cells of Gaussian-stable regular clusters must be connected, we will show that if there is a disconnected cell of a stationary, regular, flat cluster $\Omega$ then
$\Omega$ is unstable.

\medskip

Let $\Omega$ be a stationary regular flat cluster, and recall our convention from Theorem \ref{thm:Almgren} \ref{it:Almgren-ii}. 
If $\Omega$ has a disconnected cell, we may suppose without loss of generality that $\Omega_1$ is disconnected. 
We can then define a $(q+1)$-cluster $\tilde \Omega$ by splitting $\Omega_1$ into
two non-trivial pieces $\tilde \Omega_1$ and $\tilde \Omega_{q+1}$, each consisting of unions of connected components of $\Omega_1$, 
while leaving the other cells unaltered, i.e. $\tilde \Omega_i = \Omega_i$ for $i=2,\ldots,q$. Note that $\partial \Omega_1 = \partial \tilde \Omega_1 \cup \partial \tilde \Omega_{q+1}$ and that $\partial^* \Omega_1 = \partial^* \tilde \Omega_1 \cup \partial^* \tilde \Omega_{q+1}$. Consequently, $\tilde \Omega$ is still a stationary and regular cluster. It also follows that $\delta^k_X V(\Omega) _1 = \delta^k_X V(\tilde \Omega) _1 + \delta^k_X V(\tilde \Omega) _{q+1}$ and $\delta^k_X A(\Omega) = \delta^k_X A(\tilde \Omega)$ for all $k \geq 0$ and $C_c^\infty$ vector-fields $X$. In addition, denoting $\tilde H_{ij,\gamma} = H_{\tilde \Sigma_{ij,\gamma}}$, 
since $\tilde H_{1j,\gamma} = \tilde H_{(q+1)j,\gamma} = H_{1j,\gamma}$ for all $j=2,\ldots,q$, it follows that $\tilde H_{ij,\gamma} = \tilde \lambda_{i} - \tilde \lambda_{j}$ for all $i\neq j = 1,\ldots,q+1$ where $\tilde \lambda_1 = \tilde \lambda_{q+1} = \lambda_1$ and $\tilde \lambda_j = \lambda_j$ for all $j=2,\ldots,q$. In particular, it follows that $Q_{\tilde \Omega}(X) = Q_{\Omega}(X)$. 

We will construct a $C_c^\infty$ vector-field $X$ with
\begin{equation}\label{eq:connected-goal}
    \delta_X V(\tilde \Omega) = e_{q+1} - e_1 \text{ and } Q_{\tilde \Omega}(X) < 0.
\end{equation}
Applying the same vector field to $\Omega$, it will follow by the previous comments that $\delta_X V(\Omega) = 0$ but $Q_{\Omega}(X)  < 0$, demonstrating that $\Omega$ is unstable.

To construct $X$, take $\epsilon > 0$ and let $\{X^{\eps}_i\}_{i=1,\ldots,q+1}$ be an $(\eps,\eps)$-approximate
family of inward fields for $\tilde \Omega$. Given $a \in \R^{q+1}$, set $X^\eps_a = \sum_{i=1}^{q+1} a_i X^\eps_i$,
and let $V_\epsilon: E^{(q)} \to E^{(q)}$ be the linear map defined by
$V_\epsilon a = \delta_{X^\eps_a} V(\tilde \Omega)$. Observe that by Lemma \ref{lem:inward-combination} \ref{it:inward-delta-V},
 $V_\epsilon \to -L_A$ as $\epsilon \to 0$ (say, in the operator norm), where $A = A(\tilde \Omega)$
and $L_A$ was defined in (\ref{eq:L_A}). By Lemma~\ref{lem:L-nondegenerate} $L_A$ is non-degenerate, 
and hence $V_\epsilon^{-1}$
exists and is bounded uniformly in $\eps$ for all $\eps > 0$ sufficiently small.
Let $a^\eps = V_{\eps}^{-1} (e_{q+1} - e_1)$, denote $X_\eps = X^{\eps}_{a^{\eps}}$, and note that 
 $\delta_{X_\eps} V(\tilde \Omega) = V_\eps a^{\eps} = e_{q+1} - e_1$ by construction. By Theorem~\ref{thm:formula}
and the flatness of $\Omega$ and hence of $\tilde \Omega$, we have:
\[
Q_{\tilde \Omega}(X_\eps) = \sum_{i < j} \int_{\tilde \Sigma_{ij}} \brac{|\nabla^\tang X^\n_\eps|^2 - (X^\n_\eps)^2 } d\gamma^{n-1}  .
\]
Applying Lemma \ref{lem:inward-combination} \ref{it:inward-nabla} and \ref{it:inward-L2}, it follows that:
\[
Q_{\tilde \Omega}(X_\eps) \leq  - (a^{\eps})^T L_A a^\eps + (q+5) \eps|a^\eps|^2  .
\]
Since $a^{\eps} \to -L_A^{-1} (e_{q+1} - e_1)$ as $\eps \to 0$,  $|a^{\eps}|$ is bounded away from $0$ and $\infty$ for small-enough $\eps > 0$. Recalling that $L_A$ is strictly positive-definite, we conclude that $Q(X_\eps) < 0$
for sufficiently small $\eps > 0$. This completes the construction of a vector-field
satisfying~\eqref{eq:connected-goal}, and thus the proof of
Theorem~\ref{thm:flat}.

\subsection{Stable regular clusters have convex cells}

To proceed with the proof of Theorem~\ref{thm:flat}, we will show that the cells of a stationary, regular, flat cluster $\Omega$ with connected open cells must necessarily be convex. 

\smallskip

We will require the following proposition, which may be of independent interest. As we could not find a reference for this in the literature (in the case that $A \neq \emptyset$ below), we provide a proof. 

\begin{proposition}[From Almost Local to Global Convexity] \label{prop:local-to-global}
Let $\Omega$ be an open connected subset of $\R^n$, and let $A \subset \partial \Omega$ be a Borel set with $\H^{n-2}(A) = 0$. Assume that for every $p \in \partial \Omega \setminus A$ there exists an open neighborhood $N_p$ of $p$ so that $\Omega \cap N_p$ is convex. Then $\Omega$ is convex. 
\end{proposition}
It is easy to see that the claim is false when $\H^{n-2}$ is replaced by $\H^{n-k}$ with $k <2$, or when $\Omega$ is not assumed open or closed. For the proof, we will require the following well-known fact (due to Tietze and to Nakajima, see e.g. \cite{KarshonEtAl-TietzeNakajima})
\begin{lemma}[From Local to Global Convexity]\label{lem:local-to-global}
Let $\Omega_1$ be an open connected subset of $\R^n$, so that for any $p \in \partial \Omega_1$, there exists an open neighborhood $N_p$ of $p$ so that $\Omega_1 \cap N_p$ is convex. Then $\Omega_1$ is convex. \end{lemma}
\begin{proof}[Proof of Proposition \ref{prop:local-to-global}]
Given distinct $x_0, z_0 \in \Omega$, we will show that:
\begin{equation} \label{eq:L2G-goal}
\forall \eps > 0 \;\;\; \exists x \in B(x_0,\eps) \;\;\; \exists z \in B(z_0,\eps) \;\;\; [x,z] \subset \Omega . 
\end{equation}
This will be enough to conclude the claim, since this will imply that $[x_0,z_0]$ lies inside $\overline{\Omega}$, which will imply that $\overline{\Omega}$ is convex, which in turn will imply that its interior $\Omega$ is convex. 

 Recall that connectedeness and path-connectedeness are equivalent topological properties on open subsets of $\R^n$. As $\Omega$ is path-connected (being open), there exists a closed path from $x_0$ to $z_0$ which lies in $\Omega$. By openness of $\Omega$ and compactness of the path, a standard argument implies that we may assume that the path is piecewise linear, i.e. consists of a finite number of intervals $\{[y_i,y_{i+1}]\}_{i=0,\ldots,N-1}$, with $y_0 = x_0$ and $y_N = z_0$, so that $[y_i, y_{i+1}] \subset \Omega$. We will show, by induction on $N$, that this implies (\ref{eq:L2G-goal}). 
 The base case $N=1$ is trivial, and we will establish the case $N=2$. To establish the case of $N+1$, by openness of $\Omega$ and compactness of intervals, there exists $\delta > 0$ such that for all $y_N' \in B(y_N , \delta)$, $[y_N' , y_{N+1}] \subset \Omega$;  by applying the induction hypothesis with $\min(\eps/2,\delta)$ to the piecewise linear path from $y_0$ to $y_N$, and the case $N=2$ to the path $[y_0', y_N'] \cup [y_N' , y_{N+1}]$, we find $x \in B(x_0,\eps)$ and $z \in B(z_0,\eps/2)$ with $[x,z] \subset \Omega$, as required. 
 
To establish the case $N=2$, let $y_0 \in \Omega$ be such that $[x_0,y_0] \subset \Omega$ and $[y_0,z_0] \subset \Omega$. Consider the affine subspace $E_{x,y,z}$ spanned by $x,y,z$.  If $E_{x_0,y_0,z_0}$ is one-dimensional, it trivially follows that $[x_0,z_0] \subset \Omega$ and (\ref{eq:L2G-goal}) is established, so let us assume that $E_{x_0,y_0,z_0}$ is two-dimensional. By openness of $\Omega$ and compactness of the intervals, there exists $\delta > 0$ so that for all $x \in B(x_0,\delta)$, $y \in B(y_0,\delta)$ and $z \in B(z_0,\delta)$, $[x,y] \subset \Omega$, $[y,z] \subset \Omega$ and $E_{x,y,z}$ is two-dimensional. We now claim that for all $\eps \in (0,\delta)$, there exist $x \in B(x_0,\eps)$, $y \in B(y_0,\eps)$ and $z \in B(z_0,\eps)$ so that $E = E_{x,y,z}$ is disjoint from $A$. 
  Once this is established, we will consider the set $\Omega_0 := \Omega \cap E$, which is relatively open in $E$; we subsequently consider the relative topology of sets in their corresponding affine hulls. We do not a-priori know that $\Omega_0$ is a (path-)connected set, but as $[x,y]$ and $[y,z]$ are both in $\Omega_0$, we do know that $x,y,z$ all lie in the same (path-)connected component $\Omega_1 \subset \Omega_0$. Since $\partial \Omega_1 \subset \partial \Omega_0 \subset \partial \Omega  \cap E \subset \partial \Omega \setminus A$, we know that for any $p \in \partial \Omega_1$ there exists an open neighborhood $N_p$ of $p$ so that $\Omega_1 \cap N_p = \Omega_0 \cap N_p = (\Omega \cap N_p) \cap E$ is convex. Applying Lemma \ref{lem:local-to-global} to the open connected $\Omega_1$, it follows that $\Omega_1$ is convex, and hence $[x,z] \subset \Omega_1 \subset \Omega$, thereby establishing (\ref{eq:L2G-goal}). 

To see the existence of a subspace $E = E_{x,y,z}$ which is disjoint from $A$, let $\sigma^n_2$ denote the natural Haar measure on $\Eps^n_2$, the homogeneous space of all two-dimensional affine subspaces in $\R^n$, normalized so that $\sigma_2 \{ E \in \Eps^n_2 : E \cap B(0,1) \neq \emptyset\} = \H^{n-2}(B^{n-2})$, where $B^{n-2}$ denotes the Euclidean unit-ball in $\R^{n-2}$. Introduce the following intergral-geometric Favard outer measure:
\begin{equation} \label{eq:Favard}
\F^{n-2}(U) := \int_{\Eps^n_2} \H^0(U \cap E) d\sigma^n_2(E). 
\end{equation}
It is known that $\F^{n-2}$ is a Borel measure, and by a generalized Crofton formula proved by Federer (see \cite[Theorem 3.2.26]{FedererBook}), that $\F^{n-2}$ and $\H^{n-2}$ coincide on $\H^{n-2}$-rectifiable sets. While this may not be the case on general Borel sets, we always have $\F^{n-2}(U) \leq \H^{n-2}(U)$ for all Borel sets $U$ (see \cite[Theorem 2.10.15 and Section 2.10.6]{FedererBook}). Given $\eps > 0$, consider the following variant:
\[
\F^{n-2}_\eps(U) := \int_{\Eps^n_2(\eps)} \H^0(U \cap E) d\sigma^n_2(E) ,
\]
where:
\[
\Eps^n_2(\eps) := \set{ E \in \Eps^n_2 : E \cap B(x_0,\eps) \neq \emptyset, E \cap B(y_0,\eps) \neq \emptyset , E \cap B(z_0,\eps) \neq \emptyset } .
\]
Clearly $\F^{n-2}_{\eps}(U) \leq \F^{n-2}(U)$ for all $U$. 

Applying this to our set $A$, we deduce that $\F^{n-2}(A) = 0$ (and in fact, the proposition holds under this weaker assumption). Consequently, also $\F^{n-2}_\eps(A) = 0$, and we deduce that for $\sigma_2^n$-a.e. $E \in \Eps^n_2(\eps)$, $\H^0(A \cap E) = 0$, i.e. $E$ is disjoint from $A$. Applying this to all $\eps \in (0,\delta)$ and using that $\sigma^n_2(\Eps^n_2(\eps)) > 0$, the existence of the desired $E$, disjoint from $A$ and intersecting $B(x_0,\eps)$, $B(y_0,\eps)$ and $B(z_0,\eps)$, is deduced, concluding the proof. 
\end{proof}

Now let $\Omega$ be a stationary, regular, flat cluster with connected open cells $\{\Omega_i\}$. Note that by Theorems \ref{thm:Almgren}, \ref{thm:regularity} and flatness of the interfaces, the local convexity condition of Proposition \ref{prop:local-to-global} holds on $\partial \Omega_i \setminus \Sigma^4$: indeed, around $p_1 \in \partial \Omega_i \cap \Sigma^1$, $p_2 \in \partial \Omega_i \cap \Sigma^2$ and $p_3 \in \partial \Omega_i \cap \Sigma^3$, $\Omega_i$ looks locally like the convex flat cell of a model simplicial cluster around its interface, triple points and quadruple points, respectively. 
As $\H^{n-2}(\partial \Omega_i \cap \Sigma^4) = 0$, we deduce by Proposition \ref{prop:local-to-global} that $\Omega_i$ are all convex. 

\subsection{Formula for polyhedral cells of stable regular clusters}

Every non-empty open convex set $\Omega \subset \R^n$ is the intersection of its open supporting halfspaces over all boundary points $\partial \Omega$ (see e.g. \cite{Schneider-Book}). Clearly, by continuity (as the boundary of a convex set is a Lipschitz manifold), it is enough to intersect the supporting halfspaces of any subset $A \subset \partial \Omega$ which is dense in $\partial \Omega$. 
Recalling our convention from Theorem \ref{thm:Almgren} \ref{it:Almgren-ii} and (\ref{eq:top-nothing-lost}), it follows that $\Omega_i$ is the intersection of its open supporting halfspaces over the points of its interfaces $\cup_{j \neq i} \Sigma_{ij}$. As the interfaces are smooth, the unique open supporting halfspace to $\Omega_i$ at $p \in \Sigma_{ij}$ is of the form $\{x \in \R^n : \scalar{\n_{ij}(p),x} < C_{ij}(p) \}$. On the other hand, by the flatness of the interfaces, we know that for all $x \in \Sigma_{ij}$:
\[
    H_{ij,\gamma} = H_{\Sigma_{ij}}(x) - \nabla_{\n_{ij}} \pot(x) = -\inr{\n_{ij}}{x},
\]
and hence $C_{ij} \equiv -H_{ij,\gamma}$ is constant on $\Sigma_{ij}$.

From this, we deduce that $\n_{ij}$ must be constant over (any non-empty) $\Sigma_{ij}$. Otherwise, if there exist $p_k \in \Sigma_{ij}$, $k=1,2$, with $\n_{ij}(p_1) \neq \n_{ij}(p_2)$, it would follow (as $\n_{ji} = -\n_{ij}$ and $H_{ji,\gamma} = -H_{ij,\gamma}$) that:
\[
\Omega_i \subset \bigcap_{k=1,2} \{x \in \R^n : \scalar{\n_{ij}(p_k),x} < -H_{ij,\gamma} \} ~,~ \Omega_j \subset \bigcap_{k=1,2} \{x \in \R^n : \scalar{\n_{ij}(p_k),x} > -H_{ij,\gamma} \} .
\]
This would either imply that one of the cells $\Omega_i$ or $\Omega_j$ is empty (if $\n^{ij}(p_1) = -\n^{ij}(p_2)$), or that $\partial \Omega_i \cap \partial \Omega_j$ is contained in the $n-2$ dimensional plane $\bigcap_{k=1,2} \{x \in \R^n : \scalar{\n^{ij}(p_k),x} = -H_{ij,\gamma} \}$ (otherwise); in either case, it would follow that $A_{ij} = \gamma^{n-1}(\Sigma_{ij}) = 0$, i.e. that $\Sigma_{ij}$ is empty, a contradiction. As $\n_{ij}$ and $H_{ij,\gamma}$ are constant on $\Sigma_{ij}$, $\Sigma_{ij}$ is contained in a single hyperplane. By Lemma \ref{lem:MKernelAndDef}, $\dim \spn \{ \n_{ij} \}_{i < j} \leq q-1$. 

It follows that the non-empty cells $\Omega_i$ are convex polyhedra, obtained as the intersection of a finite number (at most $q-1$) open supporting halfspaces to its non-empty interfaces $\Sigma_{ij}$. The explicit formula in (\ref{eq:polyhedron-formula}) is obtained after recalling that $H_{ij,\gamma} = \lambda_i  - \lambda_j$, where $\lambda \in E^*$ is given by Lemma \ref{lem:first-order-conditions} \ref{it:first-order-cyclic}. This concludes the proof of Theorem \ref{thm:flat}.

\section{Proof of the Multi-Bubble Theorem} \label{sec:multi-bubble}

All of the work we have done so far culminates in the following key theorem, which constitutes a rigorous version of our sought-after matrix-valued partial differential inequality $\nabla^2 I \le -L_A^{-1}$ (the latter is presently only formal in view of the fact that we don't yet know $I$ to be differentiable). 

\begin{theorem}[Main Differential Inequality] \label{thm:hessian-bound-for-I}
    Let $\Omega$ be a Gaussian-stable regular $q$-cluster with $\gamma(\Omega) = v \in \interior \simplex^{(q-1)}$ and $2 \leq q \leq n+1$. Let $A = A(\Omega) = \set{A_{ij}}$ denote the collection of its interface measures. 
    Then for every $y \in E^{(q-1)}$ and every $\epsilon > 0$,
    there exists a $C_c^\infty$ vector-field $X$ such that $|\delta_X V - y| \le \epsilon$
    and $Q(X) \le - y^T L_A^{-1} y + \epsilon$.
\end{theorem}
\begin{proof}[Proof of Theorem \ref{thm:hessian-bound-for-I}]
    Given $y \in E^{(q-1)}$, since $L_A$ is full-rank by Lemma~\ref{lem:L-nondegenerate}, there exists $a \in E^{(q-1)}$ so that $L_A a = -y$. 
    Fix $\eps \in (0,1)$ and construct a family 
    $\{X_1, \dots, X_q\}$ of $(\eps,\eps)$-approximate inward fields by Proposition~\ref{prop:inward-fields}. 
    Setting $X = \sum_k a_k X_k$, it follows by Lemma \ref{lem:inward-combination} \ref{it:inward-delta-V} that for all $i$:     \[
    |(\delta_X V(\Omega) + L_A a)_i| \leq 2 \epsilon \sqrt{q} |a| .
    \]
    As $L_A a = -y$, we conclude that $|\delta_X V - y| \le 2 \epsilon q |a|$.

    Since $\Omega$ is stable and regular and $q \leq n+1$, Theorem~\ref{thm:flat} asserts that it is in addition flat. Hence, Theorem \ref{thm:formula} implies that:
    \[
        Q(X)
        = \sum_{i < j} \int_{\Sigma_{ij}} \brac{|\nabla^\tang X^\n|^2 - (X^\n)^2 } \, d\gamma^{n-1}.
    \]
    Applying Lemma \ref{lem:inward-combination} \ref{it:inward-nabla} and \ref{it:inward-L2}, it follows that:
                                                               \[
        Q(X) \le - a^T L_A a + (q+4) \epsilon |a|^2 .
    \]
    Finally, since $a^T L_A a = y^T L_A^{-1} y$, the assertion is proved (after adjusting $\epsilon$).
\end{proof}

We will also require the following standard fact:

\begin{lemma}\label{lem:semi-cont}
    $I$ is lower semi-continuous on $\simplex$.
\end{lemma}
\begin{proof} 
The proof is identical to the one of Theorem \ref{thm:Almgren} (i).  If $v_r \to v \in \simplex$, let $\Omega^r$ be a minimizing cluster with $\gamma(\Omega^r) = v_r$.  Note that for all $r$, $P_\gamma(\Omega^r) = I(v_r) \leq I_m(v_r) \leq \max_{v \in \Delta} I_m(v) =: C < \infty$, and that $P_\gamma(\Omega^r_i) \leq P_\gamma(\Omega^r)$ for each $i$. Consequently, by passing to a subsequence, each of the cells $\Omega^r_i$ converges in $L^1(\gamma)$ to $\Omega_i$. By Dominated Convergence (as the total mass is finite), we must have $\gamma(\Omega_i) = v_i$, and 
   the limiting $\Omega$ is easily seen to be a cluster (possibly after a measure-zero modification to ensure disjointness of the cells).  Consequently:
  \[
    I(v) \le P_\gamma(\Omega) \le \liminf_{r \to \infty} P_\gamma(\Omega^r) = \liminf_{r \to \infty} I(v_r).
  \]
\end{proof}

We are now ready to prove the Gaussian Multi-Bubble Theorem~\ref{thm:main-I-I_m}. 

\begin{proof}[Proof of Theorem~\ref{thm:main-I-I_m}]     Our task is to prove that $I^{(q-1)} = I^{(q-1)}_m$ on $\simplex^{(q-1)}$ for $2 \leq q \leq n+1$. Fixing $n$, we will prove this by induction on $q$. 
    Our base case is $q=2$, which is the classical (single-bubble) Gaussian isoperimetric
    inequality
    (in fact, we can start the induction with $q=1$ if we define $I^{(0)}= I^{(0)}_m \equiv 0$ on $\simplex^{(0)} = \{ 1\}$, thereby yielding an independent proof of the single-bubble case). 
    By induction, we will assume that $I^{(q-2)} = I^{(q-2)}_m$ on $\simplex^{(q-2)}$. Since
    (by Lemma~\ref{lem:tripod-profile-continuous}) $I^{(q-1)}_m$ agrees with $I^{(q-2)}_m$
    on $\partial \simplex^{(q-1)}$, and since $I^{(q-1)}$ agrees with $I^{(q-2)}$ on $\partial \simplex^{(q-1)}$
    by definition, it follows by the induction hypothesis that $I^{(q-1)} = I^{(q-1)}_m$ on $\partial \simplex^{(q-1)}$.
    For brevity, we will henceforth write $I$, $I_m$ and $\simplex$ for $I^{(q-1)}$, $I^{(q-1)}_m$ and $\simplex^{(q-1)}$, respectively. 

    By the construction of Section \ref{sec:model}, $I \le I_m$, so we will show the converse inequality.
    Since $I_m$ is continuous on $\simplex$ by Lemma \ref{lem:tripod-profile-continuous} and $I$ is lower semi-continuous on $\simplex$ by Lemma~\ref{lem:semi-cont},
    $I - I_m$ must attain a global minimum at some $v_0 \in \simplex$. Assume in the contrapositive that $I(v_0) - I_m(v_0) < 0$; by the comments in the previous paragraph,
    necessarily $v_0 \in \interior \simplex$.

    Let $\Omega$ be a Gaussian-minimizing $q$-cluster with $\gamma(\Omega) = v_0$, as ensured by Theorem \ref{thm:Almgren}. 
    By the results of Section \ref{sec:minimizers}, $\Omega$ is interface-regular, stationary and stable with (uniformly) volume and perimeter regular cells,
    and by the results of Section \ref{sec:higher}, $\Omega$ is in addition regular (all with respect to $\gamma$).
    Since $v_0$ is a minimum of $I - I_m$, for any smooth flow $F_t$ along an admissible vector-field $X$,
    we have
    \begin{equation}\label{eq:I_m-and-I}
        I_m(\gamma(F_t(\Omega))) \le I(\gamma(F_t(\Omega)) - (I - I_m)(v_0) \le \per_{\gamma}(F_t(\Omega)) - (I - I_m)(v_0),
    \end{equation}
    where the second inequality follows from the definition of $I$. The two sides
    are equal when $t = 0$, and twice differentiable in $t$. By comparing first and second
    derivatives at $t=0$, it follows that:
    \begin{equation}
        \label{eq:compare-first}
        \inr{\nabla I_m(v_0)}{\delta_X V} = \delta_X A ,
    \end{equation}
    \begin{equation}
        \label{eq:compare-second}
        (\delta_X V)^T \nabla^2 I_m(v_0) \delta_X V  + \inr{\nabla I_m(v_0)}{\delta_X^2 V} \le \delta_X^2 A.
    \end{equation}
    
    On the other hand, recall that $\delta_X A = \inr{\lambda}{\delta_X V}$ by Lemma \ref{lem:Lagrange}, where $\lambda$ is the
    unique element of $E^*$ such that $\lambda_i - \lambda_j = H_{ij,\gamma}(\Omega)$.
    Comparing with~\eqref{eq:compare-first}, and using that (by Theorem~\ref{thm:hessian-bound-for-I})
    for $y = \nabla I_m(v_0) - \lambda$ and any $\epsilon > 0$ there is a $C_c^\infty$ vector-field $X$
    with $|\delta_X V - y| \le \epsilon$, it follows that:
    \begin{equation} \label{eq:compare-fourth}
     \nabla I_m(v_0) = \lambda .
     \end{equation}
    Plugging this into~\eqref{eq:compare-second}, we obtain:
    \begin{equation} \label{eq:compare-third}
        (\delta_X V)^T \nabla^2 I_m(v_0) \delta_X V  \le \delta_X^2 A - \inr{\lambda}{\delta_X^2 V} = Q(X).
    \end{equation}
    Now by Theorem~\ref{thm:hessian-bound-for-I}, for any $y \in E$ and $\epsilon > 0$
    we may find a $C_c^\infty$ vector-field $X$ so that $|\delta_X V - y| \le \epsilon$
    and
    \[
       (\delta_X V)^T \nabla^2 I_m(v_0) \delta_X V \le Q(X) \le - y^T L_A^{-1} y + \epsilon
    \]
    (where recall $A = A(\Omega) = \set{A_{ij}}$ is the collection of interface measures). Taking $\eps \rightarrow 0$, 
    it follows that $\nabla^2 I_m(v_0) \le - L_A^{-1}$ in the positive semi-definite sense, and since $L_A$ is positive-definite, we deduce that $-\nabla^2 I_m(v_0)^{-1} \le L_A$. It follows by Proposition~\ref{prop:I_m-equation} that:
    \[
        2 I_m(v_0) = -\tr[(\nabla^2 I_m(v_0))^{-1}] \le \tr(L_A) = 2 \sum_{i<j} A_{ij} = 2 I(v_0) ,
    \]
    in contradiction to the assumption that $I(v_0) < I_m(v_0)$.
\end{proof}

\section{Proof of the Double-Bubble Theorem -- a simplified argument} \label{sec:double-bubble}

For the reader only interested in the Double-Bubble Theorem (when $q=3$), we include a simplified argument which completely avoids the contents of Sections \ref{sec:higher} through \ref{sec:stable} (as well as Appendices \ref{sec:calculation} and \ref{sec:Schauder}). Instead of invoking the complicated argument leading up to the proof of our main differential inequality in Theorem \ref{thm:hessian-bound-for-I}, we present an alternative argument for deducing the case $q=3$ of this theorem. This alternative argument will also provide us with vital information when analyzing the question of uniqueness of minimizers in the next section. 

\smallskip
Once the main differential inequality is established, the short proof of Theorem~\ref{thm:main-I-I_m} described in the previous section remains unchanged (and in fact slightly simplifies since we may simply use $\eps=0$ above), and so we concentrate on proving:

\begin{theorem}[Main Differential Inequality for $q=3$] \label{thm:hessian-bound-for-I2}
Let $\Omega$ be a Gaussian-stable $3$-cluster with $\gamma(\Omega) = v \in \interior \simplex^{(2)}$ and perimeter regular cells. Let $A = A(\Omega) = \set{A_{ij}}$ denote the collection of its interface measures. Then for any $y \in E^{(2)}$, there exists an admissible vector-field $X$ such that:
\begin{equation} \label{eq:Qinq}
    \delta_X V = y \;\; \text{ and } \;\; Q(X) \le - y^T L_A^{-1} y . 
\end{equation}
\end{theorem}

\subsection{Main Differential Inequality assuming $M$ is surjective}

Given a $q$-cluster $\Omega$, recall once more that $M : \R^n \rightarrow E^{(q-1)}$ (defined in Subsection \ref{subsec:M}) denotes the first variation of volume operator under translation fields:
\[
M w := \delta_w V . 
\]
Our simplified argument for establishing Theorem \ref{thm:hessian-bound-for-I2} relies on the following observation, which in fact holds for arbitrary $q$.

\begin{theorem}[Main Differential Inequality assuming $M$ is surjective] \label{thm:hessian-bound-for-I-surjective-M}
    Let $\Omega$ be a Gaussian-stationary $q$-cluster with $\gamma(\Omega) = v \in \interior \simplex^{(q-1)}$ and perimeter regular cells. 
    Let $A = A(\Omega) = \set{A_{ij}}$ denote the collection of its interface measures. Assume that $M : \R^n \rightarrow E^{(q-1)}$ is surjective. 
    Then for every $y \in E^{(q-1)}$, there exists a constant (translation) field $w \in \R^n$ such that:
\[ 
    \delta_w V = y \;\; \text{ and } \;\; Q(w) \le - y^T L_A^{-1} y . 
\]
\end{theorem}

For the proof, we will require the following matrix form of the Cauchy--Schwarz inequality:

\begin{lemma}\label{lem:cs}
    If $X$ is a random vector in $\R^n$ and $Y$ is a random vector in $\R^k$ such that
    $\E |X|^2 < \infty$, $\E |Y|^2 < \infty$, and $\E Y Y^T$ is non-singular,
    then as quadratic forms on $\R^n$:
    \[
        (\E XY^T) (\E Y Y^T)^{-1} (\E YX^T) \le \E X X^T ,
    \]
    with equality if and only if $X = BY$ almost-surely, for a deterministic matrix $B$.
\end{lemma}

\noindent
While this inequality may be found in the literature (e.g.~\cite{Tripathi:99}), we present a short independent proof, including that of the equality case, which we will require for establishing uniqueness of minimizing clusters. 

\begin{proof}[Proof of Lemma \ref{lem:cs}]
  Let $Z = X - (\E X Y^T) (\E Y Y^T)^{-1} Y$. Then $Z Z^T \ge 0$ (in the positive semi-definite sense),
  and since convex combinations of positive semi-definite matrices are also positive semi-definite,
  $\E Z Z^T \ge 0$. A computation verifies that:
  \[
    \E Z Z^T = \E X X^T - (\E XY^T) (\E Y Y^T)^{-1} (\E Y X^T) ,
  \]
  completing the proof of the inequality.

  To check the equality case, note that equality occurs iff $\E Z Z^T = 0$, iff $Z = 0$ a.s., iff $X = (\E X Y^T) (\E Y Y^T)^{-1} Y$ a.s. .
\end{proof}

\begin{proof}[Proof of Theorem \ref{thm:hessian-bound-for-I-surjective-M}]
Recall that by Corollary \ref{cor:Gaussian-regular}, the cells of $\Omega$ are in addition (uniformly) volume regular (with respect to $\gamma$). 
Define the ``average'' normals $\navg_{ij}$ by
\[
  \navg_{ij} = \begin{cases} \frac{1}{A_{ij}} \int_{\Sigma_{ij}} \n_{ij}\, d\gamma^{n-1} & A_{ij} > 0 \\ 0 & A_{ij} = 0 \end{cases} . 
\]
By (\ref{eq:first-variation-volume}), we have for all $w \in \R^n$:
\[
(\delta_w V)_i = \sum_{j \neq i} \int_{\Sigma_{ij}} \inr{w}{\n_{ij}} \, d\gamma^{n-1} = \sum_{j \neq i} A_{ij} \inr{w}{\navg_{ij}} .
\]
Since $\navg_{ji} = -\navg_{ij}$, we can write this as:
\begin{equation} \label{eq:Mw}
   M w =\delta_w V = \sum_{i<j} A_{ij} (e_i - e_j) \navg_{ij}^T w ,
\end{equation}
and we see that the linear operator $M : \R^n \rightarrow E^{(q-1)}$ is given by:
\begin{equation} \label{eq:M}
    M = \sum_{i<j} A_{ij} (e_i - e_j) \navg_{ij}^T .
\end{equation}

By Gaussian-stationarity and perimeter regularity we may invoke \eqref{eq:trans-formula}; applying the Cauchy--Schwarz (or Jensen) inequality we conclude that for all $w \in \R^n$:
\begin{equation} \label{eq:Nw}
    Q(w) = -\sum_{i<j} \int_{\Sigma_{ij}} \inr{w}{\n_{ij}}^2\, d\gamma^{n-1} \le -\sum_{i<j} \inr{w}{\navg_{ij}}^2 A_{ij} = - w^T N w ,
\end{equation}
where $N : \R^n \rightarrow \R^n$ denotes the following linear operator:
\begin{equation} \label{eq:N}
    N := \sum_{i<j} A_{ij} \navg_{ij} \navg_{ij}^T .
\end{equation}

Since $\gamma(\Omega) \in \interior \Delta^{(q-1)}$, clearly $\sum_{i<j} A_{ij} > 0$ (see Remark \ref{rem:empty-interfaces}). 
Now select random vectors $X \in \R^n$ and $Y \in E^{(q-1)}$, which are assigned the values  $X = \navg_{ij}$  and $Y = e_i - e_j$ with probability $A_{ij} / \sum_{i < j} A_{ij}$.
Then, by definition of $N$, $M$ and $L_A$:
\[
\E X X^T = \frac{1}{ \sum_{i<j} A_{ij}} N ~,~ \E Y X^T = \frac{1}{ \sum_{i<j} A_{ij}} M ~,~ \E Y Y^T = \frac{1}{ \sum_{i<j} A_{ij}} L_A .
\]
Recalling that $L_A$ is invertible on $E^{(q-1)}$ by Lemma~\ref{lem:L-nondegenerate}, Lemma~\ref{lem:cs} implies that
\begin{equation} \label{eq:MLAN}
M^T L_A^{-1} M \le N, 
\end{equation}
yielding:
\[
Q(w) \leq -w^T N w \leq -w^T M^T L_A^{-1} M w \;\;\; \forall w \in \R^n . 
\]
Consequently, when $M$ is surjective, for any $y \in E$, we may choose $w \in \R^n$ so that
$\delta_w V = M w = y$ and $Q(w) \leq -y L_A^{-1} y$, completing the proof. 
\end{proof} 

\subsection{Effectively One-Dimensional Clusters}

It remains to deal with the case that $M$ is not surjective. Without invoking the results of Sections \ref{sec:higher} through \ref{sec:stable} and Appendices \ref{sec:calculation} and \ref{sec:Schauder}, we only know how to do this when $q=3$, as we shall briefly explain. Note that \emph{a-posteriori}, we will show that the unique Gaussian-minimizing clusters $\Omega$ (with $\gamma(\Omega) \in \interior \Delta^{(q-1)}$ and $2 \leq q \leq n+1$) are (up to null-sets) the model simplicial clusters, for which $M$ \emph{is always} surjective, but we cannot exclude the possibility that $M$ is not full-rank \emph{beforehand}. Moreover, we will show in Theorem \ref{thm:minimizers-char} that $M$ is necessarily \emph{not} surjective for a Gaussian-stable $q$-cluster which is non-minimizing, and so if our only tool is stability, we need to take this possibility into account. 

\medskip

Recall that by Lemma \ref{lem:MKernelAndDef}, the effective-dimension of a Gaussian-stable $q$-cluster $\Omega$ (whose cells are perimeter regular) is equal to the rank of $M : \R^n \rightarrow E^{(q-1)}$, and so its effective dimension is at most $q-1$, with equality if and only if $M$ is surjective. Of course, since we assume $\gamma(\Omega) = v \in \interior \simplex^{(q-1)}$, the cluster's effective dimension is strictly positive (see Remark \ref{rem:empty-interfaces}). So if $M$ is not surjective, the effective dimension must be between $1$ and $q-2$. We immediately deduce the following dichotomy, which is particular to the case $q=3$:

\begin{corollary}[Dichotomy when $q=3$] Let $\Omega$ be a Gaussian-stable $3$-cluster whose cells are perimeter regular (in particular, this holds if $\Omega$ is Gaussian-minimizing) with $\gamma(\Omega) \in \interior \simplex^{(2)}$. Then exactly one of the following possibilities holds: either $M: \R^n \to E^{(2)}$ is surjective, or $\Omega$ is effectively one-dimensional.
\end{corollary}

We conclude that when $q=3$, it remains to deal with the case that $\Omega$ is effectively one-dimensional.
Note that for $q > 3$, the effective dimension could be as high as $q-2 > 1$ if $M$ is not surjective, and we do not know how to handle such effectively higher-dimensional clusters without first establishing that their interfaces are flat and running the full multi-bubble argument described in this work. For effectively one-dimensional clusters, the interfaces are automatically flat (in particular, the curvature is automatically integrable), the higher codimensional boundary is trivial (in particular, there is no need for any regularity theory), and we may simply construct our inward fields by hand -- this is what we do below.

\medskip

To lighten our notation, we will assume that $\Omega$ itself is a cluster in $\R$; by the product structure of the Gaussian measure, everything
that we do here can be lifted to the original cluster in $\R^n$ by taking Cartesian product with $\R^{n-1}$. At this point, there is nothing special about $q=3$, and we can assume that $\Omega$ is a Gaussian-stationary $q$-cluster on $\R$. 

\smallskip

Now, the fact that $H_{ij,\gamma}$ is constant implies that each $\Sigma_{ij}$ can contain at most two points.
For each $i \ne j$, define the vector-field $X_{ij}$ by defining $X_{ij} = \n_{ij}$ on $\Sigma_{ij}$,
$X_{ij} = 0$ on all other interfaces, and extending $X_{ij}$ to a $C_c^\infty$
vector-field on $\R$ (it is possible to make it have compact support because there are at most finitely many points
in $\Sigma_{ij}$).
Given $y \in E$, let $a = L_A^{-1} y$, and set $X = \sum_{i<j} (a_i - a_j) X_{ij}$.
Because the first two derivatives of the
one-dimensional Gaussian density $\varphi$ are $\varphi'(x) = -x\varphi(x)$ and $\varphi''(x) = (x^2 - 1) \varphi(x)$,
we easily compute that 
\begin{align*}
  \delta_X V_i &= \sum_{j \ne i} (a_i - a_j) \sum_{x \in \Sigma_{ij}} \varphi(x) = (L_A a)_i = y_i ~,\\
  \delta_X^2 A &= \sum_{i<j} (a_i - a_j)^2 \sum_{x \in \Sigma_{ij}} (x^2 - 1) \varphi(x) ~,\\
  \delta_X^2 V_i &= \sum_{j \ne i} (a_i - a_j)^2 \sum_{x \in \Sigma_{ij}} \n_{ij}(x) \cdot (-x \varphi(x)) ~. 
\end{align*}
On the other hand, at $x \in \Sigma_{ij}$ one has $H_{ij,\gamma} = \n_{ij}(x) \cdot (-x)$.
Since $\n_{ij} \in \{-1, 1\}$, it follows that
$H_{ij,\gamma} \n_{ij}(x) \cdot (-x) = x^2$, and so:
\begin{align*}
& \inr{\lambda}{\delta_X^2 V} = \sum_{i} \lambda_i \delta_X^2 V_i = \sum_{i<j} (\lambda_i - \lambda_j) (a_i - a_j)^2 \sum_{x \in \Sigma_{ij}} \n_{ij}(x) \cdot (-x \varphi(x)) \\
& = \sum_{i<j} H_{ij,\gamma} (a_i - a_j)^2 \sum_{x \in \Sigma_{ij}} \n_{ij}(x) \cdot (-x \varphi(x)) = \sum_{i<j} (a_i - a_j)^2 \sum_{x \in \Sigma_{ij}} x^2 \varphi(x) .
\end{align*}
Consequently:
\begin{align*}
    Q(X) &= \delta_X^2 A - \inr{\lambda}{\delta_X^2 V} = -\sum_{i<j} (a_i - a_j)^2 \sum_{x \in \Sigma_{ij}} \varphi(x) \\
    &= -\sum_{i<j} (a_i - a_j)^2 A_{ij}  = -a^T L_A a = -y^T L_A^{-1} y.
\end{align*}
This completes the proof of Theorem~\ref{thm:hessian-bound-for-I2} and thus of Theorem~\ref{thm:main-I-I_m} in the case $q=3$.

\section{Uniqueness of Minimizers} \label{sec:uniqueness}

Now that we know Theorem~\ref{thm:main-I-I_m}, the idea is to revisit the various inequalities
used in that proof, and observe that they must be sharp in the case of minimizing clusters.
We begin by restating the relationship between variations and the derivatives of $I$.
We actually already used this relationship in the proof of Theorem~\ref{thm:main-I-I_m},
but this is easier to state now that we know that $I = I_m$ is smooth on $\interior \Delta$. 

\begin{lemma}\label{lem:minimizer-variation}
  Let $\Omega$ be a minimizing $q$-cluster in $\R^n$ with respect to $\gamma$
    with $\gamma(\Omega) = v \in \interior \simplex^{(q-1)}$ and $2 \leq q \leq n+1$, and let $\lambda \in E^*$ be given by Lemma \ref{lem:first-order-conditions}.     Then $\nabla I(v) = \lambda$, and for any admissible vector-field $X$ we have:
  \begin{equation} \label{eq:minimizer-variation-conc}
  (\delta_X V)^T \nabla^2 I(v) \delta_X V \le \delta^2_X A - \inr{\lambda}{\delta^2_X V} = Q(X) . 
  \end{equation}
\end{lemma}
\begin{proof} 
As in the proof of Theorem \ref{thm:main-I-I_m} in Section \ref{sec:multi-bubble}, $I(\gamma(F_t(\Omega))) \le P_\gamma(F_t(\Omega))$, with equality at $t=0$. Differentiating twice at $t=0$ (since we now know that $I=I_m$ is smooth), we obtain (\ref{eq:compare-first}) and (\ref{eq:compare-second}) at \emph{any} point $v_0 \in \interior \simplex^{(q-1)}$. As explained in the proof of Theorem \ref{thm:main-I-I_m}, since $\delta_X V$ spans the entire $E$ (e.g. by Theorem~\ref{thm:hessian-bound-for-I}, or Theorem~\ref{thm:hessian-bound-for-I2} when $q=3$), we deduce (\ref{eq:compare-fourth}) and hence (\ref{eq:compare-third}), as asserted. 
\end{proof}

We immediately deduce from Lemma~\ref{lem:minimizer-variation} and Theorem~\ref{thm:hessian-bound-for-I} that $\nabla^2 I(v) \le - L_A^{-1}$ in the positive-definite sense.
We now observe that this in fact must be an equality, and that this already determines the boundary measures $A_{ij}$ of a minimizing cluster.

\begin{lemma}\label{lem:equality}
            Let $\Omega$ be a minimizing $q$-cluster in $\R^n$ with respect to $\gamma$
    with $\gamma(\Omega) = v \in \interior \simplex^{(q-1)}$ and $2 \leq q \leq n+1$. Let $\Omega^m$ be a model $q$-cluster in $\R^n$ with $\gamma(\Omega^m) = v$. 
    Let $A_{ij}$ and $A^m_{ij}$ denote the Gaussian-weighted areas of the interfaces of $\Omega$ and $\Omega^m$, respectively. Then:
                 \begin{enumerate}[(i)]
        \item $\nabla^2 I(v) = - L_A^{-1}$;
        \item $A_{ij} = A_{ij}^m$ for all $i \ne j$; and
        \item the first and second variations of $\Omega$ satisfy the following inequality for every admissible vector-field $X$:
            \begin{equation} \label{eq:uniqueness-Q}
                -(\delta_X V)^T L_A^{-1} (\delta_X V) \leq Q(X).
            \end{equation}
    \end{enumerate}
\end{lemma}
\begin{proof}
    By the preceding comments, we know that $\nabla^2 I(v) \le - L_A^{-1}$, and hence $-(\nabla^2 I(v))^{-1} \leq L_A$. On the other hand, by Proposition~\ref{prop:I_m-equation}:
    \[
    \tr[-(\nabla^2 I(v))^{-1}] = \tr[-(\nabla^2 I_m(v))^{-1}] = \tr(L_{A^m}) = 2 I_m(v) = 2 I(v) = \tr(L_A) . 
    \]
    Since $A \geq 0$ and $\tr(A) = 0$ imply that $A=0$, it follows that necessarily $-(\nabla^2 I(v))^{-1} = L_A$, thereby concluding the proof of the first assertion. The second follows by inspecting the off-diagonal elements of the matrices on either side of the following equality:
    \[
    L_A = - (\nabla^2 I(v))^{-1} = - (\nabla^2 I_m(v))^{-1} = L_{A^m} .
    \]
      The third assertion follows from the first one and (\ref{eq:minimizer-variation-conc}).
\end{proof}

\smallskip

For the rest of this section, we fix a minimizing $q$-cluster $\Omega$
with $\gamma(\Omega) = v \in \interior \simplex^{(q-1)}$ and $2 \leq q \leq n+1$. Note that $A^m_{ij} > 0$
for all $i \ne j$, and so Lemma~\ref{lem:equality} implies that $A_{ij} > 0$ for all $i \neq j$ as well -- this is the \emph{crucial} property of a minimizing cluster which we will repeatedly use (see also the next section for a more general result involving this property, which yields an alternative proof of uniqueness of minimizers). 
It follows by Theorem \ref{thm:flat} that $\n_{ij}$ is constant on (the non-empty) $\Sigma_{ij}$ for all $i \neq j$. 

\smallskip

Recall from (\ref{eq:Mw}) and (\ref{eq:Nw}) that for the constant (translation) field $w \in \R^n$, since $\n_{ij}$ is constant on $\Sigma_{ij}$, we have:
\[
    \delta_w V = M w ~,~ M = \sum_{i < j} A_{ij} (e_i - e_j) \n_{ij}^T  ,
\]
and 
\[
Q(w) = - \scalar{N w,w} ~,~ N =  \sum_{i < j} A_{ij} \n_{ij} \n_{ij}^T . 
\]
Comparing with (\ref{eq:uniqueness-Q}), it follows that as quadratic forms on $\R^n$: \begin{equation}\label{eq:matrix-inequality}
    M^T L_A^{-1} M \geq N .
\end{equation}
On the other hand, the Cauchy-Schwarz Lemma \ref{lem:cs} implies the converse inequality (\ref{eq:MLAN}). 
It follows that the latter inequality is in fact an equality, and so by the equality case of Lemma \ref{lem:cs}, there exists a linear map $B:  E^{(q-1)} \rightarrow \R^n$ so that $\n_{ij} = B (e_j - e_i)$ for all $i \ne j$ (here we crucially use that $A_{ij} >0$ in the random vector construction leading up to (\ref{eq:MLAN})). In particular, $\sum_{(\ell,m) \in \cyclic(i,j,k)} \n_{\ell m} = B \sum_{(\ell, m) \in \cyclic(i,j,k)} (e_m - e_\ell) = 0$ for all distinct $i,j,k$,
and it follows that $\n_{ij}$, $\n_{jk}$, and $\n_{ki}$ are unit vectors with $120^\circ$ angles
between them. 

Moreover,  since $E_{ij} = (e_j - e_i) (e_j - e_i)^T$ span the space of all symmetric matrices on $E^{(q-1)}$ (by checking dimensions),  
and $\tr(B^T B E_{ij} ) = |\n_{ij}|^2 = 1$ for all $i < j$ (as $A_{ij} > 0$), it follows that $B^T B = \frac 1 2 \Id$, i.e. $\sqrt 2 B$ is an isometry from $E^{(q-1)}$ onto $\Img B = \spn\{\n_{ij}\}$.

Applying Theorem \ref{thm:flat} and using that $A_{ij} > 0$ for all $i \neq j$ one last time, we know that after modifying $\Omega$ by null-sets according to our convention from  Theorem \ref{thm:Almgren} \ref{it:Almgren-ii}, we have:
\begin{align*}
\Omega_i  & =  \{ x \in \R^n : \forall j \neq i  \;\; \inr{\n_{ij}}{x} < \lambda_j - \lambda_i \} \\
& = \{ x \in \R^n : \forall j \neq i  \;\; \inr{e_j - e_i}{B^T x} < \lambda_j - \lambda_i \} \\
& = \{ x \in \R^n : \forall j \neq i  \;\; \inr{e_j - e_i}{B^T x - \lambda} < 0 \} .
\end{align*}
In other words, $\Omega$ is the pull-back via $B^T : \R^n \rightarrow E^{(q-1)}$ of the canonical model $q$-cluster on $E^{(q-1)}$ centered at $\lambda$. Since $\sqrt{2} B$ is an isometry, it follows that $\{\Omega_i\}$ are the Voronoi cells of $q$ equidistant points in $\R^n$, i.e. that $\Omega$ is a simplicial $q$-cluster. This concludes the proof of Theorem~\ref{thm:main-uniqueness}.

\section{Concluding Remarks}  \label{sec:conclude}

\subsection{Towards a characterization of stable regular clusters}

It is an interesting problem to obtain a complete characterization of stable regular $q$-clusters for the Gaussian measure when $q \leq n+1$. 
We add here an additional necessary property of stable regular clusters on top of the information given in Theorem \ref{thm:flat}.
As usual, we denote by $A_{ij} = \gamma^{n-1}(\Sigma_{ij})$ the measures of the interfaces of $\Omega$. 
 
\begin{theorem} \label{thm:pull-back}
Let $\Omega$ denote a stable regular $q$-cluster in $\R^n$ with respect to $\gamma$, and assume that $q \leq n+1$. 
Then $\Omega$ is the pull-back via a linear map $\R^n \mapsto E^{(q-1)}$ of a canonical model simplicial $q$-cluster in $E^{(q-1)}$. 
Specifically:
\begin{enumerate}
\item There exists a linear operator $B : E^{(q-1)} \rightarrow \R^n$ so that
\begin{equation} \label{eq:n-B}
\n_{ij} = B (e_j - e_i)
\end{equation}
for all $i \neq j$ such that $A_{ij} > 0$; and
\item  $\Omega$ is the pull-back via $B^T$ of the canonical model simplicial $q$-cluster in $E = E^{(q-1)}$ centered at $\lambda$: 
\begin{equation} \label{eq:pull-back-polyhedron-formula}
    \Omega_i = \bigcap_{j \neq i } \set{  x \in \R^n : \scalar{e_j - e_i,B^T x - \lambda} < 0 } ,
    \end{equation} 
  where $\lambda \in E^*$ is given by Lemma \ref{lem:first-order-conditions} \ref{it:first-order-cyclic}.
\end{enumerate}
\end{theorem}

Note that this is indeed a strengthening of Theorem \ref{thm:flat}; the additional information that (\ref{eq:n-B}) holds is a strong way to encapsulate the fact that $\Omega$ is stationary, since it implies that $\n_{ij} + \n_{jk} + \n_{ki} = 0$ for all distinct $i,j,k$ so that $A_{ij}, A_{jk} , A_{ki} > 0$ (even without knowing whether $\Sigma_{ijk}$ is non-empty). 
Before providing a proof of Theorem \ref{thm:pull-back}, let us remark that it immediately yields an alternative proof of the uniqueness of minimizers from the previous section, and in fact implies the following strengthening:

\begin{corollary} \label{cor:complete-graph}
Let $\Omega$ denote a stable regular $q$-cluster in $\R^n$ with respect to $\gamma$, and assume that $q \leq n+1$. Assume that $A_{ij} > 0$ for all $i \neq j$, where $A_{ij} = \gamma^{n-1}(\Sigma_{ij})$ denote as usual the interface measures. Then up to null-sets, $\Omega$ is a simplicial cluster. 
\end{corollary}

Indeed, for any minimizing $\Omega$ with $\gamma(\Omega) \in \interior \simplex^{(q-1)}$, Lemma \ref{lem:equality} implies that $A_{ij} = A^m_{ij} > 0$ for all $i \neq j$, which together with Corollary \ref{cor:complete-graph} yields Theorem \ref{thm:main-uniqueness} on the uniqueness of minimizers. To see why Corollary \ref{cor:complete-graph} follows from Theorem \ref{thm:pull-back}, simply note that the only other crucial property which was used in the proof of Theorem \ref{thm:main-uniqueness} in the previous section, besides the property that $A_{ij} > 0$ for all $i \neq j$, was that $\n_{ij} = B(e_j - e_i)$ for all $i \neq j$; once this is known, it follows that $\sqrt{2} B$ is an isometry onto its image, and so Theorem \ref{thm:pull-back} implies that $\{\Omega_i\}$ are the Voronoi cells of $q$ equidistant points in $\R^n$, i.e. that $\Omega$ is a simplicial $q$-cluster.

\medskip

The proof of Theorem \ref{thm:pull-back} is based on an interesting topological observation. Consider the following two-dimensional simplicial complex $\S$ associated to a regular  stable cluster $\Omega$ with respect to $\gamma$: its vertices are given by $\{ (i) : \Omega_i \neq \emptyset \}$, its edges are given by $\{ (i,j) : \Sigma_{ij} \neq \emptyset \}$, and its triangles are given by $\{ (i,j,k) : \Sigma_{ijk} \neq \emptyset \}$. This is indeed a simplicial complex since the face of any simplex in $\S$ is clearly also in $\S$. 
By a trivial adaptation of the proof of Lemma \ref{lem:L-nondegenerate}, it is immediate to see that the $1$-skeleton of $\S$, i.e. the graph obtained by considering only its edges, is necessarily connected. In other words, the zeroth (simplicial) homology of $\S$ is trivial. To this we add:

\begin{theorem} \label{thm:homology}
The first (simplicial) homology of $\S$ is trivial. 
\end{theorem}

\noindent
It is easy to check that Theorem \ref{thm:homology} is false for a general finite convex tessellation of $\R^n$. The crucial property we will use in the proof is the combinatorial incidence structure of cells along triple-points $\Sigma^2$, and the fact that $\H^{n-2}(\Sigma^3 \cup \Sigma^4) = 0$ (by regularity).  As this is not our main focus in this work, we only sketch the proof. 

\begin{proof}[Sketch of Proof of Theorem \ref{thm:homology}]
Let $\C$ denote a directed cycle in the $1$-skeleton of $\S$. Our goal will be to show that it is the boundary of a $2$-chain $\sum_{h=1}^T \Tr_h$, where $\Tr_h$ are oriented triangles in $\S$. Consider a closed piecewise linear directed path $P$ in $\R^n$ which emulates $\C$, meaning that it crosses from $\Omega_i$ to $\Omega_j$ transversely through $\Sigma_{ij}$ in the order specified by $\C$ (and without intersecting $\Sigma^3 \cup \Sigma^4$). Note that it is always possible to construct such a path thanks to the connectivity of the cells (and moreover, the convexity of the cells makes the construction especially simple). Assume that $P = \cup_{i=0}^{N-1} [y_i,y_{i+1}]$ with $y_N := y_0$, 
and fix a point $o$ in one of the (non-empty open) cells which is not co-linear with any of these segments. Consider the convex interpolation $P_t = (1-t) \cdot P + t \cdot o$, so that $[0,1] \ni t \mapsto P_t$ is a contraction of $P$ onto $o$. We claim that there exists a perturbation of the points $\mathcal{Y} := \{y_i\}_{i=0,\ldots,N-1} \cup \{o \}$ so that:
\begin{enumerate}[(i)]
\item $o$ remains in its original cell,  $P$ still emulates $\C$, and $o$ is not co-linear with any of the segments $[y_i,y_{i+1}]$. \label{it:homology-1}
\item $P_t$ does not intersect $\Sigma^3 \cup \Sigma^4$ for all times $t \in [0,1]$. \label{it:homology-2}
\item $P_t$ does not intersect $\Sigma^2$ except for a finite set of times $t \in (0,1)$. \label{it:homology-3}
\item For all times $t \in [0,1]$, $P_t$ crosses $\Sigma^1$ transversely.  \label{it:homology-4}
\end{enumerate}
The argument is similar to the one detailed in the proof of Proposition \ref{prop:local-to-global}, and we employ the same notation introduced there. 
Consider open balls of radius $\eps>0$ around the points of $\mathcal{Y}$; by selecting $\eps$ small-enough, we may always ensure that \ref{it:homology-1} holds for any perturbation inside these balls. Let $K$ denote a closed ball containing the convex hull of the above small balls around $\mathcal{Y}$. 
Consider a perturbation of the points of $\mathcal{Y}$ selected uniformly and independently inside each small ball. 
For each $i=0,\ldots,N-1$, consider the two-dimensional affine subspace $F_i$ spanned by the perturbed points $\{y_i,y_{i+1},o\}$; it is easy to see that the distribution of $F_i$ is absolutely continuous with respect to $\sigma_2^n$, the Haar measure on the homogeneous space of two-dimensional affine subspaces of $\R^n$. 
 Recall the definition (\ref{eq:Favard}) of the integral-geometric Favard measure $\F^{n-2}$, for which we know by Theorem \ref{thm:regularity}  that $\F^{n-2}(\Sigma^3 \cup \Sigma^4) = \H^{n-2}(\Sigma^3 \cup \Sigma^4) = 0$ and that $\F^{n-2}(K \cap \Sigma^2) = \H^{n-2}(K \cap \Sigma^2) < \infty$ (note that the latter statement employs for simplicity the local finiteness asserted in Theorem \ref{thm:regularity} \ref{it:regularity-Sigma2}, but as promised in Remark \ref{rem:relaxed-regularity}, this can be avoided by performing the above perturbation in two separate steps). It follows that with probability $1$, a random two-dimensional subspace selected according to $\sigma_2^n$ will be disjoint from $\Sigma^3 \cup \Sigma^4$ and intersect $K \cap \Sigma^2$ a finite number of times. By absolute continuity, the same holds for all of our randomly selected $F_i$ with probability $1$. In particular, this ensures \ref{it:homology-2} and \ref{it:homology-3}. 
Finally, note that the linear segments of $P_t$ remain parallel to those of $P$ along the contraction to $o$, and since 
$\Sigma^1$ has a finite number of normal directions $\n_{ij}$, with probability $1$ the segments of $P$ (and hence of $P_t$) will cross $\Sigma^1$ transversely, thereby ensuring \ref{it:homology-4}. 

The strategy is now clear: for every $t \in [0,1)$, consider the directed cycle $\C_t$ in the $1$-skeleton of $\S$ which the path $P_t$ traverses; it is well-defined since $P_t$ crosses from one cell to the next transversely through $\Sigma^1$, and is of finite length (the length is bounded e.g. by $N (q-1)$ since every segment can cross at most $q-1$ cells by convexity). 
By openness, $\C_t$ remains constant between consecutive times when $P_t$ intersects $\Sigma^2$, and at such times there are two possibilities: either a directed edge $(i,j)$ in $\C_t$ will transform into two consecutive directed edges $(i,k),(k,j)$, or vice versa, two consecutive directed edges $(i,k),(k,j)$ will collapse into a single directed edge $(i,j)$ (depending on how $P_t$ traverses $\Sigma_{ijk}$ as $t$ varies). By \ref{it:homology-3}, this can happen only a finite number of times, and for times $t$ close enough to $1$, $\C_t$ must be empty since the path $P_t$ is already close enough to $o$ so as to remain inside a single cell. This procedure produces a description of our original cycle $\C = \C_0$ as the boundary of the 2-chain  $\sum_{h=1}^T \Tr_h$, where $\Tr_h$ are the oriented triangles $(i,j,k)$ or $(i,k,j)$ corresponding to the above two possibilities for how $P_t$ traversed $\Sigma_{ijk}$. This concludes the proof. 
\end{proof}

\begin{proof}[Proof of Theorem \ref{thm:pull-back}] 
We first claim that for any directed cycle $\C$ in the $1$-skeleton of $\S$:
\begin{equation} \label{eq:sum-on-cycle}
 \sum_{(i,j) \in \C} \n_{i j} = 0 .
\end{equation}
Indeed, by Theorem \ref{thm:homology}, any such directed cycle is the boundary of a $2$-chain $\sum_{h=1}^T \Tr_h$, where $\Tr_h$ are oriented triangles in $\S$. Note that $\n_{ij} + \n_{jk} + \n_{ki} = 0$ for any oriented triangle $\Tr = (i,j,k)$ in $\S$, since this holds by Corollary \ref{cor:boundary-normal-sum} at any point $x$ in the non-empty $\Sigma_{ijk}$. Summing this up over all oriented triangles $\{ \Tr_h \}_{h = 1,\ldots,T}$, and using that $\n_{m \ell} = -\n_{\ell m}$, (\ref{eq:sum-on-cycle}) immediately follows. 

Clearly, for each coordinate $a=1,\ldots,n$, (\ref{eq:sum-on-cycle}) implies the existence of $b_a \in \E^{(q-1)}$ so that $\scalar{\n_{i j}, e_a} = \scalar{b_a, e_j -e_i}$ for all edges $(i,j) \in \S$; this may be seen as the triviality of the first (simplicial) cohomology. In fact, $b_a$ is necessarily unique if there are no empty cells, by connectivity of the $1$-skeleton on the set of all vertices $\{1,\ldots,q\}$. Defining $B = \sum_{a=1}^n e_a  b_a^T$, this establishes (\ref{eq:n-B}).

It follows by Theorem \ref{thm:flat} that:
\[
 \Omega_i = \bigcap_{j \neq i : A_{ij} > 0} \set{ x \in \R^n : \scalar{e_j - e_i,B^T x - \lambda} < 0 } .
\]
Note the subtle difference with (\ref{eq:pull-back-polyhedron-formula}), where the intersection is taken over all $j \neq i$ (no requirement that $A_{ij} > 0$) -- let us denote the latter variant by $\tilde \Omega_i$. We claim that $\Omega_i = \tilde \Omega_i$ for all $i$. Indeed, $\tilde \Omega$ is a genuine cluster, being the pull-back of a canonical model cluster in $E^{(q-1)}$ centered at $\lambda$, and so its cells cover $\R^n$ (up to null-sets). On the other hand, $\Omega$ is itself a cluster by assumption, and clearly $\tilde \Omega_i \subset \Omega_i$ for all $i$. Since all cells are open, it follows that necessarily $\Omega = \tilde \Omega$, concluding the proof.
\end{proof}

\medskip

An interesting question, which we leave for another investigation, is whether Theorem \ref{thm:pull-back} admits a converse:
\smallskip
\begin{addmargin}[3em]{3em}Is every $q$-cluster which is the pull-back via a linear map $B^T: \R^n \rightarrow E^{(q-1)}$ of a canonical model $q$-cluster in $E^{(q-1)}$, so that $\abs{B (e_i - e_j)}=1$ for all $i \neq j$ such that $A_{ij} > 0$, necessarily stable when $q \leq n+1$?
\end{addmargin}
\smallskip
A positive answer would yield a full characterization of stable regular $q$-clusters when $q \leq n+1$. To appreciate the difficulty in determining the stability of a given cluster, note that we do not even know how to directly show that the model simplicial clusters are stable, without invoking our main theorem to show that they are in fact minimizing (and hence in particular stable). 

\subsection{Characterizations of minimizing clusters}

Let us summarize the characterizing properties of minimizing clusters among all stable regular clusters which we have obtained so far (as usual, when $q \leq n+1$), and add an additional characterizing property to the list.

\begin{theorem} \label{thm:minimizers-char}
Let $\Omega$ denote a stable regular $q$-cluster in $\R^n$ with respect to $\gamma$ with $\gamma(\Omega) \in \interior \simplex^{(q-1)}$, and assume that $q \leq n+1$. 
Let $A_{ij} = \gamma^{n-1}(\Sigma_{ij})$ denote the measures of its interfaces. 
Then the following are equivalent:
\begin{enumerate}[(i)]
\item $\Omega$ is a minimizing cluster. \label{it:equiv-min}
\item $\Omega$ is (up to null-sets) a model simplicial cluster. \label{it:equiv-model} 
\item $A_{ij} > 0$ for all $i \neq j$.  \label{it:equiv-positive}
\item $\Omega$ is effectively $(q-1)$-dimensional (recall Definition \ref{def:effective}). \label{it:equiv-dim}
\end{enumerate}
\end{theorem}

Indeed, \ref{it:equiv-min} implies \ref{it:equiv-positive} by Lemma \ref{lem:equality}, \ref{it:equiv-positive} implies \ref{it:equiv-model} by Corollary \ref{cor:complete-graph}, and \ref{it:equiv-model} implies \ref{it:equiv-min} by Theorem \ref{thm:main-I-I_m}. To this list of equivalences, we add \ref{it:equiv-dim}; recall that by Lemma \ref{lem:MKernelAndDef}, the effective dimension of any stable regular $q$-cluster is at most $q-1$. We remark that for the proof that \ref{it:equiv-dim} implies \ref{it:equiv-positive} and hence \ref{it:equiv-model}, we do not need to invoke our Partial Differential Inequality argument from Section \ref{sec:multi-bubble} nor the simplified argument from Section \ref{sec:double-bubble} (as none of the required implications invoke Theorem \ref{thm:main1} or Lemma \ref{lem:equality}).
In other words, if for a given $v \in \interior \simplex^{(q-1)}$, a minimizing cluster were known to be effectively $(q-1)$-dimensional,
we could directly deduce that it must be a model simplicial cluster, thereby simultaneously verifying Theorems \ref{thm:main-I-I_m} and \ref{thm:main-uniqueness}. While this would result in a conceptual simplification of the overall argument for establishing our main results (even beyond the simplified argument of Section \ref{sec:double-bubble}), unfortunately we do not know how to a-priori exclude the possibility that
a given minimizing cluster is effectively lower-dimensional.

\smallskip

We first establish an almost obvious lemma regarding the convex polyhedral cells of $\Omega$. 

\begin{lemma} \label{lem:F=F}
With the same assumptions as in Theorem \ref{thm:minimizers-char}, assume that $\Sigma_{ij}$ is non-empty for some $i \neq j$.  If $F_{ij}$ denotes the (closed) facet of $\Omega_i$ which contains $\Sigma_{ij}$, then necessarily $F_{ij} = \overline{\Sigma_{ij}}$. In particular, $F_{ij} = F_{ji}$. 
\end{lemma}
\begin{proof}
Recall that the interfaces $\Sigma_{k \ell}$ are relatively open in $\Sigma = \overline{\Sigma^1}$  by Theorem \ref{thm:Almgren}, and hence $\overline{\Sigma_{ij}} \setminus \Sigma_{ij} \subset \Sigma^2 \cup \Sigma^3 \cup \Sigma^4$. 
Assume in the contrapositive that $F_{ij} \neq \overline{\Sigma_{ij}}$. Since $F_{ij}$ is the closure of its relative interior (by convexity), it follows that $\H^{n-1}(F_{ij} \setminus \overline{\Sigma_{ij}}) > 0$; in addition $\H^{n-1}(\Sigma_{ij}) > 0$ by relative-openness. It follows by the relative isoperimetric inequality inside the convex $F_{ij}$ that $\H^{n-2}(\relint F_{ij} \cap (\overline{\Sigma_{ij}} \setminus \Sigma_{ij})) > 0$ (for instance, this can be deduced by endowing $F_{ij}$ with a Gaussian density and employing the Gaussian single-bubble isoperimetric inequality inside a convex set, as in Theorem \ref{thm:CaffarelliCor} below). Since $\overline{\Sigma_{ij}} \setminus \Sigma_{ij} \subset \Sigma^2 \cup \Sigma^3 \cup \Sigma^4$ and $\H^{n-2}(\Sigma^3 \cup \Sigma^4) = 0$ by Theorem \ref{thm:regularity}, it follows that there exists a point $p \in \relint F_{ij} \cap \Sigma^2$. But this contradicts Theorem~\ref{thm:regularity} and Corollary~\ref{cor:boundary-normal-sum},  because $p \in \relint F_{ij}$ implies that $\Omega_i$ has only one outer normal at $p$, and so it cannot have two outer normals with a $120^\circ$ angle between them.
\end{proof}

\begin{proof}[Proof of equivalence with \ref{it:equiv-dim} in Theorem \ref{thm:minimizers-char}]
Clearly, \ref{it:equiv-model} implies \ref{it:equiv-positive} and \ref{it:equiv-dim}, so let us assume that \ref{it:equiv-dim} holds and establish \ref{it:equiv-positive}. 
By Lemma~\ref{lem:dim-deficient} and the product structure of the Gaussian measure, we may assume that $\Omega$ is not dimension-deficient (if it were, we could pass to the subspace perpendicular to the directions of dimension-deficiency). That is, $\Omega$ is full-dimensional and so $q=n+1$. 

As a first step, we observe that there is no cell $\Omega_i$ for which
$\overline \Omega_i$ contains a line.  Assume in the contrapositive  that
$\overline{\Omega_i}$ contains a line parallel to the one-dimensional subspace $E$.
Since $\Omega_i$ is convex and open, it follows that $\Omega_i$ can be
written as $\tilde \Omega_i \times E$. Now let $j \ne i$ be such that $A_{ij} > 0$;
we will show that $\overline \Omega_j$ also contains a line parallel to $E$. Since the undirected graph in which $(i, j)$
is an edge if $A_{ij} > 0$ is connected by Lemma \ref{lem:L-nondegenerate}, it will then follow that every cell contains a line
parallel to $E$, which will imply that every cell $\Omega_j$ can be written as $\tilde \Omega_j \times E$,
contradicting the assumption that $\Omega$ is full-dimensional.

To see that $\overline \Omega_j$ contains a line parallel to $E$, let $F_{ij}$ be the facet of $\Omega_i$
that contains $\Sigma_{ij}$ and let $F_{ji}$ be the facet of $\Omega_j$
that contains $\Sigma_{ij}$. Since $\Omega_i$ is a cylinder in the direction $E$, it follows that $F_{ij}$
contains a line parallel $E$, and so by Lemma \ref{lem:F=F} so does $F_{ji}$ and hence $\overline{\Omega_j}$.

Now fix $i \in \{1, \dots, q\}$.  Since $\overline \Omega_i$ does not contain a line, it follows that $\Omega_i$ has
a face of dimension zero (if not, a face of minimal dimension would necessarily contain a line). So let
$\{p\} \subset \overline \Omega_i$ be a face of dimension zero. Then $p$ is contained in at least $n = q - 1$ facets of $\Omega_i$.
Since $\Omega_i$ has at most $q-1$ facets by (\ref{eq:polyhedron-formula}), 
it follows that $\Omega_i$ has exactly $q-1$ facets. Since every facet has positive $\calH^{n-1}$-measure, it follows by Lemma \ref{lem:F=F} that $A_{ij} > 0$ for every $j$. Since $i$ was arbitrary, this proves~\ref{it:equiv-positive}. Moreover, it actually follows that $p \in  \overline{\Omega_i} \cap (\cap_{j \neq i} \overline{\Sigma_{ij}}) \subset \cap_{j} \overline{\Omega_j}$, and hence our two-dimensional simplicial complex $\S$ is complete (i.e. contains all vertices, edges and triangles), and so trivially has vanishing first homology, without requiring to pass through the proof of  Theorem \ref{thm:homology} to deduce Corollary \ref{cor:complete-graph}. 
\end{proof}

\subsection{Are all stable regular clusters flat?}

We have shown in Theorem \ref{thm:flat} that whenever $2 \leq q \leq n+1$, a stable regular $q$-cluster in $\R^n$ must be flat (thus having convex polyhedral cells). 
In fact, by Proposition \ref{prop:dimension-deficient-implies-flat}, no restriction on $q$ is necessary for the latter statement to hold if the cluster is dimension-deficient. 
One could wonder whether this is also the case for full-dimensional clusters. In other words:
\smallskip
\begin{addmargin}[3em]{3em}Is every stable regular cluster necessarily flat? 
\end{addmargin}
\smallskip
A positive answer would resolve a conjecture of Corneli-et-al \cite[Conjecture 2.1]{CorneliCorwinEtAl-DoubleBubbleInSandG}, which asserts that all minimizing clusters in $\R^2$ are flat. 

\smallskip

One possible way to tackle the above question is to show that stability tensorizes, namely that if $\Omega$ is a stable cluster in $\R^n$ then so is $\Omega \times \R$ in $\R^{n+1}$, since then we could invoke Proposition \ref{prop:dimension-deficient-implies-flat} for the dimension-deficient cluster $\Omega \times \R$. Unfortunately, we do not have a clear feeling of how reasonable this might be. 

\subsection{Further Extensions}

For completeness, we mention that in a follow-up work after our resolution in \cite{EMilmanNeeman-GaussianDoubleBubbleConj} of the Gaussian Double-Bubble Conjecture, and in concurrence to circulating our resolution of the Gaussian Multi-Bubble Conjecture  and posting it on the arXiv in \cite{EMilmanNeeman-GaussianMultiBubbleConj}, Heilman posted \cite{Heilman-FundamentalTone}, where he attempted to further push our simplified double-bubble argument from \cite{EMilmanNeeman-GaussianDoubleBubbleConj} to treat the triple-bubble case $q=4$. Note that with our simplified approach, one just needs to find a single additional vector-field beyond translations for which (\ref{eq:Qinq}) holds, to treat the possibility that the minimizing cluster is effectively two-dimensional (the other two possibilities, that $M$ is surjective or that the cluster is effectively one-dimensional, were already treated in our work). Heilman attempted to show that under a certain supplementary assumption (\cite[Assumption 1.6]{Heilman-FundamentalTone}) on the minimizing cluster's structure, such an additional vector-field may be produced, 
and so by repeating our MPDI argument (see \cite[pp. 6343--6344]{Heilman-FundamentalTone}), this would yield a conditional verification of Theorem 1.1 in the case $q=4$ (under the supplementary assumption).
Note that this supplementary assumption does not appear to hold even for the model simplicial $q$-cluster in $\R^n$ whenever $q \leq n$, but that seems correctable; see \cite{HeilmanGaps} for further comments regarding Heilman's work. 

\medskip

Using standard machinery, it is also possible to obtain a functional version of the Gaussian multi-bubble isoperimetric inequality for locally Lipschitz functions $f : \R^n \rightarrow \Delta^{(q-1)}$, in the spirit of Bobkov's functional version of the classical single-bubble one \cite{BobkovGaussianIsopInqViaCube}. 
This and additionally related analytic directions will be described in \cite{EMilmanNeeman-FunctionalVersions}.
 
\medskip

Finally, we mention that Theorem \ref{thm:main-I-I_m} may be immediately extended to probability measures having strongly convex potentials.

\begin{definition}
A probability measure $\mu$ on $\R^n$ is said to have a $K$-strongly convex potential, $K > 0$, if 
$\mu = \exp(-W(x)) dx$ with $W \in C^\infty$ and $\text{Hess} \; W \geq K \cdot \Id$. 
\end{definition}

\begin{theorem} \label{thm:CaffarelliCor}
Let $\mu$ be a probability measure having $K$-strongly convex potential. 
Denote by $I^{(q-1)}_\mu : \Delta^{(q-1)} \rightarrow \Real_+$ its associated $(q-1)$-bubble isoperimetric profile, given by:
\[
I^{(q-1)}_\mu(v) := \inf \set{P_\mu(\Omega) : \text{$\Omega$ is a $q$-cluster with $\mu(\Omega) = v$}} .
\]
Then:
\[
I^{(q-1)}_\mu \geq \sqrt{K} I^{(q-1)}_m \text{ on } \Delta^{(q-1)} . 
\]
\end{theorem}

The proof is an immediate consequence of the following remarkable theorem by L.~Caffarelli \cite{CaffarelliContraction} (see also \cite{KimEMilmanGeneralizedCaffarelli} for an alternative proof and an extension to a more general scenario):

\begin{theorem}[Caffarelli's Contraction Theorem]
If $\mu$ is a probability measure having a $K$-strongly convex potential,
there exists a $C^{\infty}$ diffeomorphism $T : \R^n \rightarrow \R^n$ which pushes forward the Gaussian measure $\gamma$ onto $\mu$ which is $\frac{1}{\sqrt{K}}$-Lipschitz,  i.e. $\abs{T(x) - T(y)} \leq \frac{1}{\sqrt{K}} \abs{x-y}$ for all $x,y \in \R^n$. 
\end{theorem}

We will require the following calculation:
\begin{lemma}
Let $T : \R^n \to \R^n$ denote a $C^{\infty}$ diffeomorphism pushing-forward $\mu_1 = \Psi_1(x) dx$ onto $\mu_2 = \Psi_2(y) dy$, where $\Psi_1,\Psi_2$ are strictly-positive $C^{\infty}$ densities. 
Let $X$ denote a $C^\infty$ vector-field $X$ on $\R^n$, and let $Y$ denote the vector-field obtained by push-forwarding $X$ via the differential $dT$, $Y = (dT)_* X$, given by:
\[
Y(y) = (dT)_* X(y) := (dT \cdot X)(T^{-1} y) .
\]
Then $Y$ is $C^{\infty}$-smooth and satisfies:
\[
(\div_{\mu_2} Y)(T x) = (\div_{\mu_1} X) (x) \;\;\; \forall x \in \R^n . 
\]
\end{lemma}
\noindent
The proof, which we leave as an exercise, is a simple application of the change-of-variables formula:
\[
\frac{\Psi_1(x)}{\Psi_2(T x)} = \text{det} \; dT (x) . 
\]

\begin{proof}[Proof of Theorem \ref{thm:CaffarelliCor}]
Let $T$ be the map from Caffarelli's theorem, pushing forward $\gamma$ onto $\mu$.  For any $C^{\infty}$ compactly-supported vector-field $X$ with $\abs{X} \leq 1$, denote $Y := (dT)_* X$, and observe that since $T$ is $\frac{1}{\sqrt{K}}$-Lipschitz, then $Y$ is also $C^{\infty}$, compactly-supported, and satisfies $\abs{Y} \leq \frac{1}{\sqrt{K}}$. In addition, for any Borel subset $U$ in $\R^n$, observe that:
\[
\int_U \div_{\mu} Y (y) d\mu(y) = \int_{T^{-1} U}  ( \div_{\mu} Y) (Tx) d\gamma(x) = \int_{T^{-1} U}  \div_{\gamma^n} X (x) d\gamma(x) .
\]
Consequently, taking supremum over all such vector-fields $X$, we deduce that:
\[
\frac{1}{\sqrt{K}} P_{\mu}(U) \geq P_{\gamma}(T^{-1} U) . 
\]
On the other hand, by definition, $\mu(U) = \gamma(T^{-1} U)$. Applying these observations to the cells of an arbitrary $q$-cluster $\Omega$, and applying Theorem \ref{thm:main-I-I_m}, we deduce:
\[
 \frac{1}{\sqrt{K}} P_{\mu}(\Omega) \geq P_{\gamma}(T^{-1} \Omega) \geq I_m(\gamma(T^{-1}(\Omega))) = I_m(\mu(\Omega)) . 
\]
It follows that $I_\mu \geq \sqrt{K} I_m$, as asserted. 
\end{proof}

\begin{remark}
It is possible to extend the above argument to measures $\mu$ with \emph{non-smooth} densities having $K$-strongly convex potentials, namely $\mu = \exp(-W(x)) dx$ with $U(x) = W(x) - \frac{K}{2} \abs{x}^2 : \R^n \rightarrow \R \cup \{+\infty\}$ being convex, 
but we do not insist on this generality here. 
\end{remark}

\appendix

\section{Volume and Perimeter Regularity} \label{app:per-regularity}

In this appendix, we derive a criterion for volume and perimeter regularity of cells with respect to a given measure $\mu = e^{-W} dx$, and verify that it holds for the Gaussian measure.  For simplicity, we assume that $\mu$ is a probability measure and that:
\begin{equation} \label{eq:decreasing}
\exists R_* \geq 0 \;\;\; [R_*,\infty) \ni R \mapsto \mu^{n-1}(\partial B_R)  \text{ is decreasing.}
\end{equation}
Recall that $B_R$ denotes the open Euclidean ball of radius $R>0$ centered at the origin. 

\begin{lemma} \label{lem:regular-cond}
For each $i \geq 0$, let $f_i : \Real_+ \rightarrow \Real_+$ denote a $C^1$ increasing function so that:
\[
\norm{\nabla^i W(x)} \leq f_i(\abs{x}) \;\;\; \forall x \in \R^n . 
\]
Assume that for all $i,j \geq 0$, there exists $\delta_{ij} > 0$, so that defining $F_{ij} : \Real_+ \rightarrow \Real_+$ by:
\[
F_{ij}(R) := f_i(R + \delta_{ij})^j e^{\delta_{ij} f_1(R+\delta_{ij})} ,
\]
we have:
\begin{equation} \label{eq:integrability-conds}
\int_{\R^n} F_{ij}(\abs{x}) d\mu < \infty ~,~ \int_{\R^n} F'_{ij}(\abs{x}) d\mu < \infty . 
\end{equation} 
Then any Borel set $U \subset \R^n$ is volume regular, and the cells of any isoperimetric $\mu$-minimizing cluster are perimeter regular (all with respect to $\mu$). \\
Furthermore, if $\delta_{ij} \geq \delta > 0$ uniformly in $i,j \geq 0$, the above sets are in fact uniformly volume and perimeter regular, respectively. 
\end{lemma}

For the proof of the perimeter regularity, we require the following perimeter decay estimate:
\begin{lemma}\label{lem:area-decay}
    If $\Omega$ is an isoperimetric $\mu$-minimizing $q$-cluster then:
    \[
        \sum_{i=1}^q \mu^{n-1}(\partial^* \Omega_i \setminus B_R) \leq 3 q \mu^{n-1}(\partial B_R) 
    \]
    for all $R \geq R_*$. 
\end{lemma}
\begin{proof}
   Let $R \geq R_*$. 
   Since $\mu(\R^n \setminus B_R) = \sum_{i=1}^q \mu(\Omega_i \setminus B_R)$, we may choose a non decreasing sequence $R = R_0 \leq R_1 \leq R_2 \leq \ldots R_{q-1} \leq R_q = \infty$ so that $\mu(B_{R_i} \setminus B_{R_{i-1}}) = \mu(\Omega_i \setminus B_R)$ for all $i=1,\ldots,q$. 
   Now define the cells of a competing cluster $\tilde \Omega$ as follows:
    \[
        \tilde \Omega_i := (\Omega_i \cap B_R) \cup (B_{R_i} \setminus B_{R_{i-1}})  .
    \]
    Clearly $\mu(\tilde \Omega) = \mu(\Omega)$.
                    Now observe that (see e.g. \cite[Lemma 12.22, Theorem 16.3]{MaggiBook}):
    \begin{align*}
        \per_\mu(\tilde \Omega_i)
        &\le \per_\mu(\Omega_i \cap B_R) + \per_\mu(B_{R_i} \setminus B_{R_{i-1}}) \\
        & \le \mu^{n-1}(\partial^* \Omega_i \cap B_R) + \mu^{n-1}(\partial B_R) + \mu^{n-1}(\partial B_{R_i}) + \mu^{n-1}(\partial B_{R_{i-1}}) \\
        &\le \mu^{n-1}(\partial^* \Omega_i\cap B_R) + 3 \mu^{n-1}(\partial B_R) .
    \end{align*}
   Summing in $i$ and dividing by $2$,
    \[
        \per_\mu(\tilde \Omega) \leq \frac{1}{2} \sum_{i=1}^q \mu^{n-1}(\partial^* \Omega_i\cap B_R) + \frac{3}{2} q \mu^{n-1}(\partial B_R) .
            \]
    On the other hand, the fact that $\Omega$ is $\mu$-minimizing and $\mu(\Omega) = \mu(\tilde \Omega)$ implies that
    \[
        \per_\mu(\tilde \Omega) \ge \per_\mu(\Omega) = \frac{1}{2} \sum_{i=1}^q \mu^{n-1}(\partial^* \Omega_i) .
    \]
    Combining these inequalities yields the assertion.
\end{proof}

\begin{proof}[Proof of Lemma \ref{lem:regular-cond}]
Note that by the intermediate value theorem: 
\[
\abs{\pot(z) -\pot(x)} \leq \abs{z-x} f_1(\max(\abs{x},\abs{z})) .
\]
Consequently:
\[
\sup_{z \in B(x,\delta_{ij})} \norm{\nabla^i \pot(z)}^j e^{-\pot(z)} \leq f_i(\abs{x}+\delta_{ij})^j e^{\delta_{ij} f_1(\abs{x}+\delta_{ij})} e^{-\pot(x)} = F_{ij}(\abs{x}) e^{-\pot(x)} ,
\]
and so if for any $i,j \geq 0$, the right-hand function is integrable on $\R^n$, then all Borel sets are volume regular with respect to $\mu$. 

Now if $\Omega$ is an isoperimetric $\mu$-minimizing $q$-cluster, then for each of its cells $U$:
\[
\int_{\partial^* U} \sup_{z \in B(x,\delta_{ij})} \norm{\nabla^i \pot(z)}^j e^{-\pot(z)} d\H^{n-1} \leq \int_{\partial^* U} F_{ij}(\abs{x}) d\mu^{n-1}(x) .
\]
Integrating by parts and applying Lemma \ref{lem:area-decay}, it follows that:
\begin{align*}
& \leq F_{ij}(R_*) \mu^{n-1}(\partial^* U) + \int_{R_*}^\infty F_{i,j}'(R) \mu^{n-1}(\partial^* \Omega_i \setminus B_R) dR \\
& \leq  F_{ij}(R_*) \mu^{n-1}(\partial^* U) + 3 q \int_{R_*}^\infty F_{i,j}'(R) \mu^{n-1}(\partial B_R) dR \\
& =  F_{ij}(R_*) \mu^{n-1}(\partial^* U) + 3 q \int_{\R^n \setminus B_{R_*}} F'_{i,j}(\abs{x}) d\mu(x) . 
\end{align*}
Since $\mu^{n-1}(\partial^* U) < \infty$ for any cell of a (minimizing) cluster, the perimeter regularity of $U$ with respect to $\mu$ follows as soon as the second term above is finite for all $i,j \geq 0$, as asserted. 
\end{proof}

\begin{corollary} \label{cor:Gaussian-regular}
For the standard Gaussian measure $\gamma$, any Borel set $U \subset \R^n$ is uniformly volume regular, and the cells of any isoperimetric minimizing cluster are uniformly perimeter regular, with respect to $\gamma$. 
\end{corollary}
\begin{proof}
Note that (\ref{eq:decreasing}) indeed holds since $\gamma^{n-1}(\partial B_R) = c_n R^{n-1} e^{-R^2/2}$. Setting $f_0(R) = \frac{R^2}{2}$, $f_1(R) = R$, $f_2(R) = \sqrt{n}$, $f_i(R) = 0$ for all $i \geq 3$ and $\delta = 1$, it is immediate to verify the integrability conditions (\ref{eq:integrability-conds}), as the Gaussian measure decays faster than any polynomial or exponential function. The assertion therefore follows by Lemma \ref{lem:regular-cond}. 
\end{proof}

\section{First variation for non-compactly-supported fields} \label{app:first-var}

In this appendix, we derive formulas for the first variations of (weighted) volume and perimeter. While these are well-known for sets with smooth boundaries and compactly supported vector-fields (see e.g. \cite{RCBMIsopInqsForLogConvexDensities}), for the purposes of this work, we need to extend them to more general volume and perimeter regular sets and admissible vector-fields. 

\medskip

The following lemma will be very useful in this work for calculating first and second variations of non-compactly-supported vector-fields. In several of these instances the vector-field will not be bounded in magnitude, and so we formulate it quite generally. 

\begin{definition}[Cutoff function $\eta_R$]
Given $R > 0$, we denote by $\eta_R : \R^n \rightarrow [0,1]$ a smooth compactly-supported cutoff function on $\R^n$ with $\eta_R(x) = 1$ for $\abs{x} \leq R$ and $\abs{\nabla \eta_R} \leq 1$.
\end{definition}

\begin{lemma} \label{lem:cutoff}
\hfill
\begin{enumerate}
\item
Let $X$ denote a $C^1$ vector-field on $\R^n$ so that $\abs{X}, \abs{\div X} \leq P(\abs{\nabla \pot},\ldots,\norm{\nabla^p \pot})$ for some real valued polynomial $P$ and $p \geq 1$. Assume that $U$ is volume regular with respect to the measure $\mu$. Then: 
\[
\int_U \div_\mu X \; d\mu = \lim_{R \rightarrow \infty} \int_U \div_{\mu}( \eta_R X) \; d\mu . 
\]
\item
Let $X$ denote a $C^1$ vector-field on a smooth hypersurface $\Sigma \subset \R^n$ so that $\abs{X}, \abs{\div_{\Sigma} X} \leq P(\abs{\nabla \pot},\ldots,\norm{\nabla^p \pot})$ for some real valued polynomial $P$ and $p \geq 1$. Assume that $\Sigma \subset \partial^*U$ with $U$ being perimeter regular with respect to the measure $\mu$. Then:
\[
\int_{\Sigma} \div_{\Sigma,\mu} X \; d\mu^{n-1} = \lim_{R \rightarrow \infty} \int_\Sigma \div_{\Sigma,\mu}( \eta_R X) \; d\mu^{n-1} . 
\]
\end{enumerate}
\end{lemma}
\begin{proof}
We will prove the first part, the proof of the second one is identical. Write: 
\[
\int_U \div_{\mu} X d\mu = \int_U \div_\mu (\eta_R X) d\mu - \int_U \scalar{\nabla \eta_R,X} d\mu + \int_U (1-\eta_R) \div_{\mu} X d\mu . 
\]
Next, note that:
\[
\abs{\div_\mu X} \leq \abs{\div X} + \abs{\nabla_X \pot} \leq Q(\abs{\nabla \pot},\ldots,\norm{\nabla^p \pot}) ~,~ \abs{\scalar{\nabla \eta_R,X}} \leq P(\abs{\nabla \pot},\ldots,\norm{\nabla^p \pot}) ,
\]
for an appropriate polynomial $Q$. Consequently:
\begin{align*}
\abs{\int_U (1-\eta_R) \div_{\mu} X d\mu} & \leq \int_{U \setminus B_R} Q(\abs{\nabla \pot},\ldots,\norm{\nabla^p \pot})  d\mu ~, \\
\abs{\int_U \scalar{\nabla \eta_R,X}  d\mu} & \leq \int_{U \setminus B_R} P(\abs{\nabla \pot},\ldots,\norm{\nabla^p \pot})  d\mu ~,
\end{align*}
and volume regularity implies that both of these terms go to zero as $R \rightarrow \infty$, yielding the claim. 
\end{proof}

\begin{proposition}\label{prop:first-variation}
Let $X$ be an admissible vector-field on $\R^n$, and let $U \subset \R^n$ denote a Borel subset with $P_\mu(U) < \infty$. \\
\begin{enumerate}
\item
If $U$ is volume regular with respect to the measure $\mu$ or if $X$ is $C_c^\infty$, then:
  \begin{equation}\label{eq:formula-first-variation-of-volume}
    \delta_X V(U) = \int_{\partial^* U} X^\n \, d\mu^{n-1} .
  \end{equation}
\item
If $U$ is perimeter regular with respect to the measure $\mu$ or if $X$ is $C_c^\infty$ then:
\[
\delta_X A(U) = \int_{\partial^* U} \div_{\n_U^{\perp},\mu} X d\mu^{n-1} .
\]
If in addition  $\partial^* U = \Sigma \cupdot \Xi$ where $\Sigma$ is a smooth hypersurface and $\H^{n-1}(\Xi) = 0$, then: 
  \begin{equation}\label{eq:formula-first-variation-of-area-before}
    \delta_X A(U) = \int_{\Sigma} (H_{\Sigma,\mu} X^\n + \div_{\Sigma,\mu} X^\tang) \, d\mu^{n-1}.
  \end{equation}
 \end{enumerate}
\end{proposition}

\begin{proof}Let $F_t$ be the flow along $X$. Under the assumptions of the first assertion, Lemma \ref{lem:regular} implies:
\[
\delta_X V(U) = \int_U \left . \frac{d}{dt} \right |_{t=0} (J F_t e^{-\pot \circ F_t}) dx . 
\]
It is well-known (e.g. \cite[(2.13)]{SternbergZumbrun}) that $\frac{d}{dt} |_{t=0} J F_t = \div X$, and therefore:
\[
  \delta_X V(U) = \int_U \brac{\div X - \nabla_X \pot} d\mu = \int_{U} \div_\mu X \, d\mu .
\]  
In order to apply integration-by-parts, we need to first make $X$ compactly supported. Applying Lemma \ref{lem:cutoff} and integrating by parts the compactly-supported vector-field $\eta_R X$ using ~\eqref{eq:integration-by-parts}, we deduce:
\[
\int_U \div_{\mu} X d\mu = \lim_{R \rightarrow \infty} \int_U \div_\mu (\eta_R X) d\mu = \lim_{R \rightarrow \infty} \int_{\partial^* U} \eta_R X^{\n} d\mu^{n-1} = \int_{\partial^* U} X^{\n} d\mu^{n-1} ,
\]
where the last equality follows by Dominant Convergence since $X$ is uniformly bounded and $P_\mu(U) < \infty$.

For the second assertion, set as usual $\Phi_t = F_t|_{\partial^* U}$ and recall that $J \Phi_t = \text{det}((d_{\n_{U}^{\perp}} F_t)^T d_{\n_{U}^{\perp}} F_t)^{1/2}$ is the Jacobian of $\Phi_t$. Under the assumptions of the second assertion, Lemma \ref{lem:regular} implies:
\[
\delta_X A(U) = \int_{\partial^* U} \left . \frac{d}{dt} \right |_{t=0} (J \Phi_t e^{-\pot \circ F_t}) d\H^{n-1}(x) . 
\]
It is well-known (e.g. \cite[(2.16)]{SternbergZumbrun}) that $\frac{d}{dt} |_{t=0} J \Phi_t = \div_{\n_U^\perp} X$, and hence: 
\[
\delta_X A(U) = \int_{\partial^* U} \brac{\div_{\n_U^\perp} X - \nabla_X \pot} d\mu^{n-1} = \int_{\partial^* U} \div_{\n_U^\perp,\mu} X \, d\mu^{n-1} .
\]
To establish the last claim, simply continue as follows:
\begin{align*}
& = \int_{\Sigma} \div_{\Sigma,\mu} X \, d\mu^{n-1} \notag \\
& = \int_{\Sigma} \brac{\div_{\Sigma,\mu} X^\tang + \div_{\Sigma,\mu} (X^\n \n) } d\mu^{n-1}  \notag \\
&= \int_{\Sigma} \brac{\div_{\Sigma,\mu} X^\tang + H_{\Sigma,\mu} X^\n } d\mu^{n-1}.
\end{align*}
\end{proof}

\section{Stationarity and Stability} \label{app:stst}

In this appendix, we verify that the usual properties of stationarity and stability of a minimizing cluster hold for all admissible (not necessarily compactly supported) vector-fields $X$, assuming its cells are volume and perimeter regular (with respect to $\mu$). When $X$ is compactly supported this is well-known and standard in the single-bubble case, and was proved for Euclidean double-bubbles in \cite{DoubleBubbleInR3} assuming higher-order boundary regularity, which we avoid in this appendix. 

\medskip

\noindent
Given a $q$-cluster $\Omega$, let $\Sigma_{ij}$ denote its interfaces. Recall that $E = E^{(q-1)} = \set{x \in \R^q : \sum_{i=1}^q x_i = 0}$. 

\medskip

\begin{lemma}\label{lem:span}
    If $\mu(\Omega) \in \interior \simplex^{(q-1)}$ then the set $\{e_i - e_j: \mu^{n-1}(\Sigma_{ij}) > 0\}$ spans $E^{(q-1)}$.
\end{lemma}

\begin{proof}
    Suppose in the contrapositive that $\{e_i - e_j: \mu^{n-1}(\Sigma_{ij}) > 0\}$ does not span $E^{(q-1)}$.
    Then there exists some non-zero $v$ in $E^{(q-1)}$ such that $v_i = v_j$ whenever $\mu^{n-1}(\Sigma_{ij}) > 0$.
    Let $S = \{i : v_i < 0\}$ and let $T = \{i : v_i \ge 0\}$. Since $v \in E^{(q-1)}$ implies $\sum_i v_i = 0$, it follows
    that both $S$ and $T$ are non-empty. Since $i \in S$ and $j \in T$ imply that $v_i \ne v_j$, it follows that
    $\mu^{n-1}(\Sigma_{ij}) = 0$ whenever $i \in S$ and $j \in T$. But this is in contradiction to the connectivity of the graph $G$ from Lemma \ref{lem:L-nondegenerate}. 
  \end{proof}

\begin{lemma}\label{lem:compensators}
    If $\mu(\Omega) \in \interior \simplex^{(q-1)}$ then
    there exists a collection of $C_c^\infty$ vector-fields $Y_1, \dots, Y_{q-1}$ with disjoint supports, so that for every $p=1,\ldots,q-1$,  $(Y_p)|_{\cup_{i<j}\Sigma_{ij}}$ is supported in $\Sigma_{ij}$ for some $i<j$, and such the set $\{\delta_{Y_p} V: p = 1, \dots, q-1\}$ spans $E^{(q-1)}$.
\end{lemma}

\begin{proof}
    For each pair $i, j$ with $\mu^{n-1}(\Sigma_{ij}) > 0$, choose some $x_{ij} \in \Sigma_{ij}$.
        By Theorem \ref{thm:Almgren}, there exists some $\eps > 0$ such that $B(x_{ij},\eps)$ is disjoint from all the cells besides $\Omega_i,\Omega_j$ (and hence disjoint from all
    the interfaces besides $\Sigma_{ij}$), and such that $\overline{B(x_{ij},\eps) \cap \Sigma_{ij}} \subset \Sigma_{ij}$.
        By replacing $\eps$ by $\eps/2$, we can ensure that all of the $B(x_{ij},\eps)$ are pairwise disjoint.  
    Let $f_{ij}$ be a non-negative $C_c^\infty$ function supported in $B(x_{ij},\eps)$
    such that $f_{ij}(x_{ij}) > 0$, and let $X_{ij}$ be a smooth extension of $f_{ij} \n_{ij}$ that is supported in $B(x_{ij},\eps)$.
    Then $\delta_{X_{ij}} V = \alpha_{ij} (e_i - e_j)$ for some $\alpha_{ij} > 0$.

    By Lemma~\ref{lem:span}, $\{\delta_{X_{ij}} V: \mu^{n-1}(\Sigma_{ij}) > 0\}$ spans $E^{(q-1)}$; hence,
    we may choose $Y_1, \dots, Y_{q-1}$ to be an appropriate subset of the $X_{ij}$.
\end{proof}

\subsection{Stationarity} 

\begin{lemma}[Stationarity] \label{lem:first-order} 
    Let $\Omega$ be a $\mu$-minimizing cluster.  Then for any $X$ in $C_c^\infty$, if $\delta_X V = 0$ then $\delta_X A = 0$. Moreover, the same holds for all admissible vector-fields $X$ if the cells of $\Omega$ are volume and perimeter regular with respect to $\mu$. 
\end{lemma}

\begin{proof}
   Let $F_t$ denote the flow along $X$, defined as usual by:
     \[
    \frac{d}{dt} F_t = X \circ F_t ~,~ F_0 = \text{Id} .
    \]
    Choose a family $Y_1, \dots, Y_{q-1}$ of vector-fields as in Lemma~\ref{lem:compensators}, having compact and pairwise disjoint supports. 
    Let $\{F_{t,s}\}_{t \in \R , s \in \R^{q-1}}$ be a family of $C^\infty$ diffeomorphisms defined by solving the following system of linear stationary PDEs:
    \begin{align*}
    \pdiff{}{s_i} F_{t,s} &= Y_i \circ F_{t,s} \;\;\; \forall i=1,\ldots,q-1 ~, \\
     F_{t,\vec 0} &= F_t .
    \end{align*}
                                                                  Observe that the above system is indeed integrable since the $Y_i$'s have disjoint supports, and hence the flows they individually generate necessarily commute (cf. the Frobenius Theorem \cite{Lang-ManifoldsBook}). Consequently:
     \begin{equation} \label{eq:concat}
     F_{t,s} = F^{q-1}_{s_{q-1}} \circ \ldots \circ F^{1}_{s_1} \circ F_t ,
     \end{equation}
     where:
     \[
     \frac{d}{ds} F^i_s = Y_i \circ F^i_s ~,~ F^i_0 = \text{Id} \;\;\; \forall i=1,\ldots,q-1 ~, 
     \]
     and all of the usual smoothness and uniform boundedness of all partial derivatives of any fixed order apply to $F_{t,s}$ for  $t \in [-T,T]$, $s \in [-T,T]^{q-1}$, for any fixed $T > 0$. 
     
    Let $V(t, s) = \mu(F_{t,s}(\Omega))$ and $A(t,s) = P_\mu(F_{t,s}(\Omega))$. By a tedious yet straightforward adaptation of the proof of Lemma \ref{lem:regular} to a concatenation of (partly non-commuting) flows as in (\ref{eq:concat}), it follows that $V$ and $A$ are both $C^\infty$ on $\{(t,s) : t \in (-\epsilon,\epsilon) , s \in (-\epsilon,\epsilon)^{q-1}\}$ for some $\epsilon > 0$.         Clearly:
    \begin{align*}
        \left.\frac{\partial^m V}{(\partial t)^m}\right|_{t=0,s=0} = \delta_X^m V ~,~ &
        \left.\frac{\partial^m A}{(\partial t)^m}\right|_{t=0,s=0} = \delta_X^m A \;\;\;\;\;\forall m =1,2,\ldots\\
        \left.\frac{\partial V}{\partial s_i}\right|_{t=0,s=0} = \delta_{Y_i} V  ~,~ &
        \left.\frac{\partial A}{\partial s_i}\right|_{t=0,s=0} = \delta_{Y_i} A.
    \end{align*}
    Since $\{\delta_{Y_i} V\}_{i=1,\ldots,k-1}$ span $E$, $\{\partial_{s_i} V_j(0, 0)\}_{ji} : \R^{q-1} \to E$ is full rank. By the implicit
    function theorem, there exists a $\delta \in (0,\epsilon)$ and a $C^\infty$ curve $s(t) : (-\delta,\delta) \rightarrow (-\epsilon,\epsilon)^{q-1}$, such that $s(0) = 0$ and $V(t,s(t)) = V(0,0) = \mu(\Omega)$ for all $|t| < \delta$. Moreover, $\frac{\partial V}{\partial t}(0,0) = \delta_X V = 0$ and the full-rank of $\{\partial_{s_i} V_j(0, 0)\}$ imply that $s'(0) = 0$. From this property and the chain rule,
    \[
        \diffat{A(t,s(t))}{t}{t=0} = \pdiffat{A(t,s)}{t}{t=0,s=0} = \delta_X A.
    \]
    Hence, we conclude that $\delta_X A = 0$; if not, there is some $t \ne 0$ such that the cluster
    $\tilde \Omega = F_{t,s(t)}(\Omega)$ has $\mu(\tilde \Omega) = \mu(\Omega)$ and
    $\per_\mu(\tilde \Omega) = A(t,s(t)) < A(0,0) = \per_\mu(\Omega)$, contradicting the minimality of $\Omega$.
\end{proof}

\subsection{First-Order Conditions}

We are now ready to provide a proof of Lemma \ref{lem:first-order-conditions} regarding first-order properties which a minimizing cluster must satisfy.

\begin{proof}[Proof of Lemma \ref{lem:first-order-conditions}]
We will sketch parts~\ref{it:first-order-constant} and~\ref{it:first-order-cyclic}, which are well-known, and provide a more detailed explanation of part~\ref{it:weak-angles} which is less standard. 
Since $\H^{n-1}(\partial^* \Omega_i \setminus \cup_{j \neq i} \Sigma_{ij}) = 0$ by (\ref{eq:nothing-lost}), and since $\Sigma_{ij}$ are all smooth for a minimizing cluster $\Omega$ by Theorem \ref{thm:Almgren}, the smoothness assumption in the second part of Proposition \ref{prop:first-variation} is satisfied and we may appeal to the formulas for first variation of volume and perimeter derived there:
\begin{align}
\label{eq:first-order-V} \delta_X V(\Omega)_i & = \sum_{j \neq i} \int_{\Sigma_{ij}} X^{\n_{ij}} \, d\mu^{n-1} \;\;\; \forall i=1,\ldots,q ~, \\
\label{eq:first-order-A} \delta_X A(\Omega) & = \sum_{i < j} \int_{\Sigma_{ij}} (H_{\Sigma_{ij},\mu} X^{\n_{ij}} + \div_{\Sigma_{ij},\mu} X^\tang) \, d\mu^{n-1} ~,
\end{align}
valid for any $X$ in $C_c^\infty$, as well as for any admissible $X$ assuming the cells $\Omega_i$ are volume and perimeter regular. 
Note that whenever the support of $X^{\tang}|_{\Sigma_{ij}}$ is contained in $\Sigma_{ij}$, we may integrate by parts to obtain:
\begin{equation} \label{eq:tang-zero}
\int_{\Sigma_{ij}} \div_{\Sigma_{ij},\mu} X^\tang \, d\mu^{n-1} = 0  . 
\end{equation}
 
To establish part~\ref{it:first-order-constant}, observe that for any two different points $x_1,x_2 \in \Sigma_{ij}$ we can find by Theorem \ref{thm:Almgren} an $\epsilon > 0$ so that $B(x_1,\epsilon)$, $B(x_2,\epsilon)$ and all the other interfaces are disjoint, $\mu^{n-1}(B(x_p,\eps) \cap \Sigma_{ij}) > 0$, 
and $\overline{B(x_p,\eps) \cap \Sigma_{ij}} \subset \Sigma_{ij}$, for $p=1,2$. Consequently, for any smooth vector-field $X$ supported in $B(x_1,\eps) \cup B(x_2,\eps)$, (\ref{eq:tang-zero}) applies. In view of (\ref{eq:first-order-V}) and (\ref{eq:first-order-A}), if $H_{\Sigma_{ij},\mu}$ were not the same at $x_1$ and $x_2$, we could then construct a smooth vector-field $X$ supported in $B(x_1,\eps) \cup B(x_2,\eps)$ with $\delta_X V(\Omega) = 0$ while $\delta_X A(\Omega) < 0$,  in violation of Lemma~\ref{lem:first-order}. 

To establish part~\ref{it:first-order-cyclic}, consider the undirected graph $G$ with vertices $\{1, \dots, q\}$ and an edge between vertex $i$ and $j$ if $\mu^{n-1}(\Sigma_{ij}) > 0$. Observe that for any directed cycle $(i_1,i_2,\ldots,i_r,i_{r+1}=i_1)$ in $G$ we must have $\sum_{p=1}^{r} H_{i_p i_{p+1} , \mu} = 0$; otherwise, we could construct a smooth vector-field compactly supported around $r$ points (one in each $\Sigma_{i_p i_{p+1}}$) that would preserve weighted volume to first order, while decreasing weighted perimeter to first order, exactly as above, again in violation of Lemma~\ref{lem:first-order}. Clearly, this implies the existence of $\lambda \in \R^{q}$ so that $H_{ij,\mu} = \lambda_i - \lambda_j$, and by adding an appropriate constant, we can always ensure that $\sum_{i=1}^q \lambda_i = 0$, i.e. that $\lambda \in E^*$ (it will soon be clear that $\lambda$ acts on $E$). Whenever $\mu(\Omega) \in \interior \simplex$, $\lambda \in E^*$ is uniquely defined by Lemma \ref{lem:span}. 

At this point, we have shown that for any vector-field $X$ as above:
\begin{align}
\notag & \sum_{i<j} \int_{\Sigma_{ij}}  H_{ij,\mu} X^{\n_{ij}}\, d\mu^{n-1} = \sum_{i < j} (\lambda_i - \lambda_j) \int_{\Sigma_{ij}} X^{\n_{ij}} \, d\mu^{n-1} \\
\notag  & = \sum_{i \neq j} \lambda_i \int_{\Sigma_{ij}} X^{\n_{ij}} \, d\mu^{n-1} = \sum_{i} \lambda_i \sum_{j \neq i} \int_{\Sigma_{ij}} X^{\n_{ij}}\, d\mu^{n-1} \\
\label{eq:inr-lambda}  &= \sum_{i} \lambda_i \; \delta_X V(\Omega_i) =  \inr{\lambda}{\delta_X V},
\end{align}
where we have used above that $\int_{\Sigma_{ij}} X^{\n_{ij}}\, d\mu^{n-1} = - \int_{\Sigma_{ji}} X^{\n_{ji}}\, d\mu^{n-1}$. 
For any such $X$, there exists by Lemma \ref{lem:compensators} a vector-field $Y = \sum_{i=1}^{q-1} c_i Y_i$ with $\delta_Y V(\Omega) = -\delta_X V(\Omega)$ so that $Y|_{\cup_{i<j} \Sigma_{ij}}$ is supported in $\cup_{i<j} \Sigma_{ij}$. In particular, we may integrate-by-parts and obtain:
 \begin{equation} \label{eq:lambda1}
 \int_{\Sigma_{ij}} \div_{\Sigma_{ij},\mu} Y^\tang \, d\mu^{n-1} = 0\;\;\; \forall i < j . 
 \end{equation}
 Since $\delta_{X+Y} V = 0$, it follows by (\ref{eq:inr-lambda}) that:
 \begin{equation} \label{eq:lambda2}
 \sum_{i<j}\int_{\Sigma_{ij}} H_{ij,\mu}  (X^{\n_{ij}} + Y^{\n_{ij}}) \, d\mu^{n-1} = \inr{\lambda}{\delta_{X+Y} V} = 0 . 
 \end{equation}
 In addition, Lemma \ref{lem:first-order} implies that:
 \begin{equation} \label{eq:lambda3}
 \delta_{X+Y} A = 0 . 
 \end{equation}
Combining (\ref{eq:lambda1}), (\ref{eq:lambda2}) and (\ref{eq:lambda3}), we obtain:
 \begin{align}
\nonumber 0 & = \delta_{X + Y} A  \\
 \nonumber &= \sum_{i<j} \int_{\Sigma_{ij}} \brac{H_{ij,\mu} (X^{\n_{ij}} + Y^{\n_{ij}}) + \div_{\Sigma_{ij},\mu} (X^\tang + Y^\tang)} \, d\mu^{n-1} \\
  \label{eq:app-weak-angles} & = \sum_{i<j} \div_{\Sigma_{ij},\mu} X^\tang  \, d\mu^{n-1} ,
 \end{align}
 concluding the proof of part~\ref{it:weak-angles}.

\end{proof}

\begin{proof}[Proof of Lemma \ref{lem:Lagrange}]
The formulas for $\delta_X V$ and $\delta_X A$ follow from (\ref{eq:first-order-V}), (\ref{eq:first-order-A}) and (\ref{eq:app-weak-angles}). 
In particular, it follows by (\ref{eq:inr-lambda}) that:
\begin{equation} \label{eq:app-formula-first-variation-of-area}
\delta_X A = \sum_{i<j} \int_{\Sigma_{ij}}   H_{ij,\mu} X^{\n_{ij}}  \, d\mu^{n-1} = \inr{\lambda}{\delta_X V} . 
\end{equation}
\end{proof}

\subsection{Stability}

Finally, we provide a proof of Lemma \ref{lem:unstable} on stability of a minimizing cluster. This is well-known and standard for compactly supported vector-fields in the single-bubble case, and was proved in \cite{DoubleBubbleInR3} assuming higher-order boundary regularity, which we avoid in this appendix.  

\begin{proof}[Proof of Lemma \ref{lem:unstable}]
    Let $Y_1, \dots, Y_{q-1}$, $F_{t,s}$, and $s(t)$ be as in the proof of Lemma~\ref{lem:first-order}. Recall that $\delta_X V = 0$ implies that $s'(0)=0$.
    By the chain rule, it follows (using $s'(0)=0$) that:
    \begin{align*}
                \left.\frac{d^2 A(t,s(t))}{(dt)^2} \right|_{t=0}
        & = \left.\frac{\partial^2 A}{(\partial t)^2}\right|_{t=0,s=0}  + \sum_{i=1}^{q-1} s_i''(0) \pdiffat{A}{s_i}{t=0,s=0} \\
        &= \delta_X^2 A + \sum_{i=1}^{q-1} s_i''(0) \delta_{Y_i} A \\
        &= \delta_X^2 A + \sum_{i=1}^{q-1} s_i''(0) \inr{\lambda}{\delta_{Y_i} V},
    \end{align*}
    where the last equality follows by (\ref{eq:app-formula-first-variation-of-area}). 
    Differentiating the relation $V(0,0) = V(t,s(t))$ twice in $t$ (and using again that $s'(0) = 0$),
    we obtain:
    \[
    0 = \left.\frac{\partial^2 V}{(\partial t)^2}\right|_{t=0,s=0} + \sum_i s_i''(0) \pdiffat{V}{s_i}{t=0,s=0} = \delta_X^2 V + \sum_i s_i''(0) \delta_{Y_i} V .
    \]
    Hence,
    \[
        \left.\frac{d^2 A(t,s(t))}{(dt)^2} \right|_{t=0}
        = \delta_X^2 A - \inr{\lambda}{\delta_X^2 V} = Q(X).
    \]
    It follows that necessarily $Q(X) \ge 0$, since otherwise, recalling that $\left.\frac{dA(t,s(t))}{dt} \right|_{t=0} = 0$ by Lemma~\ref{lem:first-order},
    we would find some $t \ne 0$ so that the cluster $\tilde \Omega = F_{t,s(t)}(\Omega)$ satisfies $\mu(\tilde \Omega) = \mu(\Omega)$
    and $\per_\mu(\tilde \Omega) = A(t,s(t)) < A(0,0) = \per_\mu(\Omega)$, contradicting the minimality of $\Omega$.
\end{proof}

\section{Second Variation for Translation Fields} \label{app:translations}

In this appendix, we calculate the second variations of volume and perimeter for constant (translation) fields. 
Recall that $\delta_w$ denotes the variation under the translation of the cluster generated by the constant vector-field $X \equiv w \in \R^n$.

\subsection{Second variation of volume}

\begin{lemma} \label{lem:delta2-V}
Let $\Omega$ be a stationary cluster with respect to $\mu = e^{-\pot} dx$ and with Lagrange multiplier $\lambda \in E^*$, and assume its cells are volume and perimeter regular. Then for any $i=1,\ldots,q$ and $w \in \R^n$:
\begin{equation} \label{eq:delta2-V-cell}
\delta_w^2 V(\Omega)_i = - \sum_{j \neq i}\int_{\Sigma_{ij}} \inr{w}{\n_{ij}} \nabla_w \pot \, d\mu^{n-1} .
\end{equation}
   In particular:
   \[
    \inr{\lambda}{\delta_w^2 V} = - \sum_{i < j} H_{ij,\mu} \int_{\Sigma_{ij}} \inr{w}{\n_{ij}} \nabla_w \pot \, d\mu^{n-1} .
    \]
\end{lemma}
\begin{proof}
Since $\Omega_i$ is volume regular, we have by Lemma \ref{lem:regular}:
\begin{align*}
  \delta_w^2 V(\Omega)_i
  &= \int_{\Omega_i} \diffIIat{e^{-\pot(x + tw)}}{t}{t}{t=0}\, dx \\
  &= \int_{\Omega_i} (\nabla_w \pot)^2 - \nabla^2_{w,w} \pot\, d\mu \\
  & = - \int_{\Omega_i} \div_{\mu} (w \nabla_w \pot) \, d\mu . 
\end{align*}
  In order to apply integration-by-parts, we need to first make $X = w \nabla_w \pot$ compactly supported. 
Noting that $\abs{X} \leq \abs{w}^2 \abs{\nabla \pot}$ and $\abs{\div X} = \abs{\nabla^2_{w,w} \pot} \leq \abs{w}^2 \norm{\nabla^2 \pot}$, we may apply Lemma \ref{lem:cutoff} to approximate $X$ by the compactly-supported $\eta_R X$. Integrating by parts using ~\eqref{eq:integration-by-parts}, we obtain:
\begin{align*}
  &= - \lim_{R \rightarrow \infty} \int_{\Omega_i} \div_\mu(\eta_R w \nabla_w \pot)\, d\mu \\
  &= -\lim_{R \rightarrow \infty} \int_{\partial^* \Omega_i} \eta_R \inr{w}{\n} \nabla_w \pot \, d\mu^{n-1} \\
  & = - \int_{\partial^* \Omega_i} \inr{w}{\n} \nabla_w \pot \, d\mu^{n-1}  ,
\end{align*}
where the last equality follows by Dominant Convergence since $\Omega_i$ is perimeter regular. Since $\H^{n-1}(\partial^* \Omega_i  \setminus \cup_{j \neq i} \Sigma_{ij}) = 0$, (\ref{eq:delta2-V-cell}) follows. 

Finally, since swapping $i$ and $j$ changes the sign of $\inr{w}{\n_{ij}}$, we have:
\begin{align*} 
    \inr{\lambda}{\delta_w^2 V}
        & = - \sum_{i=1}^q \lambda_i \sum_{j \ne i} \int_{\Sigma_{ij}} \inr{w}{\n_{ij}} \nabla_w \pot \, d\mu^{n-1} \\    
    &= - \sum_{i < j} (\lambda_i - \lambda_j)\int_{\Sigma_{ij}} \inr{w}{\n_{ij}} \nabla_w \pot \, d\mu^{n-1} \notag \\
    &= - \sum_{i < j} H_{ij,\mu} \int_{\Sigma_{ij}} \inr{w}{\n_{ij}} \nabla_w \pot \, d\mu^{n-1}.
\end{align*}
\end{proof}

\subsection{Second variation of perimeter}

\begin{lemma}\label{lem:delta2-A}
Let $\Omega$ be a stationary cluster with respect to $\mu = e^{-\pot} dx$, and assume that its cells are perimeter regular. Then for any $w \in \R^n$:
\[
\delta_w^2 A = -  \sum_{i < j} \int_{\Sigma_{ij}} \brac{ H_{ij,\mu} \scalar{w,\n_{ij}} \nabla_w \pot  + \scalar{w, \n_{ij}} \nabla^2_{\n_{ij},w} \pot } d\mu^{n-1} .
\]
\end{lemma}
\begin{proof}
Since $\Omega_i$ is perimeter regular, we have by Lemma \ref{lem:regular} that for every $i$:
\begin{align*}
  \delta_w^2 A(\Omega_i)
&  = \int_{\partial^* \Omega_i} \diffIIat{e^{-\pot(x + tw)}}{t}{t}{t=0}\, d\calH^{n-1}(x) \\
&  = \int_{\partial^* \Omega_i} \brac{(\nabla_w \pot)^2 - \nabla^2_{w,w} \pot} d\mu^{n-1} \\
&  = \sum_{j \neq i} \int_{\Sigma_{ij}} \brac{(\nabla_w \pot)^2 - \nabla^2_{w,w} \pot} d\mu^{n-1}  . 
\end{align*}
On the other hand, we now have:
\begin{equation}\label{eq:div-constant}
  \div_{\Sigma_{ij},\mu} (w \nabla_w \pot)
  = \div_{\Sigma_{ij}} (w \nabla_w \pot) - (\nabla_w \pot)^2
  = \nabla^2_{w^\tang,w} \pot - (\nabla_w \pot)^2,
\end{equation}
where $w^\tang$ is the tangential part of $w$. Consequently, we obtain:
\[
 \delta_w^2 A(\Omega_i) = - \sum_{j \neq i} \int_{\Sigma_{ij}} \brac{\div_{\Sigma_{ij},\mu} (w \nabla_w \pot) + \nabla^2_{w,w} \pot - \nabla^2_{w^\tang,w} \pot} d\mu^{n-1} . 
 \]
Summing over $i$ and dividing by $2$, we obtain:
\[
\delta_w^2 A = - \sum_{i < j} \int_{\Sigma_{ij}} \brac{\div_{\Sigma_{ij},\mu} (w \nabla_w \pot) + \scalar{w, \n_{ij}} \nabla^2_{\n_{ij},w} \pot  } d\mu^{n-1} . 
\]

We now claim that the contribution of the tangential part of divergence terms vanishes. This is essentially the content of Lemma~\ref{lem:first-order-conditions} part~\ref{it:weak-angles}, applied to the vector-field $X = w \nabla_w \pot$; however, $X$ is not bounded and so does not satisfy the required assumption (\ref{eq:field-bdd}). To remedy this, note as before that $\abs{X} \leq \abs{w}^2 \abs{\nabla \pot}$ and $\abs{\div_\Sigma X} = \abs{\nabla^2_{w^{\tang},w} \pot} \leq \abs{w}^2 \norm{\nabla^2 \pot}$, and so we may apply Lemma \ref{lem:cutoff} to approximate $X$ by the compactly-supported $\eta_R X$. Now applying Lemma~\ref{lem:first-order-conditions} part~\ref{it:weak-angles} to $\eta_R X$, we deduce:
\begin{align*}
\delta_w^2 A  & = - \lim_{R \rightarrow \infty} \sum_{i < j} \int_{\Sigma_{ij}} \brac{\div_{\Sigma_{ij},\mu} (\eta_R w \nabla_w \pot) + \scalar{w, \n_{ij}} \nabla^2_{\n_{ij},w} \pot } d\mu^{n-1} \\
& =  - \lim_{R \rightarrow \infty} \sum_{i < j} \int_{\Sigma_{ij}} \brac{\div_{\Sigma_{ij},\mu} (\eta_R \n_{ij} \scalar{w,\n_{ij}} \nabla_w \pot) + \scalar{w, \n_{ij}} \nabla^2_{\n_{ij},w} \pot } d\mu^{n-1} \\
& =  - \lim_{R \rightarrow \infty} \sum_{i < j} \int_{\Sigma_{ij}} \brac{ \eta_R H_{ij,\mu} \scalar{w,\n_{ij}} \nabla_w \pot  + \scalar{w, \n_{ij}} \nabla^2_{\n_{ij},w} \pot } d\mu^{n-1} \\
& = -  \sum_{i < j} \int_{\Sigma_{ij}} \brac{ H_{ij,\mu} \scalar{w,\n_{ij}} \nabla_w \pot  + \scalar{w, \n_{ij}} \nabla^2_{\n_{ij},w} \pot } d\mu^{n-1} ,
\end{align*}
where the last equality follows by Dominant Convergence since the cells $\Omega_i$ are perimeter regular. \end{proof}

\section{Calculation of First and Second Variations} \label{sec:calculation}

Appendix \ref{sec:calculation} is dedicated to establishing Lemmas~\ref{lem:formula} and~\ref{lem:polarization}. 
In the first two subsections below, let $U$ denote a Borel set of locally finite perimeter so that $\partial^* U = \Sigma \cup \Xi$, where $\Sigma$ is a smooth $(n-1)$-dimensional manifold co-oriented by the outer unit-normal $\n$ and $\H^{n-1}(\Xi) = 0$. 
Let $X$ be an admissible vector-field, and let $F_t$ denote the associated flow generated by $X$ via (\ref{eq:flow-gen-by-X}).

For a fixed point $p \in \Sigma$, let $\tang_1, \dots, \tang_{n-1}$ be normal coordinates for $\Sigma$. Then $\tang_1, \dots, \tang_{n-1}, \n$ is an orthonormal basis for $\R^n$ at $p$; we write $X^i$ for $\inr{\tang_i}{X}$ and $X^\n$ for $\inr{\n}{X}$. Moreover, we will write $X^\tang$ for the tangential component of $X$: $X^\tang = X^i \tang_i$. We freely employ Einstein summation convention, summing over repeated matching upper and lower indices. 

Recall that the second fundamental form is the symmetric bilinear form defined on $T \Sigma$ by $\II(Y, Z) = \inr{Y}{\nabla_Z \n}$. We write $\{\II_{ij}\}_{i,j=1}^{n-1}$ for its coordinates $\II_{ij} = \inr{\tang_i}{\nabla_j \n} = -\inr{\n}{\nabla_j \tang_i}$, where $\nabla_j$ means $\nabla_{\tang_j}$. The mean-curvature $H_\Sigma$ is defined as $\tr \II$. Given a measure $\mu = e^{-\pot} dx$ with $\pot \in C^\infty(\R^n)$, the weighted mean curvature $H_{\Sigma,\mu}$ is defined as $H_\Sigma - \nabla_\n \pot$. The surface unweighted and weighted divergences are defined, respectively, by 
\begin{align*}
  \div_\Sigma X &= \div X - \inr{\n}{\nabla_\n X} ~, \\
\div_{\Sigma,\mu} X &= \div_\Sigma X - \nabla_X \pot.
\end{align*}

\subsection{Second variation of perimeter}

Assuming $X$ is compactly supported or $U$ is perimeter regular, Lemma~\ref{lem:regular} implies that
\begin{multline}
  \label{eq:second-variation-of-area}
  \delta_X^2 A(U) = \int_{\partial^* U} \left . \frac{d^2}{(dt)^2} \right |_{t=0} (J \Phi_t e^{-\pot \circ F_t}) d\H^{n-1} \\
  = \int_{\Sigma} \Bigg[
  \left . \frac{d^2 J \Phi_t}{(dt)^2} \right |_{t=0}   - 2 \nabla_X \pot \left . \frac{d J \Phi_t }{dt} \right |_{t=0}
          + \left . \frac{d^2 e^{-\pot \circ F_t}}{(dt)^2} \right |_{t=0} e^{+\pot} 
  \Bigg]\, d\mu^{n-1},
\end{multline}
where, recall from Subsection \ref{subsec:variations}, $J \Phi_t$ denotes the (surface) Jacobian of $F_t|_{\Sigma}$. 
An immediate calculation yields:
\begin{equation}\label{eq:second-derivative-of-exponent}
  \diffIIat{e^{-\pot \circ F_t}}{t}{t}{t=0}  e^{+\pot} 
  = - \inr{\nabla_X X}{\nabla \pot}  + (\nabla_X \pot)^2  - \nabla^2 _{X,X}\pot .
\end{equation}
An explicit computation of the first and second derivatives of $J \Phi_t$ in $t$ (see e.g. ~\cite[formula (2.16)]{SternbergZumbrun}) reveals that: 
\begin{align*}
\left . \frac{d J \Phi_t }{dt} \right |_{t=0}& = \div_\Sigma X \\
    \left . \frac{d^2 J \Phi_t}{(dt)^2} \right |_{t=0}
  & = \div_\Sigma \nabla_X X + (\div_\Sigma X)^2 + \sum_{i=1}^{n-1} \inr{\n}{\nabla_i X}^2 - \sum_{i,j=1}^{n-1} \inr{\tang_i}{\nabla_j X} \inr{\tang_j}{\nabla_i X}.
\end{align*}
Plugging the above into~\eqref{eq:second-variation-of-area} yields
\begin{multline}
  \delta_X^2 A(U) = \\
  \int_{\Sigma} \Big[\div_{\Sigma} \nabla_X X + (\div_\Sigma X)^2 + \sum_{i=1}^{n-1} \inr{\n}{\nabla_i X}^2 - \sum_{i,j=1}^{n-1} \inr{\tang_i}{\nabla_j X} \inr{\tang_j}{\nabla_i X} \\
  \qquad  - 2 \nabla_X \pot \div_\Sigma X - \inr{\nabla_X X}{\nabla \pot} + (\nabla_X \pot)^2 - \nabla^2_{X,X} \pot \Big]\, d\mu^{n-1} .
  \label{eq:second-variation-of-area-1}
\end{multline}

\begin{lemma}\label{lem:div-nabla-nabla-div}
  \begin{multline*}
    \div_{\Sigma} (\nabla_{X^t} X) + (\div_\Sigma X)^2 + \sum_{i=1}^{n-1} \inr{\n}{\nabla_i X}^2 - \sum_{i,j=1}^{n-1} \inr{\tang_j}{\nabla_i X} \inr{\tang_i}{\nabla_j X} \\
    = \div_\Sigma (X \div_\Sigma X) + |(\nabla X^\n)^\tang|^2 - (X^\n)^2 \|\II\|_2^2  - \div_{\Sigma} (X^\n \nabla_{X^\tang} \n) + X^\n \nabla_{X^\tang} H_\Sigma.
  \end{multline*}
\end{lemma}

\begin{proof}
  In coordinates 
    \begin{align*}
    \nabla_{X^\tang} (\div_\Sigma X)
    &= \sum_{i,j=1}^{n-1} X^j \nabla_j \inr{\tang_i}{\nabla_i X} \\
    &= \sum_{i,j=1}^{n-1} X^j \inr{\tang_i}{\nabla_j \nabla_i X} + X^j \inr{\nabla_j \tang_i}{\nabla_i X} \\
    &= \sum_{i,j=1}^{n-1} X^j \inr{\tang_i}{\nabla_j \nabla_i X} - X^j \II_{ij} \inr{\n}{\nabla_i X}.
  \end{align*}
  The last term can be manipulated by writing
  \[
    X^j \II_{ij} = \inr{\nabla_i \n}{X} = \nabla_i X^\n - \inr{\n}{\nabla_i X},
  \]
  and hence
  \begin{equation}\label{eq:nabla-div}
    \nabla_{X^\tang} (\div_\Sigma X)
    = \sum_{i,j=1}^{n-1} X^j \inr{\tang_i}{\nabla_j \nabla_i X} + \inr{\n}{\nabla_i X}^2 - (\nabla_i X^\n)^2 + X^j \II_{ij} \nabla_i X^n.
  \end{equation}
  On the other hand,
  \begin{align*}
    \div_\Sigma (\nabla_{X^\tang} X)
    &= \sum_{i,j=1}^{n-1} \inr{\tang_i}{\nabla_i (X^j \nabla_j X)} \\
    &= \sum_{i,j=1}^{n-1} X^j \inr{\tang_i}{\nabla_i \nabla_j X} + \nabla_i X^j \inr{\tang_i}{\nabla_j X} \\
    &= \sum_{i,j=1}^{n-1} X^j \inr{\tang_i}{\nabla_i \nabla_j X} + \inr{\tang_j}{\nabla_i X} \inr{\tang_i}{\nabla_j X} - \II_{ij} X^\n \inr{\tang_i}{\nabla_j X}.
  \end{align*}
  Comparing this with~\eqref{eq:nabla-div}, we obtain (using that $\nabla_i \nabla_j X = \nabla_j \nabla_i X$ since $\nabla$ is the flat
  connection on $\R^n$ and $[\tang_i, \tang_j] = 0$):
      \begin{multline}\label{eq:nabla-div-div-nabla}
    \div_{\Sigma} (\nabla_{X^t} X) + \sum_{i=1}^{n-1} \inr{\n}{\nabla_i X}^2 - \sum_{i,j=1}^{n-1}\inr{\tang_j}{\nabla_i X} \inr{\tang_i}{\nabla_j X} \\
    = \nabla_{X^\tang} (\div_\Sigma X) + \sum_{i=1}^{n-1} (\nabla_i X^\n)^2 - \sum_{i,j=1}^{n-1} (X^j \II_{ij} \nabla_i X^\n + \II_{ij} X^\n \inr{\tang_i}{\nabla_j X}).
  \end{multline}
  We now concentrate on the last two terms above. First, note that
  \[
   \inr{\tang_i}{\nabla_j X} = \nabla_j \inr{\tang_i}{X} - \inr{\nabla_j \tang_i}{X} = \nabla_j X^i + X^\n \II_{ij} .
   \]
  In conjunction with the Codazzi equation $\sum_{i=1}^{n-1} \nabla_i \II_{ij} = \sum_{i=1}^{n-1} \nabla_j \II_{ii}$, we obtain:
  \begin{align*}
    & \sum_{i,j=1}^{n-1} \brac{X^j \II_{ij} \nabla_i X^\n + \II_{ij} X^\n \inr{\tang_i}{\nabla_j X}} \\
    &\qquad = (X^\n)^2 \|\II\|_2^2  + \sum_{i,j=1}^{n-1} \II_{ij} (X^j \nabla_i X^\n + X^\n \nabla_j X^i) \\
    &\qquad = (X^\n)^2 \|\II\|_2^2  + \sum_{i,j=1}^{n-1} \II_{ij} \nabla_i (X^\n X^j) \\
    &\qquad = (X^\n)^2 \|\II\|_2^2  + \sum_{i,j=1}^{n-1} \nabla_i (\II_{ij} X^\n X^j) - X^\n X^j \nabla_i \II_{ij} \\
    &\qquad = (X^\n)^2 \|\II\|_2^2  + \sum_{i=1}^{n-1} \nabla_i \inr{X^\n \nabla_{X^\tang} \n}{\tang_i} - X^\n \nabla_{X^\tang} \II_{ii} \\
    &\qquad = (X^\n)^2 \|\II\|_2^2  + \div_\Sigma (X^\n \nabla_{X^\tang} \n) - X^\n \nabla_{X^\tang} H_\Sigma,
  \end{align*}
  where in the last equality we used that $\nabla_{X^\tang} \n$ is tangential, and
  $\nabla_i \inr{Y}{\tang_i} = \inr{\nabla_i Y}{\tang_i}$ for tangential fields $Y$.
  Plugging this final equation into~\eqref{eq:nabla-div-div-nabla} yields 
  \begin{multline*} 
    \div_{\Sigma} (\nabla_{X^t} X) + \sum_{i=1}^{n-1} \inr{\n}{\nabla_i X}^2 - \sum_{i,j=1}^{n-1} \inr{\tang_j}{\nabla_i X} \inr{\tang_i}{\nabla_j X} \\
    = \nabla_{X^\tang} (\div_\Sigma X) + |\nabla^\tang X^\n|^2 - (X^\n)^2 \|\II\|_2^2  - \div_{\Sigma} (X^\n \nabla_{X^\tang} \n) + X^\n \nabla_{X^\tang} H_\Sigma.
  \end{multline*}
  Finally, add $(\div_\Sigma X)^2$ to both sides and note that
  \[
    \div_\Sigma (X \div_\Sigma X) = \nabla_{X^\tang} (\div_\Sigma X) + (\div_\Sigma X)^2.
  \]
\end{proof}

Going back to~\eqref{eq:second-variation-of-area-1} and plugging in Lemma~\ref{lem:div-nabla-nabla-div}, we obtain
\begin{multline*}
  \delta_X^2 A(U)
  = \int_{\Sigma} \Big[\div_{\Sigma} (X^\n \nabla_\n X) + \div_{\Sigma} (X \div_\Sigma X) - \div_\Sigma (X^\n \nabla_{X^\tang} \n) \\
    + |\nabla^\tang X^\n|^2 - (X^\n)^2 \|\II\|_2^2  + X^\n \nabla_{X^\tang} H_\Sigma \\
- \inr{\nabla_X X}{\nabla \pot} - 2 \nabla_X \pot \div_\Sigma X + (\nabla_X \pot)^2 - \nabla^2_{X,X} \pot\Big]\, d\mu^{n-1}.
\end{multline*}
Now, expanding the definition of weighted divergence, we observe that
\begin{align}
  &\div_{\Sigma,\mu} (X \div_{\Sigma,\mu} X) \label{eq:expanding-tangential-divergence} \\
  &\quad= \div_\Sigma (X \div_\Sigma X) - 2 \nabla_X \pot \div_\Sigma X + (\nabla_X \pot)^2 - \nabla_{X^\tang} \nabla_X \pot \notag \\
  &\quad= \div_\Sigma (X \div_\Sigma X) - 2 \nabla_X \pot \div_\Sigma X + (\nabla_X \pot)^2 - \nabla^2_{X^\tang,X} \pot - \inr{\nabla_{X^\tang} X}{\nabla \pot}. \notag
\end{align}
Hence,
\begin{multline*}
  \delta_X^2 A(U)
  = \int_{\Sigma} \Big[\div_{\Sigma} (X^\n \nabla_\n X) + \div_{\Sigma,\mu} (X \div_{\Sigma,\mu} X) - \div_\Sigma (X^\n \nabla_{X^\tang} \n) \\
    + |\nabla^\tang X^\n|^2 - (X^\n)^2 \|\II\|_2^2  + X^\n \nabla_{X^\tang} H_\Sigma \\
- X^\n \inr{\nabla_\n X}{\nabla \pot} - X^\n \nabla^2_{X^\tang,\n} \pot - (X^\n)^2 \nabla^2_{\n,\n} \pot\Big]\, d\mu^{n-1}.
\end{multline*}
Now, the first term of the first line combines with the first term of the third
line to make a weighted divergence, and $X^\n \nabla^2_{X^\tang,\n} \pot =
X^\n \nabla_{X^\tang} \nabla_\n \pot - X^\n \inr{\nabla \pot}{\nabla_{X^\tang} \n}$. The second
of these terms combines to make $\div_{\Sigma} (X^\n \nabla_{X^\tang} \n)$ a weighted
divergence, while the first one combines with $X^\n \nabla_{X^\tang} H_\Sigma$
to make the mean-curvature weighted. To summarize:

\begin{lemma}\label{lem:second-variation-of-perimeter}
For $U$ and $X$ as above, 
\begin{multline*}
  \delta_X^2 A(U)
  = \int_{\Sigma} \Big[\div_{\Sigma,\mu} (X^\n \nabla_\n X) + \div_{\Sigma,\mu} (X \div_{\Sigma,\mu} X) - \div_{\Sigma,\mu} (X^\n \nabla_{X^\tang} \n) \\
    + |\nabla^\tang X^\n|^2 - (X^\n)^2 \|\II\|_2^2  + X^\n \nabla_{X^\tang} H_{\Sigma,\mu} - (X^\n)^2 \nabla^2_{\n,\n} \pot\Big]\, d\mu^{n-1}.
\end{multline*}
\end{lemma}

\subsection{Second variation of volume}

Assuming $X$ is compactly supported or $U$ is volume and perimeter regular, Lemma~\ref{lem:regular} implies that:
\begin{multline*}
    \delta_X^2 V(U) = \int_{U} \left . \frac{d^2}{(dt)^2} \right |_{t=0} (J F_t e^{-\pot \circ F_t}) dx \\
  = \int_{U} \Bigg[
  \left . \frac{d^2 J F_t}{(dt)^2} \right |_{t=0}   - 2 \nabla_X \pot \left . \frac{d J F_t }{dt} \right |_{t=0}
          + \left . \frac{d^2 e^{-\pot \circ F_t}}{(dt)^2} \right |_{t=0} e^{+\pot} 
  \Bigg]\, d\mu,
\end{multline*}
where, recall from Subsection \ref{subsec:variations}, $J F_t$ denotes the Jacobian of $F_t$. Recalling~\eqref{eq:second-derivative-of-exponent}, and using (see e.g.~\cite[formula (2.13) and the derivation after (2.35)]{SternbergZumbrun}):
\[
\left . \frac{d J F_t }{dt} \right |_{t=0} = \div X ~,~ \left . \frac{d^2 J F_t}{(dt)^2} \right |_{t=0}  = \div(X \div X) ,
\]
we obtain:
\begin{align*}
    \delta_X^2 V(U)
    &= \begin{multlined}[t]
        \int_{U} \Big[ \div(X \div X) - 2 \nabla_X \pot \div X \\
            - \inr{\nabla_X X}{\nabla \pot} + (\nabla_X \pot)^2 - \nabla^2_{X,X} \pot \Big] \, d\mu
        \end{multlined}  \\
    &= \int_{U} \div_\mu(X \div_\mu X)\, d\mu ,
\end{align*}
where the last equality follows by expanding out the weighted divergences as in (\ref{eq:expanding-tangential-divergence}). 
We now claim that we may integrate-by-parts and proceed as follows:
\[
= \int_{\partial^* U} X^\n \div_\mu X\, d\mu^{n-1}  .
\]
Indeed, if $X$ is compactly supported, this holds by the Gauss--Green--De Giorgi theorem (\ref{eq:integration-by-parts}). For admissible $X$ and volume and perimeter regular $U$, this follows by employing in addition a truncation argument, exactly as the one employed in Appendix \ref{app:translations}.  
To summarize:
\begin{lemma} \label{lem:second-variation-of-measure}
For $U$ and $X$ as above:
\[
\delta_X^2 V(U) = \int_{\Sigma} X^\n \div_\mu X\, d\mu^{n-1}  .
\]
\end{lemma}

\subsection{Index form of a stationary cluster}

Let us apply the above calculations to the cells $\{\Omega_i\}$ of a stationary cluster $\Omega$, assuming that either $X$ is compactly supported, or that $X$ is admissible and $\{\Omega_i\}$ are volume and perimeter regular. Stationarity implies there is some $\lambda \in E^*$ such that $H_{ij,\mu} = \lambda_i - \lambda_j$ for all $i \ne j$. Define $\Psi_{ij} = \int_{\Sigma_{ij}} X^\n \div_\mu X\, d\mu^{n-1}$ and note that $\Psi_{ji} = - \Psi_{ij}$ and that $\delta_X^2 V(\Omega)_i = \sum_{j \neq i} \Psi_{ij}$ by Lemma \ref{lem:second-variation-of-measure} and (\ref{eq:nothing-lost}). It follows  that:
\begin{align}
    \inr{\lambda}{\delta_X^2 V(\Omega)} \notag
    &= \sum_i \lambda_i \sum_{j \ne i} \Psi_{ij} \notag \\
    &= \frac 12 \sum_{i \ne j} (\lambda_i - \lambda_j) \Psi_{ij} \notag \\
    &= \sum_{i < j} H_{ij,\mu} \int_{\Sigma_{ij}} X^\n \div_\mu X \, d\mu^{n-1} . \label{eq:second-variation-of-volume-for-one-component}
\end{align}
As for $\delta^2_X A(\Omega)$, stationarity implies that the term $\nabla_{X^\tang} H_{\Sigma_{ij},\mu}$
in Lemma~\ref{lem:second-variation-of-perimeter} vanishes. We are left with:
\begin{multline*}
  \delta_X^2 A(\Omega) - \inr{\lambda}{\delta_X^2 V(\Omega)} \\
  = \sum_{i < j} \int_{\Sigma_{ij}} \Big[
    \begin{aligned}[t]
        &\div_{\Sigma,\mu}(X^\n \nabla_\n X) + \div_{\Sigma,\mu} (X \div_{\Sigma,\mu} X) - \div_{\Sigma,\mu} (X^\n \nabla_{X^\tang} \n) \\
               & + |\nabla^\tang X^\n|^2 - (X^\n)^2 \|\II\|_2^2  - (X^\n)^2 \nabla^2_{\n,\n} \pot - H_{ij,\mu} X^\n \div_\mu X \Big]\, d\mu^{n-1},
    \end{aligned}
\end{multline*}
which is the formula claimed in Lemma~\ref{lem:formula}.

\subsection{Polarization with respect to constant vector-fields}

Finally, assume that $\Omega$ is a stationary cluster whose cells are volume and perimeter regular. Let $X$ be an admissible vector-field, and set $Y = X + w$, where $w \in \R^n$ is a constant vector-field. For brevity, denote $\delta^2_Z A = \delta^2_Z A(\Omega)$ and $\delta^2_Z V = \delta^2_Z V(\Omega)$ for $Z = Y,X,w$. 
    Since $\nabla w = 0$ and $\div_\Sigma w = 0$, we may plug $Y$ into~\eqref{eq:second-variation-of-area-1}
    to obtain
    \begin{multline*}
        \delta_Y^2 A = \\
      \sum_{i < j} \int_{\Sigma_{ij}} \Big[\div_{\Sigma} \nabla_Y X + (\div_\Sigma X)^2 + \sum_{i=1}^{n-1} \inr{\n}{\nabla_i X}^2 - \sum_{i,j=1}^{n-1} \inr{\tang_i}{\nabla_j X} \inr{\tang_j}{\nabla_i X} \\
      \qquad - \inr{\nabla_Y X}{\nabla \pot} - 2 \nabla_Y \pot \div_\Sigma X + (\nabla_Y \pot)^2 - \nabla^2_{Y,Y} \pot \Big]\, d\mu^{n-1} .
    \end{multline*}
    Expanding $Y = X + w$ in all the remaining places, we obtain:
    \begin{align*}
        & \delta_Y^2 A = \delta_X^2 A + \delta_w^2 A  + \sum_{i<j} \int_{\Sigma_{ij}} \Big [ \div_\Sigma \nabla_w X - \inr{\nabla_w X}{\nabla \pot} \\
        &  \hspace{120pt} - 2\nabla_w \pot \div_\Sigma X + 2\nabla_X \pot \nabla_w \pot - 2\nabla^2_{w,X} \pot \Big ] d\mu^{n-1} \\
        & = \delta_X^2 A + \delta_w^2 A  + \sum_{i<j} \int_{\Sigma_{ij}} \Big [  \div_{\Sigma,\mu} (\nabla_w X) -     2\div_{\Sigma,\mu} (\nabla_w W \cdot X) - 2 X^\n \nabla^2_{w,\n} W    \Big ] d\mu^{n-1} .
    \end{align*}
   For the second variation of volume,~\eqref{eq:second-variation-of-volume-for-one-component}
    implies that
    \begin{align*}
        \inr{\lambda}{\delta_Y^2 V}
        &= \sum_{i < j} H_{ij,\mu} \int_{\Sigma_{ij}} Y^\n \div_\mu Y\, d\mu^{n-1} \\
        &= \sum_{i < j} H_{ij,\mu} \int_{\Sigma_{ij}} (X^\n + w^\n) (\div_\mu X - \nabla_w W)\, d\mu^{n-1} \\
        &= \inr{\lambda}{\delta_X^2 V + \delta_w^2 V} + \sum_{i < j} H_{ij,\mu} \int_{\Sigma_{ij}} \brac{w^\n \div_\mu X - X^\n \nabla_w \pot} d\mu^{n-1}.
    \end{align*}
    The polarization formula for $Q(X+w) = \delta_Y^2 A - \inr{\lambda}{\delta_Y^2 V}$ asserted in Lemma~\ref{lem:polarization} then follows by subtracting the above expressions.

\section{Curvature Blow-Up Near Quadruple Points} \label{sec:Schauder}

Appendix \ref{sec:Schauder} is dedicated to providing a proof of Proposition \ref{prop:Schauder}. The latter is an easy consequence of the following statement:
\begin{proposition} \label{prop:B1}
Let $\Omega$ be a stationary regular $\mathsf{q}$-cluster, and let $i,j,k,l \in \{1,\ldots,\mathsf{q}\}$ be distinct. For any $q \in \Sigma_{ijkl}$ and open neighborhood $N_q$ of $q$, there exists an open sub-neighborhood $\tilde N_q \subset N_q$ of $q$ so that:
\begin{enumerate}[(i)]
\item \label{it:B1-1}
\begin{equation} \label{eq:Schauder-goal}
\exists C_q > 0 \;\;\; \forall p \in \Sigma_{ijk} \cap \tilde{N}_q \;\;\;  \norm{\II^{ij}(p) }\leq \frac{C_q}{d(p,\Sigma_{ijkl} \cap N_q)^{1-\alpha}}    .
\end{equation}
\item \label{it:B1-2}
\begin{equation} \label{eq:Schauder-goal2}
\exists C_q > 0 \;\;\; \forall w \in \Sigma_{ij} \cap \tilde{N}_q \;\;\;  \norm{\II^{ij}(w)} \leq \frac{C_q}{d(w,\Sigma_{ijkl} \cap N_q)^{1-\alpha}}   .
\end{equation}
\end{enumerate}
\end{proposition}

Of course statement (\ref{eq:Schauder-goal2}) is more general than (\ref{eq:Schauder-goal}), as $\Sigma_{ij}$ is $C^\infty$ smooth all the way up to $\Sigma_{ijk}$, but it will be more convenient to first establish (\ref{eq:Schauder-goal}) and then obtain (\ref{eq:Schauder-goal2}). Let us start by showing that Proposition \ref{prop:B1} implies Proposition \ref{prop:Schauder}.

\begin{proof}[Proof that Proposition \ref{prop:B1} implies Proposition \ref{prop:Schauder}]
As explained in the comments following the formulation of Proposition \ref{prop:Schauder}, it is enough by compactness to verify the asserted integrability in a small-enough open neighborhood of a quadruple point in $\Sigma^3$. By Theorem \ref{thm:regularity}, given $q \in \Sigma_{ijkl}$, there exists a small enough neighborhood $N_q$ of $q$ where $\Sigma$ is an embedded $C^{1,\alpha}$-diffeomorphic image of $\T \times \R^{n-3}$; namely, there exists a neighborhood $U_q$ of the origin in $E^{(3)} \times \R^{n-3}$ and a $C^{1,\alpha}$-diffeomorphism $\varphi : U_q \rightarrow N_q$ so that $\varphi(0) = q$, $\varphi((\T \times \R^{n-3}) \cap U_q) = \Sigma \cap N_q$ and so that $\snorm{\varphi}_{C^{1,\alpha}(U_q)}$ and $\snorm{\varphi^{-1}}_{C^{1,\alpha}(N_q)}$ are bounded (see e.g. \cite{SardAndInverseFunctionTheoremInHolderAndSobolevSpaces} for an inverse function theorem for $C^{1,\alpha}$ functions). 

Denote $\Sigma^m_{12} := \{ x \in E^{(3)} : x_1 = x_2 > x_3, x_4 \} \times \R^{n-3}$, $\Sigma^m_{123} := \{ x \in E^{(3)} : x_1 = x_2 = x_3 > x_4 \} \times \R^{n-3}$ and similarly for $\Sigma^m_{23}$ and $\Sigma^m_{31}$. 
We may suppose without loss of generality that $\varphi$ maps $\Sigma^m_{12} \cap U_q$ onto $\Sigma_{ij} \cap N_q$, $\Sigma^m_{23} \cap U_q$ onto $\Sigma_{jk} \cap N_q$, and $\Sigma^m_{31} \cap U_q$ onto $\Sigma_{ki} \cap N_q$ (and hence $\Sigma^m_{123} \cap U_q$ onto $\Sigma_{ijk} \cap N_q$). 

We now apply Proposition \ref{prop:B1}. By part \ref{it:B1-1}, there exists a sub-neighborhood $\tilde N_q \subset N_q$ of $q$ satisfying (\ref{eq:Schauder-goal}), and we denote $\tilde U_q = \varphi^{-1}(\tilde N_q)$. To see that $\int_{\Sigma_{ijk} \cap \tilde N_q} \norm{\II^{ij}} d\mu^{n-2} < \infty$, observe that
\begin{align*}
\int_{\Sigma_{ijk} \cap \tilde{N}_q} \norm{\II^{ij}(p)} d\mu^{n-2}(p) & \leq M_q \int_{\Sigma^m_{123} \cap \tilde{U}_q}  \norm{\II^{ij}}(\varphi(z)) J(z) d\H^{n-2}(z) \\
& \leq M_q J_q \int_{\Sigma^m_{123} \cap \tilde{U}_q}  \frac{C_q}{d(\varphi(z) , \Sigma_{ijkl} \cap N_q)^{1-\alpha}}  d\H^{n-2}(z) ,
\end{align*} 
where $M_q$ is an upper bound on the density of $\mu$ on $\tilde{N}_q$, $J(z)$ is the Jacobian at $z$ of $\varphi$ restricted as a map from $\Sigma^m_{123} \cap \tilde{U}_q$ to $\Sigma_{ijk} \cap \tilde{N}_q$ and $J_q$ is an upper bound on the latter Jacobian on $\tilde{U}_q$ (as it only depends on first-order derivatives of $\varphi$). Since $\abs{\varphi(z_1) - \varphi(z_2)}$ and $\abs{z_1 - z_2}$ are equivalent on $U_q$ up to a factor $D_q$, we deduce that:
\[
\int_{\Sigma_{ijk} \cap \tilde{N}_q} \norm{\II^{ij}}(p) d\mu^{n-2}(p) \leq M_q J_q D_q^{1-\alpha} C_q \int_{\Sigma^m_{123} \cap \tilde{U}_q}  \frac{1}{d(z,(\set{0} \times \R^{n-3}) \cap U_q)^{1-\alpha}} d\H^{n-2}(z)  .
\]
But the latter integral is immediately seen to be finite after applying Fubini's theorem, since $\overline{\Sigma^m_{123}} = (1,1,1,-3) \Real_+ \times \R^{n-3}$ and the function $r \mapsto 1/r^{1-\alpha}$ is integrable at $0 \in \R_+$.

The proof that  $\int_{\Sigma_{ij} \cap \tilde N_q} \norm{\II^{ij}}^2 d\mu^{n-1} < \infty$ follows similarly from Proposition \ref{prop:B1} part \ref{it:B1-2}. Indeed, after pulling back via $\varphi$ to the model cluster, we obtain:
\begin{multline*}
\int_{\Sigma_{ij} \cap \tilde{N}_q} \norm{\II^{ij}(w)}^2 d\mu^{n-1}(w) \leq \\
M_q J_q D_q^{2(1-\alpha)} C_q^2 \int_{\Sigma^m_{12} \cap \tilde{U}_q}  \frac{1}{d(z,(\set{0} \times \R^{n-3}) \cap U_q)^{2(1-\alpha)}} d\H^{n-1}(z)  .
\end{multline*}
To see that the integral on the right-hand-side is finite, note that:
\[
 \overline{\Sigma^m_{12}} = \brac{(1,1,1,-3) \Real_+ + (1,1,-3,1) \Real_+} \times \R^{n-3},
\]
and so an application of Fubini's theorem and integration in polar coordinates on the two dimensional sector above, which incurs an additional $r$ factor from the Jacobian, boils things down to the integrability of the function $r \mapsto r / r^{2(1-\alpha)}$ at $0 \in \R_+$. 
\end{proof}

\subsection{Schauder estimates}

For the proof of Proposition \ref{prop:B1} we will require the following weighted Schauder estimate for systems of linear elliptic and coercive PDEs. As explained in \cite{KNS}, this is a generalization of \cite[Theorem 9.1]{ADN1} for systems, modeled on the generalization of \cite[Theorem 7.2]{ADN1} to systems in \cite[Theorem 9.2]{ADN2}. We will only formulate the version which we will require here. 

\smallskip

Let $\mathring{\Omega}_R$ denote an open hemisphere of radius $R \in (0,1]$ around the origin in $\R^n$, let $\Sigma_R$ denote the open $(n-1)$-dimensional disc which constitutes its flat boundary, and set $\Omega_R := \mathring{\Omega}_R \cup \Sigma_R$. We will use the following notation for a function $f$ defined on $\Omega_R$ or $\Sigma_R$, a non-negative integer $l$ and $\alpha \in (0,1)$:
\begin{align*}
[f]_l & := \max_{\abs{\mu} = l} \norm{D^\mu f}_\infty ~,~  [f]_{l+\alpha}  := \max_{\abs{\mu} = l} \norm{\frac{\abs{D^\mu f(x) - D^\mu f(y)}}{\abs{x-y}^{\alpha}}}_\infty , \\
|f|_l & := \sum_{i=0}^{l} [f]_i ~,~ |f|_{l+\alpha}  := |f|_{l} + [f]_{l + \alpha} ,
\end{align*}
where $D^{\mu} = \partial_{x_1}^{\mu_1} \ldots \partial_{x_n}^{\mu_n}$ denotes partial differentiation with respect to the multi-index $\mu$, and $\abs{\mu}$ denotes its total degree $\mu_1 + \ldots + \mu_n$. We will also require the following weighted versions of the above norms and semi-norms. Let $d_S(x)$ denote the distance of a point $x \in \Omega_R$ from the spherical part of $\partial \Omega_R$. Given a homogeneity excess parameter $p \geq 0$, denote:
\begin{align*}
& \widetilde{[f]}_{p,l} := \max_{\abs{\mu} = l} \norm{d_S^{p+l} D^\mu f}_\infty , \\
& \widetilde{[f]}_{p,l+\alpha} := \max_{\abs{\mu} = l} \; \sup \set{ d_S^{p+l+\alpha}(x) \frac{\abs{D^\mu f(x) - D^\mu f(y)}}{\abs{x-y}^\alpha} \; ; \;  4 \abs{x-y} < \min(d_S(x),d_S(y)) } , \\
& \widetilde{\abs{f}}_{p,l} := \sum_{i=0}^{l} \widetilde{[f]}_{p,i}  ~,~ \tabs{f}_{p,l+\alpha} := \tabs{f}_{p,l} + \tbra{f}_{p,l + \alpha} .
\end{align*}
Finally, set:
\[
\widetilde{[f]}_{a} := \widetilde{[f]}_{0,a} ~,~ \widetilde{\abs{f}}_{a} := \widetilde{\abs{f}}_{0,a}  .
\]

\begin{theorem}[Agmon--Douglis--Nirenberg] \label{thm:ADN}
Let $\alpha \in (0,1)$. Fix integers $N,M \geq 1$, let $i,j=1,\ldots,N$, $k=1,\ldots, M$, and denote the following linear differential operators on $x \in \Omega_R$ and $y \in \Sigma_R$, respectively:
\[
L_{ij} := \sum_{\abs{\beta} \in \set{0,1} , \abs{\mu} \in \set{0,1}} D^\beta a_{ij\beta\mu}(x) D^{\mu} ~,~  B_{kj}  := \sum_{\abs{\mu} \in \set{0,1}} b_{kj\mu}(y) D^{\mu} ,
\]
where $a_{ij\beta\mu} \in C^{0,\alpha}(\Omega_R)$ and $b_{kj\mu} \in C^{0,\alpha}(\Sigma_R)$ satisfy $|a_{ij\beta\mu}|_{\alpha} , |b_{kj\mu}|_{\alpha} \leq K$ for some parameter $K > 0$. Here and below differentiation is understood in the distributional sense. Given $F_{i\beta} \in C^{0,\alpha}(\Omega_R)$ and $\Phi_k \in C^{0,\alpha}(\Sigma_R)$, denote also:
\[
F_i := \sum_{\abs{\beta} \in \set{0,1}} D^\beta F_{i\beta} .
\]
Assume that $\set{f^j} \subset C^{1,\alpha}(\Omega_R)$ satisfy the following system of equations in the distributional sense:
\[
\forall i \;\;\; \sum_{j} L_{ij} f^j = F_i  \text{ on } \Omega_R  ~,~ \forall k \;\;\; \sum_{j} B_{kj} f^j = \Phi_k \text{ on } \Sigma_R .
\]
Finally, assume that the above system is uniformly elliptic and coercive on $\Omega_R$. Then:
\[
\forall j \;\;\;  \tabs{f^j}_{1+\alpha} \leq C \brac{ \sum_{i} \sum_{\abs{\beta} \in \set{0,1}} \tabs{F_{i\beta}}_{2-\abs{\beta},\alpha} + \sum_{k} \tabs{\Phi_k}_{1,\alpha} + \sum_{j} \tabs{f^j}_0 } .
\]
where the constant $C$ depends only on $n,N,M,K,\alpha$ and the ellipticity and coercitivity constants. 
\end{theorem}

We refer to \cite{ADN2} for the definitions of ellipticity and coercitivity (the latter is a certain compatibility requirement for the boundary conditions, a generalization of an oblique derivative condition). We will use Theorem \ref{thm:ADN} in the following form:

\begin{corollary} \label{cor:ADN}
With the same notation and assumptions of Theorem \ref{thm:ADN}, assume in addition that $F_{i\beta} \equiv 0$ and $\Phi_{k} \equiv 0$ for all $i,k,\beta$. Then:
\[
\forall j \;\;\;  \tbra{f^j}_{1} \leq C \brac{ \sum_{i,j} \tabs{a_{ij00}}_{2,\alpha} \abs{f^j(0)} + \sum_{k,j} \tabs{b_{kj0}}_{1,\alpha} \abs{f^j(0)} + \sum_{j} \tabs{f^j - f^j(0)}_0 } .
\]
\end{corollary}
\begin{proof}
Immediate after applying Theorem \ref{thm:ADN} to $g^j = f^j - f^j(0)$ which satisfies:
\begin{align*}
\forall i \;\;\; \sum_{j} L_{ij} g^j = -\sum_j a_{ij00}(x) f^j(0)  \text{ on } \Omega_R  , \\
\forall k \;\;\; \sum_{j} B_{kj} g^j = - \sum_j b_{kj0}(y) f^j(0) \text{ on } \Sigma_R ,
\end{align*}
and noting that $\tabs{g^j}_{1+\alpha} \geq \tbra{g^j}_1 = \tbra{f^j}_1$. 
\end{proof}

We are now ready to prove Proposition \ref{prop:B1}. 
When referring to the topology of a set $U$ (and in particular, its topological boundary $\partial U$), we always employ the relative topology in $U$'s affine hull.

\subsection{Blow-up on $\Sigma_{ijk}$}

To establish (\ref{eq:Schauder-goal}), first note that by replacing the neighborhood $N_q$ by a smaller one if necessary, we may always assume that $\Sigma \cap N_q$ is an embedded local $C^{1,\alpha}$-diffeomorphic image of $\T \times \R^{n-3}$ as in the proof that Proposition \ref{prop:B1} implies Proposition \ref{prop:Schauder}, since this only makes the estimate (\ref{eq:Schauder-goal}) stronger. We continue with the notation introduced in the latter proof and proceed as in \cite{KNS}. Let $H$ denote the tangent plane to $\overline{\Sigma_{ki}}$ at $q$ (which is well-defined by continuity and the fact that $\varphi$ is a $C^1$-diffeomorphism).  By again shrinking $N_q$ if necessary, we may represent $\overline{\Sigma_{ij}}$, $\overline{\Sigma_{jk}}$ and $\overline{\Sigma_{ki}}$ in $N_q$ as graphs of $C^{1,\alpha}$ functions, denoted $u^1,u^2,u^3$ respectively, over their orthogonal projections onto $H$, denoted $\Omega^1$, $\Omega^2$ and $\Omega^3$ respectively. Indeed, this is immediate by the implicit function theorem 
for $\Sigma_{ki}$, but also follows for $\Sigma_{ij}$ and $\Sigma_{jk}$ as they form $120^{\circ}$ degree angles with $\Sigma_{ki}$ on $\Sigma_{ijk}$, and this extends by $C^1$ regularity of $\varphi$ to $q \in \Sigma_{ijkl}$; hence $\overline{\Sigma_{ij}}$ and $\overline{\Sigma_{jk}}$ form $60^{\circ}$ degree angles with $H$ at $q$, and the angle remains bounded away from $90^{\circ}$ by $C^1$ regularity at a small-enough neighborhood of $q$. 

Denote by $\Gamma$ the orthogonal projection $P_H$ of $\overline{\Sigma_{ijk}} \cap N_q$ onto $H$. 
Note that $\Gamma$ is the common boundary of $\Omega^1$, $\Omega^2$ and $\Omega^3$ in $P_H N_q$, and that $\Omega^1$ and $\Omega^2$ lie on one side of $\Gamma$, whereas $\Omega^3$ lies on the other side. 
Also denote $\Gamma_0 := P_H (\Sigma_{ijkl} \cap N_q) \subset \Gamma$. 
As we shall essentially reproduce below, it was shown in \cite{KNS} that in fact $u^i$ are $C^\infty(\Omega^i \setminus \Gamma_0)$.  
For convenience, we choose a coordinate system so that $H$ is identified with $\R^{n-1}$,  $q$ is identified with the origin, $T_q \Gamma$ is identified with $T_0 \{ x_{1} = 0 \}$ with  $\partial_{x_{1}}$ pointing into $\Omega_1$ and $\Omega_2$ at the origin, and $T_q \Sigma_{ijkl}$ is identified with $T_0 \{ x_1 = x_2 = 0 \}$. 

To elucidate the picture, let us describe how this works for a model 4-cluster in $\R^3$: all the $\Sigma_{ij}$'s are two-dimensional sectors of angle $\cos^{-1}(-1/3) \simeq 109^{\circ}$, and after projecting these onto the plane tangent to (and spanned by) $\Sigma_{ki}$, which we identify with $\R^2$, clearly $\Omega_3$ is the same sector of angle $\cos^{-1}(-1/3)$, but $\Omega_1 = \Omega_2$ are sectors of angle $\cos^{-1}(-1/\sqrt{3}) \simeq 125^{\circ}$, with common boundary $\Gamma = \R_+ e_2$ and with $\Gamma_0 = \{0\}$.

Now assume, by relabeling indices if necessary, that $u_2 \geq u_1$ on $\Omega^1 \cap \Omega^2$.  Each height function $u^i$, $i=1,2,3$, satisfies the following constant weighted mean-curvature equation:
\begin{equation} \label{eq:system0}
\sum_{a=1}^{n-1} \partial_{x_a}\brac{ \frac{\partial_{x_a} u^i }{\sqrt{ 1 + \abs{\nabla u^i}^2}} } - \scalar{\nabla W(x,u^i(x)) , \frac{(\nabla u^i(x) , -1)}{\sqrt{1 + \abs{\nabla u^i(x)}^2}}} \equiv C^i  \text{ in } \interior \Omega^i  .
\end{equation}
The boundary conditions are obtained from the fact that $\Sigma_{ij}$, $\Sigma_{jk}$ and $\Sigma_{ki}$ meet at $120^{\circ}$ degrees on $\Sigma_{ijk}$ according to Corollary \ref{cor:boundary-normal-sum} (and this extends by $C^1$ regularity to $\Sigma_{ijkl}$), and hence:
\begin{equation} \label{eq:bc0}
\begin{cases}
\scalar{\nabla u^2 , \nabla u^3} + 1 - \cos(60^{\circ}) \sqrt{1 + \abs{\nabla u^2}^2} \sqrt{1 + \abs{\nabla u^3}^2} = 0 ~, \\
\scalar{\nabla u^1 , \nabla u^2} + 1 - \cos(120^{\circ}) \sqrt{1 + \abs{\nabla u^1}^2} \sqrt{1 + \abs{\nabla u^2}^2} = 0 ~, \\
u^1 = u^2 = u^3  ~,
\end{cases}
\text{ on } \Gamma . 
\end{equation}
Of course, this still leaves a vertical degree of freedom (adding a common constant to all $u^i$'s), which we may remove by setting $u^3(0) = 0$, but we do not require this here. 

Denoting $x' = (x_2,\ldots,x_{n-1})$, we now apply the following hodograph transform $T$ mapping $x \in \Omega^1 \cap \Omega^2$ to $y$, where:
\[
\begin{cases}
w(x) := u^2(x) - u^1(x) ~,\\
y = T(x) := (w(x),x')  ~.
\end{cases}
\]
The inverse $T^{-1}$ of this transform is given by:
\[
x = T^{-1}(y) = (\psi(y),y') ,
\]
where $\psi$ is determined by the inverse function theorem as solving $x_{1} = \psi(y)$. This is always well-defined, perhaps after replacing $N_q$ by a smaller neighborhood, since $\partial_{x_{1}} w(0) = 2 \tan(60^{\circ}) > 0$, and $C^1$ regularity ensures that $\partial_{x_{1}} w$ will remain positive on the entire $\Omega^1 \cap \Omega^2$ if $N_q$ is chosen sufficiently small. In particular $w(x) > 0$ for $x \in (\Omega^1 \cap \Omega^2) \setminus \Gamma$, and $\partial_{y_1} \psi > 0$ on $T(\Omega^1 \cap \Omega^2)$. 

Denote $V_{12} := T(\Omega_1 \cap \Omega_2)$, $\Xi := T(\Gamma)$ and $\Xi_0 := T(\Gamma_0)$, and observe that $\Xi \subset \set{y \in \R^{n-1} : y_1 = 0}$, $V_{12} \setminus \Xi\subset \set{y \in \R^{n-1} : y_1 > 0}$ and $T(0) = 0$. Also note that $V_{12}$ is ``convex in the first coordinate", i.e. if $(t_0,y'),(t_1,y') \in V_{12}$ then $((1-\lambda) t_0 + \lambda t_1 , y') \in V_{12}$ for all $\lambda \in [0,1]$, by continuity of $w$. Finally note that $\psi \in C^{1,\alpha}(V_{12}) \cap C^\infty(V_{12} \setminus \Xi_0)$ by the inverse function theorem.

 We now apply the reflection method of \cite{KNS}, and for $y \in V_{12}$, denote:
\[
S(y) := (\psi(y) - C y_1 , y') ~,~ C := \sup_{y \in V_0} \abs{\nabla \psi(y)} + 1 . 
\]
Clearly $S$ is a diffeomorphism of class $C^{1,\alpha}$ from $V_{12}$ to $S(V_{12})$ and of class $C^\infty$ from $V_{12} \setminus \Xi_0$ to $S(V_{12} \setminus \Xi_0)$, and hence by the inverse function theorem, so is its inverse on the corresponding domains. 
Obviously $S$ maps $\Xi$ onto $\Gamma$ and $\Xi_0$ onto $\Gamma_0$, but also note that it maps $V_{12} \setminus \Xi$ to the complement of $\Omega^1 \cap \Omega^2$, since otherwise we would have $z \in V_{12} \setminus \Xi$ and $y \in V_{12}$ with $(\psi(z) - C z_1 , z') = (\psi(y),y')$, which is impossible since $t \mapsto \psi(t,y')$ is increasing whereas $t \mapsto \psi(t,y') - C t$ is decreasing on $\{ t \geq 0 : (t,y') \in V_{12} \}$. Finally, we set $V := S^{-1}(\Omega_3) \subset V_{12}$, which by the preceding comments must contain an open neighborhood of $K$ in $\set{y \in \R^{n-1} : y_1 \geq 0}$ for any compact subset $K$ of $\relint(\Xi)$; we will return to the geometry of $V$ later. We can now define for $y \in V$:
\begin{align*}
\Phi^+(y) & := u^1(x) ~,~ x = T^{-1}(y) ~,~\\
\Phi^{-}(y) & := u^3(x) ~,~ x = S(y) ~,
\end{align*}
and note that:
\[
u^2(x) = w(x) + u^1(x) = y_1 + \Phi^+(y) ~,~ y = T(x) . 
\]
Also note that since $u^i \in C^{1,\alpha}(\Omega^i) \cap C^{\infty}(\Omega^i \setminus \Gamma_0)$ then $\Phi^+,\Phi^- \in C^{1,\alpha}(V) \cap C^\infty(V \setminus \Xi_0)$. 

As verified in \cite{KNS}, the hodograph and reflection maps transform our quasilinear system of PDEs in divergence form (\ref{eq:system0}), for functions $u^i$ defined on different domains $\Omega^i$, into a system of quasilinear PDEs, still in divergence form, for $\Phi^+,\Phi^-,\Psi$, defined on the same domain $V$, which is of the following form:
\begin{equation} \label{eq:system}
\begin{cases}
M_1(D^2 \Psi,D^2\Phi^+,D \Psi, D \Phi^+, \Phi^+) = 0 \\
M_2(D^2 \Psi,D^2\Phi^+,D \Psi, D \Phi^+, \Phi^+) = 0 \\
M_3(D^2 \Psi,D^2\Phi^-,D \Psi, D \Phi^-, \Phi^-) = 0 
\end{cases}
\text{ in $\interior V$} .
\end{equation}
Our boundary conditions (\ref{eq:bc0}) on $\Gamma$ are transformed into boundary conditions on $\Xi$ of the following form:
\begin{equation} \label{eq:bc}
\begin{cases}
N_1(D\Psi,D\Phi^+,D \Phi^-) = 0 \\
N_2(D\Psi,D\Phi^+) = 0 \\
N_3(\Phi^+,\Phi^-) = \Phi^+ - \Phi^- = 0 
\end{cases} 
\text{ on $\Xi$}. 
\end{equation}
Unfortunately, these boundary conditions are quadratic in the first-order leading terms, which prevents us from directly employing $C^{2,\alpha}$ Schauder estimates for elliptic linear systems \cite{ADN2} after treating the first and zeroth order terms as coefficients, which we do know are in $C^{0,\alpha}(V)$.  

Consequently, as in \cite{KNS}, we linearize (\ref{eq:system}) and (\ref{eq:bc}) by considering their first order perturbation in a given direction $\theta \in \R^{n-1}$, treating the lower order terms as $C^{0,\alpha}$ coefficients. As verified in \cite{KNS}, this results in a linear system of PDEs with boundary conditions which is elliptic and coercive at $y=0$ (the coercitivity ultimately follows from the fact that the surfaces meet in non-zero angles). Since the latter two reqruirements are open (cf. \cite{ADN1,ADN2}) and only depend on the $C^{0,\alpha}(V)$ coefficients, by replacing $N_q$ by a smaller neighborhood if necessary, we may always assume that the system is uniformly elliptic and coercive on the entire $V$. 

It is worth pointing out that the linearized system remains in divergence form with $C^{0,\alpha}$ coefficients -- to see why, recall that the system before linearization consisted of quasi-linear terms in divergence form:
\[
\sum_{\abs{\beta} \in \set{0,1} , \abs{\mu} \in \set{0,1}} D^\beta c_{i \beta \mu}(D f^1,D f^2, Df^3,f^1,f^2,f^3) D^{\mu} f^i ,
\]
with $c_{i \beta \mu}$ being smooth functions of their arguments and $\set{f^1,f^2,f^3} = \set{\Phi^+,\Phi^-,\Psi} \subset C^{1,\alpha}(V) \cap C^\infty(V \setminus \Xi_0)$. Taking partial derivative $\partial_\theta$ in the direction of $\theta \in S^{n-2}$, we obtain:
\begin{align*}
& \sum_{\abs{\beta} \in \set{0,1} , \abs{\mu} \in \set{0,1}} D^\beta \rbrac{c_{i \beta \mu}(D f^1,D f^2, Df^3,f^1,f^2,f^3)} D^{\mu} (\partial_\theta f^i) + \\
& \sum_{j=1}^3 \sum_{\abs{\beta} \in \set{0,1} , \abs{\delta} = 1} D^\beta \rbrac{\sum_{\abs{\mu} \in \set{0,1}} D^{\mu} f^i \cdot D_j^\delta c_{i \beta \mu}(D f^1,D f^2, Df^3,f^1,f^2,f^3)} D^{\delta} (\partial_\theta f^j) + \\
& \sum_{j=1}^3 \sum_{\abs{\beta} \in \set{0,1}} D^\beta \rbrac{\sum_{\abs{\mu} \in \set{0,1}} D^{\mu} f^i \cdot \partial_{j+3} c_{i \beta \mu}(D f^1,D f^2, Df^3,f^1,f^2,f^3)} \partial_\theta f^j ,
\end{align*}
where $D_j^\delta$ denotes partial differentiation of $c_{i \beta \mu}$ according to the multi-index $\delta$ of its $j$'th argument. Since $D f^i \in C^{0,\alpha}(V_0)$, we see that the system for $\set{\partial_\theta f^i}_{i=1,2,3}$ is still in divergence form with $C^{0,\alpha}(V)$ coefficients (delineated by rectangular brackets above). After linearization, the quadratic boundary conditions become linear in $D^\mu \partial_\theta f^i$ with $C^{0,\alpha}(\Xi)$ coefficients. Note that these coefficients (of both the system and the boundary conditions) are bounded in $C^{0,\alpha}$ norm on the entire $V$ and $\Sigma$, respectively, since all of the estimates depend only on our initial $C^{1,\alpha}$ estimates for $\varphi$ and $\varphi^{-1}$. We are thus in a position to apply the Schauder estimates of Corollary \ref{cor:ADN}. 

Let $p \in \Sigma_{ijk} \cap N_q$. Then $x_p := P_H p \in \Gamma \setminus \Gamma_0$ and $y_p := T(x_p) \in \Xi \setminus \Xi_0$.  Let $\Omega_R \subset \R^{n-1}$ denote a maximal hemisphere centered at $y_p$ of radius $R = R_p$ so that its flat part $\Sigma_R \subset \set{y_1 = 0}$ and $\Omega_R \subset V \setminus \Xi_0$. By the preceding discussion and since $\Xi_0$ is a relatively closed set in $P_H N_q$, necessarily $R > 0$. Applying Corollary \ref{cor:ADN} to our linearized system for $\{\partial_\theta f^1,\partial_\theta f^2,\partial_\theta f^3\} = \{\partial_\theta \Phi^+,\partial_\theta \Phi^-,\partial_\theta \psi \} \subset C^\infty(\Omega_R)$, and recalling the weighing used in the norms appearing in the resulting estimate, we deduce for all $j=1,2,3$:
\begin{align*}
& R \abs{D \partial_\theta f^j(y_p)}  \leq C_q  \left ( \sum_{i,j} R^{2+\alpha} \norm{a_{ij00}}_{C^{0,\alpha}(\Omega_R)} \abs{\partial_\theta f^j(y_p)} \right . \\
&\left .  + \sum_{k,j} R^{1+\alpha} \norm{b_{kj0}}_{C^{0,\alpha}(\Sigma_R)} \abs{\partial_\theta f^j(y_p)} + \sum_j \norm{\partial_\theta f^j - \partial_\theta f^j(y_p)}_{C^0(\Omega_R)} \right ) ,
\end{align*}
where $a_{ij00}$ and $b_{ij0}$ are the zeroth order coefficients of the linearized system and boundary conditions, and $C_q$ depends on quantities which are uniform for all $p \in \Sigma_{ijk} \cap N_q \setminus \{q\}$. Since $\norm{a_{ij00}}_{C^{0,\alpha}(V)} , \norm{b_{kj0}}_{C^{0,\alpha}(\Xi)} < \infty$, $\norm{\partial_\theta f^j}_{C^0(V)} < \infty$ and since we may assume $R \leq 1$ by choosing $N_q$ small enough, it follows that for another constant $C_q'$ which is independent of $p$ we have:
\begin{equation} \label{eq:Schauder-estimate}
R \abs{D \partial_\theta f^j(y_p)} \leq C_q' \brac{R^{1+\alpha} + \sum_j \norm{\partial_\theta f^j - \partial_\theta f^j(y_p)}_{C^0(\Omega_R)} }.
\end{equation}
Using the fact that $f^j \in C^{1,\alpha}(V)$, it follows that $\norm{\partial_\theta f^j - \partial_\theta f^j(y_p)}_{C^0(\Omega_R)} \leq c_q R^{\alpha}$ where $c_q$ is independent of $p$, and we conclude that:
\[
\forall j=1,2,3 \;\;\; \norm{D^2 f^j(y_p)} \leq \frac{C_q''}{R_p^{1-\alpha}} ,
\]
for some constant $C_q''$ independent of $p$. As a side note, we remark that Kinderlehrer--Nirenberg--Spruck deduce in \cite{KNS} that $f^j \in C^{2,\alpha}$ all the way up to the triple points, by considering the difference-quotients of $f^j$ (instead of working with $D^\theta f^j$ as we did above, since we already know the functions are smooth by \cite{KNS}), and using the resulting $C^{1,\alpha}$ Schauder estimate on the difference-quotient and the Arzel\`a--Ascoli theorem to pass to the limit, yielding $C^{2,\alpha}$ regularity. 

We now claim that:
\begin{equation} \label{eq:D2u-decay}
\forall j=1,2,3 \;\;\; \norm{D^2 u^j(x_p)} \leq \frac{C_q^{(3)}}{R_p^{1-\alpha}} ,
\end{equation}
where $C_q^{(3)}$ is independent of $p$. Indeed, by the chain-rule, if we denote $\bar T = T^{-1}$:
\[
D^2_{ab} T_c(x_p) = - D_b T_e(x_p) D^2_{a e} \bar T_d(y_p) D_d T_c(x_p) . 
\]
Recalling that $\bar T(y) = (\psi(y),y')$, we know that $\norm{D^2 \bar T(y_p)} \leq A_q/R_p^{1-\alpha}$; and recalling that $T(x) = (u_2(x)-u_1(x),x')$, since $u^i \in C^{1,\alpha}(\Omega^1 \cap \Omega^2)$, we see that $\norm{D T(x_p)} \leq B_q$, where $A_q,B_q$ are constants independent of $p$. Consequently, $\norm{D^2 T(x_p)} \leq C_q /R_p^{1-\alpha}$. Taking two derivatives of $u^1(x_p) = \Phi^+(T(x_p))$ and applying the chain-rule, the latter estimates on $\norm{D T(x_p)}$ and $\norm{D^2 T(x_p)}$ in conjunction with $\norm{D^2 \Phi^+(y_p)} \leq A_q / R_p^{1-\alpha}$ and $\norm{D \Phi^+(y_p)} \leq B_q$, readily establishes (\ref{eq:D2u-decay}) for $u^1$. As $u^2(x_p) = w(x_p) + u^1(x_p)$ and $w(x) = T_1(x)$, the previous estimates imply (\ref{eq:D2u-decay}) for $u^2$ as well; the case of $u^3$ follows similarly.

As $\II^{ij}(p)$ is computed quasilinearly from $D^2 u^1(x_p)$ (leading order) and $D u^1(x_p)$, it immediately follows that:
\[
\forall p \in \Sigma_{ijk} \cap N_q \;\;\; \norm{\II^{ij}(p)} \leq \frac{C_q^{(4)}}{R_p^{1-\alpha}} .
\]
It remains to establish that:
\begin{equation} \label{eq:Rp-goal}
\forall p \in \Sigma_{ijk} \cap \tilde{N_q} \;\;\;  R_p \geq c_q d(p,\Sigma_{ijkl} \cap N_q) ,
\end{equation}
for some constant $c_q > 0$ independent of $p$ and some sub-neighborhood $\tilde{N_q} \subset N_q$ of $q$. 
We will use $\tilde{N_q} = \varphi(\tilde{U_q})$, where:
\[
\tilde{U_q} := \set{z \in U_q : d(z,0) < d(z,\partial U_q)} .
\]

Indeed, observe that $R_p = \min(d(y_p,\Xi_0),d(y_p, \partial V \setminus \Xi))$, 
and so if we show that:
\begin{equation} \label{eq:Rp-goal2}
\forall y_p \in \Xi \setminus \Xi_0  \cap T(P_H(\tilde{N_q}))  \;\;\;  d(y_p, \partial V \setminus \Xi)) \geq c_q' d(y_p,\Xi_0) ,
\end{equation}
then (\ref{eq:Rp-goal}) will follow, since $d(y_p,\Xi_0)$ is equivalent to $d(p,\Sigma_{ijkl} \cap N_q)$ up to a factor of $D_q$ (as $T \circ P_H$ is a $C^{1,\alpha}$ diffeomorphism on $\Sigma_{ijk} \cap N_q$). To show (\ref{eq:Rp-goal2}), observe that this is indeed the case on our model cluster:
\[
\forall z \in \Sigma_{123}^m \;\;\; d(z , \partial \Sigma^m_{12} \setminus \Sigma^m_{123}) \geq d(z, \Sigma^m_{1234} ) ,
\]
(in fact, with equality above, since $\Sigma^m_{123}$ and $\Sigma^m_{123}$, the two boundary components of $\Sigma^m_{12}$, form an obtuse angle of $\cos^{-1}(1/3) \simeq 109^{\circ}$). Consequently, it is easy to see that:
\[
\forall z \in \Sigma_{123}^m \cap \tilde{U}_q \;\;\; d(z , \partial (\Sigma^m_{12} \cap U_q) \setminus \Sigma^m_{123}) \geq d(z, \Sigma^m_{1234} \cap U_q) .
\]
The same holds with $\Sigma^m_{12}$ replaced by $\Sigma^m_{23}$ and $\Sigma^m_{31}$. But since $P_H \circ \varphi$ is a $C^{1,\alpha}$ diffeomorphism from $\overline{\Sigma^m_{12}} \cap U_q$ to $\Omega^1$, from $\overline{\Sigma^m_{23}} \cap U_q$ to $\Omega^2$, and from $\overline{\Sigma^m_{31}} \cap U_q$ to $\Omega^3$, we have:
\[
\forall x \in \Gamma \setminus \Gamma_0 \cap P_H(\tilde{N}_q)\;\;\; \forall i=1,2,3 \;\;\;  d(x , \partial \Omega^i \setminus \Gamma) \geq c_q''' d(x, \Gamma_0) .
\]
Finally, since $V = T(\Omega^1 \cap \Omega^2) \cap S^{-1}(\Omega^3)$, and $T$ and $S^{-1}$ are $C^{1,\alpha}$ diffeomorphisms on their corresponding domains, (\ref{eq:Rp-goal2}) follows with an appropriate constant $c_q'$. This concludes the proof of (\ref{eq:Schauder-goal}).

\subsection{Blow-up on $\Sigma_{ij}$}

Let us now sketch the argument for establishing (\ref{eq:Schauder-goal2}). Let $N_q$ be a (possibly readjusted) neighborhood of $q$ as in the proof of (\ref{eq:Schauder-goal}) described above. 
We will use $\tilde N_q = \tilde N^1_q(c_q) \cap \tilde N^2_q$, where:
\[
\tilde N^1_q(c) := \varphi( \set{ z \in U_q : d(z,0) < c d(z,\partial U_q)})  ,
\]
and:
\[
\tilde N^2_q := \set{ w \in N_q : d(w,q) < d(w,\partial N_q) } . 
\]
The constant $c_q \leq 1$ above is chosen to ensure that if $w \in \tilde N^1_q(c_q) $ and $p \in \partial \Sigma_{12}$ realizes the distance $d(w,\partial \Sigma_{12})$, then necessarily $p \in \tilde N^1_q(1)$, which was exactly the requirement on $p$ which we used in the proof of (\ref{eq:Schauder-goal}). Indeed, it is always possible to choose such a constant $c_q$ since $\varphi$ is a $C^{1,\alpha}$ diffeomorphism from $U_q$ to $N_q$, and thus distances are preserved up to constants. 

We first establish (\ref{eq:Schauder-goal2}) for points $w \in \Sigma_{ij} \cap \tilde{N}_q$ so that:
\begin{equation} \label{eq:Schauder-case1}
d(w,\Sigma_{ijkl} \cap N_q) > A_q d(w,\partial (\Sigma_{ij} \cap N_q)) , 
\end{equation}
for some constant $A_q>1$ to be prescribed. Let $p \in \partial (\Sigma_{ij} \cap N_q)$ realize the distance on the right-hand-side above. Then necessarily $p \in \partial \Sigma_{ij}\cap N_q$, since $d(w,\partial N_q) > d(w,q) \geq d(w,\partial \Sigma_{ij})$ by our assumption that $w \in \tilde N^2_q$. Hence $p \in \overline{\Sigma_{ijk}} \cap N_q$ or $p \in \overline{\Sigma_{ijl}} \cap N_q$, and we assume without loss of generality it is the former case (otherwise exchange the index $k$ with the index $l$ in all of our previous arguments). 
Clearly $p \notin \Sigma_{ijkl}$ since otherwise we would have $d(w,\Sigma_{ijkl} \cap N_q) = d(w,p)$, in violation of (\ref{eq:Schauder-case1}) and the fact that $A_q > 1$. 
In addition, since $w \in \tilde N^1_q(c_q)$, we are guaranteed that $p \in \tilde N^1_q(1)$.  

Now let $y_p = T(P_H p) \in \Xi \setminus \Xi_0 \cap T(P_H(\tilde N^1_q(1)))$. By (\ref{eq:Rp-goal2}), we know that there exists a constant $c_q' \in (0,1]$ so that $\Omega_{R_{y_p}}(y_p) \subset V$ with $R_{y_p} \geq c_q' d(y_p,\Xi_0)$, where recall $\Omega_{R}(y)$ denotes the hemisphere of radius $R$ centered at $y$ with flat part in $\set{ y_1 = 0}$. Since $T$ and $S^{-1}$ are $C^{1,\alpha}$ diffeomorphisms from $\Omega^1 \cap \Omega^2$ and $\Omega^3 \cap S(V_{12})$ into $V$, respectively, 
there exists a constant $c_q'' > 0$ so that, denoting $R_{x_p} := c_q'' d(x_p, \Gamma_0)$, we have $B_{R_{x_p}}(x_p) \subset (\Omega^1 \cap \Omega^2) \cup \Omega^3$, where $B_R(x)$ is a ball of radius $R$ around $x$ in $H$. 

By (\ref{eq:Schauder-case1}), we know that:
\begin{equation} \label{eq:Schauder-dwp}
d(w,p) < \frac{1}{A_q} d(w,\Sigma_{ijkl} \cap N_q).
\end{equation}
Since $P_H$ is a $C^{1,\alpha}$ diffeomorphism from $\overline{\Sigma_{ij}} \cap N_q$ to $\Omega^1$, there exists a constant $C_q > 1$ so that $d(x_w,x_p) < \frac{1}{C_q A_q} d(x_w,\Gamma_0)$. By the triangle inequality, it follows that:
\begin{equation} \label{eq:Schauder-dist1}
d(x_w,x_p) \leq \frac{1}{C_q A_q - 1} d(x_p,\Gamma_0) . 
\end{equation}
Therefore, choosing $A_q$ large enough, we can ensure that $\frac{1}{C_q A_q - 1} \leq c_q''$, and conclude that $x_w \in B_{R_{x_p}}(x_p)$; in particular, $x_w \in \Omega^1 \cap \Omega^2 \cap S(V)$ and $y_w := T(x_w) \in V$. 

In addition, since $T$ is a $C^{1,\alpha}$ diffeomorphism from $\Omega^1 \cap \Omega^2$ into $V$, (\ref{eq:Schauder-dist1}) implies that by choosing $A_q$ large enough, we can also ensure that $d(y_w,y_p) \leq (c_q'/2) d(y_p,\Xi_0)$, which recall is at most $R_{y_p}/2$. Hence the distance of $y_w$ to the spherical part of the boundary of $\Omega_{R_{y_p}}(y_p)$ is at least $R_{y_p}/2$. Applying the Schauder estimate for systems as in (\ref{eq:Schauder-estimate}) and testing it at $y_w$, we conclude that: 
\[
\norm{D^2 \Phi^+(y_w)} \leq \frac{C_q^{(2)}}{R_{y_p}^{1-\alpha}} .
\]
Arguing as in (\ref{eq:D2u-decay}), we deduce that:
\[
\norm{D^2 u^1(x_w)} \leq \frac{C_q^{(3)}}{R_{y_p}^{1-\alpha}} ,
\]
which implies:
\begin{equation} \label{eq:Schauder-final1}
\norm{\II^{ij}(w)} \leq \frac{C_q^{(4)}}{R_{y_p}^{1-\alpha}} .
\end{equation}
Recall that $R_{y_p} \geq c_q d(p, \Sigma_{ijkl} \cap N_q)$ by (\ref{eq:Rp-goal}). On the other hand, (\ref{eq:Schauder-dwp}) implies that:
\[
 d(p, \Sigma_{ijkl} \cap N_q) \geq d(w,\Sigma_{ijkl} \cap N_q) - d(w,p) \geq \frac{A_q - 1}{A_q} d(w,\Sigma_{ijkl} \cap N_q) . 
 \]
Hence, we conclude that $R_{y_p} \geq c_q \frac{A_q - 1}{A_q} d(w,\Sigma_{ijkl} \cap N_q)$, which together with (\ref{eq:Schauder-final1}) concludes the proof of (\ref{eq:Schauder-goal2}) under the assumption that (\ref{eq:Schauder-case1}). 

The case when (\ref{eq:Schauder-case1}) does not hold is much simpler, and there is no need for systems of PDEs nor boundary conditions. 
Set $x_w = P_H w$ and let $B_{R_w}$ be a ball centered at $x_w$ of maximum radius $R_w > 0$ in $\Omega^1 \subset H$. Since $u^1$ is $C^{1,\alpha}$ uniformly on the entire $\Omega^1$, a standard text-book $C^{2,\alpha}$ Schauder estimate for the quasilinear elliptic equation (\ref{eq:system0}) satisfied by $u^1$ in $B_{R_w}$ immediately verifies that:
\[
\norm{D^2 u^1(x_w)} \leq \frac{C_q'}{R_w^{1-\alpha}} ,
\]
which implies as before that:
\[
\norm{\II^{ij}(w)} \leq \frac{C_q''}{R_w^{1-\alpha}} .
\]
Since $w$ violates (\ref{eq:Schauder-case1}), it remains to establish that $R_w \geq c_q d(w,\partial (\Sigma_{ij} \cap N_q))$. But this is immediate, since $P_H$ is a $C^{1,\alpha}$ diffeomorphism from $\overline{\Sigma_{ij}} \cap N_q$ to $\Omega^1$. This concludes the proof of (\ref{eq:Schauder-goal2}) and thus of Proposition \ref{prop:B1}.

\bibliographystyle{plain}

\def\cprime{$'$} \def\textasciitilde{$\sim$}

\end{document}